\documentclass{article}
\usepackage{preamble}

\usepackage[style=alphabetic, 
maxbibnames=5, backend=biber, 
doi=false, isbn=false, url=false]{biblatex}
\title{The equifibered approach to $\infty$-properads}
\author{Shaul Barkan and Jan Steinebrunner}
\date{\today}

\begin{document}

\maketitle

\begin{abstract}
    We define a notion of $\infty$-properads that generalizes $\infty$-operads by allowing operations with multiple outputs.
    Specializing to the case where each operation has a single output provides a simple new perspective on $\infty$-operads,
    but at the same time the extra generality allows for examples such as bordism categories.
    We also give an interpretation of our $\infty$-properads as Segal presheaves on a category of graphs by comparing them to the Segal $\infty$-properads of Hackney--Robertson--Yau.
    Combining these two approaches yields a flexible tool for doing higher algebra with operations that have multiple inputs and outputs.
    Crucially, this allows for a definition of algebras over an $\infty$-properad such that, for example, topological field theories are algebras over the bordism $\infty$-properad.
    
    The key ingredient to this paper is the notion of an \emph{equifibered} map between $\bbE_\infty$-monoids, which is a well-behaved generalization of free maps.
    We also use this to prove facts about free $\bbE_\infty$-monoids,
    for example that free $\bbE_\infty$-monoids are closed under pullbacks along arbitrary maps.
\end{abstract}

\tableofcontents

\section{Introduction}

\subsubsection{Historical context}
Properads are a generalization of operads in which operations can have multiple outputs as well as inputs, as illustrated in \cref{figure:gluing-operations} below.

They were introduced by Vallette \cite{Val07} to study Koszul duality for PROPs over a field of characteristic $0$
and subsequently Merkulov--Vallette \cite{MV09, MV09b} studied their deformation theory.
The associativity of composition in a properad $\calP$ ensures that there is a unique way to form the composite of any collection of operations in $\calP$ which label the vertices of a connected directed graph with no directed cycles.
Using this insight, Markl and Johnson--Yau define properads in terms of the combinatorics of such graphs \cites{Mar08, JohnsonYau}.
Batanin--Berger also give a definition of properads as algebras for a certain polynomial monad built from graphs \cite[\S 10.4]{BB17},
and Kaufmann--Ward (using the language of Feynman categories) note that properads may be described as algebras for a certain coloured operad \cite[\S 2.2.4]{KW17}.

An $\infty$-properad is a generalization of this concept, 
where the sets of operations are replaced by spaces of operations 
and the gluing maps are associative up to specified higher coherence data.
The first model for such homotopy coherent properads are the Segal $\infty$-properads defined by Hackney--Robertson--Yau \cite{HRY15} as certain presheaves on a category of graphs.
In this paper we introduce a simple, equivalent theory of $\infty$-properads,
which has the advantage of admitting a good notion of algebras and not relying on the combinatorics of graphs.

The approach we take is somewhat unusual:
rather than defining $\infty$-properads in terms of colours and operations, we will define $\infty$-properads in terms of the free PROPs they generate, i.e.~as symmetric monoidal $\infty$-categories satisfying certain freeness conditions. We then derive an interpretation of such symmetric monoidal $\infty$-categories in terms of spaces of operations with multiple inputs and outputs, equipped with a coherently defined composition operation.
To justify this approach, we shall also prove that our $\infty$-properads are equivalent to Segal $\infty$-properads \cite{HRY15}.

A benefit of this approach is that the reader is not assumed to be familiar with the definition of $1$-properads.
We will begin the introduction by explaining how \operads{} are viewed from the perspective of this paper. 
This will naturally lead to the definition of $\infty$-properads.
(The curious reader may jump to \cref{defn:intro-properad} on page \pageref{defn:intro-properad}.)

\subsubsection{Envelopes and PROPs}
For a coloured operad $\calO$ we let \hldef{$\Env(\calO)$} denote the PROP generated by $\calO$, which we also refer to as the \hldef{envelope of $\calO$}.
This is the symmetric monoidal category
whose objects are tuples of colours $(c_1, \dots, c_n)$ and 
where morphisms $(f,\{\alpha_i\}_{i=1}^m)\colon (c_1, \dots, c_n) \to (d_1, \dots, d_m)$ consist of a map $f:\{1,\dots,n\} \to \{1,\dots,m\}$ 
and an operation $\alpha_i$ of arity $f^{-1}(i)$ for each $i \in \{1,\dots,m\}$ (with suitable input and output colours).
This is illustrated in \cref{figure:morphism-in-Env}.
In \cite[\S 2.2.4]{HA} Lurie generalizes this and constructs an envelope functor
$
    \Env\colon \Op \longrightarrow \SM
$,
from the $\infty$-category of \operads{} to the $\infty$-category of symmetric monoidal $\infty$-categories.%
\footnote{We define symmetric monoidal $\infty$-categories as functors $\Fin_\ast \to \Cat$ satisfying the Segal condition.}
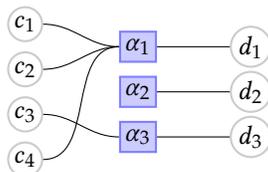
\begin{figure}[ht]
\centering
\begin{tikzpicture}[%
    col/.style={circle,thick, draw=black!20,
                 inner sep=1pt,minimum size=3mm},
   operation/.style={rectangle,draw=blue!50,fill=blue!20,thick,
                      inner sep=2pt,minimum size=4mm}]
    \def\xspace{1.5}
    \def\yspace{.6}
    \foreach \i in {1,...,4}
    {
        \node[col] (source-col\i) at (0,{(-\i)*\yspace})  {$c_\i$};
    }
    \foreach \i in {1,...,3}
    {
        \node[operation] (op\i) at (\xspace, {(-\i-.5)*\yspace}) {$\alpha_\i$};
        \node[col] (target-col\i) at (2*\xspace, {(-\i-.5)*\yspace})  {$d_\i$};
        \draw [-] (op\i) to [out = 0, in = 180] (target-col\i);
    }
    \draw [-] (source-col1) to [out = 0, in = 180] (op1);
    \draw [-] (source-col2) to [out = 0, in = 180] (op1);
    \draw [-] (source-col3) to [out = 0, in = 180] (op3);
    \draw [-] (source-col4) to [out = 0, in = 180] (op1);
\end{tikzpicture}
\caption{A morphism in the monoidal envelope of an operad.}
\label{figure:morphism-in-Env}
\end{figure}

The maximal subgroupoid of the envelope $\Env(\calO)^\simeq \subseteq \Env(\calO)$ is equivalent to the free symmetric monoidal $\infty$-groupoid on the $\infty$-groupoid whose objects are colours of $\calO$ and whose morphisms are invertible $1$-ary operations.
In the language of higher category theory a symmetric monoidal $\infty$-groupoid is the same as an $\bbE_\infty$-monoid in the $\infty$-category of spaces.
The \hldef{free $\bbE_\infty$-monoid} on a space $X$ is given by the formula $\hldef{\xF(X)} = \coprod_{n\ge 0} X^n_{h\Sigma_n}$,
and we say that an $\bbE_\infty$-monoid $M$ is free if there is a subspace $X \subset M$ such that the induced map $\xF(X) \to M$ of $\bbE_\infty$-monoids is an equivalence.
We call a map of free $\bbE_\infty$-monoids $f \colon \xF(X) \to \xF(Y)$ a free ($\bbE_\infty$-)map if $f \simeq \xF(g)$ for some map of spaces $g \colon X \to Y$.%
\footnote{Note that this is a sensible condition because $\xF \colon \Map_\calS(X,Y) \to \Map_\CMon(\xF(X),\xF(Y))$ induces an equivalence on the connected components it hits (see \cref{lem:free-is-a-property} and \cref{obs:elementary-subspace}).}
Motivated by this we propose the following reformulation of a definition of Haugseng--Kock \cite[Definition 2.4.9]{HK21}:
\begin{defnA}
    An \hldef{$\infty$-PROP} is a symmetric monoidal $\infty$-category $\calP$ such that $\calP^\simeq$ is a free $\bbE_\infty$-monoid.
    A morphism of $\infty$-PROPs is a symmetric monoidal functor $F:\calP \to \calQ$ such that $F:\calP^\simeq \to \calQ^\simeq$ is a free $\bbE_\infty$-monoid map.
    We let $\hldef{\PROP}\subset \SM$ denote the resulting (non-full) subcategory.
\end{defnA}

Haugseng-Kock show in \cite{HK21} that the envelope induces a fully faithful embedding:
\[
    \Env \colon \Op \hookrightarrow \PROP.
\]
In particular, this means that $\Op$ is equivalent to a (non-full) subcategory of $\SM$ and therefore the theory of \operads{} can (at least in principle) be developed entirely within the $\infty$-category $\SM$.
We offer two possible motivations for such a pursuit.
Firstly, \operads{} are often viewed through their algebras in symmetric monoidal $\infty$-categories,
hence it makes sense to put 
them on the same footing.
Secondly, the pleasant properties of equifibered maps established in \cref{sec:eqf} indicate that such a theory might be more elementary.
With this in mind we move on to describe the essential image of $\Env$.

\subsubsection{Characterizing the image of $\Env$}
As an example, consider the terminal \operad{} $\calO = \bbE_\infty$.
Its envelope is $\Env(\bbE_\infty) = \Fin$, the category of finite sets with disjoint union as its monoidal structure.
This symmetric monoidal category has the property that its tensor product $\amalg$ is ``disjoint'' in the following precise sense:
for any two finite sets $A, B \in \Fin$ the map
\[
    \amalg\colon \Fin^\simeq_{/A} \times \Fin^\simeq_{/B} \too \Fin^\simeq_{/A \amalg B}
\]
is an equivalence. 
Generalizing this we obtain a characterization of \operads{} within $\infty$-PROPs:
\begin{corA}[\ref{lem:monic-properad-tensor-disjunctive}]\label{corA:characterization-operads}
    An $\infty$-$\mrm{PROP}$ $\calP$ lies in the essential image of $\Env \colon \Op \hookrightarrow \PROP$ if and only if:
    \begin{itemize}
        \item[$(\star)$]\label{eq:tensor-disj}
        For every $x, y \in \calP$ the natural map $\otimes \colon \calP^\simeq_{/x} \times \calP_{/y}^\simeq \too \calP_{/x\otimes y}^\simeq$ is an equivalence.
    \end{itemize}
\end{corA}

In the $1$-categorical setting condition $(\star)$ resembles the hereditary condition that has frequently appeared in connection with operads in the literature \cite{BM08, getzler, KW17, BKW18}.
To see \cref{corA:characterization-operads} in action, consider the symmetric monoidal $\infty$-category $\Disk_d$, whose objects are $d$-manifolds of the form $J \times \bbR^d$ for some finite set $J$ and where the mapping spaces are spaces of smooth embeddings $J \times \bbR^d \hookrightarrow K \times \bbR^d$. 
The maximal subgroupoid $\Disk_d^\simeq$ is equivalent to $\xF(\BO_d)$ and hence $\Disk_d$ is an $\infty$-PROP.
To check $(\star)$ we rewrite
$(\Disk_d)^\simeq_{\!/J \times \bbR^d} \simeq \Conf_\bullet(J \times \bbR^d)$ as the space of unordered configurations in $J \times \bbR^d$ and observe that the map
\[
    \amalg\colon \Conf_\bullet(J \times \bbR^d) \times \Conf_\bullet(K \times \bbR^d) 
    \iso \Conf_\bullet((J \amalg K) \times \bbR^d)
\]
is indeed an equivalence.
We thus conclude that $\Disk_d$ is the envelope of an $\infty$-operad.
Indeed, it is the envelope of the \textit{framed} little $d$-disc operad $fD_d$. (See for example \cite{salvatore2003framed}.)

\subsubsection{Equifibered morphisms}
To better understand condition $(\star)$ consider the square of $\bbE_\infty$-monoids
\[
    \begin{tikzcd}
    	{\Ar(\calP)^\simeq \times \Ar(\calP)^\simeq} & \Ar(\calP)^\simeq \\
    	\calP^\simeq \times \calP^\simeq & \calP^\simeq.
    	\arrow["{\ev_1 \times \ev_1}"', from=1-1, to=2-1]
    	\arrow["{\otimes}", from=1-1, to=1-2]
    	\arrow["\ev_1", from=1-2, to=2-2]
    	\arrow["{\otimes}", from=2-1, to=2-2]
    	\arrow["\lrcorner"{anchor=center, pos=0.125}, draw=none, from=1-1, to=2-2]
    \end{tikzcd}
\]
Here $\Ar(\calP) \coloneq \Fun([1], \calP)$ is the arrow category with its pointwise symmetric monoidal structure and $\ev_1\colon \Ar(\calP) \to \calP$ is the functor $(f\colon x \to y) \mapsto y$.
Passing to vertical fibers at $(x,y) \in \calP^\simeq \times \calP^\simeq$ and $x \otimes y \in \calP^\simeq$ recovers the map $\otimes \colon \calP_{/x}^\simeq \times \calP_{/y}^\simeq \to \calP_{/x \otimes y}^\simeq$ from $(\star)$.
Hence, $\calP$ satisfies condition $(\star)$ if and only if the above square is cartesian.
We encapsulate this in the following definition, which is the driving force behind most of the results in this paper.
\begin{defnA}\label{defnA:equifibered}
    A morphism of $\bbE_\infty$-monoids $f\colon M \to N$ is called \hldef{equifibered} if the natural square
        \[\begin{tikzcd}
    	{M \times M} & M \\
    	N \times N & N
    	\arrow["{f \times f}"', from=1-1, to=2-1]
    	\arrow["{+}", from=1-1, to=1-2]
    	\arrow["f", from=1-2, to=2-2]
    	\arrow["{+}", from=2-1, to=2-2]
    	\arrow["\lrcorner"{anchor=center, pos=0.125}, draw=none, from=1-1, to=2-2]
        \end{tikzcd}\]
        is cartesian.
\end{defnA}

We think of equifibered maps as 
a generalization of free maps.
Indeed, we show a morphism of free $\bbE_\infty$-monoids $g \colon \xF(X) \to \xF(Y)$ is equifibered if and only if it is free.
Curiously, the assumption that the source is free can be removed:~an equifibered map $M \to \xF(Y)$ necessarily gives rise to an equivalence $\xF(Y \times_{\xF(Y)} M) \simeq M$.
In contrast to free maps, equifibered maps have excellent categorical properties, for example they form the right class of a factorization system.
One can also use the theory of equifibered maps to study free $\bbE_\infty$-monoids.
For instance, we prove the following surprising fact:
\begin{propA}[\ref{cor:retract-of-free} and \ref{cor:finite-limits}]
    Let $\Mon_{\bbE_\infty}(\calS)^{\rm free} \subset \Mon_{\bbE_\infty}(\calS)$ 
    denote the \textbf{full} subcategory on those $\bbE_\infty$-monoids that are free.
    Then $\Mon_{\bbE_\infty}(\calS)^{\rm free}$ is closed under finite limits and retracts.
\end{propA}

Returning to envelopes,  suppose $\calP \in \SM$ lies in the essential image of $\Env$. 
Using our newly acquired terminology, we may interpret \cref{corA:characterization-operads} as telling us that $\calP^\simeq$ is free and that the map $\ev_1\colon \Ar(\calP)^\simeq \to \calP^\simeq$ is equifibered.
In particular, it follows that $\Ar(\calP)^\simeq$ is also free as an $\bbE_\infty$-monoid.
Indeed, writing $\calP \simeq \Env(\calO)$ for $\calO \in \Op$ one checks that the $\infty$-groupoid of arrows $\Ar(\Env(\calO))^\simeq$ is freely generated by the space of operations of $\calO$
and $\ev_1\colon \Ar(\Env(\calO))^\simeq \to \Env(\calO)^\simeq$ is free on the map which assigns to each operation its target colour.
Note however that $\ev_0\colon \Ar(\Env(\calO))^\simeq \to \Env(\calO)^\simeq$ is \textit{not} free:
it sends an operation $\alpha \colon (c_1,\dots,c_k) \to d$ to the sum of its input-colours $\sum_{i=1}^k c_i \in \Env(\calO)^\simeq$.

\subsubsection{The nerve}
The $\bbE_\infty$-monoids $\calP^\simeq$ and $\Ar(\calP)^\simeq$ considered above
are the first two levels of the nerve $\xN_\bullet(\calP)$.
Recall that for $\calC \in \SM$ the $n$-th level of the nerve $\hldef{\xN_n(\calC)} \coloneq \Fun([n],\calC)^\simeq$ is naturally an $\bbE_\infty$-monoid
with respect to the point-wise tensor product.
We will think of the \hldef{nerve} as a functor
\[
    \hldef{\xN_\bullet}\colon \SM \hookrightarrow \Fun(\Dop, \Mon_{\bbE_\infty}(\calS)).
\]
In terms of this we can now say that $\calC$ is an $\infty$-PROP if $\xN_0\calC = \calC^\simeq$ is free,
and $\calC$ is in the image of $\Env$ if moreover $d_0\colon \xN_1\calC \to \xN_0\calC$ is equifibered.
In the latter case, the basic properties of equifibered maps imply that $\xN_n\calC$ is free for all $n$
and $d_i\colon \xN_n\calC \to \xN_{n-1}\calC$ is equifibered for all $0 \le i < n$.

\subsubsection{\texorpdfstring{$\infty$-properads}{ Infinity properads}}
Condition $(\star)$ in \cref{corA:characterization-operads} in particular enforces that each operation has a single output colour.
In order to generalize from single output to multiple outputs we must find a replacement for $(\star)$.
Our guiding example will be the bordism category $\Bord_d$.
This is the symmetric monoidal $(\infty,1)$-category where objects are closed $(d-1)$-manifolds and the morphism spaces are disjoint unions of $\BDiff_\partial(W)$ where $W\colon M \to N$ is a compact $d$-dimensional bordism.

The nerve $\xN_n(\Bord_d)$ has a geometric interpretation as a certain space of $d$-manifolds in $\bbR \times \bbR^\infty$ equipped with $n+1$ regular values for the first coordinate projection.%
\footnote{See \cref{ex:bord} for an explanation of why we do not have to worry about Rezk-completeness.}
This is an $\bbE_\infty$-monoid under disjoint union,
and as such it is freely generated by connected manifolds.
However, even though $\xN_\bullet(\Bord_d)$ is level-wise free, $\Bord_d$ is not in the essential image of $\Env$.
Indeed, the face map $d_0\colon \xN_1(\Bord_d) \to \xN_0(\Bord_d)$ is not free (as a connected bordism may have a disconnected outgoing boundary) and hence not equifibered.
Instead, $\Bord_d$ is an example of an $\infty$-properad.
\begin{defnA}\label{defn:intro-properad}
    An \hldef{$\infty$-properad} is a symmetric monoidal $\infty$-category $\calP$ such that:
    \begin{enumerate}
        \item $\xN_1(\calP) = \Ar(\calP)^\simeq$ is a free \Emon{} and 
        \item the face map $d_1\colon \xN_2(\calP) \to \xN_1(\calP)$ is equifibered.
    \end{enumerate}
    Let $\hldef{\Prpd} \subset \SM$ denote the (non-full) subcategory with objects $\infty$-properads and morphisms equifibered symmetric monoidal functors.
\end{defnA}

We will see that the first condition is equivalent to asking $\xN_n(\calP)$ to be free for all $n$ and the second condition is equivalent to asking $\lambda^*\colon \xN_m\calP \to \xN_n\calP$ to be equifibered for all active $\lambda\colon [n] \to [m] \in \simp$.
In particular, $\infty$-properads form a (non-full) subcategory of $\PROP$.
Recently, Kaufmann and Monaco \cite{KM22} defined a notion of ``hereditary unique factorization category'' (UFC), which looks like a $1$-categorical version of the above.
As discussed in \cref{rem:UFCs}, we believe that hereditary UFCs are exactly the $\infty$-properads that also happen to be $1$-categories, and are thus equivalent to $1$-properads that have no $(0,0)$-ary operations.

\subsubsection{Colours and operations}

    For an $\infty$-properad $\calP$ the $\bbE_\infty$-monoids 
    $\xN_0 \calP = \calP^\simeq$ and $\xN_1 \calP = \Ar(\calP)^\simeq$ 
    are freely generated by subspaces
    $\hldef{\mrm{col}(\calP)} \subset \calP^\simeq$ and
    $\hldef{\mrm{ops}(\calP)} \subset \Ar(\calP)^\simeq$,
    which we respectively refer to as the \hldef{space of colours} of $\calP$ and the \hldef{space of operations} of $\calP$.
    
    Given an operation $o$ in $\calP$, i.e.\ a morphism $o\colon x \to y \in \calP$ that is a generator in $\Ar(\calP)^\simeq$,
    its source and target can be written as tensor products of colours:
    \[
        o\colon x_1 \otimes \dots \otimes x_n \too y_1 \otimes \dots \otimes y_m.
    \]
    We say that such an operation is of \hldef{arity $(n,m)$}.
    We refer to the $x_i \in \mrm{col}(\calP)$ as the inputs and 
    to the $y_j \in \mrm{col}(\calP)$ as the outputs of $o$.
    These are unique up to reordering.
    The map that encodes the inputs and outputs of operations is
    \[
        \mrm{ops}(\calP) \subset \xN_1 \calP \xtoo{(s,t)} 
        \xN_0 \calP \times \xN_0 \calP \simeq \xF(\mrm{col}(\calP)) \times \xF(\mrm{col}(\calP)).
    \]
    We may sometimes write \hldef{$\calP(x_1,\dots,x_n; y_1,\dots,y_m)$}
    for the fiber of this map at the point given by the objects $(x,y) \in \xN_0\calP \times \xN_0\calP$.
    Note that this is a union of connected components of $\Map_\calP(x,y)$,
    and a general morphism in $\calP$ may be decomposed as a monoidal product of such operations as illustrated in \cref{figure:decomposition-of-morphism}.
    \begin{figure}[ht]
    \centering
    \begin{tikzpicture}[%
    col/.style={circle,thick, draw=black!20,
                 inner sep=1pt,minimum size=3mm},
   operation/.style={rectangle,draw=blue!50,fill=blue!20,thick,
                      inner sep=2pt,minimum size=4mm}]
    \def\xspace{1.5}
    \def\yspace{.6}
    \def\hoffset{-3*\xspace}
    
    \foreach \i in {1,...,4}
    {
        \node[col] (p-source-col\i) at (\hoffset,{(-\i)*\yspace})  {$c_\i$};
    }
    \foreach \i in {1,...,3}
    {
        \node[operation] (prop\i) at (\hoffset+\xspace, {(-\i-.5)*\yspace}) {$\alpha_\i$};
    }
    \foreach \i in {1,...,4}
    {
        \node[col] (p-target-col\i) at (\hoffset+2*\xspace,{(-\i)*\yspace})  {$d_\i$};
    }
    \draw [-] (p-source-col1) to [out = 0, in = 180] (prop2);
    \draw [-] (p-source-col2) to [out = 0, in = 180] (prop2);
    \draw [-] (p-source-col3) to [out = 0, in = 180] (prop2);
    \draw [-] (p-source-col4) to [out = 0, in = 180] (prop3);
    \draw [-] (prop1) to [out = 0, in = 180] (p-target-col1);
    \draw [-] (prop1) to [out = 0, in = 180] (p-target-col3);
    \draw [-] (prop3) to [out = 0, in = 180] (p-target-col2);
    \draw [-] (prop3) to [out = 0, in = 180] (p-target-col4);
    
    \def\xspace{1}
    \def\vsep{.5}
    
    \node at (-.25, -2.5*\yspace) {$\simeq$};
    \node[operation] (a1) at (\xspace, {-2.5*\yspace}) {$\alpha_1$};
    \node[col] (d1) at (2*\xspace,{-2*\yspace})  {$d_1$};
    \node[col] (d3) at (2*\xspace,{-3*\yspace})  {$d_3$};
    \draw [-] (a1) to [out = 0, in = 180] (d1);
    \draw [-] (a1) to [out = 0, in = 180] (d3);
    
    \node (ot1) at (2.5*\xspace+.5*\vsep, {-2.5*\yspace}) {$\otimes$};
    
    \node[operation] (a2) at (4*\xspace+\vsep, {-2.5*\yspace}) {$\alpha_2$};
    \foreach \i in {1,...,3}
    {
        \node[col] (c\i) at (3*\xspace+\vsep,{-(\i+.5)*\yspace})  {$c_\i$};
        \draw [-] (c\i) to [out = 0, in = 180] (a2);
    }
    
    \node (ot2) at (4.5*\xspace+1.5*\vsep, {-2.5*\yspace}) {$\otimes$};
    
    \node[operation] (a3) at (6*\xspace+2*\vsep, {-2.5*\yspace}) {$\alpha_3$};
    \node[col] (c4) at (5*\xspace+2*\vsep,{-2.5*\yspace})  {$c_4$};
    \node[col] (d2) at (7*\xspace+2*\vsep,{-2*\yspace})  {$d_2$};
    \node[col] (d4) at (7*\xspace+2*\vsep,{-3*\yspace})  {$d_4$};
    \draw [-] (c4) to [out = 0, in = 180] (a3);
    \draw [-] (a3) to [out = 0, in = 180] (d2);
    \draw [-] (a3) to [out = 0, in = 180] (d4);
    
\end{tikzpicture}
    \caption{A morphism in a properad decomposes into a product of operations.}
    \label{figure:decomposition-of-morphism}
    \end{figure}
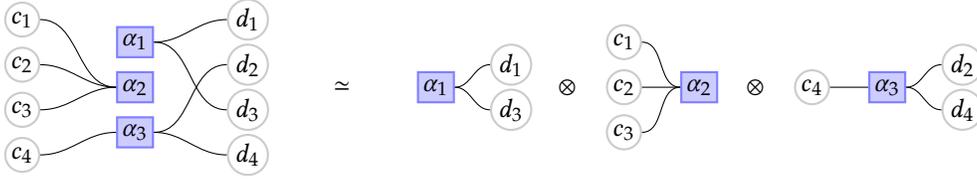
    Given an operation $o$ as above, another operation $p \in \calP(z_1,\dots,z_l; w_1,\dots, w_k)$, 
    and equivalences $\{\alpha_i\colon y_i \simeq z_i\}_{i=1}^a$,
    one can form a \hldef{composite $o \circ_{(\alpha)_i} p$} by using the monoidal product and the composition structure of the symmetric monoidal $\infty$-category $\calP$ as indicated in \cref{figure:gluing-operations}.
\begin{figure}[ht]
\centering
\begin{tikzpicture}[%
    auto,
    col/.style={circle,thick, draw=black!20,
                 inner sep=1pt,minimum size=3mm},
   operation/.style={rectangle,draw=blue!50,fill=blue!20,thick,
                      inner sep=2pt,minimum size=4mm}]
    \def\xspace{1.2}
    \def\yspace{.6}
    
    \node[operation] (op-pq) at (-3*\xspace, -2*\yspace) {$p\circ_{\alpha} q$};
    \node[col] (in-col1) at (-4*\xspace,{-1*\yspace})  {$x_1$};
    \node[col] (in-col2) at (-4*\xspace,{-2*\yspace})  {$x_2$};
    \node[col] (in-col3) at (-4*\xspace,{-3*\yspace})  {$z_3$};
    \node[col] (out-col1) at (-2*\xspace,{-1.5*\yspace})  {$w_1$};
    \node[col] (out-col2) at (-2*\xspace,{-2.5*\yspace})  {$y_3$};
    
    \foreach \i in {1,...,3}{
    \draw [-] (in-col\i) to [out = 0, in = 180] (op-pq);}
    \foreach \i in {1,...,2}{
    \draw [-] (op-pq) to [out = 0, in = 180] (out-col\i) ;}
    
    \node at (-1*\xspace,{-2*\yspace})  {$:=$};
    
    \node[operation] (op-p) at (\xspace, -2*\yspace) {$p$};
    \node[operation] (op-q) at (4*\xspace, -2*\yspace) {$q$};
    \node[operation] (id-z) at (1*\xspace, -4*\yspace) {$\mathrm{id}_{z_3}$};
    \node[operation] (id-y) at (4*\xspace, -4*\yspace) {$\mathrm{id}_{y_3}$};
    
    \foreach \i in {1,...,2}
    {
        \node[col] (xcol\i) at (0,{(-\i)*\yspace})  {$x_\i$};
        \draw [-] (xcol\i) to [out = 0, in = 180] (op-p);
    }
    \foreach \i in {1,...,3}
    {
        \node[col] (ycol\i) at (2*\xspace,{(-\i)*\yspace})  {$y_\i$};
        \draw [-] (op-p) to [out = 0, in = 180] (ycol\i);
    }
    \foreach \i in {1,...,3}
    {
        \node[col] (zcol\i) at (3*\xspace,{(-\i)*\yspace})  {$z_\i$};
        \draw [-] (zcol\i) to [out = 0, in = 180] (op-q);
    }
    \foreach \i in {1,...,1}
    {
        \node[col] (wcol\i) at (5*\xspace,{(-\i)*\yspace})  {$w_\i$};
        \draw [-] (op-q) to [out = 0, in = 180] (wcol\i);
    }
    
    \draw[dashed] (ycol1) to node {$\alpha_1$} (zcol1);
    \draw[dashed] (ycol2) to node {$\alpha_2$} (zcol2);
    
    \node[col] (xcol4) at (0,{(-4)*\yspace})  {$z_3$};
    \node[col] (ycol4) at (2*\xspace,{(-4)*\yspace})  {$z_3$};
    \node[col] (zcol4) at (3*\xspace,{(-4)*\yspace})  {$y_3$};
    \node[col] (wcol4) at (5*\xspace,{(-4)*\yspace})  {$y_3$};
    \draw[-] (ycol3) to [out = 0, in = 180] (zcol4);
    \draw[-] (ycol4) to [out = 0, in = 180] (zcol3);
    \draw[-] (xcol4) to (id-z);
    \draw[-] (id-z) to (ycol4);
    \draw[-] (zcol4) to (id-y);
    \draw[-] (id-y) to (wcol4);
\end{tikzpicture}
\caption{Gluing operations $o$ and $p$ along two colours.}
\label{figure:gluing-operations}
\end{figure}
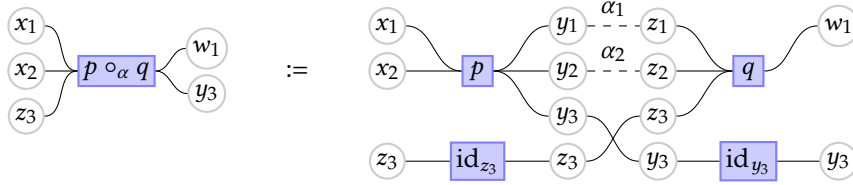

\subsubsection{Comparison to Segal $\infty$-properads}
    The composition operations described above are associative up to suitable higher coherence because they are obtained as certain compositions in a symmetric monoidal $\infty$-category.
    An informal way of summarizing this coherence is to say that in an $\infty$-properad there is a unique (i.e.~contractible) way to form a composite, given a connected directed acyclic graph whose vertices are suitably labelled by operations of the $\infty$-properad.
    This is made precise in the definition of Segal $\infty$-properads of \cite{HRY15}.
    We will only sketch the definition here and refer the reader to \cite{HRY15} and \cite{Koc16} for a careful elaboration of the necessary combinatorics.
\begin{defnA}
    Let $\mbf{G}$ denote the $1$-category whose objects are finite, connected, directed graphs $\Gamma$ with no directed cycles
    and where a morphism $f\colon \Gamma \to \Lambda$ consists of a subgraph $\Lambda_f \subset \Lambda$ and a map $\Gamma \leftarrow \Lambda_f$ whose fibers are connected.
    A \hldef{Segal $\infty$-properad} is a functor $\calP\colon \mbf{G}^\op \to \calS$ such that the canonical map
    \[
        \calP(\Gamma) \too \lim_{\Gamma_0 \subset \Gamma} \calP(\Gamma_0)
    \]
    is an equivalence for all graphs $\Gamma$, where the limit runs over all elementary subgraphs $\Gamma_0 \subset \Gamma$ (that is, corollas or edges).
    We let $\hldef{\Seg_{\mbf{G}^\op}(\calS)}\subset \Fun(\mbf{G}^\op,\calS)$
    denote the full subcategory of Segal $\infty$-properads.
    We say that a Segal $\infty$-properad is \hldef{complete} if its restriction to the subcategory of linear graphs $\Dop \subset \mbf{G}^\op$ is a complete Segal space in the sense of Rezk \cite{rezk},
    and we let $\hldef{\CSeg_{\mbf{G}^\op}(\calS)} \subset \Seg_{\mbf{G}^\op}(\calS)$ denote the full subcategory of these.
\end{defnA}

Using a result of \cite{CH22}
that replaces $\mbf{G}^\op$ with a certain category $\bfL^\op$ of levelled graphs, we will prove in \cref{sec:segal} the following comparison result: 
\begin{thmA}[\ref{cor:segal-G-envelope}]\label{thmA:Segal-properads}
    There is an envelope functor
    $\Env\colon \Seg_{\mbf{G}^\op}(\calS) \to \SM$ that restricts to an equivalence:
    \[
        \Env\colon \CSeg_{\mbf{G}^\op}(\calS) \simeq \Prpd \subset \SM.
    \]
\end{thmA}
This relates our notion of $\infty$-properads to the only previously existing notion of higher homotopical properads.
We wish to emphasize here that both sides of \cref{thmA:Segal-properads} can be useful for different purposes and its strength lies in allowing them to be used simultaneously.
The left side provides formulas for free $\infty$-properads, whereas the right side interfaces with symmetric monoidal $\infty$-categories.
In particular, this equivalence together with the adjunction $\Prpd \adj \SM$ allows us to define the endomorphism $\infty$-properad of an object in a symmetric monoidal $\infty$-category, and thus to define algebras over $\infty$-properads.
\cref{thmA:Segal-properads} was conjectured in the second author's thesis \cite[Conjecture 2.31]{Jan-tropical}.
The $1$-categorical part of this conjecture was recently proven by Beardsley--Hackney \cite{BH22}, who compare the ``labelled cospan categories'' (LCCs) of \cite[\S2]{Jan-tropical} to the classical definition of properads.
By expressing $\Prpd$ as part of a semi-recollement we in \cref{sec:5} are also able to show that the full subcategory of $1$-properads $\Prpdone \subset \Prpd$ is equivalent to the $(2,1)$-category of LCCs.
Combining the two results we see that $1$-properads in our sense are equivalent to the more classical definitions of properads.

\subsubsection{The terminal $\infty$-properad}
The proof of \cref{thmA:Segal-properads} will proceed by first identifying the terminal $\infty$-properad. 
The combinatorics of graphs will subsequently emerge from a careful study of that terminal case.
To find the terminal $\infty$-properad we again draw inspiration from bordism categories.
In any dimension, extracting the set of connected components defines a functor
\[
    \pi_0\colon \Bord_d \to  \hldef{\Csp} \coloneq \Cospan(\Fin) ,
    \qquad
    (W\colon M \to N) \mapsto (\pi_0M \to \pi_0W \leftarrow \pi_0N).
\]
Here $\Csp$ denotes the symmetric monoidal $(2,1)$-category whose objects are finite sets, whose morphisms are cospans of finite sets, and whose monoidal structure is the disjoint union.
In \cref{subsec:Csp-and-disjunctive} we check that $\Csp$ is an $\infty$-properad and give more general conditions under which $\Cospan(\calC)$ is an $\infty$-properad.
In \cref{subsec:pre-properads} we prove the following theorem, which constitutes the technical heart of the paper.
\begin{thmA}[\ref{cor:Csp-is-final}]\label{thmA:Csp-terminal}
    The symmetric monoidal $\infty$-category $\Csp$ is the terminal $\infty$-properad.
\end{thmA}
In particular, this implies that any $\infty$-properad is canonically a symmetric monoidal $\infty$-category equifibered over $\Csp$.
The converse of this will not be difficult to see and hence we conclude that the functor $\SMeq{\Csp} \to \SM$ restricts to an equivalence
\[
    \SMeq{\Csp} \simeq \Prpd.
\]
This generalizes the equivalence $\SMeq{\Fin} \simeq \Op$ established in \cite[Corollary D]{envelopes}, which itself is a variation on the main result of \cite{HK21}.
Indeed, restricting to $\infty$-properads where every operation has precisely one output colour recovers the $\infty$-category of \operads{} on the right and the $\infty$-category of symmetric monoidal $\infty$-categories equifibered over $\Fin \subset \Csp$ on the left.
In \cref{sec:prpd-eqf-Csp} we will use \cref{thmA:Csp-terminal} to show that $\Prpd$ is a compactly generated $\infty$-category and that various adjoints exist.

\subsubsection{The theory of $\infty$-properads}

In \cref{sect:properad-theory} we develop basic tools for working with $\infty$-properads.
For example, we will discuss how to characterize sub-$\infty$-properads, and how \emph{monic} $\infty$-properads (i.e.~those $\infty$-properads where every operation has exactly one output) are equivalent to $\infty$-operads.

Crucially, we will give a description of the free $\infty$-properad on a given space of operations,
in terms of the factorization system spanned by equifibered symmetric monoidal functors.
While this might be complicated in general, we are able to give a simple formula in the case of the \hldef{free corolla $\mfr{c}_{A,B}$}, which is defined as the free $\infty$-properad on an operation whose set of input and output colours are in bijection with finite sets $A$ and $B$.
\begin{lemA}[\ref{lem:free-corolla-pushout}]
    The free $(A,B)$-corolla fits into a pushout square of symmetric monoidal $\infty$-categories
    \[
        \begin{tikzcd}
            {\xF(* \amalg *)} \ar[d] \ar[r, "\Delta_A \oplus \Delta_B"] \ar[dr, very near end, phantom, "\ulcorner"] &
            {\xF(A \amalg B)} \ar[d] \\
            {\xF([1])} \ar[r] &
            \mfr{c}_{A,B}
        \end{tikzcd}
    \]
    where the top horizontal functor sends the two points to
    $\sum_{a \in A} a$ and $\sum_{b \in B} b$, respectively.
\end{lemA}

This description of the free corolla allows us to better understand the \hldef{morphism $\infty$-properad} functor, which we define to be the right adjoint of the inclusion functor:

\[
    \mrm{include} \colon \Prpd \adj \SM \cocolon \hldef{\calU}.
\]
By mapping $\mfr{c}_{A,B}$ into $\calU(\calC)$ we show that $\calU(\calC)$ is an $\infty$-properad whose colours are the objects of $\calC$ and whose operations from a collection of colours $(c_1,\dots, c_n)$ to another collection $(d_1, \dots, d_m)$ are precisely the morphisms $c_1 \otimes \dots \otimes c_n \to d_1 \otimes \dots \otimes d_m$ in $\calC$.
When restricting to the subproperad of $\calU(\calC)$ on a single colour $c \in \calC$ one obtains the \hldef{endomorphism $\infty$-properad} of $c$,
i.e.~the $\infty$-properad whose arity $(k,l)$ operations are $\Map_\calC(c^{\otimes k}, c^{\otimes l})$.
Restricting further to those operations with a single output colour recovers the endomorphism \operad{} of $c$.
In analogy with the situation for operads, we define the $\infty$-category of \hldef{$\calP$-algebras} in a symmetric monoidal $\infty$-category $\calC$ to be
\[
    \hldef{\Alg_\calP(\calC)} \coloneq \Fun_{\Prpd}(\calP, \calU(\calC)) \simeq \Fun_{\SM}(\calP, \calC).
\]
Previous models of $\infty$-properads such as \cite{HRY15} did not yet have a notion of a $\calP$-algebra in a symmetric monoidal $\infty$-category%
\footnote{
        Note that while Chu--Hackney in \cite[\S4]{CH22} discuss algebras over $\infty$-properads, these are $\calP$-algebras in $\calQ$ where both $\calP$ and $\calQ$ are $\infty$-properads. Thus, in our language these would be simply morphisms of $\infty$-properads and their work is to establish a $\Cat$-enrichment of $\Prpd$. (We can obtain such an enrichment by suitably restricting the one of $\SM$, but we do not attempt to compare it to theirs.)
        In order to set up $\calP$-algebras in a symmetric monoidal $\infty$-category $\calC$ in their setting, one would have to construct the morphism $\infty$-properad $\calU(\calC)$ to then take $\calP$-algebras in $\calU(\calC)$.
        This is one of the key achievements of \cref{thmA:Segal-properads}.
}
and the simplicity of the above definition of algebras is one of the key advantages of the definition of $\infty$-properads proposed here.
This will be particularly useful given the comparison result in \cref{thmA:Segal-properads}.

In the special case where all the operations in $\calP$ have a single output colour, equivalently $\calP \simeq \Env(\calO)$ for some $\calO \in \Op$, the definition above agrees with Lurie's definition of algebras (see \cite[Proposition 2.2.4.9]{HA}).
In contrast, substituting $\calP= \Bord_d$ the above lets us interpret topological field theories (TFTs) in the sense of Atiyah and Witten \cites{witten1988topological,atiyah1988topological} as algebras over the properad $\Bord_d$.
To demonstrate the difference between these two examples, recall that the $1$-dimensional cobordism-hypothesis postulates an equivalence $\Alg_{\Bord_1^\fr}(\calC) \simeq (\calC^\dbl)^\simeq$ where $(\calC^\dbl)^\simeq \subseteq \calC^\simeq$ is the space of dualizable objects and $\Bord_1^\fr$ is the $1$-dimensional framed bordism category.

In \cref{sec:reduced-projective} we study \hldef{reduced $\infty$-properads}, i.e.~those $\calP$ for which the space of $(0,0)$-ary operations $\calP(\emptyset; \emptyset)$ is contractible.
We show that the $\infty$-category of these is equivalent to the $\infty$-category of \hldef{projective $\infty$-properads} which are obtained by passing to the cofiber 
$\overline{\calP} \coloneq \calP / \calP_0$
where $\calP_0 \subset \calP$ is the full subcategory on the monoidal unit.
These correspond to the ``reduced labelled cospan categories'' introduced by the second author in \cite{Jan-tropical}.
When $\pi_0|\calP|$ is a group, we in \cref{obs:proj-fiber-sequence} obtain a fiber sequence on classifying spaces, generalizing \cite[Proposition 3.4]{Jan-tropical}:
\[  \Omega^\infty \Sigma^{\infty+1} \calP(\emptyset;\emptyset)_+ \too |\calP| \too |\overline{\calP}|.\]
We also prove that $\Prpd$ sits in a semi-recollement between the $\infty$-category of spaces and the $\infty$-category of reduced/projective $\infty$-properads.
From this it follows that $\Prpd$ can be written as a pullback $\Prpd^\proj \times_\calS \Ar(\calS)$ and that any $\infty$-properad can be recovered from its projective $\infty$-properad, its space of $(0,0)$-ary operations, and a certain gluing map.

There is a further left-adjoint $(-)^\ext$ that freely adds $(0,0)$-ary operations to a projective $\infty$-properad.
We call an $\infty$-properad $\calP$ ``extended'' if $\calP \simeq \Red{\calP}^\ext$ and in \cref{sect-3.4} we give a concrete characterization of such $\infty$-properads in terms of the factorization category $\calF(\calP)\subset \calP_{\unit//\unit}$.
As an example, we show in \cref{cor:Bord-extended} that the bordism $\infty$-properad $\Bord_d^\theta$ is always extended for any dimension $d \ge 1$ and tangential structure $\theta$.
Since, as mentioned above, TFTs are algebras over the $\infty$-properad $\Bord_d^\theta$ this means that the value of a TFT on closed $(d+1)$-manifolds is always uniquely and coherently determined by its values on connected manifolds with boundary.

In \cref{subsec:n-properads} we define a full subcategory $\mrm{Prpd}_n \subset \Prpd$ of $n$-properads.
For $n=1$ we use the aforementioned pullback description to establish an equivalence between $\Prpdone$ and the $(2,1)$-category of labelled cospan categories (based on \cite[\S2]{Jan-tropical}).
Further using the main result of \cite{BH22} this connects our definition of $1$-properads to the original definition of (coloured) properads used e.g.~in \cite{HRY15}.

\subsubsection{Outlook}
We hope that the theory of $\infty$-properads developed here might serve as foundations for an alternative
approach to higher algebra.
In forthcoming work, we intend to follow this idea in various directions.

\begin{itemize}
    \item 
    \textbf{Equifibered higher algebra:} 
    We intend to revisit some of the foundational results on \operads{} established by Lurie in \cite{HA} such as the Boardmann--Vogt tensor product and develop them independently within $\SM$, relying on the theory of equifibered maps.
    (For this Boardmann--Vogt tensor product this has now been achieved in \cite{segalification}.)

    \item \textbf{Modular operads:}
    We intend to expand the theory of $\infty$-properads developed here to encompass other operad-like structures such as cyclic operads and modular operads.
    In particular, we hope to show that modular \operads{} embed fully faithfully in $\infty$-properads,
    which implies a version of the cobordism hypothesis ``with singularities'' in dimension $1$.
    When applied to other bordism categories, we also expect this to be useful for studying the stable homology of certain diffeomorphism groups.

    \item \textbf{Bisymmetric sequences:} 
    Using the theory of equifibered maps one should be able to show that $\infty$-properads embed fully faithfully into $\bbE_1$-algebras in an $\infty$-category of bisymmetric sequences endowed with a coherently defined composition product.
    This would restrict to an equivalence between \operads{} and algebras in symmetric sequences for the composition product in the sense of Baez--Dolan \cite[\S 2.3]{baez1997higher}.
    Such a comparison theorem was proven by Haugseng \cite{Hau22} for a possibly different choice of composition product.
\end{itemize}

\subsubsection{Acknowledgements}
We would like to thank Philip Hackney, Rune Haugseng, and Oscar Randal-Williams for several useful conversations about the topic of the paper. 
We would also like to thank Joachim Kock and Lior Yanovski for comments on an earlier version of this draft, Az\'elie Picot and Adela Zhang for pointing out an error, and the referee for their very detailed feedback and patience as we revised and extended this paper.

The first author would like to thank the Hausdorff Research Institute for Mathematics for their hospitality in the 2022 trimester program during which part of this work was written.
Part of this work was completed while the second author was at the
Mathematical Sciences Research Institute program
(NSF grant no. DMS-1928930) at the UNAM campus in Cuernavaca.
The second author is supported by the ERC grant no.~772960, and would like to thank the Copenhagen Centre for Geometry and Topology for their hospitality.

\section{Commutative monoids and equifibered maps}\label{sec:eqf}

\subsection{Equifibered theory}\label{subsec:Eqf-theory}

In this section we introduce the notion of an equifibered map between commutative monoids and investigate its properties.
We begin by briefly recalling some basic facts on commutative monoids.

\subsubsection{Recollection on commutative monoids}
    We let $\hldef{\Fin_*}$ denote the category of finite pointed sets.
    Up to isomorphism its objects are of the form $A_+ = A \sqcup \{\infty\}$ where $A$ is an unpointed finite set.
    For $n \in \bbN$ we also let $n$ denote the set $\{1,\dots,n\}$ and accordingly $n_+ = \{1,\dots,n,\infty\} \in \Fin_*$.
    For each $a \in A$ we have a canonical map $\rho_a\colon A_+ \to \{a\}_+ \cong 1_+$ that sends every element except $a$ to the base point.
    
    A \hldef{commutative monoid} (in spaces) is a functor
    $ M\colon \Fin_* \to \pcal{S} $
    satisfying that for all $A_+ \in \Fin_*$ the Segal map
    \[
        M(A_+) \xrightarrow{(\rho_a)_{a \in A}} \prod_{a \in A} M(\{a\}_+),
    \]
    is an equivalence. 
    We let $\hldef{\CMon} \subset \Fun(\Fin_*, \calS)$ denote the full 
    subcategory of commutative monoids.

\begin{rem}
    Commutative monoids in this sense are often called $\mbb{E}_{\infty}$-monoids. 
    We will work entirely in the $\infty$-categorical setting where these notions are interchangeable.
\end{rem}

    The forgetful functor $U\colon  \CMon \too \calS$ is defined by sending
    $M\colon \Fin_* \to \calS$ to $M(1_+)$. 
    By abuse of notation we will usually write $M$ to denote $M(1_+)$.
    By \cref{lem:basic-properties-for-xF}.\ref{it:xF-U} the forgetful functor has a left adjoint, which we denote: 
    \[
        \hldef{\xF}\colon  \calS \too \CMon.
    \]
    We call $\xF(X)$ the \hldef{free commutative monoid} on $X$.
    We say that a commutative monoid $M$ is \hldef{free} if it is in the essential image of $\xF$, and we let $\hldef{\CMon^\free} \subset \CMon$ denote the full subcategory of free commutative monoids.

\begin{lem}\label{lem:basic-properties-for-xF}
    The free-forgetful adjunction $\xF\colon \calS \adj \CMon :\! U$ has the following properties:
    \begin{enumerate}[(1)]
        \item\label{it:xF-presentable}
        $\CMon$ is an accessible localization of $\Fun(\Fin_*, \calS)$ and hence presentable.
        \item\label{it:xF-U}
        The forgetful functor
        $U$ is a conservative right adjoint.
        Moreover, it preserves sifted colimits.
        \item\label{it:xF-formula} 
        The free functor $\xF$ can be explicitly computed as
        \[
            \xF(X) \simeq \colim_{A \in \Fin^\simeq} \Map(A, X) 
        \simeq \coprod_{n \ge 0} X^n_{h\Sigma_n}.  \]
        \item\label{it:xF-semiadd}
        $\CMon$ is semi-additive, i.e.~the categorical coproduct and product coincide.
        We refer to both as the \hldef{direct sum}, which we denote by $\hldef{M \oplus N} \coloneq M \times N$.
        For $X, Y \in \calS$ we have $\xF(X \sqcup Y) \simeq \xF(X) \oplus \xF(Y)$.
        \item\label{it:xF-contractible-limits}
        The free functor $\xF\colon  \calS \too \CMon$ preserves weakly contractible limits, and in particular pullbacks.
    \end{enumerate}
\end{lem}
\begin{proof}
    One can also show these properties directly from the definition, but for simplicity we shall cite the literature instead.
    \ref{it:xF-presentable}
    This follows from \cite[Proposition 5.5.4.15]{HTT}, see \cite[Propositon 4.1]{gepner-universality}.
    \ref{it:xF-U}
    Inspecting the Segal condition we see that $\CMon \subset \Fun(\Fin_\ast,\calS)$ is closed under limits and sifted colimits.
    It follows that the forgetful functor $U \coloneq \mathrm{ev}_{1_+} \colon \CMon \to \calS$ preserves limits and sifted colimits and hence has a left adjoint by the adjoint functor theorem.
    $U$ is moreover conservative by the Segal condition.
    
    \ref{it:xF-formula} was shown by Lurie \cite[Example 3.1.3.14]{HA}, though for the case of $\calS$ a simpler proof can be given using algebraic patterns \cite[Example 8.13]{CH19}.
    \ref{it:xF-semiadd} Semi-additivity is shown in \cite[Corollary 2.5]{gepner-universality}. The claim about the free functor follows because, being a left adjoint, it preserves coproducts.
    \ref{it:xF-contractible-limits} follows because the formula in \ref{it:xF-formula} is a colimit indexed by an $\infty$-groupoid and in $\calS$ such colimits commute with weakly contractible limits \cite[Lemma 2.2.8]{GHK22}.%
    \footnote{
        Alternatively, one could say that the formula in \ref{it:xF-formula} also shows that $\xF(X)$ is a polynomial functor and these preserve weakly contractible limits by \cite[Theorem 2.2.3]{GHK22}.
    }
\end{proof}

Being free is a \emph{property} of a commutative monoid in the following sense:
\begin{lem}\label{lem:free-is-a-property}
    The free functor restricts to an equivalence on maximal subgroupoids:
    $\calS^\simeq \iso (\CMon^\free)^\simeq$.
\end{lem}
\begin{proof}
    The functor $\xF\colon  \calS \to \CMon$ induces the map
    \[
        \Map_{\calS}(X, Y) \too \Map_{\CMon}(\xF(X), \xF(Y)) \simeq \Map_{\calS}(X, \xF(Y))
    \]
    that is given by post-composition with the unit $Y \hookrightarrow \xF(Y)$,
    which is a monomorphism by \cref{lem:basic-properties-for-xF}.\ref{it:xF-formula}.
    It thus suffices to observe that for any equivalence $f\colon \xF(X) \iso \xF(Y)$ the induced map on components $\pi_0\xF(X) \to \pi_0 \xF(Y)$ must preserve indecomposables and thus $f(X) \subseteq Y$. 
\end{proof}

\subsubsection{The definition of equifibered maps}

\begin{defn}\label{defn:eqf}
    A morphism of commutative monoids $f\colon M \to N$ is called \hldef{equifibered} if the natural square 
    \[
        \begin{tikzcd}
            M \times M \ar[r, "+"] \ar[d, "f\times f"'] &
            M \ar[d, "f"] \\
            N \times N \ar[r, "+"] & 
            N
        \end{tikzcd}
    \]
    is a pullback in $\pcal{S}$. 
\end{defn}

\begin{rem}
    Equifibered morphisms were introduced in \cite{envelopes} under the name of ``active-equifibered morphisms'' in the context of Segal objects over arbitrary algebraic patterns. 
    In general, a natural transformation is called equifibered if all of its naturality squares are cartesian:
    a morphism of commutative monoids is equifibered if its restriction to $\Fin \simeq \Fin_*^\act \subset \Fin_*$ is an equifibered natural transformation in this sense (this follows from \cref{prop:TFAE-to-equifibered}.\ref{item:eqf-n} and the pullback pasting lemma).
    In the present paper we shall only consider the pattern $\Fin_\ast$ and drop the word ``active''.
    This notion is also closely related to the ``CULF'' maps of \cite{GCKT18}, as discussed in \cref{obs:CULF}.
\end{rem}

\begin{example}\label{ex:free-is-eqf}
    For any map of spaces $f\colon X \to Y$ the resulting map of free commutative
    monoids $\xF(f)\colon \xF(X) \to \xF(Y)$ is equifibered. 
    Indeed, by \cref{lem:basic-properties-for-xF}.\ref{it:xF-semiadd}, the relevant square is equivalent to 
    \[
        \begin{tikzcd}
            \xF(X \amalg X) \ar[r, "\xF(\nabla)"] \ar[d, "\xF(f\amalg f)"'] &     
            \xF(X) \ar[d, "\xF(f)"] \\
            \xF(Y \amalg Y) \ar[r, "\xF(\nabla)"] &     
            \xF(Y),
        \end{tikzcd}
    \]
    which is cartesian because $\xF$ preserves pullbacks by \cref{lem:basic-properties-for-xF}.\ref{it:xF-contractible-limits}.
\end{example}

\begin{rem}
    Note that \cref{ex:free-is-eqf} fails for free commutative monoids in the $1$-category $\Sets$ of sets.
    The relevant square for the map $f\colon \{a, b\} \to \{c\}$ is 
    \[
        \begin{tikzcd}
            \bbN\langle a_1, b_1, a_2, b_2 \rangle \ar[r, "\bbN(\nabla)"] \ar[d, "\bbN(f\amalg f)"'] &     
            \bbN\langle a, b \rangle  \ar[d, "\bbN(f)"] \\
            \bbN\langle c_1, c_2 \rangle  \ar[r, "\bbN(\nabla)"] &     
            \bbN\langle c \rangle,
        \end{tikzcd}
    \]
    where the horizontal maps send $a_i \mapsto a$ etc.\ 
    and the vertical maps send $a_i \mapsto c_i$ and $b_i \mapsto c_i$.
    This is not a pullback since $a_1+b_2$ and $a_2+b_1$ are sent to the same element by $\bbN\langle a_1, b_1, a_2, b_2\rangle \to \bbN\langle c_1, c_2 \rangle \times \bbN\langle a, b \rangle$.
\end{rem}

\begin{rem}
    Below we will see that a map between free commutative monoids is equifibered
    if and only if it is free. 
    Motivated by this, we will often think of equifibered maps 
    as a more well-behaved notion, generalizing free maps.
\end{rem}

\begin{obs}\label{obs:eqf-cancellation}
    Equifibered maps are closed under composition and satisfy the following cancellation property:
    for any two morphisms $f\colon M \to N$ and $g\colon N \to L$ in $\CMon$, if $g$ and $g \circ f$ are equifibered, then so is $f$.
    (This follows from pullback pasting, or alternatively from \cref{lem:formal-ctf-eqf}.)
\end{obs}

\begin{obs}\label{obs:eqf=equi-fibered}
    For a morphism $f\colon M \to N$ and $n \in N$ write $f^{-1}(n)$
    for the (homotopy) fiber of $f$ at $n$.
    For all $a, b \in N$ addition yields a well-defined map
    \[
        +\colon  f^{-1}(a) \times f^{-1}(b) \too f^{-1}(a+b).
    \]
    Since these are exactly the fibers of the square in \cref{defn:eqf},
    the morphism $f$ is equifibered if and only if 
    the above map is an equivalence for all $a,b\colon * \to N$.
    In fact, it suffices to check this for one representative in each component.
\end{obs}

\begin{example}
    An example of a non-free, equifibered map can be obtained as follows.
    Let $f\colon \Fin_*^\simeq \to \Fin^\simeq$ denote the functor 
    that forgets from the groupoid of pointed finite sets 
    to the groupoid of finite sets. 
    Both groupoids are commutative monoids with respect to the cartesian
    product and $f$ is a map of commutative monoids.
    As a map of spaces $f$ may be described as:
    \[
        \coprod_{n \ge 0} B\Sigma_{n-1} \too 
        \coprod_{n \ge 0} B\Sigma_{n}
    \]
    where we interpret $B\Sigma_{-1} = *$.
    Note that the right side is not the free monoid $\xF(*)$,
    since the monoid structure is given by cartesian product,
    not disjoint union.
    In particular, it would make sense to restrict to the submonoid
    where $n$ is of the form $p^k$ for some fixed $p$.
    
    To check that $f$ is equifibered we use \cref{obs:eqf=equi-fibered}. 
    In the case at hand $a, b \in \Fin^\simeq$ are finite sets 
    and their ``sum'' is the product $a \times b$.
    The fiber of $f\colon \Fin_*^\simeq \to \Fin^\simeq$ at a finite set
    $a$ is canonically identified with the set $a$ itself.
    Hence, the map in question is 
    $\id_{a\times b}\colon  f^{-1}(a) \times f^{-1}(b) \to f^{-1}(a \times b)$,
    which is an equivalence.
\end{example}

As sifted colimits (\cref{lem:basic-properties-for-xF}.\ref{it:xF-U}) and finite coproducts of monoids (\cref{lem:basic-properties-for-xF}.\ref{it:xF-semiadd}) tend to be easier to compute than arbitrary colimits, the following lemma and corollary will be very useful for checking that certain functors preserve all colimits.
\begin{lem}[Lurie]\label{realizations+coproducts}
    In a cocomplete $\infty$-category $\calC$ any colimit can be written as a geometric realization of coproducts.
\end{lem}
\begin{proof}
    Let $F\colon J \to \calC$ be a diagram. 
    The colimit of $F$ is the left Kan extension of $F$ along the map $J \to \pt$. 
    Equivalently, $\colim_J F = F'(\pt) $ where $F' \colon \Psh(J) \to \calC$ is the unique colimit preserving extension of $F$ (\cite[Theorem 5.1.5.6]{HTT}) and $\pt$ is the terminal presheaf.
    By \cite[Lemma 5.5.8.13]{HTT} we may write $\pt$ as a geometric realization of coproducts of representables.
    Applying the colimit-preserving functor $F'$ to this gives the desired description of $\colim_J F$.
\end{proof}

Since small coproducts can be written as filtered colimits over finite coproducts (which in turn are either initial objects or iterated binary coproducts), we have:
\begin{cor}\label{sifted+binary+initial=all}
    For a functor $F\colon \calC \to \calD$ between cocomplete $\infty$-categories the following are equivalent:
    \begin{enumerate}[(1)]
        \item $F$ preserves small colimits,
        \item $F$ preserves geometric realizations and small coproducts,
        \item $F$ preserves sifted colimits, binary coproducts, and the initial object.
    \end{enumerate}
\end{cor}

We now record several equivalent characterizations of equifibered maps,
which will be useful throughout the paper:
\begin{prop}\label{prop:TFAE-to-equifibered}
    Let $f\colon  M \to N$ be morphism of commutative monoids. The following are equivalent:
    \begin{enumerate}[(1)]
        \item \label{item:eqf}
        $f$ is equifibered.
        \item \label{item:eqf-rlp}
        $f$ is right orthogonal to 
        $\Delta\colon  \xF(*) \to \xF(*) \oplus \xF(*)$.
        \item \label{item:eqf-n}
        For all $n \ge 0$ the following square is cartesian:
        \[
        \begin{tikzcd}
            M^n \ar[r] \ar[d, "f^n"'] &  M \ar[d, "f"] \\
            N^n \ar[r] & N.
        \end{tikzcd}
        \]
        \item \label{item:eqf-xF}
        The following square obtained from the counits of the adjunction
        $(\xF \dashv U)$ is cartesian:
        \[
        \begin{tikzcd}
            \xF(M) \ar[r, "+"] \ar[d, "\xF(f)"'] &  M \ar[d, "f"] \\
            \xF(N) \ar[r, "+"] &  N.
        \end{tikzcd}
        \]
        \item \label{item:eqf-rep-free}
            $f$ is representably free:
            for any space $X$ and map $X \to N$ the following square is cartesian:
            \[
                \begin{tikzcd}
                    \xF(X \times_N M) \ar[r] \ar[d] & M \ar[d] \\
                    \xF(X) \ar[r] & N.
                \end{tikzcd}
            \]
        \item \label{item:exp}
        $f$ is exponentiable, i.e.~the base change functor $f^\ast \colon \CMon_{/N} \to \CMon_{/M}$ preserves colimits.
    \end{enumerate}
\end{prop}
\begin{proof}
    \ref{item:eqf} $\Leftrightarrow$ \ref{item:eqf-rlp}:
    A morphism $f\colon M \to N$ is right orthogonal \cite[Definition 5.2.8.1]{HTT} with respect
    to the diagonal map $\Delta\colon  \xF(*) \to \xF(*) \times \xF(*)$
    if and only if the following square of spaces is cartesian:
    \[
        \begin{tikzcd}
            {\Map_\CMon(\xF(*) \times \xF(*), M)} \ar[r, "\Delta^*"] \ar[d, "f_*"'] &
            {\Map_\CMon(\xF(*), M)} \ar[d, "f_*"] \\
            {\Map_\CMon(\xF(*) \times \xF(*), N)} \ar[r, "\Delta^*"] &
            {\Map_\CMon(\xF(*), N)}
        \end{tikzcd}
    \]
    Using that $\xF(*) \times \xF(*) \simeq \xF(* \amalg *)$ 
    and using the adjunction $(\xF \dashv U)$ 
    this can be identified with the square in \cref{defn:eqf}.
    
    \ref{item:eqf} $\Leftrightarrow$ \ref{item:eqf-n}:
    We show that if the square in \ref{item:eqf-n} is cartesian for $n=2$ (this is \ref{item:eqf}), then it is also cartesian for all other $n$.
    For $n=0$ condition \ref{item:eqf-n} says that the fiber $f^{-1}(0) = \{0\} \times_N M$
    is contractible. 
    By \cref{obs:eqf=equi-fibered}
    the addition map 
    $+\colon f^{-1}(0) \times f^{-1}(0) \to f^{-1}(0)$ is an equivalence.
    This is only possible for the $0$-monoid (\cref{rem:addition-equivalence}),
    hence $f^{-1}(0) = \{0\} \times_N M$ is contractible.
    Now suppose condition \ref{item:eqf-n} is satisfied for all $k \le n$, where $n \ge 2$.
    Then the $(n+1)$-square may be written as a composite of squares:
    \[
        \begin{tikzcd}
            M^n \times M \ar[r, "+ \times \id"] \ar[d, "f^n\times f"'] &
            M \times M \ar[r, "+"] \ar[d, "f\times f"] &  M \ar[d, "f"] \\
            N^n \times N \ar[r, "+ \times \id"] & 
            N \times N \ar[r, "+"] & N.
        \end{tikzcd}
    \]
    The left square is cartesian by condition \ref{item:eqf-n} for $n$
    and the right square is cartesian by the condition for $n=2$.
    Hence, the entire rectangle is cartesian and condition \ref{item:eqf-n}
    is satisfied for $n+1$.
    The claim now follows by induction.
    
    \ref{item:eqf-n} $\Leftrightarrow$ \ref{item:eqf-xF}:
    Consider the commutative diagram
    \[
    \begin{tikzcd}
        \coprod_{n\ge 0} M^n \ar[r] \ar[d] & 
        \xF(M) \ar[r, "+"] \ar[d] & 
        M \ar[d] \\
        \coprod_{n \ge 0} N^n \ar[r] & \xF(N) \ar[r, "+"] & N.
    \end{tikzcd}
    \]
    The left square is cartesian since the horizontal maps 
    both have equivalent fibers: the fiber of $M^n \to M^n_{h\Sigma_n}$ 
    at any point is the finite set $\Sigma_n$, independently of $M$.
    Condition \ref{item:eqf-n} says that the outside square is cartesian
    and condition \ref{item:eqf-xF} says that the right square is cartesian.
    Since the map $\coprod_{n \ge 0} N^n \to \xF(N)$ is surjective
    on connected components, it follows from the pullback pasting
    lemma that the two conditions are equivalent.
    
    \ref{item:eqf-xF} $\Leftrightarrow$ \ref{item:eqf-rep-free}:
    Suppose $f\colon M \to N$ satisfies condition \ref{item:eqf-xF}.
    For all $g\colon X \to N$ the map $\xF(X) \to N$ can be factored as 
    $+ \circ \xF(g)\colon  \xF(X) \to \xF(N) \to N$ and so
    the square in \ref{item:eqf-rep-free} may be factored as
    \[
    \begin{tikzcd}
        \xF(X \times_N M) \ar[r] \ar[d] & 
        \xF(M) \ar[r, "+"] \ar[d] & 
        M \ar[d] \\
        \xF(X) \ar[r] & \xF(N) \ar[r, "+"] & N.
    \end{tikzcd}
    \]
    The right square is cartesian by assumption and the left
    square is cartesian because $\xF$ preserves pullbacks.
    Hence, the entire square is cartesian, which is exactly
    condition \ref{item:eqf-rep-free}.
    For the converse we simply set $X = N$.

    \ref{item:eqf-n} $\Leftrightarrow$ \ref{item:exp} 
    By \cref{sifted+binary+initial=all} $f^\ast$ preserves small colimits if it preserves finite coproducts and sifted colimits.
    All colimits in the slice $\CMon_{/N}$ are computed in $\CMon$.
    Sifted colimits can be computed in $\calS$ by \cref{lem:basic-properties-for-xF},
    and are thus preserved as the functor $f^*\colon \calS_{/N} \to \calS_{/M}$ preserves all colimits.
    So $f$ is exponentiable if and only if $f^*$ preserves finite coproducts.
    Since $\CMon$ is semi-additive, finite coproducts are finite direct sums, which may be computed as products in $\calS$.
    Now suppose that $f$ satisfies \ref{item:eqf-n}, then we need to show that $f^*$ preserves any finite coproduct $\bigoplus_{i \in I} A_i$.
    In the diagram
\[\begin{tikzcd}
	{\bigoplus_{i \in I}f^*(A_i)} & {\bigoplus_{i \in I}M} & M \\
	{\bigoplus_{i \in I}A_i} & {\bigoplus_{i \in I}N} & N
	\arrow[from=1-1, to=1-2]
	\arrow[from=1-1, to=2-1]
	\arrow["\lrcorner"{anchor=center, pos=0.125}, draw=none, from=1-1, to=2-2]
	\arrow["{+}", from=1-2, to=1-3]
	\arrow["{\bigoplus_{i \in I}f}"', from=1-2, to=2-2]
	\arrow["\lrcorner"{anchor=center, pos=0.125}, draw=none, from=1-2, to=2-3]
	\arrow["f", from=1-3, to=2-3]
	\arrow[from=2-1, to=2-2]
	\arrow["{+}", from=2-2, to=2-3]
\end{tikzcd}\]
    the left square is cartesian because it is the $I$-indexed product of cartesian squares and the right square is cartesian by \ref{item:eqf-n}.
    Hence, pullback pasting implies $f^*(\bigoplus_{i \in I} A_i) \simeq \bigoplus_{i \in I} f^*(A_i)$,  proving \ref{item:eqf-n}~$\Rightarrow$~\ref{item:exp}.
    Considering the special case where $A_i = N$ yields the converse implication.
\end{proof}

\begin{rem}\label{rem:addition-equivalence}
    In the above proof we used that if $M$ is a commutative monoid for which the addition map $+\colon M \times M \to M$ is an equivalence, then $M = 0$ is the $0$-monoid.
    Indeed, then $0 \times \id_M\colon M \to M \times M$ must be an equivalence because it is a section of $+$.
    But $0 \times \id_M$ contains $0\colon \pt \to M$ as a retract, so this map is also an equivalence, i.e.~$M$ is contractible.
    (Alternatively, one can argue that $\pi_k(M)$ is trivial for all $k$.)
\end{rem}

Given a commutative monoid $M \in \CMon$ we let $\hldef{\CMon^\eqf_{/M}} \subset \CMon_{/M}$ denote the full subcategory spanned by equifibered maps with target $M \in \CMon$.
(By the cancellation property from \cref{obs:eqf-cancellation} this agrees with the slice category of $\CMon^\eqf$ over $M$.)
Applying characterization \ref{item:eqf-rep-free} of \cref{prop:TFAE-to-equifibered} in the special case where $N = \xF(X)$ yields the following corollary.

\begin{cor}\label{cor: equifibered-over-free-equals-spaces}
    For a space $X \in \calS$ the functor $\xF \colon  \calS_{/X} \too \CMon_{/\xF(X)}$ induces an equivalence of $\infty$-categories $\calS_{/X} \simeq \CMon^{\eqf}_{/\xF(X)}$.
    The inverse can be described as the composite
    \[
        \CMon_{/\xF(X)}^\eqf \hookrightarrow \CMon_{/\xF(X)} 
        \xtoo{\text{forget}} \calS_{/\xF(X)} 
        \xtoo{\iota^*} \calS_{/X}
    \]
    where the last functor pulls back along $\iota\colon X \hookrightarrow \xF(X)$. 
\end{cor}

\begin{obs}\label{obs:elementary-subspace}
    The space of equifibered maps $\xF(X) \to \xF(\ast)$ is contractible.
    Indeed, it corresponds to the subspace of $\Map_\calS(X, \xF(\ast))$ where all of $X$ is mapped to the subspace $\ast \subset \xF(\ast)$.
    Therefore, $\xF(*)$ is a terminal object of the replete subcategory $\CMon^{\free, \eqf} \subset \CMon$ 
    of free commutative monoids and equifibered maps.
    So we have $\CMon^{\free, \eqf} \simeq (\CMon^{\free, \eqf})_{/\xF(\ast)} \simeq \CMon^\eqf_{/\xF(\ast)}$.
    Applying \cref{cor: equifibered-over-free-equals-spaces} in the case of $X=\ast$ shows that the free functor 
    $\xF\colon  \calS \too \CMon$ restricts to an equivalence $\calS \simeq \CMon^{\free, \eqf}$.
    We denote the inverse equivalence by:
    \[
        (-)^\el \colon \CMon^{\free, \eqf} \too \calS.
    \]
    For a free commutative monoid $M$ one can also describe $M^\el \in \calS$ as the unique subspace $M^\el \subset M$ such that $\xF(M^\el) \to M$ is an equivalence.
\end{obs}

We now record some formal properties of equifibered maps.
Some of these properties can also be seen as formal consequences of the fact that equifibered maps are the right class of a factorization system, which we prove in \cref{lem:formal-ctf-eqf}.

\begin{lem}\label{lem:eqf-maps-closed-under-limits}
    The full subcategory of $\Ar(\CMon)$ on the equifibered morphisms
    is closed under all limits and filtered colimits.
\end{lem}
\begin{proof}
    This holds because both the product and the pullback used in the definition are preserved under all limits and filtered colimits. 
\end{proof}

\begin{lem}\label{lem:eqf-base-change}
    Suppose we have a cartesian square of commutative monoids
    \[
        \begin{tikzcd}
            M_1 \ar[r] \ar[d, "f_1"'] \ar[dr, very near start, phantom, "\lrcorner"] & M_2 \ar[d, "f_2"] \\
            N_1 \ar[r, "g"] & N_2.
        \end{tikzcd}
    \]
    If $f_2$ is equifibered, then so is $f_1$.
    Conversely, if $f_1$ is equifibered and 
    $\pi_0(g)\colon  \pi_0 N_1 \to \pi_0 N_2$
    is surjective, then $f_2$ is equifibered.
\end{lem}
\begin{proof}
    We will use the characterization from \cref{obs:eqf=equi-fibered}.
    For all $a, b \in N_1$ we have the following square of fibers:
    \[
    \begin{tikzcd}
        {f_1^{-1}(a) \times f_1^{-1}(b)} \ar[r, "+"] \ar[d, "\simeq"] &
        {f_1^{-1}(a+b)} \ar[d, "\simeq"] \\
        {f_2^{-1}(g(a)) \times f_2^{-1}(g(b))} \ar[r, "+"] &
        {f_2^{-1}(g(a)+g(b))}
    \end{tikzcd}
    \]
    where the vertical maps are equivalences because the square in 
    the statement of the lemma is cartesian.
    Now suppose $f_2$ is equifibered, then the bottom map is an equivalence
    for all $a, b \in N_1$ and hence so is the top map. 
    This shows that $f_1$ is equifibered.
    The other direction follows similarly:
    since we assume that $\pi_0(g)$ is surjective, it suffices
    to check \cref{obs:eqf=equi-fibered} at $g(a), g(b) \in N_2$
    for all $a,b \in N_1$.
\end{proof}

\begin{lem}\label{lem:submonoid-eqf}
    A monomorphism $i\colon M \into N$ of commutative monoids is equifibered 
    if and only if $\pi_0 M \subseteq \pi_0 N$ closed under factoring, 
    i.e.\ whenever $[x_1],[x_2] \in \pi_0 N$ satisfy 
    $[x_1] + [x_2] \in \pi_0 M$ then $[x_1]$ and $[x_2] $ are both in $\pi_0 M$.
\end{lem}
\begin{proof}
    The inclusion is equifibered if and only if the following square is cartesian:
    \[
    \begin{tikzcd}
        M \times M \ar[r, "+"] \ar[d, "i\times i"', hook] & M \ar[d, "i", hook] \\
        N \times N \ar[r, "+"] & N .
    \end{tikzcd}
    \]
    Since both vertical maps are monomorphisms, the square is cartesian exactly when it satisfies that the connected component 
    $([x_1], [x_2]) \in \pi_0(N \times N)$ is hit by $i\times i$
    if and only if its image $[x_1] + [x_2] \in \pi_0(N)$ is hit by $i$.
    This is exactly equivalent to the condition on the submonoid $\pi_0(N) \subset \pi_0(M)$ described in the lemma.
\end{proof}

\begin{example}\label{ex:unit-inclusion-eqf}
    For any commutative monoid $M$ the inclusion of the submonoid of invertible elements $M^\times \subset M$ is equifibered.
    Indeed, if $m+m' \in M^\times$ is invertible, then so are $m$ and $m'$.
\end{example}

The following lemma shows that there are no interesting equifibered
maps between grouplike commutative monoids. See \cite[\S 1]{gepner-universality} for a recollection on grouplike commutative monoids.
\begin{lem}\label{lem:eqf-grouplike}
    Suppose $f\colon M \to N$ is equifibered and $N$ is group-like,
    then $f$ is an equivalence.
\end{lem}
\begin{proof}
    By \cref{prop:TFAE-to-equifibered}.\ref{item:eqf-n} the kernel $f^{-1}(0)$ of an equifibered map is contractible.
    Moreover, \cref{obs:eqf=equi-fibered} tells us that for all $x \in N$ the map
    \[
        + \colon f^{-1}(x) \times f^{-1}(-x) \too f^{-1}(0) \simeq \pt
    \]
    is an equivalence because $f$ is equifibered. 
    This implies that $f^{-1}(x)$ is a retract of $\pt$ and thus is contractible.
    Since we showed this for all $x \in N$, $f$ is an equivalence.
\end{proof}

\subsubsection{The contrafibered-equifibered factorization system}
By \cref{prop:TFAE-to-equifibered}.\ref{item:eqf-rlp} equifibered morphisms are characterized by a lifting property.
We now study the resulting factorization system on $\CMon$, which will imply several pleasant properties of equifibered maps.
We refer the reader to the appendix for a brief introduction to factorization systems.

\begin{defn}
    We say that a morphism of commutative monoids $f\colon M \to N$ is \hldef{contrafibered}
    if it is left-orthogonal to all equifibered morphisms.
\end{defn}

\begin{example}\label{ex:diagonal-ctf}
    For any finite set $A$ the diagonal map
    \[
        \Delta_A\colon \xF(*) \too \prod_A \xF(*) = \xF(A)
    \]
    which sends the generator $\ast \in \xF(\ast)$ to the sum $\sum_{a \in A} a \in \xF(A)$ is contrafibered.
    Indeed, it has the left lifting property with respect to any equifibered map $f\colon M \to N$ because the relevant square
    \[\begin{tikzcd}
        \Map_{\CMon}(\xF(A), M) \ar[r, "\Delta_A^*"] \ar[d, "f_!"] &  
        \Map_{\CMon}(\xF(*), M) \ar[d, "f_!"] \\
        \Map_{\CMon}(\xF(A), N) \ar[r, "\Delta_A^*"] &  
        \Map_{\CMon}(\xF(*), N)
    \end{tikzcd}\]
    can be identified with the square from \cref{prop:TFAE-to-equifibered}.\ref{item:eqf-n}, which is cartesian.
\end{example}

\begin{lem}\label{lem:formal-ctf-eqf}
    The contrafibered and equifibered morphisms form a factorization system on $\CMon$.
\end{lem}
\begin{proof}
    It follows from \cite[Proposition 5.5.5.7]{HTT} 
    (see \cite[Proposition 3.1.18]{ANEL} or \cite[04PN]{Kerodon})
    that for any small collection of morphisms $S$ in a presentable $\infty$-category there is a factorization system $({}^\bot(S^\bot), S^\bot)$.
    The claim then follows by setting $S = \{ (\Delta\colon \xF(*) \to \xF(*) \oplus \xF(*)) \}$ such that $S^\bot$ are the equifibered maps.
\end{proof}

\begin{example}\label{ex:cft/eqf-factorization}
    For a finite set $A$ let $A\cdot -\colon \xF(*) \to \xF(*)$ be the unique map that sends the generator to $A \in \Fin^\simeq = \xF(*)$.
    We can construct the contra/equifibered factorization of this by hand as
    \[  
        \xF(*) \xtoo{\Delta_A} \xF(A) \xtoo{\nabla} \xF(*)
    \]
    where the first map is the diagonal, which is contrafibered by \cref{ex:diagonal-ctf}, and the second map is the fold map, i.e.~the free map on $A \to *$.
\end{example}

There are many contrafibered maps between non-free monoids:

\begin{example}\label{ex:contrafibered-grouplike}
    Suppose we are given a group-like commutative monoid $G$ and a morphism $f\colon M \to G$.
    Since equifibered and contrafibered maps form a factorization system
    there is a factorization $f\colon M \to G' \to G$ where the first map
    is contrafibered and the second map is equifibered.
    However, \cref{lem:eqf-grouplike} implies that the second map is an equivalence.
    Consequently, any morphism into a group-like commutative monoid is contrafibered.
\end{example}

\begin{lem}\label{lem:colim-of-eqf-slice}
    For every commutative monoid $M$ the full subcategory 
    $\CMon_{/M}^\eqf \subset \CMon_{/M}$ on the equifibered maps 
    is closed under small limits and under sifted colimits.
    If $M$ is a free commutative monoid, then this subcategory 
    is in fact closed under small colimits.
\end{lem}
\begin{proof}
    First, we note that because equifibered maps are the right-class
    of a factorization system the inclusion
    $\CMon_{/M}^\eqf \subset \CMon_{/M}$
    is a right adjoint \cite[Observation 2.3.6]{envelopes} 
    and hence preserves all limits.
    
    For a sifted diagram $F\colon I \to \CMon_{/M}$ the colimit may be computed
    on underlying spaces since $\CMon_{/M} \to \CMon$ preserves colimits
    and $\CMon \to \calS$ preserves sifted colimits.
    Because colimits in $\calS$ are universal (i.e.~stable under base change) \cite[Lemma 6.1.3.14.(1)]{HTT} we can compute
    \[
        M^2 \times_M \colim_{i \in I} F(i) 
        \simeq \colim_{i \in I} (M^2 \times_M F(i))
        \simeq \colim_{i \in I} F(i)^2
        \simeq \colim_{i \in I} F(i) \times \colim_{j \in I} F(j)
        \simeq (\colim_{i \in I} F(i))^2
    \]
    where the penultimate equivalence uses that $I$ is sifted.
    This shows that $\colim_I F(i) \to M$ is equifibered.
    
    To prove the second part of the lemma it suffices by \cref{sifted+binary+initial=all} to show that 
    $\CMon_{/M}^\eqf \subset \CMon_{/M}$ contains the initial object and is closed under binary coproducts when $M$ is free.
    For the initial object we know that $0 \to \xF(X)$ is equifibered because it is free on $\emptyset \to X$.
    The coproduct of $N_1 \to M$ and $N_2 \to M$ is the composite map
    \[
        N_1 \times N_2 \too M \times M \xrightarrow{\ +\ } M.
    \]
    The first map is equifibered as a product of equifibered maps 
    and the second map is equifibered because for
    $M = \xF(X)$ the addition $+\colon  \xF(X) \times \xF(X) \to \xF(X)$
    is equivalent to the free map on the fold map 
    $\nabla\colon  X \amalg X \too X$.
\end{proof}

\begin{rem}\label{rem:spans}
    For a finite covering $p\colon Y \to X$ one can construct a transfer map $\mrm{trf}_p\colon \xF(X) \to \xF(Y)$ in $\CMon$ by summing over the fibers of $p$, by restricting the pullback functor $p^*\colon \calS_{/X} \to \calS_{/Y}$ to the groupoids of finite sets over $X$ and $Y$.
    When $X$ and $Y$ are finite sets this agrees with a sum of diagonal maps $\bigoplus_{x \in X} \Delta_{p^{-1}(x)}$ as in \cref{ex:diagonal-ctf} and is thus contrafibered.
    One can assemble these transfer maps into an equivalence of $\infty$-categories 
    \[
        \Span^{\rm fcov, all}(\calS) \simeq \CMon^\free
    \]
    that sends a space $X$ to $\xF(X)$.
    A morphism in $\Span^{\rm fcov, all}(\calS)$ is a span $X \xleftarrow{p} Y \xto{f} Z$ where $p$ is a finite covering and the equivalence sends them to $\xF(f) \circ \mrm{trf}_p$.
    Under this equivalence of $\infty$-categories the forward maps $X = X \to Z$ correspond to the equifibered (i.e.~free) maps and the backward maps $X \leftarrow Y = Y$ correspond to the contrafibered maps.
    In particular, a map $\xF(X) \to \xF(Y)$ is contrafibered if and only if it is $\mrm{trf}_p$ for some finite covering $p\colon Y \to X$.
\end{rem}

\begin{war}\label{war:ctf-pullback}
    Contrafibered maps are closed under small colimits in $\Ar(\CMon)$ (\cite[Proposition 5.2.8.6.(7)]{HTT}), which includes direct sums and thus products.
    In an earlier version of this paper, we implicitly and incorrectly assumed that contrafibered maps are also closed under pullbacks, but this is not the case.
    To illustrate this, consider the diagram
\[\begin{tikzcd}
	{\xF(*)} & {\xF(*)} & {\xF(*)} \\
	{\xF(*)} & {\xF(*)} & {\xF(\{a, b\})}
	\arrow["{\times 2}", from=2-1, to=2-2]
	\arrow["\nabla"', from=2-3, to=2-2]
	\arrow["{\times 2}", from=1-1, to=1-2]
	\arrow["{\times 2}"', from=1-3, to=1-2]
	\arrow["\Delta", from=1-3, to=2-3]
	\arrow[Rightarrow, no head, from=1-2, to=2-2]
	\arrow[Rightarrow, no head, from=1-1, to=2-1]
\end{tikzcd}\]
    which defines a cospan in $\Ar(\CMon)$ such each of the objects involved is a contrafibered morphism, namely $\id_{\xF(*)}$ and $\Delta$.
    By \cref{ex:free-pullbacks} the pullback will be some map $f\colon \xF(X) \to \xF(\{1,2,3,4\})$ where $X$ is a $1$-type with infinitely many connected components.
    For this map to be contrafibered, by \cref{rem:spans}, it would have to be $\trf_{p}$ for some finite covering $p\colon \{1,2,3,4\} \to X$.
    Such a covering can hit at most four components, so $f=\mrm{trf}_p$ would have to be $0$ on all but at most four components of $X$.
    But the kernel of $f$ is the pullback of the kernels of the vertical maps in the above diagram, which are all trivial, so $f^{-1}(0) = 0$ -- a contradiction.
\end{war}

\subsubsection{Pseudo-free monoids}

The proof of \cref{lem:colim-of-eqf-slice} did not really use that $M$ is free, but only the property that the addition map is equifibered.
Studying this in more detail we will see that this condition almost implies that $M$ is free.
In particular, we will be able to use this to show that free commutative monoids are closed under retracts and finite limits in $\CMon$.

\begin{defn}\label{def:pseudo-free}
    A commutative monoid $M$ is \hldef{pseudo-free} if the 
    addition map $+\colon M \times M \to M$ is equifibered.
\end{defn}

Every free monoid is pseudo-free since the addition map
$+\colon  \xF(X) \times \xF(X) \to \xF(X)$
is equivalent to the free map on the fold map $\nabla\colon  X \amalg X \too X$.
We have a partial converse as follows:

\begin{lem}\label{lem:eqf-addition}
    For every pseudo-free commutative monoid $M$ there is a free submonoid $\xF(X) \subset M$ where $\pi_0(X) \subset \pi_0(M)$ consists of the indecomposable elements,
    i.e.\ those non-zero $a \in \pi_0(M) \setminus \{0\}$ for which $a = b+c$ implies $b=0$ or $c=0$.
    In particular, if $\pi_0(M)$ is generated by indecomposables,
    then $M$ is free.
\end{lem}
\begin{proof}
    Let $X \subseteq M$ denote the subspace on those connected components that are indecomposables in $\pi_0 M$.
    We will show that the induced map $f\colon  \xF(X) \to M$ is a monomorphism.
    
    We begin by showing that the $0$-component $[0] \subset M$ is contractible.
    By \cref{ex:unit-inclusion-eqf} the inclusion $M^\times \hookrightarrow M$ of the units is equifibered 
    and applying \cref{lem:eqf-base-change} to the pullback square in the definition of ``equifibered'' shows that the addition map on $M^\times$ is also equifibered.
    By \cref{lem:eqf-grouplike} and \cref{rem:addition-equivalence} $M^\times$ is contractible (as it is grouplike and pseudo-free) and hence so is $[0] \subset M$.
    
    Next we show that $f$ is equifibered. 
    Since $X \subset M$ corresponds to the indecomposables in $M$, 
    its preimage under the addition map $+\colon M^2 \too M$ is a disjoint
    union of the form $(X \times [0]) \amalg ([0] \times X)$.
    As we have shown that $[0] \subset M$ is contractible 
    this results in the cartesian square
    \[
    \begin{tikzcd}
        X \amalg X \ar[r] \ar[d, "\nabla"'] \ar[dr, phantom, very near start, "\lrcorner"] &  M \times M \ar[d, "+"] \\
        X \ar[r, ] & M.
    \end{tikzcd}
    \]
    Because $+\colon M \times M \to M$ is equifibered 
    characterization \ref{item:eqf-rep-free} in \cref{prop:TFAE-to-equifibered}
    applied to $X \to M$ yields the cartesian square
    \[
    \begin{tikzcd}
        \xF(X \amalg X) \ar[r] \ar[d, "\xF(\nabla)"']  \ar[dr, phantom , very near start, "\lrcorner"] &
        M \times M \ar[d, "+"] \\
        \xF(X) \ar[r] & 
        M.
    \end{tikzcd}
    \]
    Note that the left most map is canonically equivalent to the addition map 
    $+\colon  \xF(X) \times \xF(X) \too \xF(X)$. 
    Furthermore, under this equivalence the top and bottom horizontal composites are identified with $f \times f$ and $f$ respectively, which proves that $f\colon \xF(X) \to M$ is equifibered.
    
    Finally, we prove that $f$ is a monomorphism.
    Let $Y \subset \xF(X)$ be the subspace of those $y$ for which $f^{-1}(f(y))$ is contractible. 
    This is a union of components, and it is closed under addition because
    \[
        f^{-1}(f(y_1+y_2))  = f^{-1}(f(y_1) + f(y_2)) \simeq f^{-1}(f(y_1)) \times f^{-1}(f(y_2)) \simeq \pt
    \]
    for any $y_i \in Y$ since $f$ is equifibered (\cref{obs:eqf=equi-fibered}).
    By construction $Y$ contains $X \subseteq \xF(X)$, and we showed that it contains $0$, so it follows that $Y = \xF(X)$ and that $f$ is a mono.
\end{proof}

\begin{example}
    Not every pseudo-free monoid is free.
    Using tools from \cref{subsec:Csp-and-disjunctive} we can argue that
    the category $\Sets^{\le\omega}$ of countable sets is $\amalg$-disjunctive and hence the coproduct map
    $ \amalg \colon  (\Sets^{\le \omega})^2 \too \Sets^{\le \omega} $
    is equifibered by \cref{lem:coproduct-disjunctive-colim} 
    applied to $J = \{0,1\}$.
    In particular, passing to maximal subgroupoids we obtain 
    a commutative monoid $M : = (\Sets^{\le \omega})^\simeq$
    such that $+\colon  M^2 \to M$ is equifibered.
    Note that $\pi_0 M \cong \mbb{N} \cup \{\infty\}$ with addition 
    defined by $n + \infty = \infty$. 
    This monoid is not generated by indecomposables.
    
    For another example, let $X_i \in \calS$ be an infinite collection of non-empty spaces. 
    Then $\prod_{i\in I} \xF(X_i)$ is pseudo-free, but not free.
    Indeed, $\prod_{i \in I} \mathbb{N}\langle \pi_0 X_i \rangle$ is not generated by indecomposables.
\end{example}

\begin{cor}\label{cor:pseudo-free-with-map-to-N}
    A pseudo-free commutative monoid $M$ is free if and only if 
    there exists a morphism $h\colon  M \to \mbb{N}$ such that the kernel $h^{-1}(0)$ is connected.
\end{cor}
\begin{proof}
    If $M$ is free, then we can use $M = \xF(X) \to \xF(*) \to \pi_0\xF(*) \cong \mbb{N}$.
    Conversely, assuming we have $h$, it suffices by \cref{lem:eqf-addition} to show that $\pi_0(M)$ is generated by indecomposables.
    To argue by contradiction, let $x \in \pi_0(M)$ be an element with minimal $h(x) \in \mbb{N}$ such that $x$ cannot be written as a sum of indecomposables.
    Since $x$ is not indecomposable we may write it as $x = a + b$, where neither $a$ nor $b$ are in $[0] = h^{-1}(0)$.
    But this means that $h(a), h(b) > 0$ and hence $h(a), h(b) < h(x)$.
    By the minimality of $x$ both $a$ and $b$ must be a sum of indecomposables -- a contradiction.
\end{proof}

\begin{obs}\label{obs:pseudo-free-limits}
    The full subcategory $\CMon^{\rm ps-free} \subseteq \CMon$ spanned by the pseudo-free commutative monoids is closed under all limits
    because equifibered maps are.
\end{obs}

\begin{cor}\label{cor:retract-of-free}
    The full subcategory $\CMon^{\free} \subseteq \CMon$ spanned by the free commutative monoids is closed under retracts.
\end{cor}
\begin{proof}
    Consider a retraction $i\colon  M \adj \xF(X) :\! r$.
    Then the addition map on $M$ is equifibered since it is a retract 
    of the equifibered addition map $+\colon \xF(X)^2 \to \xF(X)$,
    so $M$ is pseudo-free.
    Now we apply \cref{cor:pseudo-free-with-map-to-N} using the map $M \to \xF(X) \to \mbb{N}$.
\end{proof}

\begin{cor}\label{cor:finite-limits}
    The full subcategory $\CMon^{\free} \subseteq \CMon$ spanned by the free commutative monoids is closed under finite limits.
\end{cor}
\begin{proof}
    The terminal commutative monoid is free on the empty set, so it suffices to show that free commutative monoids are closed under pullbacks.
    Any pullback $M := \xF(X) \times_{\xF(Z)} \xF(Y)$ is automatically pseudo-free by \cref{obs:pseudo-free-limits}.
    So by \cref{cor:pseudo-free-with-map-to-N} we only need to construct a morphism $M \to \mbb{N}$ with trivial kernel.
    To do so we fit $M$ in another pullback square as follows
    \[
   \begin{tikzcd}
	M & {\xF(X) \times \xF(Y)} & {\mathbb{N}} \\
	{\xF(Z)} & {\xF(Z) \times \xF(Z).}
	\arrow[from=1-1, to=1-2]
	\arrow[from=1-1, to=2-1]
	\arrow[from=1-2, to=2-2]
	\arrow["\Delta", from=2-1, to=2-2]
	\arrow["\lrcorner"{anchor=center, pos=0.125}, draw=none, from=1-1, to=2-2]
	\arrow[from=1-2, to=1-3]
    \end{tikzcd}
    \]
    Here $\xF(X) \times \xF(Y) \to \mbb{N}$ is any morphism with trivial kernel.
    The map $M \to \xF(X) \times \xF(Y)$ also has trivial kernel because $\Delta\colon  \xF(Z) \to \xF(Z) \times \xF(Z)$ does.
    So \cref{cor:pseudo-free-with-map-to-N} applies and $M$ is free.
\end{proof}

\begin{war}\label{war:discrete-pullback}
    The analogue of \cref{cor:finite-limits} for discrete commutative monoids is \textit{false}. 
    Indeed, consider the submonoid $M \coloneq \left\{(a,b): a+b \text{ is even} \right\} \subseteq \bbN \times \bbN$,
    which may be written as a pullback of free discrete commutative monoids
    \[\begin{tikzcd}
	M & {\bbN \oplus \bbN} \\
	\bbN & \bbN.
	\arrow[from=1-1, to=2-1]
	\arrow[hook, from=1-1, to=1-2]
	\arrow["\lrcorner"{anchor=center, pos=0.125}, draw=none, from=1-1, to=2-2]
	\arrow["+", from=1-2, to=2-2]
	\arrow["\cdot 2", hook, from=2-1, to=2-2]
    \end{tikzcd}\]
    However, $M$ is not free since, for example, it has the relation $(1,1)+(1,1)=(2,0)+(0,2)$.
\end{war}

\begin{example}\label{ex:free-pullbacks}
    We now consider the homotopical analogue of the pullback in \cref{war:discrete-pullback} to see that it is indeed free.
    Concretely, we will show that there are pullback squares in $\CMon$
\[\begin{tikzcd}
	{\xF(\{(a,a), (a,b), (b,a), (b,b)\})} & {\xF(\{a, b\})} && {\xF(X)} & {\xF(*)} \\
	{\xF(*)} & {\xF(*)} && {\xF(*)} & {\xF(*).}
	\arrow[from=1-1, to=2-1]
	\arrow["{\times 2}", from=2-1, to=2-2]
	\arrow["\nabla", from=1-2, to=2-2]
	\arrow[from=1-1, to=1-2]
	\arrow[from=1-4, to=2-4]
	\arrow["{\times 2}", from=2-4, to=2-5]
	\arrow[from=1-4, to=1-5]
	\arrow["{\times 2}", from=1-5, to=2-5]
\end{tikzcd}\]
    where $X$ is a $1$-type with $\pi_0(X) \cong \bbN$.
    The forgetful functor $\CMon \to \calS$ detects limits and since all spaces involved are $1$-types it will suffice to compute the following (homotopy) pullbacks of $1$-groupoids:
\[\begin{tikzcd}
	P & {(\Fin_{/\{a,b\}})^\simeq} && Q & {\Fin^\simeq} \\
	{\Fin^\simeq} & {\Fin^\simeq} && {\Fin^\simeq} & {\Fin^\simeq}
	\arrow[from=1-1, to=2-1]
	\arrow["{\times \{0,1\}}", from=2-1, to=2-2]
	\arrow["\mrm{forget}", from=1-2, to=2-2]
	\arrow[from=1-1, to=1-2]
	\arrow[from=1-4, to=2-4]
	\arrow["{\times \{0,1\}}", from=2-4, to=2-5]
	\arrow[from=1-4, to=1-5]
	\arrow["{\times \{a,b\}}", from=1-5, to=2-5]
\end{tikzcd}\]
    Objects of $P$ can be presented as pairs $(C,\alpha)$ of a finite set $C$ and a map $\alpha\colon C \times \{0,1\} \to \{a,b\}$ with the symmetric monoidal structure given by 
    \[(C_1,\alpha_1)  + (C_2,\alpha_2) = \left(C_1\sqcup C_2, (C_1 \sqcup C_2)\times \{0,1\} \xrightarrow{\alpha_1 \sqcup \alpha_2} \{a,b\} \sqcup \{a,b\} \xrightarrow{\nabla}  \{a,b\}\right). \]
    This symmetric monoidal groupoid is freely generated by objects $(\ast,\alpha)$ where $\alpha$ runs over the four elements of $\Map(\{0,1\}, \{a,b\})$.
    It is interesting to contrast this with the pullback in \cref{war:discrete-pullback} where $M$ was (non-freely) generated by three elements:
    the difference is that in $P$ the two generators $(a, b)$ and $(b,a)$ differ.
    
    Objects of $Q$ can be presented as triples $(C, D, \alpha\colon C \times \{0,1\} \cong D \times \{0, 1\})$, with the symmetric monoidal structure given by disjoint union.
    We could compute $Q$ by hand, but for simplicity we will use that by \cref{cor:finite-limits} $Q = \xF(X)$ for $X \subset Q$ the indecomposables. 
    It formally follows that $X$ must be a $1$-type with countably many components, but we still have to argue that $\pi_0(X)$ is not finite.
    Consider the object $(\bbZ/n, \bbZ/n, \alpha)$ where $\alpha(k,0) = (k,1)$ and $\alpha(k,1) = (k+1,0)$.
    This cannot be written as a disjoint union in $Q$ because all the elements of the $\bbZ/n$ are ``interlinked''.
    Therefore, we have exhibited an infinite family on non-isomorphic indecomposable objects of $Q$.
\end{example}

\subsection{Equifibered symmetric monoidal functors}
\label{subsec:eqf-sm-functors}

In this section we generalize the notion of equifibered maps from commutative monoids, i.e.~symmetric monoidal $\infty$-groupoids, to arbitrary symmetric monoidal $\infty$-categories.

\begin{defn}\label{defn:sm-cat}
    A \hldef{symmetric monoidal $\infty$-category} is a commutative monoid in $\Cat$,
    i.e.\ a functor $\calC\colon  \Fin_* \to \Cat$ such that the map
    \[
        \calC(A_+) \longrightarrow \prod_{a \in A} \calC(\{a\}_+)
    \]
    is an equivalence for all $A_+ \in \Fin_*$.
    We let $\hldef{\SM} \subset \Fun(\Fin_*, \Cat)$ denote the full 
    subcategory of symmetric monoidal $\infty$-categories.
    We refer to morphisms in this category as \hldef{symmetric monoidal functors}.
\end{defn}

\begin{example}\label{ex:cocartesian-monoidal}
    Let $\calC$ be an $\infty$-category with finite coproducts.
    Then the coproduct defines a symmetric monoidal structure on $\calC$ (see \cite[Section 2.4.3]{HA}).
    Let $\hldef{\Cat^\amalg} \subseteq \Cat$ 
    denote the subcategory whose objects are $\infty$-categories with finite coproducts
    and whose morphisms are finite-coproduct-preserving functors.
    By \cite[Variant 2.4.3.12.]{HA} we may regard $\Cat^\amalg$ as a full subcategory of
    $\SM$.
\end{example}

\begin{rem}\label{rem:HA-apologists}
    It should be possible to set up the theory $\infty$-properads entirely independently of Lurie's book project Higher Algebra \cite{HA}.
    We believe that there would be some pedagogical value in this, as one in particular obtains a theory of $\infty$-operads as symmetric monodial $\infty$-categories equifibered over $\Fin$ (see \cref{thm:monic-prpd=operad}) without having to use Lurie's rather subtle definition in terms of inert-cocartesian lifts.
    In the present work there are some mild dependencies on \cite{HA}. 
    Crucially, we will need the cocartesian symmetric monoidal structure on $\Fun(J, \Fin)$ for certain categories $J$. 
    (Namely in defining $\Csp \in \SM$ in \cref{defn:Csp} and, relatedly, in describing $\widehat{\bmC}$ in \cref{cor:boldC-description}.)
    These dependencies could be avoided by defining $\Fin^\amalg \coloneq \Ar^\act(\Fin_*)$ and taking \cref{lem:cocart-operad-Fun} as the definition of $\Fun(J,\Fin)^\amalg$, but for clarity of exposition we shall take \cite[\S 2.4.3]{HA} as our definition instead.
\end{rem}

A substantial portion (but not all) of the theory developed in \cref{subsec:Eqf-theory} carries over to the setting of symmetric monoidal $\infty$-categories.
We begin with the definition:

\begin{defn}\label{defn:sm-eqf}
    A symmetric monoidal functor $F\colon  \calC \to \calD$ is called \hldef{equifibered} 
    if the square 
    \[
        \begin{tikzcd}
            \calC \times \calC \ar[r, "\otimes"] \ar[d, "F\times F"'] &
            \calC \ar[d, "F"] \\
            \calD \times \calD \ar[r, "\otimes"] & 
            \calD
        \end{tikzcd}
    \]
    is a pullback in $\Cat$.
\end{defn}

To fully analyse equifibered functors we will crucially rely on the fact $\SM$ embeds into the $\infty$-category $\Fun(\Dop,\CMon)$ of simplicial commutative monoids. 
With this in mind we recall some basic facts about Segal spaces.

\subsubsection{Recollection on Segal spaces}

A \hldef{Segal space} is a simplicial space $X_{\bullet}\colon  \simp^{\op} \too \calS$ 
    satisfying the Segal condition, i.e.\ the natural map $X_n \to X_1 \times_{X_0} \cdots \times_{X_0} X_1$
    is an equivalence for all $n$.
    We denote by $\hldef{\Seg_{\simp^{\op}}(\calS)} \subseteq \Fun(\simp^{\op},\calS)$ the full subcategory of Segal spaces. 
To a Segal space $X_\bullet$ one can associate a homotopy category $\mathrm{ho}(X)$ whose objects are the points of $X_0$ and whose mapping sets are the connected components of the fibers of $(d_1,d_0)\colon X_1 \to X_0 \times X_0$.
We refer the reader to \cite{rezk} for a detailed description.
We let $X_\bullet^{\rm eq} \subset X_\bullet$ denote the largest Segal subspace such that $\mathrm{ho}(X_\bullet^{\rm eq})$ is a groupoid.
A Segal space $X_\bullet$ is called \hldef{complete} if the map $s_0\colon  X_0 \to X_1^{\rm eq}$ is an equivalence, or equivalently if $X_\bullet^{\rm eq}$ is a constant simplicial space (that is, all its face and degeneracy maps are equivalences).
    We denote by $\hldef{\CSeg_{\simp^{\op}}(\calS)} \subseteq \Seg_{\simp^{\op}}(\calS)$ the full subcategory of complete Segal spaces.
\begin{defn}
    We define the \hldef{nerve} functor as the Yoneda embedding followed by restriction along the inclusion
    $\simp \hookrightarrow \Cat$.
    \[\hldef{\xN_\bullet} \colon \Cat \too \Fun(\simp^\op,\calS), \qquad \calC \longmapsto \left(\hldef{\xN_\bullet(\calC)} \colon [n] \mapsto  \Map_{\Cat}([n],\calC)\right)\]
\end{defn}

It is fully faithful, and the essential image is characterized by the Segal and completeness conditions.

\begin{thm}[Joyal--Tierney]\label{complete-Segal-spaces}
    The nerve functor
    $\xN_{\bullet}\colon  \Cat \to \Fun(\simp^{\op},\calS)$ 
    is fully faithful and 
    a simplicial space $X\colon \simp^\op \to \calS$ lies in its essential image
    if and only if it is a complete Segal space.
    Hence, the nerve induces an equivalence of $\infty$-categories:
    \[ \Cat \simeq \CSeg_{\simp^{\op}}(\calS).\]
\end{thm}
\begin{proof}[Remarks on the theorem]
    If we choose complete Segal spaces as our model for $\infty$-categories, then this statement is a tautology.
    However, our preferred model is quasicategories, so the theorem amounts to the equivalence between quasicategories and complete Segal spaces, which was shown by Joyal and Tierney \cite{JT06}.
    See \cite[Corolary 4.3.17]{Lurie-2Cat} for a more detailed explanation of how to translate Joyal and Tierney's result to $\infty$-categories.
    Alternatively, see \cite{HebestreitSteinebrunner} for a model-independent proof.
\end{proof}

Note that the Yoneda image of $[1] \in \simp$ generates $\Seg_\Dop(\calS)$ under colimits.
Indeed, $\Seg_\Dop(\calS)$ is a localization of $\Fun(\Dop,\calS)$, and is therefore generated under colimits by the simplices $[n] \in \simp$.
Since $[0]$ is a retract of $[1]$ it remains to observe that, by definition, the inclusion $[1] \amalg_{[0]} \cdots \amalg_{[0]} [1] \hookrightarrow [n]$ is a Segal equivalence.
By \cref{complete-Segal-spaces}, $\Cat$ is a localization of $\Seg_\Dop(\calS)$ and thus $[1]$ also generates $\Cat$ under colimits. 
By \cite[Corollary 2.5]{Yan21} this is equivalent to $\xN_1\colon \Cat \to \calS$ being conservative, which one can also see directly.

\begin{obs}\label{obs:N1-conservative}
    $\Cat$ is generated under colimits by $[1]$ and $\xN_1 \colon \Cat \to \calS$ is conservative.
\end{obs}

Note that since the nerve functor
$\xN_\bullet \colon \Cat \hookrightarrow \Fun(\Dop,\calS)$ 
is fully faithful and limit preserving, the Segal condition on a commutative monoid in $\Cat$
may be checked level-wise on the nerve $\xN_\bullet(\calC)$.
Concretely, a functor
$\calC \colon \Fin_\ast \to \Cat$
defines a symmetric monoidal $\infty$-category if and only if 
$\xN_n(\calC)$
is a commutative monoid for all $n$.
See \cite[Example 5.7]{CH19} for a more general discussion of how to combine two Segal-type structures.
We record this for future use.

\begin{cor}
    The nerve functor 
    $\xN_\bullet \colon \Cat \hookrightarrow \Fun(\Dop,\calS)$ 
    gives rise to a pullback square of fully faithful functors
    \[
    \begin{tikzcd}
        \SM \ar[r, "\xN_\bullet", hook] \ar[d, hook] \ar[dr, very near start, phantom, "\lrcorner"] & \Fun(\simp^\op, \CMon) \ar[d, hook] \\
        \Fun(\Fin_*, \Cat) \ar[r, "\xN_\bullet", hook] & \Fun(\Fin_* \times \simp^\op, \calS).
    \end{tikzcd}
    \]
    We refer to the top horizontal functor as \hldef{symmetric monoidal nerve}.
    This is a pullback square in $\PrR$ and in particular the symmetric monoidal nerve has a left adjoint.
\end{cor}

\subsubsection{Equifibered functors, free functors, and the nerve} 

By \cite[Corollary 5.2.8.18]{HTT} the contrafibered-equifibered factorization lifts to a factorization system $(\Fun(\Dop,\CMon^\ctf),\Fun(\Dop,\CMon^\eqf))$ on $\Fun(\Dop,\CMon)$ whose right part is related to equifibered functors via the nerve functor: 

\begin{lem}\label{lem:eqf-detected-on-N1}
    For a symmetric monoidal functor $F\colon \calC \to \calD$ the following are equivalent:
    \begin{enumerate}[(1)]
        \item $F$ is equifibered,
        \item the map $\xN_{\bullet}(F)\colon \xN_\bullet \calC \to \xN_\bullet \calD$
            of simplicial commutative monoids is level-wise equifibered,
        \item the map $\xN_1(F)\colon \xN_1 \calC \to \xN_1 \calD$
            of commutative monoids is equifibered,
        \item $F$ has the right-lifting-property with respect to the diagonal map $\Delta\colon  \xF([1]) \to \xF([1]) \times \xF([1])$.
    \end{enumerate}
\end{lem}
\begin{proof}
    The equivalence $(1) \Leftrightarrow (2)$ holds because the nerve functor $\xN_\bullet\colon \SM \to \Fun(\Dop, \CMon)$ is conservative and commutes with limits.
    This suffices because the definition of equifibered maps only involves limits.
    Similarly, the equivalence $(1) \Leftrightarrow (3)$ holds because $\xN_1\colon \SM \to \CMon$ is conservative (\cref{obs:N1-conservative}) and commutes with limits.
    The equivalence $(3) \Leftrightarrow (4)$ holds since $\Delta \colon \xF([1]) \to \xF([1]) \times \xF([1])$ corepresents the addition map on $\xN_1$. 
\end{proof}

\begin{cor}\label{cor:full-subcat-eqf}
    The inclusion of a full symmetric monoidal subcategory $\calC \subset \calD$ is equifibered if and only if
    for all $x,y \in \calD$ with $x \otimes y \in \calC$ we have that $x, y \in \calC$.
\end{cor}
\begin{proof}
    By \cref{lem:submonoid-eqf} the condition is equivalent to $\xN_0\calC \subset \xN_0\calD$ being equifibered.
    This also implies that $\xN_1\calC \to \xN_1\calD$ is equifibered as, by fully faithfulness, it is the base-change of $\xN_0\calC^2 \to \xN_0\calD^2$ along $\xN_1\calD \to \xN_0\calD^2$.
    Therefore, the claim follows from \cref{lem:eqf-detected-on-N1}.
\end{proof}

\begin{example}\label{ex:maximal-subgpd-eqf}
    For a symmetric monoidal $\infty$-category $\calC$ the following are equivalent:
    \begin{enumerate}[(1)]
        \item The functor $\otimes\colon  \calC \times \calC \to \calC$ is conservative.
        \item The inclusion of the maximal subgroupoid $\calC^\simeq \subset \calC$ is equifibered.
        \item $s_0\colon  \xN_0\calC \to \xN_1\calC$ is equifibered.
    \end{enumerate}
\end{example}
\begin{proof}
    For $(1) \Leftrightarrow (2)$ we observe that by definition the functor $\calC^\simeq \to \calC$ is equifibered 
    if and only if $\calC^{2\simeq} \to \calC^{\simeq} \times_{\calC} \calC^2$ is an equivalence.
    This is always the inclusion of a wide subcategory.
    It is full if and only if:
    whenever $f,g \in \Ar(\calC)$ are such that $f \otimes g$ is an equivalence then $f$ and $g$ are both equivalences.
    In other words, if and only if $\otimes$ is conservative.
    
    For $(2) \Leftrightarrow (3)$ we use that by \cref{lem:eqf-detected-on-N1} the functor $\calC^\simeq \to \calC$ is equifibered 
    if and only if $\xN_1(\calC^\simeq) \to \xN_1\calC$ is.
    But $s_0\colon \xN_0(\calC^\simeq) \to \xN_1(\calC^\simeq)$ and
    $\xN_0(\calC^\simeq) \to \xN_0\calC$ are equivalences,
    so we may equivalently ask $s_0\colon \xN_0\calC \to \xN_1 \calC$ to be equifibered.
\end{proof}

For cocartesian symmetric monoidal functors we have the following generalization of \cref{obs:eqf=equi-fibered}.
\begin{lem}\label{lem:eqf-cocartesian}
    Let $p\colon  \calE \to \calC$ be a symmetric monoidal functor that is also a cocartesian fibration.
    Suppose furthermore that 
    $\otimes \colon \calE \times \calE \to \calE$ 
    sends $(p\times p)$-cocartesian edges to $p$-cocartesian edges.
    Then the following are equivalent:
    \begin{enumerate}
        \item $p$ is an equifibered symmetric monoidal functor.
        \item For any two objects $x,y \in \calC$ the functor
        \[
            \calE_x \times \calE_y \longrightarrow \calE_{x \otimes y},
        \]
        obtained by restricting the monoidal structure of $\calE$ to the fibers, is an equivalence.
    \end{enumerate}
\end{lem}
\begin{proof}
    The assumptions guarantee that the square
    \[
        \begin{tikzcd}
            \calE \times \calE \ar[r, "\otimes"] \ar[d, "p \times p"'] &
            \calE \ar[d, "p"] \\
            \calC \times \calC \ar[r, "\otimes"] &
            \calC 
        \end{tikzcd}
    \]
    induces a morphism $\calE \times \calE \to (\calC \times \calC) \times_\calC \calE$ of cocartesian fibrations over $\calC \times \calC$.
    A morphism of cocartesian fibrations is an equivalence if and only if it induces an equivalence on all fibers.
    We may therefore check whether the square is cartesian by comparing the fiber of $p \times p$ at every point $(x, y) \in \calC \times \calC$ with the fiber of $p$ at its image $x \otimes y \in \calC$.
    This is precisely condition (2).
\end{proof}

\begin{rem}\label{rem:eqf-cocartesian}
    There is a variant of the straightening/unstraightening construction that induces an equivalence \cite[Proposition A.2.1]{Hinich-rectification}
    (see also \cite[\S 2.1, \S 2.4]{HA} and \cite[\S 2]{Ram22})
    \[
        \St\colon  \Catover{\calC}^{\otimes, \rm cocart}
        \xleftrightarrow{\;\simeq\;}
        \Fun^{\rm lax-\otimes}(\calC, \Cat^\times) :\!\Un
    \]
    between symmetric monoidal cocartesian fibrations over $\calC$
    and lax symmetric monoidal functors from $\calC$ to $\Cat$, equipped with the Cartesian symmetric monoidal structure.
    \cref{lem:eqf-cocartesian} says that under this equivalence the \emph{strong} symmetric monoidal functors $F\colon \calC \to \Cat^\times$
    exactly correspond to the \emph{equifibered} cocartesian symmetric monoidal functors $\Un(F) \to \calC$.
\end{rem}

We now relate equifibered functors to free symmetric monoidal functors:
\begin{prop}\label{prop:free-nerve-commute}
    The square of $\infty$-categories 
    \[\begin{tikzcd}
	\Cat && \SM \\
	{\Fun(\Dop,\calS)} && {\Fun(\Dop,\CMon)}
	\arrow["\xF", from=1-1, to=1-3]
	\arrow[phantom, very near start, "\lrcorner", from=1-1, to=2-3]
	\arrow["{\xN_\bullet}", hook, from=1-1, to=2-1]
	\arrow["\xF", from=2-1, to=2-3]
	\arrow["{\xN_\bullet}", hook, from=1-3, to=2-3]
    \end{tikzcd}\]
    canonically commutes and is a pullback square.
\end{prop}
\begin{proof}
    To show that the square commutes we need to show that for all $\calC \in \Cat$ and all $n$ the canonical map 
    $\xF(\xN_n(\calC)) \to \xN_n(\xF(\calC))$ 
    is an equivalence of spaces.
    The formula 
    $\xF(X) \simeq \colim_{A \in \Fin^\simeq} X^A$
    from \cite[Example 8.13]{CH19} that was recalled in \cref{lem:basic-properties-for-xF} also holds when $X$ is an $\infty$-category.
    The map of interest is therefore the induced map on maximal subgroupoids of the functor
    \[
        \colim_{A \in \Fin^\simeq} \Fun([n], \calC)^A
        \simeq \colim_{A \in \Fin^\simeq} \Fun([n], \calC^A)
        \to \Fun([n], \colim_{A \in \Fin^\simeq} \calC^A).
    \]
    Since $[n] \in \simp \subseteq \Cat$ is weakly contractible, this functor is an equivalence by \cref{lem:weakly-contractible-lim-commute},
    which we prove at the end of this subsection.
    
    It remains to show that the square is cartesian.
    Because the vertical functors are fully faithful, it suffices to show that if $X_\bullet$ is a simplicial space such that $\xF(X_\bullet)$ is in the essential image of $\xN_\bullet$, then $X_\bullet$ was already in the essential image of $\xN_\bullet$.
    We have a disjoint decomposition $U(\xF(X_\bullet)) = X_\bullet \sqcup \bigsqcup_{n\ge 0, n\neq 1} (X_\bullet)^{\times n}_{h\Sigma_n}$.
    The claim follows because if $X_\bullet$ and $Y_\bullet$ are simplicial spaces such that $X_\bullet \sqcup Y_\bullet$ is in the essential image of $\xN_\bullet$, then $X_\bullet$ and $Y_\bullet$ must have both been in the essential image.
    (To see this, note that if $X_\bullet$ is non-empty then it is a retract of $X_\bullet \sqcup Y_\bullet$ and the essential image of $\xN_\bullet$ is closed under all limits, in particular retractions.)
\end{proof}

\begin{rem}
    The proof given above also shows that for any weakly contractible $\infty$-category  $I \in \Cat$ and any $\calC \in \Cat$ we have a canonical equivalence $\xF(\Fun(I,\calC)) \simeq \Fun(I,\xF(\calC))$.
\end{rem}

\begin{cor}\label{cor:free-functor-on-Cat}
    The free functor $\xF \colon \Cat \to \SM$ has the following properties:
    \begin{enumerate}[(1)]
        \item 
        For any functor $F\colon \calC \to \calD$ the symmetric monoidal functor $\xF(F)\colon \xF(\calC) \to \xF(\calD)$ is equifibered.
        \item 
        The free functor $\xF \colon \Cat \to \SM$ commutes with contractible limits.
        \item 
        The free functor $\xF \colon \Cat \to \SM$ induces an equivalence
        \(\xF \colon \Cat \iso \SMeq{\xF(\ast)}\).
    \end{enumerate}
\end{cor}
\begin{proof}
    Claim $(1)$ holds because the free functor $\xF$ is computed level-wise
    by \cref{prop:free-nerve-commute} and equifiberedness can be checked level-wise by \cref{lem:eqf-detected-on-N1}.
    Similarly, claim (2) holds because $\xF$ and limits are both computed level-wise and for $\CMon$ we know that $\xF$ commutes with contractible limits.
    Alternatively, (2) is an instance of the more general \cite[Proposition 10.6]{CH22}.
    For claim $(3)$, we pass to simplicial objects in \cref{cor: equifibered-over-free-equals-spaces} to get an equivalence 
    \[\xF \colon \Fun(\Dop,\calS) \iso \Fun(\Dop,\CMon^{\eqf}_{/\xF(\ast)}) \simeq \Fun(\Dop,\CMon)^{\eqf}_{/\xF(\ast)}\]
    which by \cref{prop:free-nerve-commute} restricts to an equivalence 
    $\xF \colon \Cat \iso \SMeq{\xF(\ast)}$.
\end{proof}

We also have a variant of \cref{cor:free-functor-on-Cat}.(3) and \cref{cor: equifibered-over-free-equals-spaces} for symmetric monoidal $\infty$-categories equifibered over the categorical delooping of $\xF(*)$.
First we recall that symmetric monoidal $\infty$-categories ``with one object'' are commutative monoids.
\begin{lem}\label{lem:con-SM=CMon}
    There is an adjunction 
    \[ 
        \frB\colon \CMon \adj \SM \cocolon \End_{(-)}(\unit)
    \] 
    such that the left adjoint is fully faithful, and its essential image are those symmetric monoidal $\infty$-categories $\calC$ for which $\calC^\simeq$ is connected.
\end{lem}
\begin{proof}
    This is well-known and can for instance be found in \cite[Remark 8.7]{gepner-universality}.
    The idea is as follows.
    The $\infty$-category $\Mon$ is defined as the full subcategory $\Mon \subset \Fun(\Dop, \calS)$ of those simplicial spaces satisfying the Segal condition $X_n \simeq X_1^{\times n}$.
    As such it is, via the Rezk nerve, equivalent to the $\infty$-category of $\infty$-categories with a \emph{pointed} connected space of objects. 
    (This can be found in the literature for example by specializing \cite[Theorem 0.26]{AF-flagged-18} to $n=1$.)
    We then get an adjunction
    \[ 
        \frB\colon \Mon \adj \mrm{Cat}_{\infty, \ast/} \cocolon \End_{(-)}(\ast)
    \] 
    where the left adjoint is fully faithful with essential image those pointed $\infty$-categories with a connected space of objects.
    Both functors preserve products and hence we can apply $\CMon(-)$ to get an adjunction
    \[ 
        \frB\colon \CMon(\Mon) \adj \CMon(\mrm{Cat}_{\infty, \ast/}) \cocolon \End_{(-)}(\ast)
    \] 
    This yields the desired adjunction because $\CMon(\Mon) \simeq \CMon$ and $\CMon(\mrm{Cat}_{\infty,\ast/}) \simeq \SM$.
    These equivalences follow from Dunn--Lurie additivity \cite[Theorem 5.1.2.2]{HA} (namely using $\bbE_\infty \otimes \bbE_1 = \bbE_\infty$ and $\bbE_\infty \otimes \bbE_0 = \bbE_\infty$), which is the argument made in \cite{gepner-universality}, but it is worth pointing out this can already be seen using more elementary means.%
    \footnote{
        We would like to thank Fabian Hebestreit for pointing out the following argument to us.
        Namely, we have that $\CMon(\Mon) \simeq \Mon(\CMon)$ and 
        for every semi-additive category $\calC$ (such as $\CMon$) the forgetful functor $\Mon(\calC) \to \calC$ is an equivalence.
        The latter can be shown by checking that if $\Dop \to \calC$ satisfies the Segal condition, then it is left Kan extended from $\simp^\op_{\le 1}$, and then showing that via left and right Kan extensions every $c \in \calC$ uniquely promotes to a functor $X\colon \Dop_{\le 1} \to \calC$ with $X_0 = 0$ and $X_1 = c$.
    }   
\end{proof}

\begin{lem}\label{lem:eqf-over-frB}
    The functor $\frB$ composed with the free commutative monoid functor induces an equivalence
    \[
        \frB \circ \xF\colon \calS \xtoo{\simeq} \CMon_{/\xF(*)}^\eqf \xtoo{\simeq} \SMeq{\frB(\xF(*))}.
    \]
\end{lem}
\begin{proof}
    The first equivalence is a special case of \cref{cor: equifibered-over-free-equals-spaces}.
    For the second equivalence we can slice \cref{lem:con-SM=CMon} over $\xF(*)$ to get a full inclusion $\CMon_{/\xF(*)} \hookrightarrow \SMover{\frB(\xF(*))}$.
    It thus suffices to argue that a for a symmetric monoidal functor $F\colon \calC \to \frB(\xF(*))$  the following are equivalent:
    \begin{enumerate}[(1)]
        \item $F$ is equifibered.
        \item $\calC^\simeq$ is connected and $\End_\calC(\unit) \to \End_{\frB(\xF(*))}(\unit) = \xF(*)$ is equifibered.
    \end{enumerate}
    If $F$ is equifibered, then so is $F^\simeq\colon \calC^\simeq \to \frB(\xF(*))^\simeq = 0$, which forces $\calC^\simeq$ to be contractible and in particular connected.
    Then $\xN_1\calC \simeq \End_\calC(\unit)$ because the space of objects is contractible, so $\xN_1(F)$ being equifibered implies that $F$ induced an equifibered map on $\End_{-}(\unit)$.
    This shows $(1) \Rightarrow (2)$.
    For the converse, suppose $(2)$.
    The fiber over $0$ of the map $\End_\calC(\unit) \to \xF(*)$ must be contractible, as this map is equifibered.
    But this fiber contains all the invertible elements, so we know that $\Aut_\calC(\unit) \simeq *$, which implies $\calC^\simeq \simeq *$.
    Now $\xN_1(\calC) \simeq \End_\calC(\unit)$ as before and we see that $\xN_1(F)$ is equifibered, which implies (1) by \cref{lem:eqf-detected-on-N1}.
\end{proof}

We can sometimes use the nerve to compute colimits of symmetric monoidal categories.
\begin{obs}\label{obs:level-wise-colimits-in-SM}
    We say that a diagram $\calC\colon I \to \SM$
    \hldef{has a level-wise colimit}
    if its colimit is preserved by $\xN_\bullet$ in the sense that the canonical map
    \[
        \colim_{i \in I} \xN_n( \calC_i ) 
        \too
        \xN_n( \colim_{i \in I} \calC_i ) 
    \]
    is an equivalence for all $[n] \in \Dop$.
    Since $\xN_\bullet$ is fully faithful,
    this is the case if and only if
    the simplicial commutative monoid $M_\bullet$ obtained as the colimit of $\xN_\bullet \circ \calC\colon I \to \SM \to \Fun(\Dop, \CMon)$ is in the essential image of $\xN_\bullet$,
    i.e.\ if and only if $M_\bullet$ is a complete Segal space.
\end{obs}

We still have to provide the category-theoretic ingredient for the proof for \cref{prop:free-nerve-commute}.
\begin{lem}\label{lem:weakly-contractible-lim-commute}
    Let $I \in \Cat$ be a weakly contractible $\infty$-category, $X \in \calS$ an $\infty$-groupoid, and $\calC_{(-)}\colon X \to \Cat$ a functor.
    Then the canonical functor 
     $\colim_{x \in X}\Fun(I,\calC_x) \to \Fun(I,\colim_{x \in X}\calC_x)$
    is an equivalence.
\end{lem}
\begin{proof}
    The colimit over an $\infty$-groupoid may be computed by unstraightening \cite[Corollary 3.3.4.3]{HTT},
    so $\colim_{x \in X} \calC_x \to X$ is the unstraightening of $\calC_{(-)}\colon X \to \Cat$.
    By powering/cotensoring this with $I$ (see e.g.~\cite[Proposition 5.3.2]{envelopes}) we get that the unstraightening (and thus the colimit) of the functor
    \(
        \Fun(I, \calC_{-}) \colon X \too \Cat
    \) 
    is given by the pullback
    \[\begin{tikzcd}
	{\Un_X\left(\Fun\left(I,\calC_{(-)}\right)\colon  X \to \Cat\right)} & {\Fun\left(I,\colim_{x\in X} \calC_x\right)} \\
	X & {\Fun\left(I,X\right).}
	\arrow[""{name=0, anchor=center, inner sep=0}, "{\mrm{const}}", from=2-1, to=2-2]
	\arrow[from=1-1, to=2-1]
	\arrow[from=1-1, to=1-2]
	\arrow[from=1-2, to=2-2]
	\arrow["\lrcorner"{anchor=center, pos=0.125}, draw=none, from=1-1, to=0]
    \end{tikzcd}\]
    Since $I$ is weakly contractible the bottom horizontal map is an equivalence and thus so is the top horizontal map, proving the claim.
\end{proof}

\subsubsection{CULF maps}
The equifibered maps we have studied so far are those natural transformations $\alpha\colon  M \to N$ of $\Fin_*$-Segal objects $M, N\colon  \Fin_* \to \calS$ that are ``active-equifibered'' in the sense of \cite{envelopes}.
It also makes sense to consider such equifibered maps in other circumstances.
In the example of $\simp^\op$ these maps have been studied under the name of ``CULF'' maps \cite{GCKT18} -- an acronym for ``Conservative and Unique Lifting of Factorizations''.
We briefly recall this definition here and recall some elementary properties that will be useful later.
In particular, we prove \cref{cor:tensor-conservative-flat},
which relates the conditions appearing in \cref{lem:colim-of-eqf-slice-categories} to the upcoming definition of $\infty$-properads.

\begin{defn}
    A map of simplicial spaces $f\colon X_\bullet \to Y_\bullet$ is called \hldef{CULF} if the square 
    \[
        \begin{tikzcd}
            X_n \ar[r, "\lambda^*"] \ar[d, "f_n"'] &
            X_m \ar[d, "f_m"] \\
            Y_n \ar[r, "\lambda^*"] &
            Y_m
        \end{tikzcd}
    \]
    is cartesian for every active $[n] \leftarrow [m] :\! \lambda \in \simp^\op$.
    Here active means that $\lambda(0) = 0$ and $\lambda(m) = n$.
\end{defn}

Just like for equifibered maps, this reduces to a simpler condition when the simplicial spaces involved satisfy a Segal condition:
\begin{lem}[{\cite[Lemma 4.3]{GCKT18}}]\label{lem:check-CULF-on-d1}
    For a map of Segal spaces $f\colon X_\bullet \to Y_\bullet$
    it suffices to check the CULF condition for the active map $\lambda\colon [1] \to [2]$,
    i.e.\ it suffices to check that the diagram
    \[
        \begin{tikzcd}
            X_2 \ar[r, "d_1"] \ar[d, "f_2"'] &         
            X_1 \ar[d, "f_1"]\\
            Y_2 \ar[r, "d_1"] &         
            Y_1
        \end{tikzcd}
    \]
    is cartesian.
\end{lem}

\begin{obs}\label{obs:CULF}
    Write $B_\bullet\colon \CMon \to \Fun(\simp^\op, \calS)$ for the functor induced by restriction along the functor $|-|\colon\simp^\op \to \Fin_*$ \cite[Example 4.9]{CH19}.
    This sends a commutative monoid $M$ to its bar construction $B_n M = M^{\times n}$.
    A morphism $f\colon M \to N$ of commutative monoids is equifibered if and only if the simplicial map $B_\bullet f\colon  B_\bullet M \to B_\bullet N$ is CULF.
\end{obs}

CULF maps between complete Segal spaces are exactly conservative flat fibrations (or conservative exponentiable functors) as was already observed in \cite[Remark 3.3]{HK22}.
\begin{lem}\label{lem:conservative-flat-fibration}
    A functor $F\colon \calC \to \calD$ is a conservative flat fibration
    if and only if its nerve $\xN_\bullet F\colon \xN_\bullet\calC \to \xN_\bullet\calD$ is CULF.
\end{lem}
\begin{proof}
    The functor $F$ is conservative if and only if the square
    \[
        \begin{tikzcd}
            \xN_0\calC \ar[r, hook, "s_0"] \ar[d, "\xN_0 F"'] &         
            \xN_1\calC \ar[d, "\xN_1 F"] \\
            \xN_0\calD \ar[r, hook, "s_0"] &         
            \xN_1\calD 
        \end{tikzcd}
    \]
    is cartesian. 
    This is the case for CULF maps by \cref{lem:check-CULF-on-d1} because $([0] \leftarrow [1] \cocolon s^0)$ is active.
    We may therefore assume that $F$ is conservative.
    
    To complete the proof we show that a conservative functor $F$ is a flat fibration if and only if the square
    \[
        \begin{tikzcd}
            \xN_2\calC \ar[r, "d_1"] \ar[d, "\xN_2 F"'] &         
            \xN_1\calC \ar[d, "\xN_1 F"]\\
            \xN_2\calD \ar[r, "d_1"] &         
            \xN_1\calD
        \end{tikzcd}
    \]
    is cartesian.
    For a fixed $(\alpha\colon x \to y) \in \xN_1\calC$
    the map on horizontal fibers of this square is
    $((\calC_{x/})_{/\alpha})^\simeq \to  ((\calD_{F(x)/})_{/F(\alpha)})^\simeq$,
    and the square is cartesian if this map is an equivalence for all $\alpha$.
    By \cite[Remark 3.3.8]{Lurie-2Cat} $F\colon \calC \to \calD$ is a flat fibration
    if and only if for any $\alpha\colon x \to y \in \calC$ and factorization 
    $\sigma=(F(x) \to \sigma_0 \to F(y)) \in (\calD_{F(x)/})_{/F(\alpha)}$ the $\infty$-category
        $(\calC_{x/})_{/\alpha} \times_{(\calD_{F(x)/})_{/F(\alpha)}} \{\sigma\}$
    is weakly contractible. 
    Since $F$ was assumed to be conservative, 
    this is an $\infty$-groupoid and hence weakly contractible
    if and only if the functor 
    $((\calC_{x/})_{/\alpha})^\simeq \to ((\calD_{F(x)/})_{/F(\alpha)})^\simeq$
    is an equivalence.
\end{proof}

\cref{lem:check-CULF-on-d1} and \cref{lem:conservative-flat-fibration} have the following consequence. 
\begin{cor}\label{cor:tensor-conservative-flat}
    For a symmetric monoidal $\infty$-category $\calC$ the following are equivalent:
    \begin{enumerate}
        \item the functor $\otimes\colon \calC \times \calC \to \calC$ is a conservative flat fibration,
        \item the simplicial map $\otimes\colon  \xN_\bullet \calC \times \xN_\bullet \calC \to \xN_\bullet \calC$ is a CULF map,
        \item the map of commutative monoids $d_1\colon \xN_2\calC \to \xN_1\calC$ is equifibered,
        \item for all active $[n] \leftarrow [m]\colon  \lambda \in \simp^\op$
        the map of commutative monoids $\lambda^*\colon \xN_n\calC \to \xN_m\calC$ is equifibered.
    \end{enumerate}
\end{cor}

\begin{rem}
    G\'alves-Carrillo--Kock--Tonks define a notion of \textit{monoidal decomposition spaces} \cite[\S 9]{GCKT18}.
    The symmetric monoidal categories satisfying the equivalent conditions of \cref{cor:tensor-conservative-flat} are precisely the symmetric monoidal decomposition spaces which are also complete Segal spaces.
\end{rem}

\begin{rem}
    In \cref{sect:properad-theory} we will define an $\infty$-properad as a symmetric monoidal $\infty$-category $\calP$ that satisfies the equivalent conditions of \cref{cor:tensor-conservative-flat}
    and moreover that $\xN_1 \calP$ is free.
\end{rem}

\subsubsection{Equifibered factorization for functors}
Just as we did for commutative monoids, we can define contrafibered morphisms of symmetric monoidal $\infty$-categories.

\begin{defn}
    A symmetric monoidal functor $F\colon \calC \to \calD$ is called \hldef{contrafibered} if it is left-orthogonal to all equifibered functors. 
\end{defn}

In \cref{lem:eqf-detected-on-N1} we saw that a symmetric monoidal functor $F\colon  \calC \to \calD$ is equifibered if and only if $\xN_1\calC \to \xN_1\calD$ is equifibered.
For contrafibered functors we only have a weaker statement.
\begin{cor}\label{cor:detect-ctf}
    The functor $\xN_\bullet \colon \SM \to \Fun(\Dop,\CMon)$ detects contrafibered functors:
    If $F \colon \calC \to \calD$ is a symmetric monoidal functor such that $\xN_n(F) \colon \xN_n(\calC) \to \xN_n(\calD)$ is contrafibered for all $n$, then $F$ is contrafibered.
\end{cor}
\begin{proof}
    Suppose that $\xN_\bullet(F)$ is contrafibered.
    By definition $F$ is contrafibered if and only if it is left orthogonal to every equifibered symmetric monoidal functor $G\colon  \calE \to \calF$.
    Since $\xN_\bullet$ is fully faithful if suffices to show that $\xN_\bullet(F)$ is left orthogonal to $\xN_\bullet(G)$.
    But this is indeed the case since $\xN_\bullet(G)$ is equifibered by \cref{lem:eqf-detected-on-N1}.
\end{proof}

The same arguments as in \cref{lem:formal-ctf-eqf}, show that every symmetric monoidal functor $F \colon \calC \to \calD$ admits a unique contrafibered-equifibered factorization:
\[F \colon \calC \xrightarrow{F^\ctf} \calE \xrightarrow{F^\eqf} \calD\]
If $\calC$ and $\calD$ happen to be symmetric monoidal $\infty$-groupoids, i.e.~commutative monoids, this factorization agrees with previously discussed contrafibered-equifibered factorization in $\CMon$.
(This follows by \cref{lem:eqf-detected-on-N1} and \cref{cor:detect-ctf} applied to groupoids.)

\begin{cor}\label{cor:eqf-slice-is-accessible}
    For any $\calC \in \SM$ we have an (accessible) adjunction
    \[\begin{tikzcd}
	{\bbL^\eqf \colon \SMover{\calC}} && {\SMeq{\calC} \colon \inc.}
	\arrow[""{name=0, anchor=center, inner sep=0}, shift left=2, from=1-1, to=1-3]
	\arrow[""{name=1, anchor=center, inner sep=0}, shift left=2, hook', from=1-3, to=1-1]
	\arrow["\dashv"{anchor=center, rotate=-90}, draw=none, from=0, to=1]
    \end{tikzcd}\]
\end{cor}
\begin{proof}
    This is a general fact about factorization systems generated by a set of morphisms, see for instance \cite[Observation 2.3.6]{envelopes}.
\end{proof}

We establish an analogue of \cref{lem:colim-of-eqf-slice} giving sufficient conditions for equifibered functors to be closed under colimits in the slice category.
\begin{lem}\label{lem:colim-of-eqf-slice-categories}
    Let $\calC$ be a symmetric monoidal category and let 
    $\SMeq{\calC} \subset \SMover{\calC}$ denote the full subcategory
    on the equifibered symmetric monoidal functors:
    \begin{enumerate}
        \item 
            If $\calC$ satisfies the equivalent conditions of \cref{cor:tensor-conservative-flat},
            then $\SMeq{\calC} \subset \SMover{\calC}$ is closed under sifted colimits.
        \item
            If $N_n\calC$ is a free commutative monoid for all $n$, 
            then $\SMeq{\calC} \subset \SMover{\calC}$ is closed under finite coproducts.
    \end{enumerate}
    In particular, when both $(1)$ and $(2)$ hold 
    (i.e.\ when $\calC$ is an $\infty$-properad in the sense of \cref{defn:properad})
    $\SMeq{\calC} \subset \SMover{\calC}$ is closed under small colimits.
\end{lem}
\begin{proof}
    In $\Cat$ base change along a flat fibration admits a right adjoint \cite[Proposition 3.4.9]{Lurie-2Cat}, and in particular preserves colimits.
    The forgetful functor $\SMover{\calC} \to \SM \to \Cat$ 
    creates and preserves 
    sifted colimits.
    (The first functor by \cite[Proposition 4.4.2.9]{HTT}, the second by \cite[Corollary 3.2.3.2]{HA}.)
    For a sifted diagram $\calD\colon I \to \SMeq{\calC}$ 
    we may therefore compute (just as in \cref{lem:colim-of-eqf-slice}):
    \[
        \calC^2 \times_\calC \colim_{i \in I} \calD(i) 
        \simeq \colim_{i \in I} (\calC^2 \times_\calC \calD(i))
        \simeq \colim_{i \in I} \calD(i)^2
        \simeq \colim_{i \in I} \calD(i) \times \colim_{j \in I} \calD(j).
    \]
    For $(2)$ we use that $\SMover{\calC} \to \SM$ preserves coproducts
    and the coproducts in $\SM$ are given by the cartesian product.
    This cartesian product is computed level-wise on the nerve,
    so it will suffice to check that $\CMon_{/N_n\calC}^\eqf \subset \CMon_{/N_n\calC}$
    is closed under coproducts, which follows from \cref{lem:colim-of-eqf-slice}
    since we have assumed that $N_n\calC$ is free for all $n$.
\end{proof}

\subsection{Cospans and \texorpdfstring{$\amalg$-disjunctive $\infty$-}{ coproduct disjunctive infinity-}categories}
\label{subsec:Csp-and-disjunctive}

In this section we study, for $\calC$ an $\infty$-category with finite colimits, the symmetric monoidal $\infty$-category $\Cospan(\calC)$ whose objects are those of $\calC$ and whose morphisms are cospans in $\calC$.
This construction is due to Barwick \cite{Bar17}, though we will follow the modified approach of Haugseng--Hebestreit--Linskens--Nuiten \cite{HHLN20}.
Applied to $\calC = \Fin$ it will yield an important example of an $\infty$-properad, which we later prove to be terminal.
We will also study a more general class of categories $\calC$ for which $\Cospan(\calC)$ is an $\infty$-properad.

\begin{notation}\label{notation:Tw}
    We let \hldef{$\Tw[n]$} denote the twisted arrow category of the poset $[n]$.
    That is, it is the poset whose objects are pairs $(i \le j)$ with $0 \le i \le j \le n$
    and where there is a unique morphism $(i \le j) \to (i' \le j')$ if and only if $i' \le i \le j \le j'$.
\end{notation}

\begin{const}\label{const:cospans-2-cat}
    Given an $\infty$-category $\calC$ with finite colimits, we construct a functor 
    \[\hldef{\mfr{C}_\bullet(\calC)} \colon \simp^\op \too \SM,\]
    which can be thought of as the \hldef{double $\infty$-category of cospans} in $\calC$.
    More precisely, it will have the following properties:
    \begin{itemize}
        \item 
        $\mfr{C}_0(\calC) = \calC$ 
        and 
        $\mfr{C}_1(\calC) = \Fun(\Tw[1],\calC)$.
        \item 
        The natural map
        $\mfr{C}_n(\calC) \too \mfr{C}_1(\calC) \times_{\mfr{C}_0(\calC)} \cdots \times_{\mfr{C}_0(\calC)}\mfr{C}_1(\calC)$ 
        is an equivalence for all $n \ge 2$.
    \end{itemize}
    Consider the composite functor:
    \[\Fun(\Tw[\bullet],\calC)  \colon \simp^\op \xrightarrow{\Tw[\bullet]} \Cat^\op \xrightarrow{\Fun(-,\calC)} \Cat^\amalg \subseteq \SM, \qquad [n] \longmapsto \Fun(\Tw[n],\calC)\]
    where we use the cocartesian monoidal structure from \cref{ex:cocartesian-monoidal}.
    For 
    $[n] \in \simp^\op$ we let
    $\mfr{C}_n(\calC) \subseteq \Fun(\Tw[n],\calC)$
    denote the full subcategory spanned by pushout preserving functors.
    It is closed under coproducts and since all maps 
    $\lambda \colon [n] \to [m]$
    induce pushout preserving functors 
    $\Tw[\lambda] \colon \Tw[n] \to \Tw[m]$
    we may consider $\mfr{C}_\bullet(\calC)$ as a subfunctor
    \[\mfr{C}_\bullet(\calC) \subseteq \Fun(\Tw[\bullet],\calC) \colon \simp^\op \too \Cat^\amalg, \qquad [n] \longmapsto \mfr{C}_n(\calC) \subseteq \Fun(\Tw[n],\calC).\]
    Note that a functor 
    $A \colon \Tw[n] \to \calC$
    preserves pushouts if and only if it is left Kan extended from the full subcategory 
    $i \colon \Tw[n]^\el:= \Tw[1]\amalg_{[0]} \cdots \amalg_{\Tw[0]} \Tw[1] \hookrightarrow \Tw[n]$.
    Consequently, the adjunction 
    $i_! \colon \Fun(\Tw[n]^\el,\calC) \adj \Fun(\Tw[n],\calC) \colon i^\ast$
    restricts to an equivalence 
    \[i_! \colon \Fun(\Tw[n]^\el,\calC) \simeq  \mfr{C}_n(\calC) \cocolon i^\ast. \]
    From this it follows that 
    $\mfr{C}_n(\calC) \too \mfr{C}_1(\calC) \times_{\mfr{C}_0(\calC)} \cdots \times_{\mfr{C}_0(\calC)}\mfr{C}_1(\calC)$ 
    is an equivalence as claimed.
\end{const}

\begin{obs}\label{obs:Csp-complete}
    The composite 
    $\simp^\op \xrightarrow{\mfr{C}_\bullet(\calC)} \SM \xrightarrow{(-)^\simeq} \CMon \xrightarrow{U} \calS$
    defines a complete Segal space.
    We checked the Segal condition above, and we refer the reader to \cite[Lemma 2.17]{HHLN20} for the completeness.
\end{obs}

\begin{defn}\label{defn:Csp}
    For a finitely cocomplete $\infty$-category $\calC$ we define \hldef{$\Cospan(\calC)$} as the unique symmetric monoidal $\infty$-category with
    $\xN_\bullet \Cospan(\calC) \simeq \mfr{C}_\bullet(\calC)^\simeq$.
    In the case $\calC=\Fin$ we write $\hldef{\Csp} := \Cospan(\Fin)$
    and simply refer to it as ``the'' \hldef{cospan category}. 
\end{defn}

\begin{lem}\label{lem:Csp-level-wise-free}
    The nerve of $\Csp$ is level-wise free
    and all active morphisms
    $\lambda \colon [m] \to [n] \in \Delta$ 
    induce equifibered maps 
    $\lambda^\ast\colon \xN_n\Csp \to \xN_m\Csp$.
\end{lem}
\begin{proof}
    \cref{const:cospans-2-cat} provides a factorization
    \[\colim_{\!\!\Tw[n]^\el}  \colon \Fun(\Tw[n]^\el,\Fin) \simeq \mfr{C}_n(\Fin) \hookrightarrow \Fun(\Tw[n],\Fin) \xrightarrow{\ev_{(0 \le n)}} \Fin.\]
    By \cref{lem:coproduct-disjunctive-colim} below the composite is an equifibered functor.
    In particular $\mfr{C}_n(\Fin)^\simeq$ is a free commutative monoid as it is equifibered over $\Fin^\simeq \simeq \xF(*)$.
    Whenever $\lambda$ is active, 
    $\Tw[\lambda]\colon \Tw[m] \to \Tw[n]$ preserves the terminal object and thus the diagram 
    \[
    \begin{tikzcd}
        {\mfr{C}_n(\Fin)} \ar[d, "\lambda^\ast"] \ar[r] &
        {\Fun(\Tw[n], \Fin)} \ar[d, "\lambda^\ast"] \ar[r, "\ev_{(0\le n)}"] & 
        \Fin \\
        {\mfr{C}_m(\Fin)} \ar[r] & 
        {\Fun(\Tw[m], \Fin).} \ar[ru, "\ev_{(0\le m)}"'] &
    \end{tikzcd}
    \]
    commutes.
    It follows by cancellation that $\lambda^\ast \colon \mfr{C}_n(\Fin) \to \mfr{C}_m(\Fin)$ is an equifibered functor. 
    In particular, $\lambda^\ast \colon \mfr{C}_n(\Fin)^\simeq \to \mfr{C}_m(\Fin)^\simeq$ is equifibered as promised. 
\end{proof}

\subsubsection{$\amalg$-disjunctive categories}

\cref{lem:Csp-level-wise-free} works equally well when replacing the category of finite sets with any $\amalg$-disjunctive category in the following sense:

\begin{defn}\label{defn:coproduct-disjunctive}
    An $\infty$-category is called \hldef{$\amalg$-disjunctive}
    if it has finite coproducts and the functor
    \[
        \amalg \colon \calC_{/x} \times \calC_{/y} \to \calC_{/x \amalg y},  \qquad
        (f\colon a \to x, g\colon b \to y) \mapsto (f \amalg g\colon  a \amalg b \to x \amalg y)
    \]
    is an equivalence for all $x, y \in \calC$.
\end{defn}

\begin{rem}
    This is a homotopical version of the $1$-categorical notion of ``extensive category'' \cite{CLW93}, where the category is moreover required to have products, which we will not need here.
    An $\infty$-categorical variant of this notion was studied by Barwick \cite[Definition 4.2]{Bar17}, who called them ``disjunctive $\infty$-categories'' and also required them have finite limits.
    We chose the above name as it is a special case of an $\otimes$-disjunctive category, which will appear again in \cref{defn:tensor-disjunctive}.
\end{rem}

\begin{example}
    The category of sets $\Sets$, the $\infty$-category of spaces $\calS$, 
    the $\infty$-category of $\infty$-categories $\Cat$, 
    and the opposite category of discrete commutative rings $\mrm{CAlg}(\mrm{Ab})^\op$
    are all $\amalg$-disjunctive. 
\end{example}

\begin{example}
    If $\calC$ is $\amalg$-disjunctive, then so is $\calC_{/z}$ for all $z \in \calC$.
    Indeed, coproducts in the over category can be computed in $\calC$ and $(\calC_{/z})_{/\alpha\colon x \to z} \simeq \calC_{/x}$.
\end{example}

\begin{obs}[Barwick]\label{obs:barwick-disjunctive}
    Let $\calC$ be an $\infty$-category with finite coproducts and finite limits. 
    Then the functor $\amalg\colon \calC_{/x} \times \calC_{/y} \too \calC_{/x \amalg y}$ admits a right adjoint given by $(z \to x \amalg y) \longmapsto (z \times_{x \amalg y} x \to x, z \times_{x \amalg y} y \to y)$ and $\calC$ is $\amalg$-disjunctive if and only if this is an adjoint equivalence.
    Inspecting the unit and counit we see that $\calC$ is $\amalg$-disjunctive if and only if finite coproducts in $\calC$ are disjoint and universal in the sense of \cite[\S6.1.1 (ii) and (iii)]{HTT}. 
    (This is taken as the definition of disjunctive in \cite[Definition 4.2]{Bar17}.)
    As these conditions are a subset of Lurie's Giraud-axioms for $\infty$-topoi \cite[Proposition 6.1.0.1]{HTT}, we see that every $\infty$-topos is $\amalg$-disjunctive.
\end{obs}

\begin{example}\label{ex:comm-algebras-coprod-disj}
    Let $\calE \in \CAlg(\PrLst)$ be a stable presentably symmetric monoidal $\infty$-category. Then the $\infty$-category $\CAlg(\calE)^{\op}:=\Alg_{\mbb{E}_{\infty}}(\calE)^{\op}$ is $\amalg$-disjunctive
    \cite[Proposition 2.39]{Akhil}.
\end{example}

The key property of $\amalg$-disjunctive categories for us is the following:
\begin{lem}\label{lem:coproduct-disjunctive-colim}
    Let $\calC$ be a $\amalg$-disjunctive $\infty$-category 
    and let $J \in \Cat$ such that $\calC$ has colimits of shape $J$. 
    Then the functor
    \[
        \colim_J\colon \Fun(J, \calC) \too \calC,
    \]
    which is symmetric monoidal with respect to the coproduct, is equifibered.
\end{lem}

\begin{proof}
    Let $J^\cone$ denote the $\infty$-category obtained by freely 
    adjoining a terminal object to $J$. The colimit of a diagram
    $F\colon J \to \calC$ can be computed by first left Kan extending it
    along the full inclusion $J \subset J^\cone$ and 
    then evaluating at the terminal object:
    \[
        \colim_J\colon  \Fun(J, \calC) \into 
        \Fun(J^\cone, \calC) \xrightarrow{\ev_\infty}
        \calC.
    \]
    Both functors are symmetric monoidal with respect to the coproduct, and we will show that they are both equifibered.
    
    Evaluation at the tip $\ev_\infty \colon \Fun(J^\cone, \calC) \to \calC$ 
    is a cocartesian fibration whose fiber $x \in \calC$ is $\Fun(J, \calC_{/x})$.
    (Indeed, it is the base change of the cocartesian fibration $\ev_1\colon \Ar(\Fun(J, \calC)) \to \Fun(J, \calC)$ \cite[Corollary 2.4.7.12]{HTT} along $\Delta \colon \calC \to \Fun(J, \calC)$.)
    The cocartesian edges in $\Fun(J^\cone,\calC)$ are precisely the natural transformations $F \to G$ which restrict to an equivalence $F|_J \simeq G|_J$.
    In particular, $\sqcup \colon \Fun(J^\cone,\calC) \times \Fun(J^\cone,\calC) \to \Fun(J^\cone,\calC)$ 
    preserves cocartesian edges and thus by \cref{lem:eqf-cocartesian} the functor $\ev_\infty$ is equifibered if and only if the map
    \[
        \amalg\colon  \Fun(J, \calC_{/x}) \times \Fun(J, \calC_{/y}) 
        \too \Fun(J, \calC_{/x \amalg y})
    \]
    is an equivalence.
    Indeed, this is the case since we assumed that $\calC$ is
    $\amalg$-disjunctive.
    
    It remains to show that the fully faithful functor
    $\Fun(J, \calC) \into \Fun(J^\cone, \calC)$,
    given by left Kan extension, is equifibered.
    This can be checked by verifying that its essential image
    is closed under cancellation in the sense of \cref{cor:full-subcat-eqf}.
    A diagram $F\colon J^\cone \to \calC$ is in the essential image
    if and only if the canonical map $\colim_J F(j) \to F(\infty)$
    is an equivalence.
    Suppose $F$ is in the essential image and we have $F = G_1 \amalg G_2$.
    Then the coproduct of the two maps 
    $\alpha_1\colon \colim_J G_1(j) \to G_1(\infty)$ and 
    $\alpha_2\colon \colim_J G_2(j) \to G_2(\infty)$ is an equivalence. 
    In other words, $(\alpha_1 \amalg \alpha_2) \in \calC_{/G_1(\infty) \amalg G_2(\infty)}$ is a terminal object.
    Since $\calC$ is $\amalg$-disjunctive
    we conclude that 
    $(\alpha_1,\alpha_2) \in \calC_{/G_1(\infty)} \times \calC_{/G_2(\infty)}$
    is a terminal object and hence $\alpha_1$ and $\alpha_2$
    both are equivalences.
    Therefore, $G_1$ and $G_2$ are both in the essential image, and we are done.
\end{proof}

\begin{cor}\label{lem:disjunctive-Csp-level-wise-free}
    Let $\calC$ be a $\amalg$-disjunctive category that has pushouts.
    Then:
    \begin{enumerate}[(1)]
        \item Active morphisms
        $\lambda \colon [m] \to [n] \in \Delta$ 
        induce equifibered maps 
        $\lambda^\ast\colon \xN_n\Cospan(\calC) \to \xN_m\Cospan(\calC)$.
        \item 
        If $\pi_0(\calC^\simeq)$ is generated by indecomposables,
        then
        $\xN_\bullet(\Cospan(\calC))$ is level-wise free.
    \end{enumerate}
    In particular, when $(2)$ holds $\Cospan(\calC)$ is an $\infty$-properad in the sense of \cref{defn:properad}.
\end{cor}
\begin{proof}
    The proof of $(1)$ is as in \cref{lem:Csp-level-wise-free}.
    For $(2)$ we apply \cref{lem:coproduct-disjunctive-colim} in the case $J = * \amalg *$
    to see that $\amalg\colon \calC \times \calC \to \calC$ is an equifibered functor.
    In particular $\calC^\simeq$ is a pseudo-free monoid and by \cref{lem:eqf-addition} it is free.
    Now the rest of the proof proceeds as in \cref{lem:Csp-level-wise-free}.
\end{proof}

\section{\texorpdfstring{$\infty$-Properads}{Infinity Properads}}\label{sect:properad-theory}

In this section we introduce the notion of $\infty$-properad as a symmetric monoidal $\infty$-category satisfying certain ``freeness'' conditions, formulated in the language of equifibered maps.
After discussing some examples we move on establish some categorical properties of the $\infty$-category of $\infty$-properads $\Prpd$, which we use to study more intricate examples such as $\infty$-properads freely generated by corollas (\cref{defn:free-corolla}) and endomorphism $\infty$-properads (\cref{defn:end-properad}).
These tools and constructions all rely on the fact that the $(2,1)$-category of cospans of finite sets $\Csp$ is the terminal $\infty$-properad (\cref{thm:Csp-is-final}), which we prove at the end of this section.
The proof of \cref{thm:Csp-is-final} can be read independently of the rest of this section.

\subsection{Definition and examples}

\begin{defn}\label{defn:properad}
    An \hldef{$\infty$-properad} is a symmetric monoidal category $\calP$ such that 
    \begin{enumerate}[(1)]
        \item $\xN_1 \calP$ is free, and
        \item the composition $\circ = d_1 \colon  \xN_2\calP \too \xN_1\calP$ is equifibered.
    \end{enumerate}
    Define the $\infty$-category of $\infty$-properads $\hldef{\Prpd}$ to be the replete subcategory of $\SM$ whose objects are $\infty$-properads 
    and whose morphisms are equifibered symmetric monoidal functors.
\end{defn}

The following example is crucial, as we shall later see that it is the terminal $\infty$-properad.
\begin{example}\label{ex:Csp}
    We have shown in \cref{lem:Csp-level-wise-free} that the $(2,1)$-category
    $\Csp$ of finite sets and cospans between them satisfies 
    \cref{defn:properad} and hence is an $\infty$-properad.
\end{example}

\begin{example}
    For any $\infty$-operad $\calO$ its monoidal envelope $\Env(\calO) \in \SM$ in the sense of Lurie \cite{HA} is an $\infty$-properad,
    as we shall see in \cref{thm:monic-prpd=operad}.
\end{example}

There are many equivalent ways of characterizing $\infty$-properads.
We now list some of them:
\begin{prop}\label{prop:TFAE-properad}
    The following are equivalent for a symmetric monoidal $\infty$-category $\calP$:
    \begin{enumerate}[(1)]
        \item \label{item:prpd}
            $\calP$ is an $\infty$-properad.
        \item \label{item:prpd-op} 
            The opposite category $\calP^\op$ is an $\infty$-properad.
        \item \label{item:prpd-active}
            $\xN_n \calP$ is free for all $n$ and $\lambda^*\colon  \xN_n \calP \to \xN_m \calP$ is equifibered for all active $[n] \leftarrow [m]\colon  \lambda \in \simp^\op$.
        \item \label{item:prpd-flat} 
            $\Ar(\calP)^\simeq$ is free and the monoidal product $\otimes \colon  \calP \times \calP \to \calP$ is a conservative flat fibration.
        \item \label{item:prpd-Csp} 
            There exists an equifibered symmetric monoidal functor $\calP \to \Csp$.
    \end{enumerate}
\end{prop}
\begin{proof}
    $\ref{item:prpd} \Leftrightarrow \ref{item:prpd-op}$ holds because the definition is symmetric.
    $\ref{item:prpd} \Leftrightarrow \ref{item:prpd-active} \Leftrightarrow \ref{item:prpd-flat}$ holds by \cref{cor:tensor-conservative-flat} and since every $[n] \in \simp$ receives an active map from $[1]$, which induces an equifibered map $\xN_n\calP \to \xN_1\calP \simeq \Ar(\calP)^\simeq$ and thus $\xN_n\calP$ is free if $\xN_1\calP$ is.
    $\ref{item:prpd-Csp} \Rightarrow \ref{item:prpd}$ holds by \cref{lem:eqf-over-properad} and \cref{ex:Csp}.
    $\ref{item:prpd} \Rightarrow \ref{item:prpd-Csp}$ follows because $\Csp$ is the terminal $\infty$-properad by \cref{thm:Csp-is-final}.
\end{proof}

Most of the above implications can be shown using only elementary facts about equifibered maps.
However, the implication $\ref{item:prpd} \Rightarrow \ref{item:prpd-Csp}$ is more complicated 
and will be the subject of \cref{subsec:pre-properads} 
where we use obstruction theory to show that $\Csp$ is the terminal $\infty$-properad.

\cref{defn:properad} is very unlike the standard definition of (1-categorical) properads. 
We now introduce the necessary language to relate our definition to the standard definition, at least conceptually.
In \cref{sec:segal} we will prove that our $\infty$-category of $\infty$-properads is indeed equivalent to previous definitions.
Recall that for a free commutative monoid $M$ we write $M^\el \subset M$ for the subspace of generators.
\begin{notation}\label{notation:properads}
    For an $\infty$-properad $\calP$ we refer to $\xN_0^\el \calP = (\calP^\simeq)^\el$ as the \hldef{space of colours} of $\calP$.
    Moreover, we refer to $\xN_1^\el \calP = (\Ar(\calP)^\simeq)^\el$ as the \hldef{space of operations} of $\calP$.
    Given an operation $o$ in $\calP$, i.e.\ a morphism $o\colon x \to y \in \calP$ that is a generator in $\Ar(\calP)^\simeq$,
    its source and target can be written as tensor products of colours.
    So every operation can be written as
    \[
        o\colon  x_1 \otimes \dots \otimes x_n \too y_1 \otimes \dots \otimes y_m.
    \]
    We say that such an operation is of \hldef{arity $(n,m)$}.
    We refer to the $x_i \in \xN_0^\el \calP$ as the inputs and 
    to $y_j \in \xN_0^\el \calP$ as the outputs of $o$.
    These are unique up to reordering.
    The map that encodes the inputs and outputs of operations is
    \[
        \xN_1^\el \calP \subset \xN_1 \calP \xtoo{(s,t)} 
        \xN_0 \calP \times \xN_0 \calP \simeq \xF(\xN_0^\el \calP) \times \xF(\xN_0^\el \calP).
    \]
    We may sometimes write \hldef{$\calP(x_1,\dots,x_n; y_1,\dots,y_m)$}
    for the fiber of this map at the point given by the objects $(x,y) \in \xN_0\calP \times \xN_0\calP$.
    Note that this is a subspace of $\Map_\calP(x,y)$.
    
    Given $o$ as above, another operation $p \in \calP(z_1,\dots,z_l; w_1,\dots, w_k)$, 
    and equivalences $\{\alpha_i\colon  y_i \simeq z_i\}_{i=1}^a$,
    we can form a \hldef{composite $o \circ_{(\alpha)_i} p$} as
    \begin{align*}
        p \circ_{(\alpha)_i} o\colon 
        x_1 \otimes \dots \otimes x_n \otimes z_{a+1} \otimes \dots \otimes z_l
        & \xto{ o \otimes \id_{z_\bullet} }
        y_1 \otimes \dots \otimes y_m \otimes z_{a+1} \otimes \dots \otimes z_l \\
        & \xtoo{\ \alpha_\bullet \ }
        y_{a+1} \otimes \dots \otimes y_m \otimes z_1 \otimes \dots \otimes z_l \\
        & \xto{ \id_{y_\bullet} \otimes p }
        y_{a+1} \otimes \dots \otimes y_m \otimes w_1 \otimes \dots \otimes w_k .
    \end{align*}
    This is illustrated in \cref{figure:gluing-operations} in the introduction.
\end{notation}

\begin{rem}
    More generally, we could also have used equivalences $\alpha_i\colon  y_{\sigma(i)} \simeq z_{\sigma'(i)}$ for any two injections $\sigma\colon  \{1,\dots,a\} \to \{1,\dots,m\}$ and $\sigma'\colon \{1,\dots,a\} \to \{1,\dots,l\}$.
    Evidently there are many such compositions one could define, and they should all come with plenty of coherence data that explains how they interact with each other.
    However, we need not worry about this as it is all encoded in the assumption that $\calP$ is a symmetric monoidal $\infty$-category.
\end{rem}

We have the following description of mapping spaces in $\infty$-properads.
This matches the $1$-categorical description of hom-sets in ``labelled cospan categories'' \cite[Lemma 2.8]{Jan-tropical}, and it generalizes the ``hereditary condition'' for $\infty$-operads as for example discussed in \cite[Remark 1.1.2/2.4.7]{HK21}.
\begin{lem}\label{lem:mapping-space-in-properad}
    If a symmetric monoidal $\infty$-category $\calP$ is an $\infty$-properad then its mapping spaces can be described in terms of the spaces of operations as
    \begin{align*}
        \Map_{\calP}(x_1 \otimes \dots \otimes x_n, y_1 \otimes \dots \otimes y_m)
        & \simeq  
        \colim_{K \in (\Fin_{I \sqcup J/})^\simeq} \prod_{k \in K} 
        \calP\left(\{x_i\}_{i \in I_k}; \{y_j\}_{j \in J_k} \right) \\
        & \simeq  
        \xF(\calP(\emptyset; \emptyset)) \times 
        \coprod_{I \amalg J \twoheadrightarrow K} \prod_{k \in K} 
        \calP\left(\{x_i\}_{i \in I_k}; \{y_j\}_{j \in J_k} \right)
    \end{align*}
    where the colimit is taken over the groupoid of finite sets under $I \sqcup J$, the coproduct is over the (discrete) groupoid of quotients of $I \sqcup J$,
    and $I_k \subset I$ and $J_k \subset J$ denote the fibers over a given $k\in K$.
\end{lem}
\begin{proof}
    The (unique) map of $\infty$-properads $\pi\colon \calP \to \Csp$ from \cref{thm:Csp-is-final} sends 
    $x = \otimes_{i \in I} x_i$ to $I \in \Csp$ and $y = \otimes_{j \in J} y_j$ to $J \in \Csp$.
    It thus induces maps of fiber sequences
    \[\begin{tikzcd}
        \Map_\calP(\otimes_{i \in I} x_i, \otimes_{j \in J} y_j) \ar[r] \ar[d, "q"'] &
        \xN_1\calP \ar[r, "{(s,t)}"] \ar[d, "{\xN_1\pi}"] & \xN_0\calP \times \xN_0\calP \ar[d, "{\xN_0\pi \times \xN_0\pi}"] \\
        \Map_{\Csp}(I, J) = (\Fin_{I \sqcup J/})^\simeq \ar[r] &
        \xN_1\Csp \ar[r, "{(s,t)}"] & \xN_0\Csp \times \xN_0\Csp .
    \end{tikzcd}\]
    We can write the source of $q$ as a colimit of its fibers \cite[Corollary 3.3.4.3]{HTT} to get
    \[
        \Map_\calP(\otimes_{i \in I} x_i, \otimes_{j \in J} y_j) 
        = \colim_{(\alpha\colon I\sqcup J \to K) \in (\Fin_{I \sqcup J/})^\simeq} \Map_\calP^\alpha(\otimes_{i \in I} x_i, \otimes_{j \in J} y_j)
    \] 
    where $\Map_\calP^\alpha(-,-)$ denotes the fiber of $q$ over $(\alpha\colon I\sqcup J \to K)$.
    In the right square of the above diagram both vertical maps are equifibered, so by applying \cref{obs:eqf=equi-fibered} to both maps we get that on fibers the maps
    \[
        \Map_\calP^\alpha(x, y) \times \Map_\calP^{\beta}(x', y') \too
        \Map_\calP^{\alpha\sqcup \beta}(x\otimes x', y \otimes y') 
    \]
    are equivalences.
    Any cospan $(\alpha\colon I\sqcup J \to K)$ canonically decomposes as 
    $\bigsqcup_{k \in K} (\alpha_k\colon I_k\sqcup J_k \to \{k\})$,
    and so we get that
    \[
        \prod_{k \in K}  \Map_\calP^{\alpha_k}(\otimes_{i \in I_k} x_i, \otimes_{j \in J_k} y_j) \too
        \Map_\calP^\alpha(\otimes_{i \in I} x_i, \otimes_{j \in J} y_j) 
    \]
    is an equivalence.
    Moreover, $\Map_\calP^{\alpha_k}(\otimes_{i \in I_k} x_i, \otimes_{j \in J_k} y_k) = \calP(\{x_i\}_{i \in I_k}; \{y_j\}_{j \in J_k})$ since the fiber over the cospan $I_k \to * \leftarrow J_k$ exactly picks out those morphisms that are indecomposable under $\otimes$.
    This proves the first claimed equivalence.
    
    The second equivalence can be obtained by rewriting the colimit to obtain the formula for $\xF(\Map_\calP(\unit, \unit))$ from \cref{lem:basic-properties-for-xF}.
    Alternatively, we have $\xN_1\calP = \xN_1(\calP_0) \oplus K_1(\calP)$ as in the proof of \cref{lem:nerve-splitting} below, which induces
    \[
        \Map_\calP(x, y) \simeq \Map_\calP(\unit, \unit) \times \Map_\calP^{\rm surj}(x, y)
    \]
    where $\Map_\calP^{\rm surj}(x,y) \subset \Map_\calP(x,y)$ is the subspace of those morphisms that are sent to a cospan $I \to K \leftarrow J$ for which $I\sqcup J \twoheadrightarrow K$ is surjective.
    We can then restrict the equivalence from the first claim to describe this subspace.
\end{proof}

\begin{example}\label{ex:bord}
    Let $\Bord_d$ be the $\infty$-category whose objects are closed unoriented
    $(d-1)$-manifolds and whose morphisms are compact unoriented $d$-bordisms.
    To give a precise definition one constructs a Segal space $\mrm{PBord}^d_\bullet$ 
    and defines $\Bord_d$ as its completion.
    We refer the reader to \cite{CS22} for the construction%
    \footnote{
    Though note that our $\Bord_d$ denotes the $(\infty,1)$-category, whereas in \cite{CS22} it denotes the fully extended $(\infty,d)$-category.
    One can recover the $(\infty,1)$-category by setting the first $d-1$ simplicial coordinates to $[0] \in \Dop$ and requiring all manifolds of dimension $\le d-2$ to be empty.
    }
    of $\mrm{PBord}^d_\bullet$ as a $1$-functor
    $\simp^\op \times \Fin_* \to \mrm{Top}$. 
    After composing with $\mrm{Top} \to \calS$ this yields a functor 
    $\simp^\op \to \Fun(\Fin_*, \calS)$, which by \cite[Proposition 7.2]{CS22} lands
    in $\CMon \subset \Fun(\Fin_*, \calS)$ and by \cite[Proposition 5.19]{CS22} satisfies the Segal condition.
    
    As a commutative monoid the space $\mrm{PBord}^d_n$ of $n$ composable bordisms is freely generated by 
    those $n$-tuples $(W_1,\dots,W_n)$ of bordisms for which the composite
    $W_1 \cup \dots \cup W_n$ is connected.
    To prove this, one checks that the map $\xN_n(\Bord_d) \to \Fin^\simeq$
    sending $(W_1,\dots,W_n)$ to the finite set $\pi_0(W_1 \cup \dots \cup W_n)$
    is equifibered.
    The face map $d_1\colon \xN_2(\Bord_d) \to \xN_1(\Bord_d)$ preserves
    this connectedness and hence is a free map.
    When $\mrm{PBord}^d_\bullet$ is complete (e.g.\ $d \le 2$)%
    \footnote{
        This completeness seems to be well-known for $d \le 2$, but we were unable to find a proof in the literature.
        However, as pointed out above, we do not actually require completeness for the purpose of this paper.
        },
    this shows that $\Bord_d$ is an $\infty$-properad.
    
    In higher dimensions the Segal space $\mrm{PBord}^d_\bullet$ 
    is not complete,
    but the above argument still shows that $\mrm{PBord}^d_\bullet$
    is a pre-properad in the sense of \cref{defn:pre-properad}.
    Hence, it follows from \cref{prop:nerves-and-completion}
    that its completion $\Bord_d$ is an $\infty$-properad.
\end{example}

\begin{example}\label{example:endomorphism-prpd}
    Let $\calC$ be a symmetric monoidal $1$-category and $x \in \calC$.
    The \textit{endomorphism properad} of $x$ is the discrete $1$-coloured properad whose set of operations at arity $(k,l)$ is given by $\mrm{Hom}_{\calC}(x^{\otimes k},x^{\otimes l})$ and whose properad structure maps are dictated by the composition in $\calC$.
    Restricting to arities of the form $(k,1)$ recovers the well known \textit{endomorphism operad} of $x$.
    In \cref{defn:end-properad} we generalize this and introduce the \textit{endomorphism $\infty$-properad} of an object in an arbitrary symmetric monoidal $\infty$-category.
    This will be related to the notion of algebra in the expected way.
\end{example}

\begin{example}\label{ex:U(discrete-monoid)}
    For a discrete commutative monoid $M \in \CMon(\Sets)$ we can define an ``$M$-weighted cospan category'' $\Csp(M)$ as follows.
    The objects of $\Csp(M)$ are finite sets and the morphisms are cospans $A \to X \leftarrow B$ together with a labelling $m_X\colon X \to M$.
    When composing cospans we add their labels in the sense that 
    $m\colon  X \amalg_B Y \to M$ is obtained from $m_X \sqcup m_Y\colon  X \amalg Y \to M$
    by summing over the fibers of $X \amalg Y \to X \amalg_B Y$.
    This can be made into a symmetric monoidal $(2,1)$-category with an equifibered forgetful symmetric monoidal functor $\Csp(M) \to \Csp$,
    but we will not construct the necessary coherence here.
    This is similar to \cite[Definition 2.13]{Jan-tropical}, and it is also a special case of the ``decorated cospan categories'' of Fong \cite{Fon15} (see also \cite{BCV22}).
    As far as we understand, this is the only connection between labelled cospan categories \cite{Jan-tropical} and decorated cospan categories \cite{Fon15}.
    
    If one ``de-loops'' $M$ into a symmetric monoidal category $\mfr{B}(M)$ with a single object $\unit$, then $\Csp(M)$ is exactly the endomorphism properad of $\unit$ in $\mfr{B}(M)$ in the sense of \cref{example:endomorphism-prpd}.
    Using the morphism $\infty$-properad of \cref{defn:underlying-properad} we have $\Csp(M) \simeq \calU(\mfr{B}(M))$.
    Note that $\calU(\mfr{B}(M))$ provides a definition of $\Csp(M)$ when $M \in \CMon$ is a not necessarily discrete commutative monoid.
\end{example}
    
In \cref{lem:disjunctive-Csp-level-wise-free} we showed the following:
\begin{lem}\label{lem:cospan-properad-coprod-disj}
    Let $\calC \in \Cat$ be a finitely cocomplete $\amalg$-disjunctive $\infty$-category and suppose that $\pi_0(\calC^\simeq)$ is generated by indecomposables.
    Then $\Cospan(\calC)$ is an $\infty$-properad.
\end{lem}

Without the ``generated by indecomposables'' assumption the nerve $\xN_\bullet(\Cospan(\calC))$ is only level-wise pseudo-free in the sense of \cref{def:pseudo-free}.

\begin{example} 
    Let $\calS^{\pifin} \subseteq \calS$ denote the full subcategory spanned by spaces with finitely many connected components.
    Then $\CoSpan(\calS^{\pifin})$ is an $\infty$-properad.
\end{example}

\begin{example}
    Let $\frX$ be an $\infty$-topos and write $\frX^{\pifin} \subseteq \frX$ 
    for the full subcategory spanned by objects $X \in \frX$ 
    whose poset of sub-objects $\mrm{Sub}(X)$ is a finite boolean algebra.
    Then $\frX^{\pifin}$ satisfies the conditions of \cref{lem:cospan-properad-coprod-disj}, hence $\Cospan(\frX^{\pifin})$ is an $\infty$-properad.
\end{example}

\begin{example}\label{ex:tangential}
    For a space $B \in \calS$ the slice category $\calS_{/B}^{\pifin}$ satisfies the conditions of \cref{lem:cospan-properad-coprod-disj} and hence $\Cospan(\calS_{/B}^{\pifin})$ is an $\infty$-properad.
    There is a symmetric monoidal functor $\tau\colon \Bord_d \to \Cospan(\calS_{/BO(d)}^{\pifin})$ that sends a manifold $M$ 
    to its underlying space equipped with the map $\tau_M\colon  M \to BO(d)$ that classifies the tangent bundle. 
    This functor is equifibered as it sends connected bordisms $W\colon M \to N$ to cospans $((M,\tau_M) \to (W,\tau_W) \leftarrow (N,\tau_N))$ where the tip $(W, \tau_W)$ is connected, and so is level-wise free.
    Given some map $\theta\colon  B \to BO(d)$, post-composition with $\theta$ also defines an equifibered symmetric monoidal functor $\theta_!$, and we may form the following pullback square in $\Prpd$:
    \[\begin{tikzcd}
    	{\Bord_d^\theta} & {\Cospan(\calS_{/B}^{\pifin})} \\
    	{\Bord_d} & {\Cospan(\calS_{/BO(d)}^{\pifin}).}
    	\arrow["{\theta_!}", from=1-2, to=2-2]
    	\arrow[from=1-1, to=2-1]
    	\arrow[from=1-1, to=1-2]
    	\arrow["\tau", from=2-1, to=2-2]
    	\arrow["\lrcorner"{anchor=center, pos=0.125}, draw=none, from=1-1, to=2-2]
    \end{tikzcd}\]
    This pullback may be computed in $\SM$ (in fact in $\Cat$) as the inclusion $\Prpd \to \SM$ preserves contractible limits by \cref{cor:forgetful-lim-and-colim}.
    The symmetric monoidal $\infty$-category $\Bord_d^\theta$ is the $\theta$-structured bordism category.
    For instance, if $\theta\colon  BSO(d) \to BO(d)$ is the orientation double-cover, then $\Bord_d^\theta$ is the oriented bordism category.
\end{example}

\begin{example}
    Let $\calE \in \CAlg(\PrLst)$ be a presentably symmetric monoidal $\infty$-category and write $\Aff(\calE)^{\Zarfin} \subseteq \CAlg(\calE)^\op$
    for the full subcategory spanned by commutative algebras 
    $A \in \CAlg(\calE)$ 
    such that the ring
    $\pi_0 \Map_{\calE}(\mbf{1},A)$
    has finitely many idempotents.
    Then $\Cospan(\Aff(\calE)^{\Zarfin})$ is an $\infty$-properad.
    By \cref{ex:comm-algebras-coprod-disj} $\CAlg(\calE)^\op$ is $\amalg$-disjunctive,
    so to apply \cref{lem:cospan-properad-coprod-disj} it suffices to show that $\pi_0(\Aff(\calE)^{\Zarfin,\simeq})$ is generated by indecomposables.
    The indecomposables are those rings for which $1 \in \pi_0 \Map_\calE(\mbf{1}, A)$ is the only non-zero idempotent.
    These generate because any $A$ which contains a non-zero idempotent different from $1$ can be split as a product $A_0 \times A_1$
    and this terminates as the $A_i$ have strictly fewer idempotents than $A$.
\end{example}

\begin{example}
    The $(2,1)$-category $\Span(\Fin)$ where objects are finite sets, 
    morphisms are spans $A \leftarrow X \to B$,
    and the monodial structure is given by disjoint union,
    is not an $\infty$-properad.
    Although $\xN_n\Span(\Fin)$ is free for all $n$
    (by a similar argument as in \cref{lem:Csp-level-wise-free}),
    the composition map $d_1\colon \xN_2\Span(\Fin) \to \xN_1\Span(\Fin)$
    is \textit{not} equifibered.
    To see this note that $\xN_1\Span(\Fin)$ is free on spans
    $(A_0 \leftarrow X \to A_1)$ such that $A_0 \amalg_X A_1$ 
    has exactly one element,
    and $\xN_2\Span(\Fin)$ is free on pairs of composable spans
    $(A_0 \leftarrow X \to A_1, A_1 \leftarrow Y \to A_2)$
    such that $A_0 \amalg_X A_1 \amalg_Y A_2$ has exactly one element.
    In particular 
    $(\emptyset \leftarrow \emptyset \to *, * \leftarrow \emptyset \to \emptyset)$
    is a generator, but the composition $(\emptyset \leftarrow \emptyset \to \emptyset)$ is not.
    Therefore, $d_1\colon \xN_2\Span(\Fin) \to \xN_1\Span(\Fin)$ is not free
    and hence $\Span(\Fin)$ not an $\infty$-properad.
    
    However, the subcategory $\Span(\Fin)^{\rm f-surj} \subset \Span(\Fin)$,
    which only contains spans $(A \leftarrow X \twoheadrightarrow B)$ where the forward map is surjective,
    is an $\infty$-properad.
    Indeed, in this case the canonical map 
    \[
        A_0 \amalg_{X \times_{A_1} Y} A_2 
        \cong A_0 \amalg_X (X \amalg_{X \times_{A_1} Y} Y) \amalg_Y A_2 
        \too A_0 \amalg_X A_1 \amalg_Y A_2
    \]
    is always a bijection because $X \amalg_{X \times_{A_1} Y} Y \cong A_1$ whenever $X \twoheadrightarrow A_1$ is surjective.
    Therefore, our previous considerations about $d_1\colon \xN_2\Span(\Fin) \to \xN_1\Span(\Fin)$ 
    show that it is equifibered when restricted to $\Span(\Fin)^{\rm f-surj}$.
\end{example}

\begin{rem}\label{rem:Span-not-prpd}
    If $\calC$ is a symmetric monoidal $1$-category, then symmetric monoidal functors $\Span(\Fin) \to \calC$ correspond to commutative bialgebras in $\calC$, whereas symmetric monoidal functors $\Span(\Fin)^{\rm f-surj} \to \calC$ correspond to \emph{non-counital} bialgebras in $\calC$.%
    \footnote{
        We do not claim to prove this here, and we do not make any claim about the situation when $\calC$ is a symmetric monoidal $\infty$-category.
        This merely serves as motivation.
    }
    Therefore, the above example can be understood as saying that non-counital bialgebras are controlled by a properad, but bialgebras are not.

    To identify the issue, let 
    $A \colon \Span(\Fin) \to \calC$ be a bialgebra encoded as a symmetric monoidal functor.
    Its unit and counit maps are given respectively as follows
    \[u :=  A(\emptyset \leftarrow \emptyset \to *) \colon \unit \to A(\ast), \qquad c := A(* \leftarrow \emptyset \too \emptyset)  \colon A(\ast) \too \unit\]
    In particular $u$ and $c$ have arity $(0,1)$ and $(1,0)$ respectively.
    One of the axioms for a bialgebra postulates a non-homogenous relation
    $c \circ u = \id_\unit$
    between the $(0,0)$-ary operation $c \circ u$ and the $(1,1)$-ary operation $\id_\unit$.
    This relation is witnessed in $\Span(\Fin)$ by the composition:
    \[(\emptyset \leftarrow \emptyset \to *) \circ (* \leftarrow \emptyset \to \emptyset) \simeq (\emptyset \leftarrow \emptyset \to \emptyset).\]
    Such a non-homogenous relation is impossible to encode using a properad as it contradicts the homogeneity of the composition map with respect to the grading by arity.
    However, while $\Span(\Fin)$ is not an $\infty$-properad, it is still a projective $\infty$-properad in the sense of \cref{defn:projective-prpd},
    as we shall see in \cref{ex:Span-proj}.
\end{rem}

\subsection{Properties and constructions of \texorpdfstring{$\infty$-properads}{infinity-properads}}\label{sec:prpd-eqf-Csp}

\subsubsection{Properads are equifibered over $\Csp$}

In \cref{subsec:pre-properads} we will show that $\Csp \in \Prpd$ is the terminal $\infty$-properad.
We will now discuss some of the consequences this has for categorical properties of $\Prpd$.

\begin{lem}\label{lem:eqf-over-properad}
    Let $f\colon  \calC \to \calP$ be an equifibered symmetric monoidal functor
    such that $\calP$ is an $\infty$-properad. 
    Then $\calC$ is an $\infty$-properad.
\end{lem}
\begin{proof}
    An equifibered symmetric monoidal functor induces an equifibered
    map $\xN_\bullet(f)\colon  \xN_\bullet\calC \to \xN_\bullet\calP$ on nerves.
    Hence, the vertical maps in the commutative square
    \[
        \begin{tikzcd}
            \xN_2\calC \ar[r, "d_1"] \ar[d, "f_2"'] & 
            \xN_1\calC \ar[d, "f_1"]\\
            \xN_2\calP \ar[r, "d_1"] & 
            \xN_1\calP
        \end{tikzcd}
    \]
    are equifibered.
    Since $\xN_1\calP$ is free it follows that $\xN_1\calC$ is free
    and since $d_1\colon \xN_2\calP \to \xN_1\calP$ is equifibered
    it follows by cancellation (\cref{obs:eqf-cancellation}) that $d_1\colon \xN_2 \calC \to \xN_1\calC$
    is equifibered. 
\end{proof}

\begin{cor}\label{cor:sub-properad}
    Let $\calP$ be an $\infty$-properad and $\calQ \subset \calP$ be 
    a replete symmetric monoidal subcategory satisfying:
    \begin{itemize}
        \item for any two morphisms $(f_1\colon x_1 \to y_1), (f_2\colon x_2 \to y_2) \in \Ar(\calP)$
        we have that if $f_1 \otimes f_2 \in \Ar(\calQ)$, 
        then $f_1 \in \Ar(\calQ)$ and $f_2 \in \Ar(\calQ)$.
    \end{itemize}
    Then $\calQ$ is an $\infty$-properad.
\end{cor}
\begin{proof}
    Combining \cref{lem:eqf-over-properad} and \cref{lem:eqf-detected-on-N1} the statement reduces to the claim that $N_1\calQ \hookrightarrow N_1\calP$ is equifibered,
    which follows from \cref{lem:submonoid-eqf}.
\end{proof}

\begin{defn}\label{defn:sub-properad}
    In the situation of \cref{cor:sub-properad} we say that $\calQ$ is a \hldef{sub-$\infty$-properad} of $\calP$.
    If the inclusion $\calQ \hookrightarrow \calP$ is full we say that $\calQ$ is a \hldef{full sub-$\infty$-properad} of $\calP$.
\end{defn}

Full sub-$\infty$-properads are classified as follows:
\begin{cor}\label{cor:subproperad-classification}
    Let $\calP$ be an $\infty$-properad.
    There is an inclusion-preserving bijection
    \[\{\text{full sub-$\infty$-properads of $\calP$}\} \xtoo{\cong}  \{\text{subsets of $\pi_0(\calP^\simeq)^\el$}\}\]
    defined by sending $\calQ \subset \calP$ to $\pi_0(\calQ^\simeq)^\el \subset \pi_0(\calP^\simeq)^\el$.
\end{cor}
\begin{proof}
    Full symmetric monoidal subcategories of $\calP$ are in bijection with submonoids $M \subset \pi_0(\calP^\simeq) \cong \bbN\langle\pi_0(\calP^\simeq)^\el\rangle$.
    By \cref{cor:sub-properad} such a submonoid corresponds to a sub-$\infty$-properad if and only if it satisfies that $a+ b \in M \Rightarrow a,b \in M$.
    Such submonoids of $\bbN\langle\pi_0(\calP^\simeq)^\el\rangle$ are precisely those generated by subsets $S \subset \pi_0(\calP^\simeq)^\el$.
\end{proof}

\begin{cor}\label{cor:Prpd-nice-properties}
    For $\calP\in \Prpd$ the $\infty$-category $(\Prpd)_{/\calP}$ is presentable. 
    Furthermore, the inclusion $(\Prpd)_{/\calP} \subset \SMover{\calP}$ admits left and right adjoints.
\end{cor}
\begin{proof}
    \cref{lem:eqf-over-properad} provides an identification $(\Prpd)_{/\calP} = \SMeq{\calP}$.
    Thus, by \cref{cor:eqf-slice-is-accessible} the inclusion $(\Prpd)_{/\calP} \hookrightarrow \SMover{\calP}$
    admits a left adjoint and this adjunction is accessible.
    By \cref{lem:colim-of-eqf-slice-categories}, 
    the inclusion $\SMeq{\calP} \hookrightarrow \SMover{\calP}$ 
    preserves all colimits and is therefore a left adjoint by the adjoint functor theorem.
\end{proof}

Once we show that $\Csp$ is a terminal object in $\Prpd$ in \cref{subsec:pre-properads},
we see that a symmetric monoidal $\infty$-category $\calC$ is an $\infty$-properad 
if and only if there is an equifibered symmetric monoidal functor
$\calC \to \Csp$. 
Moreover, this functor is canonical in the following sense:
\begin{thm}\label{thm:Csp-eqf=prpd}
    The forgetful functor $\SMover{\Csp} \too \SM$ restricts to an equivalence of $\infty$-categories: 
    \[ \SMeq{\Csp} \simeq \Prpd\]
\end{thm}
\begin{proof}
    \cref{lem:eqf-over-properad} provides an identification $(\Prpd)_{/\calP} = \SMeq{\calP}$.
    Note that the forgetful functor
    $(\Prpd)_{/\calP} \too \Prpd$ 
    is an equivalence if and only if $\calP$ is a terminal object in $\Prpd$.
    Hence, the theorem follows from \cref{thm:Csp-is-final}
    where we show that $\Csp \in \Prpd$ is terminal.
\end{proof}

\begin{rem}\label{rem:Prpd-presentable}
    Since $\SMeq{\Csp}$ is presentable by \cref{cor:Prpd-nice-properties}, a particular consequence of \cref{thm:Csp-eqf=prpd} is that $\Prpd$ is presentable.
    We shall see in \cref{cor:Prpd-compact-generation} that $\Prpd$ is in fact compactly generated.
    Note that, a priori, it is not at all clear that $\Prpd$ is presentable, when thought of as a replete subcategory of $\SM$.
\end{rem}

\begin{cor}\label{cor:forgetful-lim-and-colim}
    The inclusion functor $\Prpd \to \SM$ preserves all colimits and all contractible limits.
    Hence, it admits a right adjoint by the adjoint functor theorem.
\end{cor}
\begin{proof}
    By \cref{thm:Csp-eqf=prpd} we may consider the functors
    $\SMeq{\Csp} \hookrightarrow \SMover{\Csp} \to \SM$
    instead.
    The first one has both adjoints by \cref{cor:Prpd-nice-properties} and the second one commutes with all colimits and contractible limits \cite[Proposition 4.4.2.9.]{HTT}.
\end{proof}

\subsubsection{\texorpdfstring{$\infty$-operads as $\infty$-properads}{ Infinity-operads as infinity-properads}}

By restricting to $\infty$-properads ``where every operation has exactly one output'' one recovers the theory of $\infty$-operads.
We will be rather brief on this here, but we hope to explore it in more detail in future work.
\begin{defn}\label{defn:monic}
    An $\infty$-properad $\calP$ is called \hldef{monic} if the target map $t\colon  \xN_1 \calP \to \xN_0 \calP$ is equifibered.
    Equivalently, $\calP$ is monic if and only if $\calP(x_1,\dots,x_k; y_1,\dots,y_l) = \emptyset$ whenever $l\neq 1$.
    We let $\hldef{\Prpd^\monic} \subset \Prpd$ denote the full subcategory of monic $\infty$-properads.
\end{defn}

\begin{example}
    The key example of a monic $\infty$-properad is the category of finite sets $\Fin$ with its symmetric monoidal structure given by disjoint union.
    This in fact turns out to be the terminal monic $\infty$-properad.
    When thought of as an $\infty$-properad it is sub-terminal: indeed, we can think of it as the subproperad $\Fin \subset \Csp$ containing only those cospans whose backwards map is an equivalence.
\end{example}

\begin{rem}
    One can also call an $\infty$-properad $\calP$ \hldef{comonic} if the source map $s\colon  \xN_1 \calP \to \xN_0 \calP$ is equifibered.
    Note that the functor $\op\colon \SM \to \SM$ restricts to an equivalence
    between the $\infty$-categories of monic and comonic $\infty$-properads.
    In particular, it follows from \cref{thm:monic-prpd=operad} that the $\infty$-category of comonic $\infty$-properads is also equivalent to the $\infty$-category of $\infty$-operads.
\end{rem}

Restricting 
\cref{prop:TFAE-properad} to the monic case yields the following characterization:
\begin{cor}\label{cor:TFAE-monic-properad}
    For a symmetric monoidal $\infty$-category $\calP$ the following are equivalent:
    \begin{enumerate}[(1)]
        \item $\calP$ is a monic $\infty$-properad.
        \item $\xN_0\calP$ is free and the target map $t\colon \xN_1 \calP \to \xN_0 \calP$ is equifibered.
        \item There exists an equifibered symmetric monoidal functor $\calP \to \Fin$.
    \end{enumerate}
    Moreover, the equivalence of \cref{thm:Csp-eqf=prpd} restricts to an equivalence:
    \(
        (\SM)^{\eqf}_{/\Fin} \simeq \Prpd^{\monic} .
    \)
\end{cor}
\begin{proof}
    (1) $\Rightarrow$ (2) holds by definition. 
    (2) $\Rightarrow$ (1) holds by pullback and cancellation of equifibered maps, see \cref{lem:tensor-disjunctive-TFAE}.(\ref{item-tdj6}$\Rightarrow$\ref{item-tdj7}) below.
    (3) $\Rightarrow$ (2) follows as in \cref{lem:eqf-over-properad}.
    To see (1) $\Rightarrow$ (3), note that if $\calP$ is a monic properad, then every operation only has one output so the unique $\calP \to \Csp$ from \cref{cor:Csp-is-final} lands in the subcategory $\Fin \subset \Csp$ and $\calP \to \Fin$ is equifibered by cancellation.
\end{proof}

Monic $\infty$-properads are equivalent to $\infty$-operads in the sense of Lurie \cite{HA}.
\begin{thm}[Haugseng--Kock, Barkan--Haugseng--Steinebrunner]\label{thm:monic-prpd=operad}
    Lurie's envelope construction restricts to an equivalence of $\infty$-categories:
    \[
        \Env\colon  \Op \xrightarrow{\ \simeq\ } \Prpd^\monic.
    \]
\end{thm}
\begin{proof}
    It was shown in \cite{HK21} that Lurie's envelope lifts to a fully faithful functor
    \[
        \Env\colon  \Op \hookrightarrow \SMover{\Fin}
    \]
    sending an $\infty$-operad $(p\colon \calO^\otimes \to \Fin_*)$ 
    to the symmetric monoidal functor $\Env(\calO) \to \Env(\Fin_*) \simeq \Fin$.
    Moreover, \cite{HK21} give a characterization of the essential image.
    In \cite{envelopes} it was observed that the essential image consists precisely of equifibered symmetric monoidal functors to $\Fin$.
    Therefore, the theorem follows from the final claim of \cref{cor:TFAE-monic-properad}.
\end{proof}

We also want to give one additional characterization that was already mentioned in the introduction and that resembles the ``hereditary condition'' \cite[\S 3.2]{BKW18}.
\begin{defn}\label{defn:tensor-disjunctive}
    A symmetric monoidal $\infty$-category $\calC$ is called \hldef{$\otimes$-disjunctive}
    if the natural functor
    \[
        \calC_{/x} \times \calC_{/y} \to \calC_{/x \otimes y},  \qquad
        (f\colon a \to x, g\colon b \to y) \mapsto (f \otimes g\colon  a \otimes b \to x \otimes y)
    \]
    is an equivalence for all $x, y \in \calC$.
\end{defn}

\begin{lem}\label{lem:tensor-disjunctive-TFAE}
    For a symmetric monoidal $\infty$-category $\calC$ 
    the following are equivalent:
    \begin{enumerate}
        \item\label{item-tdj1} $\calC$ is $\otimes$-disjunctive.
        \item\label{item-tdj3} The target fibration $t\colon \Ar(\calC) \to \calC$ is equifibered.
        \item\label{item-tdj6} The commutative monoid map $d_0\colon  \xN_1\calC \to \xN_0\calC$ is equifibered.
        \item\label{item-tdj7} For all $0 \le i < n$ the commutative monoid map 
        $d_i\colon  \xN_n\calC \to \xN_{n-1}\calC$ is equifibered.
        \item\label{item-tdj8} The monoidal product $\otimes\colon \calC \times \calC \to \calC$ is a right-fibration.
    \end{enumerate}
\end{lem}
\begin{proof}
    (\ref{item-tdj1}) $\Leftrightarrow$ (\ref{item-tdj3}):
    The symmetric monoidal cocartesian fibration $t\colon \Ar(\calC) \to \calC$ classifies the functor $\calC_{/-} \colon \calC \to \Cat$ given on objects by $x \mapsto \calC_{/x}$.
    The cocartesian edges in $\Ar(\calC)$ are natural transformations inducing an equivalence on the source object, \cite[Corollary 2.4.7.12]{HTT}.
    In particular, $\otimes \colon \Ar(\calC) \times \Ar(\calC) \to \Ar(\calC)$ preserves cocartesian edges so by \cref{lem:eqf-cocartesian} $t$ is equifibered if and only if the functor 
    $\calC_{/x} \times \calC_{/y} \to \calC_{/x \otimes y}$,
    obtained by restricting the monoidal product to the fibers,
    is an equivalence. 
    This is exactly saying that $\calC$ is $\otimes$-disjunctive.

    (\ref{item-tdj3}) $\Rightarrow$ (\ref{item-tdj6}):
    By \cref{lem:eqf-detected-on-N1}, if $t$ is equifibered then so is $\xN_0(t) \colon \xN_0(\Ar(\calC)) \simeq \xN_1(\calC) \xrightarrow{d_0} \xN_0(\calC)$.
    
    (\ref{item-tdj6}) $\Rightarrow$ (\ref{item-tdj7}):
    The map $d_0\colon \xN_n\calC \to \xN_{n-1}\calC$ is equifibered because it is equivalent to $(\id, d_0)\colon \xN_{n-1}\calC \times_{\xN_0\calC} \xN_1\calC \to \xN_{n-1}\calC \times_{\xN_0\calC} \xN_0\calC$
    and equifibered maps are closed under limits in the arrow category.
    For $0\le i < n$ the face map $d_i\colon \xN_n\calC \to \xN_{n-1}\calC$ satisfies $(d_0)^{n-1} \circ d_i = (d_0)^{n}$, so it follows by cancellation that $d_i$ is equifibered.
    
    (\ref{item-tdj7}) $\Rightarrow$ (\ref{item-tdj3}):
    In order to show that $t\colon \Ar(\calC) \to \calC$ is equifibered
    is suffices, by \cref{lem:eqf-detected-on-N1}, to show that
    $\xN_1(t)\colon \xN_1\Ar(\calC) \to \xN_1\calC$ is equifibered.
    Indeed, we may write this map as the composite of equifibered maps as follows:
    \[
        \xN_1(t)\colon  \xN_1\Ar(\calC) 
        \simeq \xN_2\calC \times_{\xN_1\calC} \xN_2\calC
        \xrightarrow{(d_1, \id)} \xN_1\calC \times_{\xN_1\calC} \xN_2\calC
        \simeq \xN_2\calC
        \xrightarrow{d_0} \xN_1\calC
    \]
    Here the first equivalence uses that we can write 
    $\Delta^1 \times \Delta^1 \simeq \Delta^2 \amalg_{\Delta^1} \Delta^2$
    where the two $2$-simplices are glued along their long edge.
    
    (\ref{item-tdj6}) $\Leftrightarrow$ (\ref{item-tdj8}):
    A functor $F\colon \calD \to \calE$ is a right fibration if and only if the square
    \[\begin{tikzcd}
	{\Ar(\calD)} & {\calD} \\
	{\Ar(\calE)} & {\calE}
	\arrow["\Ar(F)"', from=1-1, to=2-1]
	\arrow["t", from=1-1, to=1-2]
	\arrow["t", from=2-1, to=2-2]
	\arrow["F", from=1-2, to=2-2]
    \end{tikzcd}\]
    is cartesian \cite[Tag 00TE]{Kerodon}. 
    In the case of $\otimes\colon \calC \times \calC \to \calC$ this square precisely says that $t\colon \Ar(\calC) \to \calC$ is equifibered.
\end{proof}

\begin{cor}\label{lem:monic-properad-tensor-disjunctive}
    A symmetric monoidal $\infty$-category $\calP$ is a monic $\infty$-properad if and only if $\calP^\simeq$ is free and $\calP$ is $\otimes$-disjunctive.
\end{cor}

\begin{example}
    Let $\mrm{Mfd}_n^{\rm or}$ denote the $\infty$-category obtained from the topologically enriched category
    where objects are compact unoriented $n$-dimensional manifolds with boundary
    and the morphisms spaces are the space of embeddings, equipped with the Whitney $C^\infty$-topology.
    This is a symmetric monoidal $\infty$-category with respect to disjoint union.

    Consider the symmetric monoidal functor $\pi_0\colon \mrm{Mfd}_n^{\rm or} \to \Fin$ that sends a manifold to its set of connected components.
    The square
    \[
        \begin{tikzcd}
            \mrm{Mfd}_n^{\rm or} \times \mrm{Mfd}_n^{\rm or} \ar[r, "\amalg"] \ar[d, "\pi_0 \times \pi_0"'] & 
            \mrm{Mfd}_n^{\rm or} \ar[d, "\pi_0"] \\
            \Fin \times \Fin \ar[r, "\amalg"] &
            \Fin
        \end{tikzcd}
    \]
    is cartesian since giving a disjoint decomposition $M = M_0 \amalg M_1$ of a manifold $M$ is equivalent to giving a disjoint decomposition $\pi_0(M) = A \amalg B$ of its set of path components.
    Therefore, $\mrm{Mfd}_n^{\rm or}$ admits an equifibered symmetric monoidal functor to $\Fin$ and is hence a monic $\infty$-properad by \cref{cor:TFAE-monic-properad}.
    
    We can also further restrict to the sub-properad $\mrm{Disk}_n \subset \mrm{Mfd}_n^{\rm or}$ where the manifolds are required to be disjoint unions of disks.
    This also is a monic $\infty$-properad and under the equivalence of \cref{thm:monic-prpd=operad} it corresponds to the ``framed'' little $n$-disks operad.
    We can also obtain the $E_n$-operad this way, if we restrict our attention to standard disks and require all inclusions to be component-wise rectilinear.
\end{example}

\begin{obs}\label{obs:monic-properad-adjoint}
    The full inclusion $\Prpd^\monic \subset \Prpd$ has a right adjoint 
    \[\Fin \times_{\Csp} (-)\colon  \SMeq{\Csp} \too \SMeq{\Fin}\]
    that discards all operations of arity $(n,m)$ with $m \neq 1$. 
    This works because both pullback along and composition with the inclusion $\Fin \to \Csp$ preserve equifibered maps.
\end{obs}

\subsubsection{Free \texorpdfstring{$\infty$-properads}{infinity-properads} and corollas}

We now construct the free $\infty$-properad on an operation of arity $(k,l)$.
This will be extremely useful later on as we can use it to compute the spaces of operations an $\infty$-properad by mapping into it.

\begin{defn}\label{def:Nel}
    We define a functor $\xN_1^\el\colon \Prpd \to \calS_{/\xN_1^\el\Csp}$ as the composite
    \[
        \xN_1^\el\colon  
        \Prpd \simeq \SMeq{\Csp} \xtoo{\xN_1} 
        \CMon^\eqf_{/\xN_1\Csp} \simeq \calS_{/\xF(*) \times \xF(*)}.
    \]
    Here the last equivalence is given by forgetting the commutative monoid structure and pulling back along the inclusion
    $\xF(*) \times \xF(*) \simeq \xN_1^\el \Csp \hookrightarrow \xN_1 \Csp$
    as in \cref{cor: equifibered-over-free-equals-spaces}.
\end{defn}

\begin{obs}\label{obs:N1el-conservative}
    As in \cref{obs:N1-conservative}, the functor $\xN_1^\el$ is conservative.
    Note, however, that unlike in \cref{obs:N1-conservative} the functor $\xN_1^\el$ is not co-represented by a single $\infty$-properad. 
    (In particular, $\Map_{\Prpd}(\xF([1]), -)$ is \emph{not} $\xN_1^\el$.)
    We therefore cannot conclude that $\Prpd$ is generated by a single compact object, but we will soon describe a countable set of compact generators given by the ``free corollas''.
\end{obs}

\begin{rem}
    The $\infty$-category $\calS_{/\xF(*) \times \xF(*)} \simeq \Fun(\Fin^\simeq \times \Fin^\simeq, \calS)$ 
    is the $\infty$-category of one-coloured bisymmetric sequences.
    Valette \cite{Val07} originally defined $1$-properads as algebras for a certain ``connected composition product'' on the category of bisymmetric sequences in chain complexes.
    We hope to show in future work that $\Prpd$ may be expressed as algebras in an $\infty$-category of \emph{coloured} bisymmetric sequences.
    In the present situation, the functor $\xN_1^\el$ groups together operations of all colours, we therefore expect the adjunction below to \textit{not} be monadic.
    (Though it should be monadic if one restricts to the $\infty$-category of one-coloured $\infty$-properads.)
\end{rem}

\begin{prop}\label{prop:free-properad}
    The functor $\xN_1^\el\colon  \Prpd \too \calS_{/\xF(*)\times \xF(*)}$
    commutes with filtered colimits and is right adjoint to the \hldef{free properad} functor
    \[
        \FPrpd\colon \calS_{/\xF(*) \times \xF(*)} \adj \Prpd :\! \xN_1^\el .
    \]
    Moreover, the free $\infty$-properad on $X \to \xF(*) \times \xF(*)$ is obtained as a contrafibered-equifibered factorization:
    \[
        \xF(X \times [1]) \xrightarrow{\dis}
        \FPrpd(X) \xrightarrow{\eqf}
        \Csp
    \]
\end{prop}
\begin{proof}
    The functor $\xN_1^\el$ can be factored as
    \[
        \xN_1^\el\colon  \Prpd \simeq \SMeq{\Csp} \hookrightarrow
        \SMover{\Csp} \xtoo{\text{forget}}
        \Catover{\Csp} \xtoo{\xN_1}
        \calS_{/\xN_1\Csp} \xtoo{\iota^*}
        \calS_{/\xF(*) \times \xF(*)} 
    \]
    where the last functor is pullback along the inclusion
    $\iota\colon \xN_1^\el \Csp \hookrightarrow \xN_1 \Csp$.
    Since each of the functors involved commute with filtered colimits, so does $\xN_1^\el$.
    Passing to left adjoints gives the following factorization of $\FPrpd$
    \[
        \FPrpd\colon 
        \calS_{/\xF(*) \times \xF(*)} \xtoo{\iota_!}
        \calS_{/\xN_1\Csp} \xtoo{(- \times [1])}
        \Catover{\Csp} \xtoo{\xF}
        \SMover{\Csp} \xtoo{\bbL^\eqf}
        \SMeq{\Csp} \simeq \Prpd.
    \]
    The middle two left adjoints are obtained by slicing the adjunction
    \[
        \xF( - \times [1])\colon \calS \adj \Cat \adj \SM :\! \xN_1 \circ \text{forget}
    \]
    over $\Csp \in \SM$.
    (See \cite[Lemma 5.2.5.2]{HTT} in the case $K=\Delta^0$, $p_0 = \xN_1\Csp$, $p_1 = (\xN_1 \Csp) \times [1]$, and $h$ is the identity.)
    The last left adjoint in the factorization 
    $\SMover{\Csp} \to \SMeq{\Csp}$
    is given by sending $\calC \to \Csp$ to the equifibered part of the contrafibered-equifibered factorization $\calC \to \bbL^\eqf (\calC) \to \Csp$, 
    see \cref{cor:eqf-slice-is-accessible}.
\end{proof}

\begin{defn}\label{defn:free-corolla}
    For finite sets $A, B \in \Fin$ we define the 
    \hldef{free $(A,B)$-corolla $\mfr{c}_{A,B}$} 
    as the free $\infty$-properad on the object
    \[
        (* \xrightarrow{(A,B)} \Fin^\simeq \times \Fin^\simeq \simeq \xF(*) \times \xF(*))
        \in \calS_{/\xF(*) \times \xF(*)}.
    \]
    We also sometimes denote this by $\mfr{c}_{k,l}$ where $k$ and $l$ are the cardinalities of $A$ and $B$. 
\end{defn}

\begin{example}
    The free $(1,1)$-corolla is $\mfr{c}_{1,1} = \xF([1])$.
    Indeed, the functor $\xF([1]) \to \Csp$ that picks the cospan $(* \to * \leftarrow *)$ factors as $\xF([1]) \to \xF([0]) \to \Csp$ and is thus equifibered.
    So we do not need to perform the contrafibered-equifibered factorization in \cref{prop:free-properad}.
\end{example}

\begin{obs}\label{obs:operations-via-corollas}
    For any $\infty$-properad $\calP$ the space of morphisms $\Map_{\Prpd}(\mfr{c}_{A,B}, \calP)$ 
    is the fiber of $\xN_1^\el \calP \to \xF(*) \times \xF(*)$ at $(A,B) \in \xF(*) \times \xF(*)$.
    This can be thought of as the space of operations with 
    set of inputs $A$ and set of outputs $B$.
    We can recover the entire space of operations $\xN_1^\el \calP$ by taking 
    the colimit over $A$ and $B$.
    \begin{align*}
        \xN_1^\el \calP & \simeq \xF(*)^2 \times_{\xF(*)^2} \xN_1^\el \calP 
        \simeq \colim_{A,B \in \Fin^\simeq} \{(A,B)\} \times_{\xF(*)^2} \xN_1^\el \calP 
        \simeq \colim_{A,B \in \Fin^\simeq} \Map_{\Prpd}(\mfr{c}_{A,B}, \calP)
    \end{align*}
\end{obs}

The existence of free corollas has the following formal consequence:

\begin{cor}\label{cor:Prpd-compact-generation}
    The $\infty$-category $\Prpd$ is compactly generated by the corollas $\{\mfr{c}_{k,l}\}_{k,l \ge 0}$.
\end{cor}
\begin{proof}
    Given a pair of finite sets $(A,B)$ we write
    $e_{A,B} := ((A,B)\colon  * \too \xF(*) \times \xF(*)) \in \calS_{/\xF(*) \times \xF(*)}$.
    We first show that the corolla 
    $\mfr{c}_{A,B} = \FPrpd(e_{A,B}) \in \Prpd$ is compact.
    By \cref{prop:free-properad}, $\xN_1^\el$ preserves filtered colimits, hence its left adjoint $\FPrpd$ preserves compact objects \cite[Proposition 5.5.7.2]{HTT} so it suffices to show that
    $e_{A,B} \in \calS_{/\xF(*) \times \xF(*)}$ 
    is compact.
    To see this, observe that its co-representing functor may be written as 
    \[
        \Map_{\calS_{/\xF(\ast) \times \xF(\ast)}}(e_{A,B},-) \colon \calS_{/\xF(*) \times \xF(*)} \simeq \Fun(\Fin^\simeq \times \Fin^\simeq, \calS) 
        \xrightarrow{\ev_{A,B}} \calS,
    \]
    which manifestly commutes with all colimits.
    
    Since $\Prpd$ is presentable (\cref{rem:Prpd-presentable}) it remains to prove that the corollas generate $\Prpd$ under colimits. 
    By \cite[Corollary 2.5]{Yan21} it suffices to show that
    the functors 
    \[\Map_{\Prpd}(\mfr{c}_{\{1,\dots,k\},\{1,\dots,l\}}, -)\colon  \Prpd \too \calS\] 
    are jointly conservative which follows from \cref{obs:operations-via-corollas}, where we write $\xN^\el_1\calP$ as a colimit of mapping spaces out of free corollas,
    and \cref{obs:N1el-conservative}, where we note that
    $\xN_1^\el\colon \Prpd \to \calS$
    is conservative.
\end{proof}

We now give a description of the free corolla $\mfr{c}_{A,B}$ as a symmetric monoidal category.
This will be useful in \cref{defn:underlying-properad} where we study the right adjoint
to the forgetful functor $\Prpd \to \SM$.
\begin{lem}\label{lem:free-corolla-pushout}
    The free $(A,B)$-corolla fits into a pushout square of symmetric monoidal categories:
    \[
        \begin{tikzcd}
            {\xF(* \amalg *)} \ar[d] \ar[r, "\Delta_A \oplus \Delta_B"] \ar[dr, very near end, phantom, "\ulcorner"]&
            {\xF(A \amalg B)} \ar[d] \\
            {\xF([1])} \ar[r] &
            \mfr{c}_{A,B}
        \end{tikzcd}
    \]
    Moreover, this is a level-wise pushout square in the sense of \cref{obs:level-wise-colimits-in-SM}.
\end{lem}
\begin{proof}
    Consider the simplicial commutative monoid $M_\bullet$ obtained as the following pushout:
    \[
        \begin{tikzcd}
            {\xN_\bullet\xF(*\amalg*)} \ar[d] \ar[r, "\Delta_A \oplus \Delta_B"] \ar[dr, phantom, near end, "\ulcorner"]&
            {\xN_\bullet\xF(A \amalg B)} \ar[d] \ar[dr, bend left] & \\
            {\xN_\bullet\xF([1])} \ar[r] &
            M_\bullet \ar[r] &
            \xN_\bullet \Csp
        \end{tikzcd}
    \]
    Here the top map is the direct sum of the two maps $\Delta_A \colon \xF(*) \to \xF(A)$ and $\Delta_B\colon \xF(*) \to \xF(B)$, given by $\Delta_A(*) = \sum_{a \in A} a$ and $\Delta_B(*) = \sum_{b \in B} b$.
    These are contrafibered by \cref{ex:diagonal-ctf} and hence so is the pushout $\xN_\bullet \xF([1]) \too M_\bullet$.
    The curved arrow is the nerve of the (unique) equifibered
    functor $\xF(A \amalg B) \to \Fin^\simeq \subset \Csp$ and the bottom composite
    is the nerve of the functor $\xF([1]) \to \Csp$ that picks out the cospan $(A \to * \leftarrow B)$.
    It suffices now to prove the following statements:
    \begin{enumerate}[$(a)$]
        \item 
        $M_\bullet$ is a Segal space,
        \item 
        $M_\bullet$ is complete, and
        \item  
        the map $M_1 \too \xN_1 \Csp$ is equifibered.
    \end{enumerate}
    Indeed, by $(a)$ and $(b)$ the monoid $M_\bullet$ is equivalent to the nerve $\xN_\bullet \calP$ of some symmetric monoidal $\infty$-category $\calP$,
    and we have symmetric monoidal functors 
    $\xF([1]) \to \calP \to \Csp$.
    By (c) becomes a level-wise contrafibered-equifibered factorization after applying $\xN_\bullet$,
    and hence by \cref{lem:eqf-detected-on-N1} and \cref{cor:detect-ctf} it was already a contrafibered-equifibered factorization in $\SM$.
    Now it follows from \cref{prop:free-properad} that $\calP \simeq \mfr{c}_{A,B}$.
    Therefore, this concludes the proof that $\mfr{c}_{A,B}$ is a level-wise pushout in the sense of \cref{obs:level-wise-colimits-in-SM}.

    We now prove $(a)$. The left map in the pushout square
    \[\xF(\ast) \oplus \xF(\ast) \simeq \xN_n\xF(* \amalg *)  \to \xN_n \xF([1]) \simeq \xF(*)^{\oplus n+1},\]
    is the inclusion of the first and last factors and therefore
    $M_n \simeq \xF(A) \oplus \xF(*)^{\oplus n-1} \oplus \xF(B)$.

    For $n=1$ we have $M_1 \simeq \xF(A) \oplus \xF(*) \oplus \xF(B)$ and the generator $*$ of the middle term is a $1$-simplex with source $\sum_{a \in A} a$ and target $\sum_{b \in B} b$.
    Consequently, we can write the map $d_0\colon M_1 \to M_0$ as a direct sum
    \[
        (d_0\colon M_1 \to M_0) = 
        \left( \id\colon \xF(A) \to \xF(A) \right) \oplus
        \left( \Delta_B + \id \colon \xF(*) \oplus \xF(B) \to \xF(B) \right)
    \]
    and the map $d_1\colon M_1 \to M_0$ as a sum of $\Delta_A$ and identities.
    We have similar descriptions for other face maps in $M_\bullet$.
    To check (a) it suffices to show that $M_n \simeq M_{n-1} \times_{M_0} M_1$, i.e.~that the square depicted below is cartesian.
    We can decompose this square as a direct sum
\[\begin{tikzcd}
	{M_n} && {M_1} && \begin{array}{c} \substack{\xF(A) \oplus \xF(*) \\\\ \oplus \\\\ \xF(*)^{\oplus n-1} \oplus \xF(B)} \end{array} &&& \begin{array}{c} \substack{ \xF(A) \oplus \xF(*) \\\\ \oplus \\\\  \xF(B)} \end{array} \\
	&&& {=} \\
	{M_{n-1}} && {M_0} && \begin{array}{c} \substack{\xF(A) \\\\ \oplus \\\\\ \xF(*)^{n-1} \oplus \xF(B)} \end{array} &&& \begin{array}{c} \substack{\xF(A) \\\\ \oplus \\\\ \xF(B)} \end{array}
	\arrow["{d_0^{n-1}}", from=1-1, to=1-3]
	\arrow["{d_n}", from=1-1, to=3-1]
	\arrow["{d_1}", from=1-3, to=3-3]
	\arrow["\begin{array}{c} \substack{\id \\ \oplus \\ \Delta_B^{+(n-1)} + \id} \end{array}"{description}, from=1-5, to=1-8]
	\arrow["\begin{array}{c} \substack{\id + \Delta_A \\ \oplus \\\id} \end{array}", from=1-5, to=3-5]
	\arrow["\begin{array}{c} \substack{\id + \Delta_A \\ \oplus \\ \id} \end{array}", from=1-8, to=3-8]
	\arrow["{d_0^{n-1}}", from=3-1, to=3-3]
	\arrow["\begin{array}{c} \substack{\id \\ \oplus \\ \Delta_B^{+(n-1)} + \id} \end{array}"{description}, from=3-5, to=3-8]
\end{tikzcd}\]
    where $M_n$ is decomposed as $(\xF(A) \oplus \xF(*)) \oplus (\xF(*)^{\oplus n-1} \oplus \xF(B))$ such that the $\xF(*)$ in the left summand is generated by the $n$-simplex $(0\le \dots \le 0 \le 1)$ in $\xN_n([1])$.
    (Applying $d_0^{n-1}$ to this $n$-simplex yields the non-degenerate simplex $(0\le 1) \in \xN_1([1])$.)
    The first summand is cartesian because its horizontal arrows are identities and the second summand is cartesian because its vertical arrows are identities.
    It follows that $M_\bullet$ is Segal.

    To check completeness, we first show that there are no non-trivial equivalences in $M_\bullet$.
    For this, consider the map $F\colon M_\bullet \to \xN_\bullet(\bbN, \le)$ to the nerve of the poset of natural numbers, which we define on objects by the commutative monoid map $\xF(A \amalg B) \to \bbN$ that sends $A$ to $0$ and $B$ to $1$.
    Because we are mapping into a poset it suffices to check that this is well-defined on $1$-simplices, and indeed the only non-trivial generator has source and target $\sum_{a \in A} a \to \sum_{b \in B} b$, which is sent to $0 \le |B| \in \bbN$.
    Since $F$ is a map of Segal spaces it must send equivalences to equivalences, so 
    \[
        M_1^{\rm eq} \subset F^{-1}(\xN_1(\bbN,\le)^{\rm eq}) 
        = F^{-1}(\{\id_0, \id_1\})
        = \xF(A) \oplus \xF(B).
    \]
    But this is exactly the space of degenerate simplices, so we must have $M_1^{\rm eq} = \xF(A) \oplus \xF(B)$, which is indeed equivalent to $M_0$, proving completeness.

    Finally, to prove (c) we need to show that the map
    \[
        \xF(A) \oplus \xF(*) \oplus \xF(B) = M_1 \too \xN_1\Csp
    \]
    is equifibered.
    Since $\xN_1\Csp$ is free, the full subcategory
    $\CMon^\eqf_{/\xN_n\Csp} \subset \CMon_{/\xN_n\Csp}$
    is closed under direct sums (see \cref{lem:colim-of-eqf-slice-categories})
    and thus it suffices to check that the maps
    $\xF(\ast) \to \xN_1 \Csp$ and
    $\xF(A \amalg B) \to \xN_1\Csp$ are equifibered.
    For the former this is the case since it picks the elementary cospan $(A \to * \leftarrow B)$.
    The latter is equifibered because it can be obtained by applying $\xN_1(-)$ to the equifibered functor $\xF(A \amalg B) \to \Csp$.
\end{proof}

\begin{cor}
    The forgetful functor $\Prpd \too \SM$ preserves compact objects
    and its right adjoint $\calU: \SM \too \Prpd$ in \cref{defn:underlying-properad} preserves filtered colimits.
\end{cor}
\begin{proof}
    The two statements are equivalent by \cite[Proposition 5.5.7.2]{HTT}.
    In \cref{lem:free-corolla-pushout} we wrote the free corolla $\mfr{c}_{A,B}$ as a finite colimit of compact objects in $\SM$ and hence it is compact as an object of $\SM$.
    Since the free corollas are compact generators for $\Prpd$ (\cref{cor:Prpd-compact-generation}), it follows that the forgetful functor preserves compact objects.
\end{proof}

\subsubsection{Morphism and endomorphism $\infty$-properads}
In this section we study the \textit{morphism $\infty$-properad} functor, namely the right adjoint 
$\calU\colon  \SM \to \Prpd$ to the forgetful functor.
We shall see that for $\calC \in \SM$, the colours of $\calU(\calC)$ are precisely the objects of $\calC$, and the operations of $\calU(\calC)$ with source $(x_1,\dots,x_n) \in \calC^{\times n}$ and target $(y_1,\dots,y_m) \in \calC^{\times m}$ are maps $x_1 \otimes \dots \otimes x_n \to y_1 \otimes \dots \otimes y_m $.
We shall then define the endomorphism $\infty$-properad of an object $x \in \calC$ by passing to the full sub-$\infty$-properad of $\calU(\calC)$ spanned by $x$.

\begin{defn}\label{defn:underlying-properad}
    For a symmetric monoidal $\infty$-category $\calC$ we define the \hldef{morphism $\infty$-properad} $\hldef{\calU(\calC)}$ to be the image of $\calC$ under the right adjoint
    \[
        \mrm{include}\colon  \Prpd \adj \SM :\! \hldef{\calU}
    \]
    which exists by \cref{cor:forgetful-lim-and-colim}.
\end{defn}

Note that, as explained in the introduction, 
we may hence define a $\calP$-algebra in $\calC$ to be a morphism of $\infty$-properads from $\calP$ to the morphism $\infty$-properad $\calU(\calC)$,
or equivalently as a symmetric monoidal functor from $\calP$ to $\calC$:
\[
    \Alg_\calP(\calC) =
    \Fun^{\otimes, \eqf}(\calP, \calU(\calC)) 
    \simeq 
    \Fun^\otimes(\calP, \calC).
\]

\begin{obs}
    The forgetful functor can be factored as
    \[
        \Prpd \simeq \SMeq{\Csp} \into \SMover{\Csp} \too \SM,
    \]
    and hence $\calU$ can be described as the composite of right adjoints
    \[
        \calU\colon  \SM \xrightarrow{(-) \times \Csp} \SMover{\Csp} 
        \xrightarrow{\mbm{R}^\eqf} \SMeq{\Csp} \xrightarrow{\simeq} \Prpd.
    \]
    Here $\mbb{R}^\eqf$ is the right adjoint to the fully faithful inclusion $\SMeq{\Csp} \subset \SMover{\Csp}$,
    which exists by \cref{cor:Prpd-nice-properties}
    because $\Csp$ is an $\infty$-properad.%
    \footnote{
        Note that the inclusion $\SMeq{\calC} \subset \SMover{\calC}$ always has a left adjoint given by equifibered factorization, but the right adjoint that we use here requires $\calC$ to be an $\infty$-properad.
    }
\end{obs}

\begin{obs}\label{obs:underlying-properad-restricts-to-underlying-operad}
    Composing the above adjunction with the restriction adjunction from \cref{obs:monic-properad-adjoint} yields:
    \[
        \Env\colon  \Op \simeq \Prpd^\monic \adj \Prpd \adj \SM.
    \]
    The composite left adjoint is the envelope and the composite right adjoint is given 
    by sending a symmetric monodial $\infty$-category
    $\calC \colon \Fin_\ast \to \Cat$ to its unstraightening 
    $\Un_{\Fin_\ast}(\calC) \to \Fin_*$ thought of as an $\infty$-operad.%
    \footnote{In \cite{HA} symmetric monoidal $\infty$-categories are \textit{defined} as cocartesian fibrations over $\Fin_\ast$ so unstraightening is unnecessary.}
\end{obs}

We can now define endomorphism $\infty$-properads.

\begin{defn}\label{defn:end-properad}
    Let $\calC$ be a symmetric monoidal $\infty$-category and let $S \subseteq \calC$ be a collection of objects. 
    We write $\hldef{\calU(\calC)_{S}} \subseteq \calU(\calC)$ for the full sub-$\infty$-properad spanned by $S$ (see \cref{defn:sub-properad}).
    For $x \in \calC$ we define the \hldef{endomorphism $\infty$-properad} of $x$ as $\hldef{\End_{\calC}^\otimes(x)} \coloneq  \calU(\calC)_{\{x\}}$.
\end{defn}

To justify \cref{defn:end-properad} we shall now extract an explicit description of the operations of the morphism $\infty$-properad by mapping into it from free corollas.

\begin{lem}\label{lem:underlying-properad}
    Let $\calC$ be a symmetric monoidal $\infty$-category.
    Then the colours of $\calU(\calC)$ are the objects of $\calC$:
    $\xN_0^\el \calU(\calC) \simeq \xN_0\calC$.
    The operations of $\calU(\calC)$ fit into a pullback square:
    \[
    \begin{tikzcd}
        \xN_1^\el \calU(\calC) \ar[r] \ar[d] & 
        \xN_1 \calC \ar[d, "{(s,t)}"] \\
        \xF(\xN_0\calC) \times \xF(\xN_0\calC) \ar[r, "{+ \,\times\, +}"] &
        \xN_0\calC \times \xN_0\calC.
    \end{tikzcd}
    \]
\end{lem}
\begin{proof}
    The free symmetric monoidal category $\xF(*) \simeq \Fin^\simeq$ is a properad
    and we can compute
    \[
        \xN_0^\el \calU(\calC) 
        \simeq \Map_{\Prpd}(\xF(*), \calU(\calC))
        \simeq \Map_{\SM}(\xF(*), \calC) \simeq \xN_0 \calC.
    \]
    Further, we may use the free corolla $\mfr{c}_{A,B}$ and its description as a pushout in \cref{lem:free-corolla-pushout} to compute
    \begin{align*}
        \Map_{\Prpd}(\mfr{c}_{A,B}, \calU(\calC))
        &\simeq \Map_{\SM}(\mfr{c}_{A,B}, \calC) \\
        &\simeq \Map_{\SM}(\xF(A \amalg B), \calC) \times_{\Map_{\SM}(\xF(*\amalg*), \calC)} \Map_{\SM}(\xF(\Delta^1), \calC) \\
        & \simeq \Map(A \amalg B, \xN_0\calC) \times_{\xN_0\calC \times \xN_0\calC} \xN_1\calC.
    \end{align*}
    Taking the colimit over $A, B \in \Fin^\simeq$ as in \cref{obs:operations-via-corollas} we get
    \begin{align*}
        \xN_1^\el \calU(\calC) 
        & \simeq \colim_{A, B \in \Fin^\simeq} \Map_{\Prpd}(\mfr{c}_{A,B}, \calU(\calC)) \\
        & \simeq \colim_{A, B \in \Fin^\simeq} \Map(A \amalg B, \xN_0\calC) \times_{\xN_0\calC \times \xN_0\calC} \xN_1\calC \\
        & \simeq \xF(\xN_0\calC)^2 \times_{\xN_0\calC^2} \xN_1\calC
    \end{align*}
    as claimed.
\end{proof}

Spelling out the description in \cref{lem:underlying-properad} we see that the colours
    of $\calU(\calC)$ are objects $a \in \calC$ and the operations between two collections
    of colours $\{a_i\}_{i\in I}, \{b_i\}_{j \in J} \in \Fin_{/\xN_0\calC}^\simeq = \xF(\xN_0\calC)$
    are maps between their tensor products:
    \[
        \calU(\calC)(\{a_i\}_{i\in I}, \{b_i\}_{j \in J}) 
        \simeq \Map_{\calC}\left(\bigotimes_{i \in I} a_i, \bigotimes_{j \in J} b_j\right).
    \]
    In accordance with \cref{obs:underlying-properad-restricts-to-underlying-operad}, this matches Lurie's description of the underlying $\infty$-operad of a symmetric monoidal $\infty$-category when restricting to $|J| = 1$.
    Restricting to the full sub-$\infty$-properad $\calU(\calC)_{\{x\}} \subseteq \calU(\calC)$ the above justifies \cref{defn:end-properad} as a generalization of the classical endomorphism properad, delivering on our promise from \cref{example:endomorphism-prpd}.

\subsection{\texorpdfstring{$\Csp$}{Csp} is the terminal \texorpdfstring{$\infty$-properad}{infinity-properad}}\label{subsec:pre-properads}
In this section we prove that $\Csp \in \Prpd$ is the terminal $\infty$-properad, thereby proving \cref{thm:Csp-eqf=prpd}.
In fact, we will prove the slightly stronger assertion that $\Csp$ is terminal in a certain larger $\infty$-category of ``non-complete $\infty$-properads'' which contains $\Prpd$ as a full subcategory.

\subsubsection{Pre-properads and completion}
The following definition makes precise the notion of a ``non-complete $\infty$-properad''.

\begin{defn}\label{defn:pre-properad}
    A \hldef{pre-properad} is a functor $Q_\bullet\colon  \Dop \to \CMon$
    such that 
    \begin{enumerate}
        \item
        $Q_\bullet$ is a Segal space, i.e.\ for all $n\ge 2$ the canonical map induces an equivalence
        \[
            Q_n \iso Q_1 \times_{Q_0} \dots \times_{Q_0} Q_1.
        \]
        \item 
        The composition map $d_1\colon Q_2 \to Q_1$ is equifibered.
        \item 
        The commutative monoid $Q_1$ is free.
    \end{enumerate}
    A morphism of pre-properads $Q_\bullet \to P_\bullet$ is a natural transformation 
    such that each $Q_n \to P_n$ is equifibered.
    We let $\hldef{\pPrpd} \subset \Fun(\Dop, \CMon)$
    denote the replete subcategory of pre-properads and morphism of pre-properads.
\end{defn}

Observe that the first condition is always satisfied when $Q_\bullet$ is a nerve of a symmetric monoidal $\infty$-category.
So $\infty$-properads are precisely the symmetric monoidal $\infty$-categories whose nerve is a pre-properad.
 Furthermore, \cref{lem:eqf-detected-on-N1} a symmetric monoidal functor is equifibered if and only if it induces an equifibered morphism on nerves.
 We record this for future use.

\begin{cor}\label{cor: properads-as-preproperads}
    The natural square of inclusions
    \[ \begin{tikzcd}
	\Prpd & \pPrpd \\
	\SM & \Seg_{\Dop}(\CMon)
	\arrow[from=1-1, to=2-1]
	\arrow[hook, "{\xN_{\bullet}}", from=1-1, to=1-2]
	\arrow[hook, "{\xN_{\bullet}}", from=2-1, to=2-2]
	\arrow[from=1-2, to=2-2]
	\arrow[phantom, very near start, "\lrcorner", from=1-1, to=2-2]
    \end{tikzcd}\]
    is cartesian. In particular, the nerve induces a fully faithful inclusion $\xN_{\bullet}\colon \Prpd \hookrightarrow \pPrpd$.
\end{cor}

The axioms chosen in the \cref{defn:pre-properad} are in some sense minimal.
We could have equivalently asked 
for all $Q_n$ to be free and all inner face maps to be equifibered.
Indeed, this follows by applying \cref{lem:check-CULF-on-d1} to the map $+\colon  M_\bullet \times M_\bullet \to M_\bullet$.
We record this in a corollary for future use.

\begin{cor}\label{cor:face-and-degeneracy-are-equifibered}
    Suppose $M_\bullet$ is a simplicial commutative monoid satisfying the Segal condition. 
    If $d_1\colon M_2 \to M_1$ is equifibered, then 
    $\lambda^*\colon M_m \to M_n$ is equifibered for all active $\lambda\colon [n] \to [m]$.
\end{cor}
Our next goal is to show that pre-properads complete to $\infty$-properads.
With this goal in mind we recall Rezk's completion functor.
For this we need the following Segal spaces:
\begin{example}
    Given a space $A \in \calS$ we let $\hldef{E_{\bullet}(A)}\colon \simp^\op \to \calS$ denote the
    right Kan extension of the constant functor $\{[0]\} \to \calS$ at $A$. 
    Concretely, this is given by $E_n(A) = A^{\times [n]} = A^{\times(n+1)}$.
    It is straightforward to verify that $E_{\bullet}(A)$ satisfies the Segal condition and thus gives rise to a functor $\hldef{E_{\bullet}}\colon  \calS \too \Seg_{\simp^{\op}}(\calS)$.
    We denote by $\hldef{E_{\bullet}[-]}\colon  \simp \too \Seg_{\simp^{\op}}(\calS)$ the restriction of $E_{\bullet}$ along the functor $\simp \too \calS$ which sends $[n]$ to its underlying set $\{0,\dots,n\}$. 
\end{example}

These $E_\bullet[k]$ can be used to give an equivalent characterization of complete Segal spaces as those Segal spaces that are local with respect to all maps $E_\bullet[k] \to E_\bullet[l]$.
However, we will not need this characterization, but just the following formula:

\begin{prop}[{\cite[Section 14]{rezk}}]\label{prop: completion-as-left-adjoint}
    The inclusion $\CSeg_{\Dop}(\calS) \hookrightarrow \Seg_{\Dop}(\calS)$ admits a left adjoint which sends a Segal space 
    $X_\bullet \in \Seg_{\Dop}(\calS)$ 
    to the simplicial space $L_C(X)_\bullet$ given as follows:
    \[
    \hldef{L_C(X)_{\bullet}} \colon [n] \longmapsto  
    \colim_{[k] \in \Dop} \Map_{\Fun(\Dop, \calS)}\left(\Delta^n \times E_{\bullet}[k], X_{\bullet}\right).
    \]
\end{prop}

We can use this completion formula to show that completion preserves equifiberedness over an already complete base:
\begin{cor}\label{cor: completion-of-equifibered}
    Let $Q_{\bullet} \to P_{\bullet}$ be an equifibered morphism of Segal commutative monoids such that $P_{\bullet}$ is complete.
    Then the canonical map from the completion $L_C Q_\bullet \to L_C P_\bullet \simeq P_\bullet$ is also equifibered.
\end{cor}
\begin{proof}
    We begin by showing that for all $U_\bullet \in \Fun(\Dop,\calS)$ the map 
    \[\Map_{\Fun(\Dop,\calS)}(U_\bullet,Q_\bullet) \too \Map_{\Fun(\Dop,\calS)}(U_\bullet,P_\bullet) \quad \in \CMon\]
    is equifibered.
    Indeed, this holds vacuously when $U_\bullet \simeq \Delta^n$ for some $n$ and since equifibered maps are closed under limits in the arrow category, the general case follows as the simplices $\Delta^n$ generate $\Fun(\Dop, \calS)$ under colimits.

    Setting $U_\bullet = \Delta^n \times E_\bullet[k]$ for varying $k$ we see that
    the completion formula from \cref{prop: completion-as-left-adjoint} expresses $(L_C Q_\bullet)_n \to (L_C P_\bullet)_n$ is a sifted colimit of equifibered maps in $\Ar(\CMon)$.
    Since $P_\bullet$ is assumed to be complete, the diagram in $\Ar(\CMon)$ has constant target and as such is a diagram in $\CMon_{/P_n}^\eqf$.
    By \cref{lem:colim-of-eqf-slice} the full subcategory $\CMon_{/P_n}^\eqf \subset \CMon_{/P_n}$ is closed under sifted colimits, so $(L_C Q_\bullet)_n \to P_n$ is equifibered as well.
\end{proof}

\begin{notation}
    By abuse of notation we will also use $L_C$ to denote the left adjoint to the nerve:
    \[
        L_C \colon \Seg_{\Dop}(\calS) \adj \Cat :\! \xN_\bullet.
    \]
\end{notation}

We now show that completion for pre-properads is compatible with completion for ordinary Segal spaces.
\begin{prop}\label{prop:nerves-and-completion}
    The nerve functors for $\infty$-properads and $\infty$-categories fit into a cartesian square
    \[\begin{tikzcd}[column sep = large, row sep = large]
	\Prpd & \pPrpd \\
	\Cat & {\Seg_{\Dop}(\calS)}
	\arrow["\text{forget}"', from=1-1, to=2-1]
	\arrow["\text{forget}",from=1-2, to=2-2]
	\arrow[hook, "{\xN_{\bullet}}"', from=1-1, to=1-2]
	\arrow[hook, "{\xN_{\bullet}}"', from=2-1, to=2-2]
	\arrow[dashed, bend right = 15, "{L_C}"', from=1-2, to=1-1]
	\arrow[dashed, bend right = 15, "{L_C}"', from=2-2, to=2-1]
	\arrow["\lrcorner"{anchor=center, pos=0.125}, draw=none, from=1-1, to=2-2]
    \end{tikzcd}\]
    of $\infty$-categories.
    Moreover, the dashed localization functors commute with the vertical functors.
\end{prop}
\begin{proof}
    We begin by noting that  $\CMon(\Seg_{\Dop}(\calS)) = \Seg_{\Dop}(\CMon)$
    as full subcategories of $\Fun(\Dop \times \Fin_*, \calS)$.
    Now consider the following diagram:
    \[\begin{tikzcd}[column sep = large, row sep = large]
	\Prpd & \SM & \CMon(\Cat) & \Cat \\
	\pPrpd & \Seg_{\Dop}(\CMon) & \CMon(\Seg_{\Dop}(\calS)) & \Seg_{\Dop}(\calS) 
	\arrow["\text{inc}", from=1-1, to=1-2]
	\arrow["\text{inc}",from=2-1, to=2-2]
	\arrow["=", from=1-2, to=1-3]
	\arrow["=",from=2-2, to=2-3]
	\arrow["\text{fgt}", from=1-3, to=1-4]
	\arrow["\text{fgt}",from=2-3, to=2-4]
	\arrow[hook, "{\xN_{\bullet}}"', from=1-1, to=2-1]
	\arrow[hook, "{\xN_{\bullet}}"', from=1-2, to=2-2]
	\arrow[hook, "{\xN_{\bullet}}"', from=1-3, to=2-3]
	\arrow[hook, "{\xN_{\bullet}}"', from=1-4, to=2-4]
	\arrow[dashed, bend right = 25, "{L_C}"', from=2-1, to=1-1]
	\arrow[dashed, bend right = 25, "{L_C}"', from=2-2, to=1-2]
	\arrow[dashed, bend right = 25, "{L_C}"', from=2-3, to=1-3]
	\arrow[dashed, bend right = 25, "{L_C}"', from=2-4, to=1-4]
	\arrow["\lrcorner"{anchor=center, pos=0.125}, draw=none, from=1-1, to=2-2]
	\arrow["\lrcorner"{anchor=center, pos=0.125}, draw=none, from=1-3, to=2-4]
    \end{tikzcd}\]
    We claim that both the left and right solid square are cartesian and vertically left-adjointable, i.e.~the dashed vertical left adjoints commute with the horizontal functors.
    
    The right-most adjunction is the one described in \cref{prop: completion-as-left-adjoint}.
    Note that the left adjoint $L_C\colon \Seg_{\Dop}(\calS) \to \Cat$ commutes with products since the formula \ref{prop: completion-as-left-adjoint} only involves sifted colimits and limits.
    Therefore, it descends to a functor on commutative monoids.
    This shows that the right square is vertically left-adjointable.
    It is cartesian since a commutative monoid $M\colon \Fin_* \to \Seg_{\Dop}(\calS)$ lifts against the nerve $\xN_\bullet$ if and only if $M(1_+)$ does.
    
    Finally, the left cartesian square was established in \cref{cor: properads-as-preproperads},
    but we still need to check that it is vertically left-adjointable.
    To do so, it suffices to show that for every pre-properad $P_\bullet \in \pPrpd$ the completion $L_C P_\bullet$ is the (nerve of) an $\infty$-properad.
    Since $\xN_\bullet\Csp$ is
    terminal in $\pPrpd{}$ (\cref{thm:Csp-is-final}) we get an equifibered map $P_\bullet \to \xN_\bullet\Csp$.
    Moreover, $\xN_\bullet \Csp$ is complete (\cref{obs:Csp-complete}), and therefore the induced map
    $L_C P_\bullet \to \xN_\bullet\Csp$ is equifibered by \cref{cor: completion-of-equifibered}.
    It follows by \cref{lem:eqf-over-properad} that $L_C P_\bullet$ is an $\infty$-properad.
    Similarly, if $f\colon Q_\bullet \to P_\bullet$ is a map of pre-properads,
    then both $L_C P_\bullet \to \xN_\bullet \Csp$ and the composite $L_CQ_\bullet \to L_C P_\bullet \to \xN_\bullet \Csp$ are equifibered,
    so it follows by cancellation that $L_C(f)$ is also equifibered.
\end{proof}

Finally, we provide a criterion for checking that a pre-properad is complete.
This will be useful later on when we compare pre-properads to Segal $\infty$-properads.
It relies on the following fact:

\begin{lem}\label{lem:checking-completeness}
    Let $f\colon A_\bullet \to B_\bullet$ be a map of Segal spaces. 
    Then $A_\bullet$ is complete if and only if the Segal space $B_\bullet^{\rm eq} \times_{B_\bullet} A_\bullet$ is complete.
\end{lem}
\begin{proof}
    The inclusion $A_\bullet^{\rm eq} \hookrightarrow A_\bullet$ factors as
    $A_\bullet^{\rm eq} \hookrightarrow B_\bullet^{\rm eq} \times_{B_\bullet} A_\bullet \hookrightarrow A_\bullet$ because $f$ restricts to a map $A_\bullet^{\rm eq} \to B_\bullet^{\rm eq}$.
    Applying $(-)^{\rm eq}$ to the factorization we get
    \[
        A_\bullet^{\rm eq} \hookrightarrow 
        (B_\bullet^{\rm eq} \times_{B_\bullet} A_\bullet)^{\rm eq}
        \hookrightarrow (A_\bullet)^{\rm eq}.
    \]
    Since the composite is an equivalence we deduce that the monomorphisms are in fact equivalences.
    Thus, $A_\bullet^{\rm eq}$ is a constant simplicial space if and only if 
    $ (B_\bullet^{\rm eq} \times_{B_\bullet} A_\bullet)^{\rm eq} $ is.
\end{proof}

\begin{lem}\label{lem:completeness-of-pprd-on-11ary}
    Let $P_\bullet$ be a pre-properad and $p\colon P_\bullet \to \xN_\bullet \Csp$ a morphism of pre-properads.
    (We will prove in \cref{thm:Csp-is-final} that the space of such $p$ is contractible.)
    Define the simplicial space $P^{(1,1)}_\bullet$ as the pullback
    \[
    \begin{tikzcd}
        P^{(1,1)}_\bullet \ar[r, hook] \ar[d] &
        P_\bullet \ar[d] \\
        * \ar[r, hook] &
        \xN_\bullet \Csp
    \end{tikzcd}
    \]
    where the bottom horizontal map is the nerve of the functor $* \to \Csp$ that picks the singleton.
    Note that this is a level-wise monomorphism.
    Then $P_\bullet$ is complete if and only if $P^{(1,1)}_\bullet$ is.
\end{lem}
\begin{proof}
    We may factor the square defining $P^{(1,1)}_\bullet$ as a composite of two cartesian squares:
    \[
    \begin{tikzcd}
        P^{(1,1)}_\bullet \ar[r, hook] \ar[d] &
        Q_\bullet \ar[r, hook] \ar[d] &
        P_\bullet \ar[d] \\
        * \ar[r, hook] &
        \xN_\bullet (\Csp^\simeq) \ar[r, hook] &
        \xN_\bullet \Csp
    \end{tikzcd}
    \]
    By \cref{lem:checking-completeness} $P_\bullet$ is complete if and only if the pullback $Q_\bullet$ is.
    The right square is a pullback square in $\CMon$ and thus $Q_\bullet \to \xN_\bullet(\Csp^\simeq)$ is equifibered.
    Since $\xN_\bullet(\Csp^\simeq)$ is the constant simplicial object on $\xF(*)$,
    it follows that $Q_\bullet \simeq \xF(P_\bullet^{(1,1)})$.
    The maximal subgroupoid is $Q_\bullet^{\rm eq} \simeq \xF((P^{(1,1)}_\bullet)^{\rm eq})$ and 
    therefore $P_\bullet^{\rm eq}$ is a constant simplicial space if and only if $(P^{(1,1)}_\bullet)^{\rm eq}$ is.
\end{proof}

\subsubsection{Obstruction theory for pre-properads}
In the remainder of this section we will use obstruction theory
to prove the following theorem:
\begin{thm}\label{thm:Csp-is-final}
    The nerve of the cospan category $\xN_\bullet(\Csp)$ is a terminal object in $\pPrpd$.
\end{thm}
Using the fully faithfulness of $\xN_\bullet \colon \Prpd \hookrightarrow \pPrpd$ 
from \cref{cor: properads-as-preproperads} this in particular implies the following:
\begin{cor}\label{cor:Csp-is-final}
    The cospan category $\Csp$ is a terminal object in $\Prpd$.
\end{cor}

In order to show that $\xN_\bullet(\Csp) \in \pPrpd$ is terminal we will develop
a general theory of how to construct maps of pre-properads
inductively over the simplicial level.
We begin by recalling the definition of latching and matching objects for
simplicial objects.
For this, we let $\simp^{\op}_{<n} \subset \simp^\op_{\le n} \subset \Dop$ denote the full subcategories on all objects $[k]$ with $k<n$ or $k\le n$, respectively.
\begin{defn}
    For a simplicial commutative monoid $Q_\bullet \colon \Dop \to \CMon$
    the \hldef{$n$th latching map} $L_n Q \to Q_n$ and the 
    \hldef{$n$th matching map} $Q_n \to M_n Q$ are defined respectively as follows
    \[
        \left(L_n Q \too Q_n\right) := 
        \colim \left( (\simp^{\op}_{<n})_{/[n]} \xrightarrow{Q_\bullet} \CMon_{/Q_n} \right), \quad \left( Q_n \too M_n Q\right) :=
        \lim \left( (\simp^{\op}_{<n})_{[n]/} \xrightarrow{Q_\bullet} \CMon_{Q_n/} \right).
    \]
\end{defn}

In the standard definition, the colimit in the latching object only runs over surjections and the limit in the matching object only over injections. 
However, this is equivalent to the above definition, as can be seen using a finality argument as in the proof of \cref{lem:L-and-M-eqf} and the fact that surjections and injections form a factorization system on $\Dop$.

\begin{lem}\label{lem:L-and-M-eqf}
    Let $f \colon Q_\bullet \to P_\bullet$ be a morphism of pre-properads.
    Then for every $n$ the commutative square
    \[\begin{tikzcd}
	{L_n Q} & {L_nP} \\
	{Q_n} & {P_n}
	\arrow[from=1-1, to=2-1]
	\arrow[from=1-2, to=2-2]
	\arrow[from=1-1, to=1-2]
	\arrow[from=2-1, to=2-2]
    \end{tikzcd}\]
    consists of equifibered maps.
\end{lem}
\begin{proof}
    The inclusion
    $(\simp^{\op,\act}_{<n})_{/[n]} \hookrightarrow (\simp^{\op}_{<n})_{/[n]}$ 
    is final (as it is a right adjoint, e.g.~\cite[Observation 2.3.6]{envelopes}) and so $L_n Q$ is the colimit of the restricted diagram
    $(\simp^{\op,\act}_{<n})_{/[n]} \to \CMon_{/Q_n}$.
    Then, by \cref{cor:face-and-degeneracy-are-equifibered}, 
    the colimit diagram factors through the full subcategory $\CMon_{/Q_n}^\eqf \subset \CMon_{/Q_n}$
    and this subcategory is closed under all colimits by \cref{lem:colim-of-eqf-slice}.
    Therefore, the latching maps $L_n Q \to Q_n$ and $L_n P \to P_n$ are equifibered.
    The bottom horizontal map is equifibered since $f$ is a morphism of pre-properads.
    Finally, the top map is also equifibered by cancellation.
\end{proof}

\begin{obs}
    As a consequence of \cref{lem:L-and-M-eqf}, if $Q$ is a pre-properad, the latching object $L_n Q$ is free for all $n$. 
    In fact, using the notation of \cref{obs:elementary-subspace}, we have by  \cref{cor:face-and-degeneracy-are-equifibered} an equivalence
    \[(L_nQ)^\el \simeq \colim_{[k] \to [n] \in (\simp^{\op,\act}_{< n})_{/[n]}} Q_k^\el .\]
    We shall henceforth write $\hldef{L_n^\el Q} := (L_n Q)^\el$.
\end{obs}

We recall a basic fact about monomorphisms that we need in the proof of \cref{cor:obstruction-for-mono}.

\begin{rem}\label{rem:mono-reminder}
    Recall that if $j \colon A \hookrightarrow B$ is a monomorphism in an $\infty$-category $\calC$ with finite limits, the functor induced by composition $(-) \circ j\colon \calC_{/A} \to \calC_{/B}$ preserves products. 
    To see this, consider
    \[\begin{tikzcd}
	{X \times_A Y} & Y & Y \\
	X & A & A \\
	X & A & B
	\arrow[from=1-3, to=2-3] \arrow[from=2-3, to=3-3] \arrow[from=3-2, to=3-3] \arrow[from=3-1, to=3-2] \arrow[equal, from=2-2, to=3-2] \arrow[equal, from=2-2, to=2-3] \arrow[equal, from=2-1, to=3-1] \arrow[from=1-2, to=2-2] \arrow[from=1-1, to=2-1] \arrow[from=1-1, to=1-2] \arrow[equal, from=1-2, to=1-3] \arrow[from=2-1, to=2-2]
    \ar[from=1-1, to=2-2, phantom, very near start, "\lrcorner"]
    \ar[from=2-1, to=3-2, phantom, very near start, "\lrcorner"]
    \ar[from=1-2, to=2-3, phantom, very near start, "\lrcorner"]
    \ar[from=2-2, to=3-3, phantom, very near start, "\lrcorner"]
    \end{tikzcd}\]
    where the bottom right square is cartesian because $j$ is a monomorphism and the top right and bottom left square are trivially cartesian.
    The pullback pasting lemma implies that $X \times_A Y \to X \times_B Y$ is an equivalence, which was the claim.
\end{rem}

\begin{prop}\label{prop:obstruction-for-preproperads}
    For any two pre-properads%
    \footnote{
        In fact, the proof does not use the Segal condition for $ Q_\bullet$
        or $ P_\bullet$.
    }
    $ Q_\bullet,  P_\bullet \in \pPrpd$ and
    $n \ge 1$ there is a cartesian square:
    \[
    \begin{tikzcd}
        {\Map_{\Fun(\simp^\op_{\le n}, \CMon)}^\eqf( Q_{|\le n}, P_{|\le n})} \ar[d] \ar[r] & 
        {\Map_\calS(Q_n^\el, P_n^\el)} \ar[d] \\
        {\Map_{\Fun(\simp^\op_{\le n-1}, \CMon)}^\eqf(Q_{|\le n-1}, P_{|\le n-1})} \ar[r] &
        {\Map_\calS(L_nQ^\el, P_n^\el)} 
        \underset{\Map_\calS(L_nQ^\el, M_nP)}{\times}
        {\Map_\calS(Q_n^\el, M_nP)} 
    \end{tikzcd}
    \]
\end{prop}
\begin{proof}
    If we drop the equifiberedness condition, then \cite[Remark A.2.9.16]{HTT}
    gives us a cartesian square:
    \[
    \begin{tikzcd}
        {\Map_{\Fun(\simp^\op_{\le n}, \CMon)}(Q_{|\le n}, P_{|\le n})} \ar[d] \ar[r] & 
        {\Map_\CMon(Q_n, P_n)} \ar[d] \\
        {\Map_{\Fun(\simp^\op_{\le n-1}, \CMon)}(Q_{|\le n-1}, P_{|\le n-1})} \ar[r] &
        {\Map_\CMon(L_nQ, P_n)} 
        \underset{\Map_\CMon(L_nQ, M_nP)}{\times}
        {\Map_\CMon(Q_n, M_nP)} 
    \end{tikzcd}
    \]
    By \cref{lem:L-and-M-eqf} the map $L_nQ \to Q_n$ is equifibered,
    and hence restricts to a well-defined map $L_n^\el Q \to Q_n^\el$.
    We can therefore use the free-forgetful adjunction to rewrite
    the right vertical map as
    \[
        {\Map_\calS(Q_n^\el, P_n)} \too
        {\Map_\calS(L_n^\el Q, P_n)} 
        \times_{\Map_\calS(L_n^\el Q, M_nP)} 
        {\Map_\calS(Q_n^\el, M_nP)} .
    \]
    Now suppose that the original map $f\colon Q_{|\le n-1} \to P_{|\le n-1}$ we started with was equifibered. 
    Then its extension to $f'\colon Q_{|\le n} \to P_{|\le n}$ 
    is equifibered if and only if the lift $Q_n \to P_n$ is equifibered.
    So to obtain the space of equifibered extensions of $f$
    we need to restrict to the subspace 
    $\Map_\calS(Q_n^\el, P_n^\el) \subset \Map_\calS(Q_n^\el, P_n)$.
    The map $L_nQ \to P_n$ is also equifibered
    by \cref{lem:L-and-M-eqf}, hence
    we can restrict to the subspace 
    $\Map_\calS(L_n^\el Q, P_n^\el) \subset \Map_\calS(L_n^\el Q, P_n)$,
    which yields the desired square.
\end{proof}

The obstruction theory of \cref{prop:obstruction-for-preproperads}
becomes particularly easy when the matching map restricted
to elementaries is a monomorphism:
\begin{cor}\label{cor:obstruction-for-mono}
    In the situation \cref{prop:obstruction-for-preproperads}, 
    suppose further that the composite 
    $P_n^\el \subset P_n \to M_nP$ is a monomorphism in $\calS$.
    Then there is a cartesian square:
    \[
    \begin{tikzcd}
        {\Map_{\Fun(\simp^\op_{\le n}, \CMon)}^\eqf(Q_{|\le n}, P_{|\le n})} \ar[d, hook] \ar[r] \ar[dr, phantom, very near start, "\lrcorner"] & 
        {\Map_\calS(Q_n^\el, P_n^\el)} \ar[d, hook] \\
        {\Map_{\Fun(\simp^\op_{\le n-1}, \CMon)}^\eqf(Q_{|\le n-1}, P_{|\le n-1})} \ar[r] &
        {\Map_\calS(Q_n^\el, M_nP).} 
    \end{tikzcd}
    \]
    where the vertical maps are monomorphisms. 
    Here the bottom map sends $f\colon  Q_{|\le n-1} \to P_{|\le n-1}$
    to the composite of $Q_n^{\el} \subset Q_n \too M_nQ$
    with $M_n(f)\colon  M_nQ \to M_nP$.
\end{cor}
\begin{proof}
    The square in question can be obtained from the square 
    in \cref{prop:obstruction-for-preproperads}
    by composing the right vertical and bottom horizontal maps with the projection:
    \[
        p: {\Map_\calS(L_nQ^\el, P_n^\el)} 
        \underset{\Map_\calS(L_nQ^\el, M_nP)}{\times}
        {\Map_\calS(Q_n^\el, M_nP)} \too
        {\Map_\calS(Q_n^\el, M_nP)} 
    \]
    By assumption we have that
    ${\Map_\calS(L_nQ^\el, P_n^\el)} \hookrightarrow  {\Map_\calS(L_nQ^\el, M_nP)}$ is a monomorphism and thus so is $p$.
    It follows from \cref{rem:mono-reminder} that the square remains cartesian after post-composing with $p$.
    Finally, note that in the resulting cartesian square the right vertical map
    is a monomorphism since it is given by post-composing with 
    $P_n^\el \into M_nP$. Since the square is cartesian it also
    follows that the left vertical map is a monomorphism.
\end{proof}

\begin{notation}
    For $Q_\bullet \in \Fun(\Dop,\CMon)$ let $Q(-)\colon  \Fun(\Dop, \calS)^\op \to \CMon$ denote the 
    right Kan extension of $\calQ_\bullet$ along the opposite Yoneda embedding.
\end{notation}

\begin{rem}\label{rem:matching}
    The definition of $Q(-)$ recovers the space of $n$-simplices 
    as $Q_n = Q(\Delta^n)$ and the $n$th matching object as
    $M_nQ = Q(\partial \Delta^n)$.
    Here we write $\Delta^n$ for the simplicial set
    $\Map_{\simp}(-, [n])\colon  \Dop \to \Sets$
    and use $\Lambda_k^n \subset \partial \Delta^n \subset \Delta^n$ to denote the $k$th horn and the boundary.
    By construction, the functor $Q(-)$ sends colimits of simplicial spaces to limits.
    In particular, by writing 
    $\partial \Delta^n = \Lambda_k^n \amalg_{\partial \Delta^{n-1}} \Delta^{n-1}$ 
    we get an equivalence:
    \[
        M_nQ = Q(\partial \Delta^n) 
        \xrightarrow{\ \simeq\ }
        Q(\Lambda_k^n) \times_{Q(\partial \Delta^{n-1})} Q(\Delta^{n-1})
        = Q(\Lambda_k^n) \times_{M_{n-1}Q} Q_{n-1}.
    \]
    If $Q_\bullet\colon \Dop \to \CMon$ satisfies the Segal condition
    then for all inner horns $\Lambda_k^n$ with $0<k<n$ the restriction
    $ Q(\Delta^n) \too Q(\Lambda_k^n) $
    is an equivalence.%
    \footnote{
        To see this, note that because it is Segal $Q$ inverts the spine inclusion 
        $\Delta^1 \cup_{\Delta^0} \dots \cup_{\Delta^0} \Delta^1 \hookrightarrow \Delta^n$,
        so it suffices to show that it inverts
        $\Delta^1 \cup_{\Delta^0} \dots \cup_{\Delta^0} \Delta^1 \hookrightarrow \Lambda^n_k$.
        For $n=1$ (and thus $k=1$) this is the identity, and for $n>1$ it can be written as an iterated pushout along inner horn inclusions of lower dimension, so the claim follows by induction.
    }
\end{rem}

\begin{prop}\label{prop:map-into-sub-terminal}
    Let $P$ and $Q$ be pre-properads such that
    $(d_0, d_1)\colon  P^{\el}_1 \too P_0 \times P_0$ is an equivalence. 
    Then restriction to the $0$-skeleton defines an equivalence:
        \[
	        \Map_{\pPrpd}(Q_{\bullet},P_{\bullet})
	        \xrightarrow{\simeq}
            \Map_{\CMon}^\eqf(Q_0, P_0) 
            \simeq
            \Map_{\calS}(Q_0^\el, P_0^\el). 
        \]
\end{prop}
\begin{proof}
    We will prove inductively for all $n \ge 1$ that:
    \begin{enumerate}[(i)]
        \item the map $\delta_n\colon  P_n^\el \to M_nP = P(\partial \Delta^n)$ is a monomorphism, and 
        \item for every equifibered 
            $f\colon Q_{|\le n-1} \to P_{|\le n-1}$ 
        the following diagram admits a (necessarily unique) dashed lift:
            \[\begin{tikzcd}
                Q(\Delta^n)^\el \ar[r, dashed] \ar[d] & 
                P(\Delta^n)^\el \ar[d, hook, "\delta_n"]\\
                Q(\partial\Delta^n) \ar[r, "f_{|\partial \Delta^n}"] &
                P(\partial \Delta^n)
.            \end{tikzcd}\]
    \end{enumerate}
    Before we begin the induction, let us argue why this implies the proposition.
    By (i) we may use \cref{cor:obstruction-for-mono} to obtain a cartesian square: 
    \[
    \begin{tikzcd}
        {\Map_{\Fun(\simp^\op_{\le n}, \CMon)}^\eqf(Q_{|\le n}, P_{|\le n})} \ar[d, hook] \ar[r] & 
        {\Map_\calS(Q_n^\el, P_n^\el)} \ar[d, hook] \\
        {\Map_{\Fun(\simp^\op_{\le n-1}, \CMon)}^\eqf(Q_{|\le n-1}, P_{|\le n-1})} \ar[r] &
        {\Map_\calS(Q_n^\el, M_nP)} 
    \end{tikzcd}
    \]
    and (ii) guarantees that the left vertical map is not only a monomorphism, but also an equivalence.
    This implies the proposition since $\Map_{\pPrpd}(Q_{\bullet},P_{\bullet}) \simeq \lim_n \Map^{\eqf}_{\Fun(\Delta^{\op \le n},\CMon)}(Q_{|\le n},P_{|\le n})$.
    
    We begin the induction by noting that (i) and (ii) hold for $n=1$ as $\delta_1\colon P_1^\el \to M_1 P = P_0 \times P_0$ was assumed to be an equivalence.
    
    For the inductive step we argue using the diagram below.
    The map $\varphi$ exists because the inner face map $d_1\colon P_n \to P_{n-1}$ is equifibered and hence restricts to elementaries.
    The bottom right square is cartesian by \cref{rem:matching} and $X$ is defined to make the top right square cartesian:
\[\begin{tikzcd}
	{Q(\Delta^n)^\el} & {P(\Delta^n)^\el} \\
	&& X && {P(\Delta^{\{0,2,\dots,n\}})^\el} \\
	{Q(\partial\Delta^n)} && {P(\partial\Delta^n)} && {P(\Delta^{\{0,2,\dots,n\}})} \\
	{Q(\Lambda^n_1)} && {P(\Lambda^n_1)} && {P(\partial\Delta^{\{0,2,\dots,n\}})}
	\arrow["\beta", hook, from=2-3, to=3-3]
	\arrow[from=3-3, to=3-5]
	\arrow[hook, from=2-5, to=3-5]
	\arrow[from=3-3, to=4-3]
	\arrow[from=4-3, to=4-5]
	\arrow["\lrcorner"{anchor=center, pos=0.05}, draw=none, from=3-3, to=4-5]
	\arrow["\lrcorner"{anchor=center, pos=0.05}, draw=none, from=2-3, to=3-5]
	\arrow["\alpha"{description}, hook, from=1-2, to=2-3]
	\arrow["\varphi", from=1-2, to=2-5]
	\arrow[dashed, from=1-1, to=1-2]
	\arrow[from=1-1, to=3-1]
	\arrow["{f_{|\partial \Delta^n}}", from=3-1, to=3-3]
	\arrow[from=3-1, to=4-1]
	\arrow["{f_{|\Lambda_1^n}}", from=4-1, to=4-3]
	\arrow["\gamma"{pos=0.3}, shift left=3, curve={height=-18pt}, hook, from=2-3, to=4-3]
	\arrow["{\delta_{n-1}}", shift left=3, curve={height=-18pt}, hook, from=2-5, to=4-5]
	\arrow["{\delta_n}"'{pos=0.6}, shift right=1, hook', from=1-2, to=3-3]
	\arrow[from=2-3, to=2-5]
	\arrow[from=3-5, to=4-5]
	\arrow[curve={height=12pt}, dotted, from=1-1, to=2-3]
\end{tikzcd}\]
    To prove (i) we first note that $\delta_{n-1}$ is a monomorphism by hypothesis (even when $n=2$) and hence its pullback $\gamma$ is also a monomorphism.
    Similarly, $\beta$ is a monomorphism because $P_{n-1}^\el \subset P_{n-1}$ is.
    The composite $\gamma \circ \alpha$ is a monomorphism since it can be factored as $P_n^\el \subset P_n \to P(\Lambda_1^n)$ where the second map is an equivalence because $P_\bullet$ is a Segal space and $n \ge 2$.
    By cancellation, we conclude that $\alpha$ is a monomorphism and thus so is $\delta_n = \beta \circ \alpha$, proving (i).
    
    For (ii) we need to show that for any equifibered 
       $f\colon Q_{|\le n-1} \to P_{|\le n-1}$ 
    the dashed lift in the diagram exists, making the square with $\delta_n$ commute.
    The map $Q_n^\el \to P(\Delta^{\{0,2,\dots,n\}})$ can be factored as 
    \[
        Q_n^\el \subset Q_n \xto{d_1} Q_{n-1} \xto{f_{n-1}} P_{n-1}
    \]
    where the latter two maps are equifibered, and so it lands in $P_{n-1}^\el$.
    This provides us with the dotted lift in the diagram.
    To lift the dotted map against $\alpha$ it suffices to do so after composing with $\gamma$, since $\gamma$ is a monomorphism.
    It remains to observe that the map $Q_n^\el \simeq Q(\Lambda_1^n)^\el \subset Q(\Lambda_1^n) \to P(\Lambda_1^n)$ factors through $P(\Lambda_1^n)^\el \simeq P_n^\el$ since $f_{|\Lambda_1^n}\colon Q(\Lambda_1^n) \to P(\Lambda_1^n)$ is equifibered.
\end{proof}

We are now ready to show that $\xN_\bullet\Csp$ is final in $\pPrpd$.
\begin{proof}[Proof of \cref{thm:Csp-is-final}] 
    The commutative monoid $\xN_1\Csp = \Fun(\Tw[1], \Fin)^\simeq$
    is free on cospans of the form $A \to \ast \leftarrow B$. 
    In particular the composite
    \[
        (\xN_1\Csp)^\el \hookrightarrow 
        \xN_1\Csp = \Fun(\Tw[1], \Fin)^\simeq 
        \xrightarrow{(d_0,d_1)} \Fin^\simeq \times \Fin^\simeq \simeq \xN_0\Csp \times \xN_0\Csp
    \]
    is an equivalence. 
    Hence, we may apply \cref{prop:map-into-sub-terminal} to
    conclude that for any pre-properad $Q_\bullet$ restriction yields an equivalence
    \[
	    \Map_{\pPrpd}(Q_{\bullet},\xN_\bullet\Csp)
	    \xrightarrow{\simeq}
        \Map_{\calS}(Q_{0}^\el, \xN_0^\el\Csp) \simeq *
    \]
    because $(\xN_0\Csp)^\el = *$.
    This shows that $\xN_\bullet\Csp$ is a terminal object in $\pPrpd$, as promised.
\end{proof}
\section{\texorpdfstring{$\infty$}{Infinity}-Properads as \texorpdfstring{$\L$}{L}-Segal spaces}\label{sec:segal}
In this section we compare our notion of $\infty$-properads to the (complete) Segal $\infty$-properads of Hackney--Robertson--Yau.
The main result of this section is \cref{thm:segal-envelope}, where we construct an envelope functor
\[
    \Env_{\L} \colon \Seg_{\L}(\calS) \too \Prpd \subseteq \SM    
\]
for a certain algebraic pattern $\L$ of ``level graphs'' introduced by Chu--Hackney \cite{CH22} (though they consider opposite category $\bfL_{\rm CH} = \L^\op$), and show that its right adjoint defines a fully faithful embedding $\Prpd \hookrightarrow \Seg_{\L}(\calS)$ whose essential image is characterized by a completeness condition.

This section generalizes work of Haugseng--Kock \cite[\S 3, \S 4]{HK21}, who prove the result in the case of monic $\infty$-properads (i.e.~$\infty$-operads).
While working with $\infty$-properads does add several complications, we owe many ideas to them.

More precisely, we will construct the envelope functor as the composition
\[
        \Env_{\L}\colon 
        \Seg_{\L}(\calS) \xhookrightarrow{\ \varphi^*\ }
        \Seg_{\bmC}(\calS) 
        \xrightarrow{\ q_!\ } 
        \Seg_{\Dop \times \Fin_*}(\calS) \simeq \Seg_{\Dop}(\CMon) \xrightarrow{L_C} \SM
\]
where $q \colon \bmC \to \Dop \times \Fin_\ast$ is the left fibration classifying $\xN_\bullet \Csp \in \Seg_{\Dop}(\CMon) \subseteq \Fun(\Dop \times \Fin_\ast,\calS)$ and where $\varphi \colon \bmC \to \L$ identifies $\L$ with the localization of $\bmC$ by $q^{-1}\left(\simp^{\op,\simeq} \times \Fin_*^\act\right)$.
This definition makes $\Env$ quite computable, and concretely we will show in \cref{cor:act-Env-formula} that the spaces of objects and morphisms in $\Env(X)$ can be computed as colimits over certain groupoids
\begin{align*}
    \xN_0\Env(X) &\simeq \colim_{A\in \Fin^\simeq} X([0],A_+) &
    &\text{and} &
    \xN_1\Env(X) &\simeq \colim_{A\in \Fun(\Tw[1], \Fin)^\simeq} X([1],A_+).
\end{align*}
We begin with a quick review on algebraic patterns as developed in \cite{CH19}.

\begin{defn}[{\cite[Definition 2.1]{CH19}}]
	An \hldef{algebraic pattern} is an $\infty$-category $\calO$ equipped with the following structure
	\begin{enumerate}
		\item Subcategories $\hldef{\calO^\xint}, \hldef{\calO^\act} \subset \calO$ of ``inert'' and ``active'' morphisms, which form a factorization system $(\hldef{\calO^{\xint}},\hldef{\calO^{\act}})$ on $\calO$. 
		\item A full subcategory $\hldef{\calO^{\el}} \subseteq \calO^{\xint}$ of \hldef{elementary objects}.
	\end{enumerate}
	A morphism $f\colon \calO \to \calP$ of algebraic patterns is a functor preserving all of the above, i.e.~it sends inert (respectively active) morphisms to inert (respectively active) morphisms and elementary objects to elementary objects.
\end{defn}

\begin{defn}[{\cite[Definition 2.7]{CH19}}]
	Let $\calO$ be an algebraic pattern and $\calC$ an $\infty$-category.
	An \hldef{$\calO$-Segal object in $\calC$} is a functor $F\colon  \calO \to \calC$ satisfying the \hldef{Segal condition}: for every $x \in \calO$ the comparison map $F(x) \to \lim_{e \in \calO^{\el}_{x/}} F(e)$ is an equivalence.%
    \footnote{
        If $\calC$ is not assumed to have limits, the Segal condition says that the diagram 
        $(\calO^\el_{x/})^\lhd \to \calO \to \calC$ is a limit diagram.
    }
	Here $\calO^{\el}_{x/} \subset (\calO^\xint)_{x/}$ denotes the full subcategory on the elementary objects under $x$.
	We denote by $\hldef{\Seg_{\calO}(\calC)} \subseteq \Fun(\calO,\calC)$ the full subcategory of $\calO$-Segal objects.
\end{defn}

\begin{example}
    The category $\Dop$ has a factorization system where the inert maps are the $([n] \leftarrow [m]:\!\lambda)$ such that $\lambda(i) - \lambda(j) = i-j$ for all $i$ and $j$,
    and the active maps are those $\lambda$ satisfying $\lambda(0)=0$ and $\lambda(m) = n$.
    Picking $[0]$ and $[1]$ as the elementary objects we get an algebraic pattern which we denote by $\Dop$.
    Note that $\Dop$-Segal objects in $\calS$ are precisely the Segal spaces in the sense of
    Rezk, see \cref{subsec:eqf-sm-functors}.
\end{example}

\begin{example}
    The category $\Fin_\ast$ has a factorization system where we declare $f\colon  A_+ \to B_+$ to be inert if its restriction to $A \setminus f^{-1}(\ast) \to B$ is bijective and active if the preimage of the base point contains only the base point.
    Recall that for $n \in \bbN$ we denote
    \(
        n_+ \coloneq \{1,\dots,n\}_+ = \{1, \dots, n, \infty\}.
    \)
    Picking $1_+$ as the only elementary object gives an algebraic pattern that we denote by $\Fin_\ast$.
    By definition, we have $\Seg_{\Fin_\ast}(\calS) = \CMon \subseteq \Fun(\Fin_\ast,\calS)$.
\end{example}

\subsection{A pattern for equifibered symmetric monoidal \texorpdfstring{$\infty$-categories}{infinity categories} over \texorpdfstring{$\Csp$}{cospan}}

In this subsection we construct an algebraic pattern $\bmC$ such that $\bmC$-Segal spaces are (up to completion) symmetric monoidal $\infty$-categories equipped with a symmetric monoidal functor to $\Csp$.
We then give criteria for what a localization $\bmC \to \L$ needs to satisfy such that complete $\L$-Segal spaces are symmetric monoidal $\infty$-categories \emph{equifibered} over $\Csp$.

\subsubsection{A pattern for Segal objects over a fixed base}
For each algebraic pattern $\calP$ and $\calP$-Segal space $X\colon \calP \to \calS$ there is a pattern structure on the unstraightening $\calP_X := \Un_\calP(X)$,
and Haugseng--Kock \cite{HK21} show that $\calP_X$-Segal spaces are $\calP$-Segal spaces equipped with a map to $X$.
We recall this construction here, as we shall need it later.

\begin{const}\label{constr:P_X-pattern}
    Let $\calP$ be an algebraic pattern and let $X\colon  \calP \to \calS$ be a functor with unstraightening $\pi\colon  \calP_X \to \calP$. 
    We consider $\calP_X$ as an algebraic pattern where a morphism is active or inert if and only if its image in $\calP$ is active or inert, respectively,
    and where an object is elementary if and only if its image in $\calP$ is elementary.%
    \footnote{This pattern structure can also be characterized as the maximal structure for which $\pi\colon  \calP_X \to \calP$ is a morphism of patterns.}
\end{const}

\begin{rem}\label{rem: functors-on-unstraightening}
    Let $X\colon  \calB \to \calS$ be a functor and let $\pi\colon  \calB_X \to \calB$ denote its unstraightening.
    The objects of $\calB_X$ are pairs $(b,x)$ where $b \in \calB$ and $x \in X(b)$.
    By \cite[Proposition 6.5.7]{cisinski} we have $\pi_!(\ast) \simeq X$ and hence left Kan extension along $\pi$ defines a functor
    $\pi_!\colon \Fun(\calB_X, \calS) \to \Fun(\calB, \calS)_{/X}$.
    Moreover, if $F\colon  \calB_X \to \calS$ is any functor and $b \in \calB$, the fiber of the natural map $(\pi_! F)(b) \to X(b)$ over a point $x \in X(b)$ is canonically equivalent to $F(b,x)$. 
\end{rem}

\begin{lem}[{\cite[Corollary 9.8]{GHN17} and \cite[Proposition 3.2.5]{HK21}}]\label{lem:P_X-Segal-spaces}
    Let $\calP$ be an algebraic pattern, $X\colon  \calP \to \calS$ a Segal space, and $\pi\colon  \calP_X \to \calP$ its unstraightening.
    Then left Kan extension along $\pi$ defines a commutative square:
     \[
     \begin{tikzcd}
        \Seg_{\calP_X}(\calS) \ar[d, hook] \ar[r, "\pi_!", "\simeq"'] &
        \Seg_{\calP}(\calS)_{/X} \ar[d, hook]\\
        \Fun(\calB_X,\calS) \ar[r, "\pi_!", "\simeq"'] &
        \Fun(\calB,\calS)_{/X}
     \end{tikzcd}
     \]
     where the horizontal functors are equivalences.
\end{lem}

\subsubsection{Symmetric monoidal $\infty$-categories over $\Csp$}

In \cref{const:cospans-2-cat} we recalled the construction of the symmetric monoidal double-category of cospans,
which we can think of as a functor
\[
    \mfr{C} \in \Seg_{\Dop}(\SM) \simeq \Seg_{\Dop \times \Fin_*}(\Cat) \subset \Fun(\Dop \times \Fin_*, \Cat).
\]
We will denote the unstraightening of this functor by 
\[
    \hldef{\widehat{q}}\colon \hldef{\widehat{\bmC}} \too \Fin_* \times \Dop.
\]
Composing with the projection to $\Dop$, we obtain a cocartesian fibration $\widehat{\bmC} \to \Dop$ whose fibers $\widehat{\bmC}_n$ are
the cocartesian symmetric monoidal categories of pushout preserving functors $\Tw[n] \to \Fin$:
\[
    \widehat{\bmC}_n = \Fun^{\rm po}(\Tw[n], \Fin)^\sqcup.
\]
Here, for any $\infty$-category $\calC$, $\calC^\sqcup \to \Fin_*$ denotes Lurie's cocartesian $\infty$-operad \cite[\S 4.3.2]{HA}.
If $\calC$ has finite coproduct this is a cocartesian fibration over $\Fin_*$ that encodes the cocartesian monoidal structure on $\calC$.

Let now $\hldef{q}\colon \hldef{\bmC} \to \Dop \times \Fin_*$ be the maximal left fibration in $\widehat{q}$.
By \cref{defn:Csp}, $q$ is the unstraightening of the Segal object $\xN_\bullet \Csp \in \Seg_{\Dop \times \Fin_*}(\calS)$.
We will give a more combinatorial description of $\bmC$ below.
We may therefore use \cite[Corollary 3.3.4]{HK21} (as recalled in \cref{lem:P_X-Segal-spaces}) to conclude that $\bmC$-Segal spaces are (up to completion) symmetric monoidal $\infty$-categories equipped with a functor to $\Csp$.

\begin{cor}\label{cor:Segal_bmC}
    Left Kan extension along $q\colon \bmC \to \Dop \times \Fin_*$ induces an equivalence of $\infty$-categories
    \[
    \Seg_{\bmC}(\calS) 
    \simeq \Seg_{\Dop \times \Fin_*}(\calS)_{/\St_{\Dop \times \Fin_*}(\bmC)}
    \simeq \Seg_{\Dop}(\CMon)_{/\xN_\bullet(\Csp)}.
    \]
\end{cor}

To describe $\bmC$ concretely, we need the following auxiliary lemma about cocartesian structures on functor categories.\footnote{
    Using \cite[Proposition 5.3.2, 5.3.6, \& 5.3.11]{envelopes}, one can interpret this lemma as saying that the functor $(-)^\sqcup \colon \Cat \to \Op$
    that sends an $\infty$-category to its cocartesian $\infty$-operad preserves cotensoring with $\infty$-categories, but we will not use this.
}

\begin{lem}\label{lem:cocart-operad-Fun}
    For any two $\infty$-categories $\calC$ and $J$ there is a canonical cartesian square of $\infty$-categories:
    \[\begin{tikzcd}
        \Fun(J, \calC)^\sqcup \ar[r] \ar[d] \ar[dr, phantom, very near start, "\lrcorner"] &  
        \Fun(J, \calC^\sqcup) \ar[d] \\
        \Fin_* \ar[r, "\Delta"] &
        \Fun(J, \Fin_*).
    \end{tikzcd}\]
\end{lem}
\begin{proof}
    For this proof we work in the model of quasicategories as in \cite{HA}.
    Freely using the definitions from \cite[\S2.4.3]{HA}, we have for every $K \in \mrm{sSet}_{/N(\Fin_*)}$ bijections:
    \begin{align*}
        \mrm{Hom}_{/N(\Fin_*)}(K, \Fun(J, \calC)^\sqcup) 
        & \cong \mrm{Hom}(K \times_{N(\Fin_*)} N(\Gamma^*), \Fun(J, \calC)) \\
        & \cong \mrm{Hom}((K \times J) \times_{N(\Fin_*)} N(\Gamma^*), \calC) \\
        & \cong \mrm{Hom}_{/N(\Fin_*)}(K \times J, \calC^\sqcup) \\
        & \cong \mrm{Hom}_{/N(\Fin_*)^J}(K, (\calC^\sqcup)^J) \\
        & \cong \mrm{Hom}_{/N(\Fin_*)}(K, \Delta^*(\calC^\sqcup)^J) 
    \end{align*}
    Here $\Delta^*$ denotes the restriction along the diagonal functor $\Delta\colon N(\Fin_*) \to N(\Fin_*)^J$.
    Therefore, $\Fun(J, \calC)^\sqcup$ is isomorphic, as a quasicategory, to the pullback $N(\Fin_*) \times_{N(\Fin_*)^J} (\calC^\sqcup)^J$.
\end{proof}

\begin{cor}\label{cor:cocart-Fun-Fin}
    For any $\infty$-category $J$, Lurie's cocartesian $\infty$-operad $\Fun(J, \Fin)^\sqcup$ is equivalent to the full subcategory
    \[
        \Fun^\act(J^\rhd, \Fin_*) \subset \Fun(J^\rhd, \Fin_*)
    \]
    on those functors that send all morphisms to active morphisms.
    In particular, evaluation at the cone point
    \[
        \ev_\infty \colon \Fun^\act(J^\rhd, \Fin_*) \too \Fin_*
    \]
    is a cocartesian fibration 
    and a morphism $\alpha \colon F \to G$ is a cocartesian edge if and only if $\alpha_j$ is inert for all $j \in J$.
\end{cor}
\begin{proof}
    The cocartesian symmetric monoidal structure on $\Fin_\ast$ is given by the cocartesian fibration $\Ar^\act(\Fin_\ast) \to \Fin_\ast$
    and cocartesian edges are inert natural transformations.
    By \cref{lem:cocart-operad-Fun} we can therefore identify $\Fun(J, \Fin_*)^\sqcup$ as the full subcategory of
    \[
        \Fun(J^\rhd, \Fin_*) 
        \simeq \Fun(J \times [1] \sqcup_{J \times \{1\}} \{\infty\}, \Fin_*)
        \simeq \Fun(J, \Ar(\Fin_*)) \times_{\Fun(J, \Fin_*)} \Fin_* 
    \]
    on those functors that send all morphisms to the cone point $j \to \infty$ to active morphisms.
    By cancellation of active morphisms these are the functors that send all morphisms to active morphisms.
    The cocartesian edges for 
    $\Fun(J, \Ar^\act(\Fin_*)) \to \Fun(J, \Fin_*)$
    are those natural transformations that are pointwise cocartesian \cite[Proposition 3.1.2.1]{HTT}, so restricting them along $\Delta$ we obtain the desired description of cocartesian edges.
\end{proof}

Restricting the cocartesian $\infty$-operad $\Fun(\Tw[n], \Fin)^\sqcup$ as described in \cref{cor:cocart-Fun-Fin} to the pushout preserving functors we see that 
$\widehat{\bmC}_n$ is equivalent to the full subcategory of $\Fun(\Tw[n]^\rhd, \Fin_*)$ on those functors that send all morphisms to active morphisms and whose restriction to $\Tw[n]$ preserves pushouts:
\[
    \widehat{\bmC}_n \simeq \Fun^{\rm po, \act}(\Tw[n]^\rhd, \Fin_*).
\]
The cocartesian edges (over $\Fin_*$) are still the pointwise inert natural transformations on $\Tw[n]$.
Unstraightening this over $\Dop$ we obtain a description of $\widehat{\bmC}$.

\begin{cor}\label{cor:boldC-description}
    The $\infty$-category $\widehat{\bmC}$ is equivalent to a $1$-category and admits the following description:
    \begin{itemize}
        \item 
        Objects are pairs $([n],A: \Tw[n]^\rhd \to \Fin_\ast)$ such that $A$ sends all morphism to active morphisms and such that $A|_{\Tw[n]}$ preserves pushouts,
        \item 
        Morphisms $([n],A) \to ([m],B)$ are pairs $(\lambda \colon [m] \leftarrow [n],\alpha\colon  \lambda^* A \to B)$.
    \end{itemize}
    Furthermore, a morphism $(\lambda,\alpha)$ is $\widehat{q}$-cocartesian if and only $\alpha$ is pointwise inert on $\Tw[n]$.
\end{cor}

\subsubsection{Equifiberedness through localizing}
So far we have found an algebraic pattern $\bmC$ such that $\bmC$-Segal spaces are (up to completion) symmetric monoidal $\infty$-categories over $\Csp$.
We would now like to modify this pattern in such a way that only the \emph{equifibered} symmetric monoidal $\infty$-categories over $\Csp$ are Segal spaces over it.
In this subsection we give an abstract criterion for how this can be achieved by passing to a localization of the pattern, assuming such a localization exists.

\begin{lem}\label{lem:abstract-eqf-pattern}
    Let $X\colon  \calB \too \calS$ be a functor and let $\pi\colon  \calB_X \too \calB$ denote its unstraightening. 
    Let $\calW \subseteq \calB$ be a wide subcategory and denote $\calW_X:=\pi^{-1}(\calW) \subseteq \calB_X$.
    Then the fully faithful functor
    \[
        \Fun(\calB_X[\calW_X^{-1}],\calS) 
        \hookrightarrow
        \Fun(\calB_X,\calS)
        \stackrel{(\ref{lem:P_X-Segal-spaces})}{\simeq}
        \Fun(\calB,\calS)_{/X}
    \]
    has as essential image precisely the \hldef{$\calW$-equifibered} morphisms,
    i.e.~those $(Y \to X)$ such that 
    \[
        \begin{tikzcd}
            Y(b) \ar[r, "Y(\omega)"] \ar[d] & Y(b') \ar[d]\\
            X(b) \ar[r, "X(\omega)"] & X(b')
        \end{tikzcd}
    \]
    is cartesian for all $(\omega\colon b \to b') \in \calW$.
\end{lem}
\begin{proof}
    A general $(Y \to X) \in \Fun(\calB, \calS)_{/X}$ is of the form $(\pi_!F \to X)$ for $F \in \Fun(\calB_X, \calS)$.
    Let $\omega \colon  b \to b'$ be a morphism in $\calB$ and consider the commutative square
    \[\begin{tikzcd}
	{(\pi_!F)(b)} & {(\pi_!F)(b')} \\
	{X(b)} & {X(b').}
	\arrow["{X(\omega)}", from=2-1, to=2-2]
	\arrow["{(\pi_!F)(\omega)}", from=1-1, to=1-2]
	\arrow[from=1-2, to=2-2]
	\arrow[from=1-1, to=2-1]
    \end{tikzcd}\]
    Using the notation from \cref{rem: functors-on-unstraightening} we may describe 
    the induced map on fibers over $(b,x) \in \calB_X \times_{\calB} \{b\} \simeq X(b)$ as $F(\omega_x)\colon F(b,x) \to F(b',\omega_!x) $,
    where $\omega_x \colon (b,x) \to (b',\omega_!x)$ is the unique lift of $\omega$ with source $(b,x)$.
    Therefore, the naturality square of $\pi_!F \to X$ at some $\omega \in \calW$ is cartesian
    if and only if $F(\omega_x)$ is an equivalence for all $x \in X(b)$. 
    The desired claim follows by quantifying over all $\omega \in \calW$.
\end{proof}

Mixing \cref{lem:abstract-eqf-pattern} with the Segal condition directly leads to the following corollary.
\begin{cor}\label{cor:ff-desc-of-eqseg}
    Let $\calP$ be an algebraic pattern, $\calW \subset \calP$ a subcategory, and $X\colon  \calP \to \calS$ a $\calP$-Segal space. 
    Suppose there exists a functor $\varphi\colon \calP_X \to \calL$ to an algebraic pattern $\calL$ satisfying:
    \begin{enumerate}[(1)]
        \item $\varphi$ exhibits $\calL$ as the localization of $\calP_X$ at $\pi^{-1}(\calW)$.
        \item $\varphi$ preserves and detects Segal objects:
        a functor $F\colon \calL \to \calS$ is an $\calL$-Segal space if and only if $F \circ \varphi\colon  \calP_X \to \calS$ is a $\calP_X$-Segal space.
    \end{enumerate}
    Then restriction along $\varphi$ followed by left Kan extension along $\pi\colon \calP_X \to \calP$ induces a fully faithful functor
    \[
        \Seg_{\calL}(\calS) \xhookrightarrow{\ \varphi^*\ }
        \Seg_{\calP_X}(\calS) 
        \xrightarrow[\simeq]{\ \pi_!\ } 
        \Seg_{\calP}(\calS)_{/X}
    \]
    the image of which consists precisely of the $\calW$-equifibered $\calP$-Segal spaces over $X$.
\end{cor}

\begin{war}\label{war:not-a-morphism}
    Note that in \cref{cor:ff-desc-of-eqseg} the functor $\varphi \colon \calP_X \to \calL$ is not necessarily a morphism of algebraic patterns as we do not require that it preserves elementary objects or the factorization system. 
    In our intended application the functor $\varphi \colon \bmC \to \L$ preserves the factorization system, but not elementary objects.
    
    It is tempting to think of a map $\varphi \colon \calP_X \to \calL$ satisfying the hypotheses of \cref{cor:ff-desc-of-eqseg} as a ``localization of patterns''. 
    Unfortunately, this intuition can be slightly misleading because
    neither the factorization system nor the elementary objects of $\calL$ are uniquely determined from the pattern structure on $\calP_X$.
\end{war}

\subsection{Constructing the localization}

We will now construct an algebraic pattern $\L$, which we exhibit as the localization of $\bmC$ at the preimage of $\simp^{\op,\simeq} \times \Fin_*^\act$ under the left fibration $q\colon \bmC \to \Dop \times \Fin_*$, and which we also prove to be equivalent to the pattern of ``level graphs'' $\bfL_{\rm CH}^\op$ introduced by Chu--Hackney \cite{CH22}.
Relying on the previous subsection we deduce that
the $\infty$-category of (pre)-properads is equivalent the $\infty$-category of $\L$-Segal spaces.
It was shown in loc.~cit.~that $\L$-Segal spaces are equivalent to the Segal $\infty$-properads of Hackney--Robertson--Yau \cite{HRY15}.
(We elaborate on this at the end of the section.)

\subsubsection{The category $\L$}
We begin by giving a concrete description of the localization of the $\infty$-category $\bmC$ at $q^{-1}\left(\simp^{\op,\simeq} \times \Fin_*^\act\right)$
and compare it to the category of level graphs from \cite{CH22}.

\begin{defn}\label{defn:category-L}
    For $[n] \in \Dop$ define $\hldef{\L_n} \subset \Fun(\Tw[n], \Fin_*)$ as the subcategory where
    \begin{itemize}
        \item Objects of $\L_n$ are pushout preserving functors $A\colon \Tw[n] \to \Fin_*$ that land in the (wide) subcategory spanned by the active morphisms $\Fin_\ast^\act \subseteq \Fin_\ast$.
        \item Morphisms of $\L_n$ are natural transformations $\alpha\colon A \Rightarrow B$ with $\alpha_{ij}\colon A_{ij} \intto B_{ij}$ inert.
    \end{itemize}
    This defines a functor $\L_\bullet\colon  \Dop \to \mathrm{Cat}_1$ by sending $([n] \leftarrow [m]:\! \lambda)$ to the restriction $\lambda^*\colon \L_n \to \L_m$ along $\Tw(\lambda)$.
    We let $\hldef{p\colon \L} \to \Dop$ denote the unstraightening of this functor.
\end{defn} 

This definition of $\L$ is indeed equivalent to the opposite of the category of the same name introduced by Chu--Hackney.
To avoid confusion we let $\bfL_{\rm CH}$ be the category from \cite[Definition 2.1.17]{CH22}. 
\begin{lem}\label{rem:L-compared-to-CH}
    There is an equivalence of categories $\L \simeq \bfL^\op_{\rm CH}$.
\end{lem}
\begin{proof}
    By \cite[Definitions 2.1.16 and 2.1.17]{CH22} the category $\bfL_{\rm CH}$ is defined as (a skeleton of) the total category of the cartesian unstraightening of the functor
    \[
        \widetilde{M}\colon \Dop \too \mrm{Cat}_1
    \]
    where $\widetilde{M}_n \subset \Fun(\Tw[n], \Fin)$ is the subcategory where objects are pushout preserving functors and morphisms are injective natural transformations such that every naturality square is a pullback (i.e.~it is a cartesian transformation).
    We therefore need to show that the two functors $\widetilde{M}_\bullet^\op, \L_\bullet\colon \Dop \to \mrm{Cat}_1$ are naturally equivalent.
    (Recall that when taking the opposite of a cartesian fibration, the resulting cocartesian fibration classifies the composite of the original functor with $(-)^\op\colon \Cat \to \Cat$).
    We can define a natural functor
    \begin{align*}
        \Phi_n\colon \widetilde{M}_n^\op & \too \L_n  \\
        (A\colon \Tw[n] \to\Fin) & \longmapsto (A_+\colon \Tw[n] \to \Fin_*)
    \end{align*}
    on objects by simply adding a disjoint base point.
    On morphisms, we send an injective cartesian transformation $A \leftarrow B\cocolon \alpha$ to the inert transformation $\alpha^\vee\colon A_+ \to B_+$ which on each component is defined by $\alpha^\vee(a) = b$ when $\alpha(b) = a$ and $\alpha^\vee(a) = \infty$ for $a \not\in \alpha(B)$.
    Note that the naturality squares for $\alpha^\vee$ commute if and only if the naturality squares for $\alpha$ commute and are cartesian:
    \[
    \begin{tikzcd}
        (A_{\lambda(i)\lambda(j)})_+ \ar[r, tail, "\alpha_{ij}^\vee"] \ar[d, squiggly] & (B_{ij})_+ \ar[d, squiggly] \\
        (A_{\lambda(i')\lambda(j')})_+ \ar[r, tail, "\alpha_{i'j'}^\vee"] & (B_{i'j'})_+ 
    \end{tikzcd}
    \qquad \Leftrightarrow \qquad
    \begin{tikzcd}
        A_{\lambda(i) \lambda(j)} \ar[d] & B_{ij} \ar[d] \ar[l, hook', "\alpha_{ij}"'] \ar[dl, phantom, near start, "\llcorner"]\\
        A_{\lambda(i') \lambda(j')} & B_{i'j'}. \ar[l, hook', "\alpha_{i'j'}"'] 
    \end{tikzcd}
    \]
    Hence $\Phi_n$ is a well-defined functor.
    We can also describe an inverse functor by removing basepoints and turning inert natural transformations into cartesian injective transformations in the opposite direction.
    Therefore, $\Phi_n$ is an equivalence.
    Moreover, $\Phi_\bullet$ is natural with respect to restriction along $\Tw(d)\colon \Tw[n] \to \Tw[m]$ for all $d\colon [n] \to [m]$ and hence it defines the desired equivalence $\widetilde{M}_\bullet^\op \simeq \L_\bullet$ in $\Fun(\Dop,\Cat)$.
\end{proof}

We now check that $\L$ is indeed a localization of $\bmC$.
The proof strategy is adopted from \cite[Proposition 4.2.3]{HK21}.
\begin{lem}\label{lem:L-is-localization-of-L}
    The functor $\varphi\colon  \bmC \to \L$ defined by 
    forgetting the value of $A\colon \Tw[n]^\rhd \to \Fin_*$ at the cone point
    exhibits $\L$ as the localization of $\bmC$ at the morphisms that lie over $(\Dop)^\simeq \times \Fin_*^\act$.
\end{lem}
\begin{proof}
    Since $\varphi \colon \bmC \to \L$ is a morphism of cocartesian fibrations over $\Dop$ and $\calW \coloneq q^{-1}((\Dop)^\simeq \times \Fin_*^\act)$ lies in the fibers we may use \cite[Proposition 4.2.6]{HK21} (which follows from \cite[Proposition 2.1.4]{Hin13}).
    It therefore suffices to show that 
    \[
        \varphi_n\colon \bmC_n \too \L_n
    \]
    is a localization at $\calW_n = q_n^{-1}(\Fin_*^\act)$,
    which we can do by checking the conditions of \cref{lem:localization-criterion} below.
    
    Inspecting the definitions of $\bmC_n$ and $\L_n$ we see that there is a pullback diagram
    \[
    \begin{tikzcd}
        \bmC_n \ar[rr, "\mrm{ev}_{0n \to \infty}"] \ar[d, "\varphi_n"'] \ar[dr, phantom, very near start, "\lrcorner"] & &
        \Ar^\act(\Fin_*) \ar[d, "s"] \\
        \L_n \ar[r, "\mrm{ev}_{0n}"] & 
        \Fin_*^\xint \ar[r] &
        \Fin_*
    \end{tikzcd}
    \]
    where $\Ar^\act(\Fin_*) \subset \Ar(\Fin_*)$ is the full subcategory on the active arrows
    and the top horizontal functor sends $A\colon \Tw[n]^\rhd \to \Fin_*$ to $(A_{0n} \to A_\infty)$.
    The right vertical functor $s\colon \Ar^\act(\Fin_*) \to \Fin_*$ 
    has a fully faithful left adjoint $\iota\colon \Fin_* \to \Ar^\act(\Fin_*)$ given by sending any finite pointed set $I_+$ to the identity morphism $(\id\colon I_+ \to I_+)$.
    Pulling back $\iota$ yields a fully faithful functor 
    $\iota\colon \L_n \too \bmC_n$, which is left adjoint to $\varphi_n$.
    This checks condition (1) of \cref{lem:localization-criterion}.
    
    The counit morphisms of this adjunction are of the form
    \[
        \begin{tikzcd}
            \{A_{ij}\} \ar[r, squiggly] \ar[d, equal] &
            A_{0n} \ar[r, equal] \ar[d, equal] &
            A_{0n} \ar[d, "\alpha_\infty"] \\
            \{A_{ij}\} \ar[r, squiggly] &
            A_{0n} \ar[r, squiggly] &
            A_\infty
        \end{tikzcd}
    \]
    where we write $\{A_{ij}\}$ to abbreviate the values of $A\colon \Tw[n] \to \Fin$ that are not $A_{0n}$.
    Note that $\alpha_\infty$ is active and hence this morphism lies in
    $\calW_n = q_n^{-1}(\Fin_*^\act)$,
    which is condition (2) of \cref{lem:localization-criterion}.
    
    Finally, we need to check that all morphisms in $\calW_n$ are indeed inverted by $\varphi_n$.
    If $\alpha\colon A \to B$ is such that $\alpha_\infty\colon  A_\infty \to B_\infty$ is active, 
    then by cancellation all $\alpha_{ij}\colon  A_{ij} \to B_{ij}$ are also active.
    But $\alpha_{ij}$ is inert by definition, so it is a bijection and therefore $\varphi(\alpha)$ is an isomorphism.
\end{proof}

It remains to check the fact about localizations that we used in the above proof.
\begin{lem}\label{lem:localization-criterion}
    Consider a functor $F\colon \calA \to \calB$ and a wide subcategory $\calW \subset \calA$ such that
    \begin{enumerate}[(1)]
        \item $F$ admits a fully faithful left adjoint $L\colon \calB \hookrightarrow \calA$,
        \item for all $a \in A$ the counit $L(F(a)) \to a$ is in $\calW$,
        \item for all $w\colon a \to a'$ in $\calW$ their image $F(w)$ is an equivalence.
    \end{enumerate}
    Then $F$ exhibits $\calB$ as the localization of $\calA$ at $\calW$.
\end{lem}
\begin{proof}
    By \cite[Proposition 5.2.7.12]{HTT} condition (1) implies that $F$ exhibits $\calB$ as the localization of $\calA$ at $F^{-1}(\calB^\simeq)$.
    Condition (3) ensures that $\calW \subset F^{-1}(\calB^\simeq)$.
    It remains to show that if $G\colon  \calB \to \calC$
    is some functor inverting $\calW$, then it also inverts $F^{-1}(\calB^\simeq)$.
    Let $f\colon a_1 \to a_2$ be a morphism in $F^{-1}(\calB^\simeq)$.
    Then we have a commutative square
    \[
        \begin{tikzcd}
            L(F(a_1)) \ar[r] \ar[d, "L(F(f))"', "\simeq"] & a_1 \ar[d, "f"] \\
            L(F(a_2)) \ar[r] & a_2.
        \end{tikzcd}
    \]
    Both horizontal morphisms are counits, so $G$ sends them to equivalences by (2) and hence it also sends $f$ to an equivalence, as claimed.
\end{proof}

\subsubsection{The pattern $\L$}
Now we construct a pattern structure on $\L$ that is compatible with the localization,
and show that this recovers the pattern of Chu--Hackney \cite{CH22}.

\begin{defn}\label{defn:pattern-L}
    We equip $\L$ with the following pattern structure:
    \begin{itemize}
        \item 
        A morphism in $\L$ is \hldef{inert} if its image in $\Dop$ is inert.
        \item
        A morphism $(\lambda, \alpha)$ in $\L$ is \hldef{active} if $\lambda$ is active and $\alpha\colon \lambda^*A \Rightarrow B$ is a natural isomorphism.%
        \footnote{
            Equivalently, the active morphisms are precisely the cocartesian lifts of active morphisms in $\Dop$.
            Note however that for a general cocartesian fibration over $\Dop$ the pair (all lifts of inerts, cocartesian lifts of actives) need not form a factorization system.
        }
        \item
        An object $([n], A)$ in $\L$ is \hldef{elementary} if $n \le 1$ and $A_{0n} \cong 1_+$.
    \end{itemize}
\end{defn}

One can check by hand that the above is indeed a factorization system, but this turns out to be a non-trivial task. 
Luckily this has been done in \cite{CH22}, and we can transport the factorization system through the equivalence in \cref{rem:L-compared-to-CH}.

\begin{lem}\label{lem:L-factorization-system}
    The above defines a factorization system on $\L$.
\end{lem}
\begin{proof}
    By \cite[Lemma 2.1.25]{CH22} there is a factorization system on $\bfL^\op_{\rm CH}$ where a map is inert if its image in $\Dop$ is inert, and $(\lambda,\alpha)$ is active if $\lambda$ is an active morphism in $\Dop$ and each $\alpha_{i,j}$ is a bijection.
    Since $\alpha_{i,j}$ is a bijection if and only if $\alpha_{i,j}^\vee$ is, these definitions correspond exactly to ours under the equivalence from \cref{rem:L-compared-to-CH}, so it follows that our active and inert morphisms form a factorization system on $\L$.
\end{proof}

\begin{lem}\label{lem:varphi-factorization-system}
    The localization functor $\varphi\colon \bmC \to \L$
    preserves the inert-active factorization system.
\end{lem}
\begin{proof}
    The functor $\varphi$ preserves the inert morphisms as they are defined in both cases as those morphisms whose image in $\Dop$ is inert and $\varphi$ is a functor over $\Dop$.
    An active morphism $([n] \leftarrow [m]:\! \lambda, \alpha\colon \lambda^*A \Rightarrow B)$ in $\bmC$ is any morphism whose image in $\Dop \times \Fin_*$ is active.
    We would like to show that in this case $\varphi(\lambda, \alpha) \in \L$ is always a cocartesian lift of $\lambda$,
    i.e.\ that $\alpha_{|\Tw[m]}$ is a natural isomorphism.
    By assumption $\alpha_\infty\colon A_\infty \to B_\infty$ is active, and hence by cancellation so are all $\alpha_{ij}\colon  A_{\lambda(i)\lambda(j)} \to B_{ij}$,
    but these are also all inert by definition, and thus isomorphisms.
\end{proof}

\subsubsection{Comparing to Segal $\infty$-properads}
We now define Segal $\infty$-properads in terms of the pattern $\L$:
\begin{defn}\label{defn:defn-segal-prpd}
    An \hldef{$\L$-Segal space} is a functor $\calP\colon  \L \to \calS$ such that the Segal map
    \[
        \calP([n], A) \too \lim_{([n], A) \intto ([m], B) \in \L^\el_{([n],A)/}} \calP([m],B)
    \]
    is an equivalence for all $([n], A) \in \L$.
    The underlying ($\Dop$-)Segal space of $\calP$ is defined as the simplicial space
    \[
        U_n(\calP) := \calP([n], 1_+) 
    \]
    where $1_+\colon  \Tw[n] \to \Fin_*$ denotes the constant functor at $1_+ = \{1,\infty\}$.
    A $\L$-Segal space is \hldef{complete} if its underlying ($\Dop$-)Segal space is complete.
    Let \hldef{$\CSeg_{\L}(\calS) \subset \Seg_{\L}(\calS)$} $\subset \Fun(\L, \calS)$ denote the full subcategories on the complete $\L$-Segal spaces and $\L$-Segal spaces, respectively.
\end{defn}

\begin{rem}
    The completeness condition described above only concerns the $(1,1)$-ary operations.
    In order to compare this to the completeness for $\bmC$-Segal spaces we will have to use \cref{lem:completeness-of-pprd-on-11ary}.
\end{rem}

Next we would like to show that $\varphi$ preserves and detects Segal objects.
For this we recall a variant of a lemma from \cite{CH22}, for which we give an independent proof.
\begin{lem}[{\cite[Proposition 3.2.9.(1 $\Leftrightarrow$ 3)]{CH22}}]\label{lem:Segal-condition-for-L}
    A functor $F\colon \L \to \calS$ is an $\L$-Segal space if and only if it satisfies the following two conditions:
    \begin{itemize}
        \item(segmentation condition) For all $([n], A) \in \L$ the map
        \[
            F([n], A) \too 
            F([1], A_{|\Tw(0\le 1)}) \times_{F([0], A_{11})}
            \dots
            \times_{F([0], A_{n-1,n-1})} F([1], A_{|\Tw(n-1 \le n)})
        \]
        given by restricting along $\Tw(i-1 \le i) \subset \Tw[n]$ and $\Tw(\{i\}) \subset \Tw[n]$,
        is an equivalence.
        \item(decomposition condition) For all $([n], A) \in \L$ with $n \le 1$ the map
        \[
            F([n], A) \too 
            \prod_{a \in A_{0n} \setminus \{*\}} F([n], A^{(a)})
        \]
        given by restriction to $A_{ij}^{(a)} = A_{ij} \times_{A_{0n}} \{a\}_+$ is an equivalence.
    \end{itemize}
\end{lem}
\begin{proof}
    Let
    $\L^{\qe} := \simp^{\op,\el} \times_{\simp^{\op,\xint}} \L^{\xint}$
    and note the fully faithful inclusions 
    $\L^{\el} \subseteq \L^{\qe} \subseteq \L^{\xint}$.
    A functor $F\colon \L \to \calS$ is Segal if and only if $F_{|\L^{\xint}}$ is right Kan extended from $\L^{\el}$, which in turn is the case if and only if $F_{|\L^{\xint}}$ is right Kan extended from $\L^{\qe}$ and $F_{|\L^{\qe}}$ is right Kan extended from $\L^{\el}$.
    Therefore, it suffices to prove the following statements:
    \begin{enumerate}[$(a)$]
        \item 
        $F$ satisfies the segmentation condition if and only if 
        $F|_{\L^{\xint}}$
        is right Kan extended from $\L^{\qe}$.
        \item 
        $F$ satisfies the decomposition condition if and only if 
        $F|_{\L^{\qe}}$
        is right Kan extended from $\L^{\el}$.
    \end{enumerate}
    To prove $(a)$ note that the projection 
    $\L^{\xint} \to \simp^{\op,\xint}$ 
    is a cocartesian fibration and thus for all $([n],A) \in \L$ 
    the induced functor 
    $\L^{\qe} \times_{\L^{\xint}} \L^{\xint}_{([n],A)/} \to \simp^{\op,\el}_{[n]/}$
    admits a left adjoint 
    $\simp^{\op,\el}_{[n]/} \to \L^{\qe} \times_{\L^{\xint}} \L^{\xint}_{([n],A)/}$ defined by cocartesian lifting
    \cite[Proposition 2.9]{CartFib}. 
    Left adjoints are always initial.
    Therefore, $F|_{\L^{\xint}}$ is right Kan extended from $\L^{\qe}$ if and only if $F([n], A)$ is equivalent to the limit over the diagram
    $\simp^{\op,\el}_{[n]/} \to \calS$
    defined by $([n] \leftarrowtail [\varepsilon]:\!\alpha) \mapsto F([\varepsilon], \alpha^*A)$.
    This is precisely the segmentation condition.

    To prove $(b)$
    we note that for all objects $([n],A) \in \L$
    with $n= 0$ or $1$ we have
    $\L^{\el} \times_{\L^{\le 1}} \L^{\le 1}_{([n],A)/} \simeq \L^{\el}_{([n],A)/}$,
    and thus it suffices to show that the Segal condition at such objects
    is equivalent to the decomposition condition.
    For $n=0$ this follows from the evident equivalence 
    $\L^{\el}_{([0],A)/} \simeq A_{0,0}$.
    For $n=1$ we can readily compute the relevant slice category
    \[ \L^{\el}_{([1],A)/} \simeq \coprod_{a\in A_{0,1}} \left( A_{0,0}^{(a)}   \amalg   A^{(a)}_{1,1} \right)^{\lhd} 
    \too
    \left( * \sqcup * \right)^\lhd
    \simeq \simp^{\op,\el}_{[1]/}. \]
    The fiber over the cone point is equivalent to the discrete category $A_{0,1}$ and is indeed initial by inspection.
\end{proof}

This allows us to check that $\varphi\colon \bmC \to \L$ satisfies condition $(3)$ from \cref{cor:ff-desc-of-eqseg}:

\begin{lem}\label{lem:varphi-creates-Segal}
    A functor $F\colon \L \to \calS$ is an $\L$-Segal space
    if and only if $F \circ \varphi\colon  \bmC \to \calS$ is a $\bmC$-Segal space.
\end{lem}
\begin{proof}
    As $q\colon \bmC \to \Dop \times \Fin_*$ is a left fibration, it induces an equivalence between the category of elementary objects under a given 
    $(A\colon \Tw[n]^\rhd \to \Fin_*) \in \bmC$ 
    and the category of elementary objects 
    under $([n], A_\infty) \in \Delta^{\op} \times \Fin_*$.
    It follows that $G\colon \bmC \to \calS$ is Segal if and only if for all 
    $(A\colon \Tw[n]^\rhd \to \Fin_*) \in \bmC$ 
    the map
        \begin{equation}\label{eqn:segal-condition-for-calC}
            {G([n], A)} \too
            {\prod\limits_{a \in A_{\infty}\setminus \{*\}}
            G([1], A_{|\Tw(0\le 1)^\rhd}^{(a)}) \underset{G([0], A_{1,1}^{(a)} \to \{a\}_+)}{\times} \dots
            \underset{G([0], A_{n-1,n-1}^{(a)} \to \{a\}_+)}{\times} G([1], A_{|\Tw(n-1 \le n)^\rhd}^{(a)})}
        \end{equation}
    is an equivalence.
    Here we write 
    $A_{|\Tw(i-1\le i)^\rhd}^{(a)} := A_{|\Tw(i-1\le i)^\rhd} \times_{A_\infty} \{a\}_+$.
    When $n=0$ the map is given instead by $G([0],A) \to \prod_{a \in A_\infty \setminus \{*\}} G([0], A_{00}^{(a)} \to \{a\}_+)$.
    
    Suppose $F\colon \L \to \calS$ is such that $G \coloneq F \circ \varphi$ is Segal.
    For $(A\colon \Tw[n] \to \Fin_*^\act) \in \L$ we define the lift 
    $(A^{\rm min}\colon \Tw[n]^\rhd \to \Fin_*^\act)$ to $\bmC$ by setting $A_\infty^{\rm min} := \{1\}_+$ with the unique active map from $A_{0,n}$.
    The Segal condition for $G([n], A^{\rm min})$ is then precisely the segmentation condition from \cref{lem:Segal-condition-for-L} for $F$.
    To check the decomposition condition for $n=0,1$ consider some
    $(A\colon \Tw[n] \to \Fin_*^\act) \in \L$ and define the lift
    $(A^{\rm max}\colon \Tw[n]^\rhd \to \Fin_*^\act)$ to $\bmC$ 
    by setting $A_\infty^{\rm max} := A_{0,n}$ with identity from $A_{0,n}$.
    The Segal condition for $G([n], A^{\rm max})$ is then equivalent to the decomposition condition for $F$ (still assuming $n =0,1$).
    Therefore, $F$ is Segal by \cref{lem:Segal-condition-for-L}.
    
    Conversely, suppose $F$ is Segal.
    We would like to show that $G$ is a $\bmC$-Segal space.
    For $G=F \circ \varphi$ the map in \cref{eqn:segal-condition-for-calC} can be identified with the composite:
    \begin{align*}
	{F([n], A) } 
	\too &\; {F([1], A_{|\Tw(0\le 1)}) \times_{F([0], A_{11})}             \dots \times_{F([0], A_{n-1,n-1})} F([1], A_{|\Tw(n-1 \le n)})} \\
	\too & {\prod\limits_{a \in A_{\infty} \setminus \{*\}}             F([1], A^{(a)}_{|\Tw(0\le 1)}) \times_{F([0], A^{(a)}_{11})}             \dots             \times_{F([0], A^{(a)}_{n-1,n-1})} F([1], A^{(a)}_{|\Tw(n-1 \le n)})}
    \end{align*}
    The first map is an equivalence by the segmentation condition,
    and the second map is an equivalence by the decomposition condition for $n=0,1$.
\end{proof}

\subsubsection{The comparison theorem}
Using the functors
\[
    \L \xleftarrow{\varphi} \bmC \xrightarrow{q} \Dop \times \Fin_*
\]
we now give the promised definition of the envelope functor for Segal $\infty$-properads modelled on the pattern $\L$.

\begin{defn}\label{defn:env-Lop}
    The envelope functor for $\L$-Segal spaces is defined as
    \[
        \hldef{\Env_{\L}}\colon
        \Seg_{\L}(\calS) \xhookrightarrow{\ \varphi^*\ }
        \Seg_{\bmC}(\calS) 
        \xrightarrow{\ q_!\ } 
        \Seg_{\Dop \times \Fin_*}(\calS)
        \simeq \Seg_{\Dop}(\CMon).
    \]
\end{defn}

\begin{thm}\label{thm:segal-envelope}
    The envelope for $\L$-Segal spaces lands in pre-properads and gives an equivalence
    \[
        \Env_{\L}\colon
        \Seg_{\L}(\calS) \xrightarrow{\ \simeq\ }
        \pPrpd \subset \Seg_{\Dop}(\CMon).
    \]
    Moreover, this restricts to an equivalence
    \[
        \Env_{\L}\colon
        \CSeg_{\L}(\calS) \xrightarrow{\ \simeq\ } \Prpd,
    \]
    between the $\infty$-category of complete $\L$-Segal spaces and the $\infty$-category of $\infty$-properads.
\end{thm}
\begin{proof}
    By \cref{lem:L-is-localization-of-L} and \cref{lem:varphi-creates-Segal},
    the functor $\varphi\colon \bmC \to \L$ satisfies the conditions of \cref{cor:ff-desc-of-eqseg}.
    Combining this with \cref{cor:Segal_bmC} we obtain a fully faithful functor
    \[
        \Env^{/\Csp}_{\L}: 
        \Seg_{\L}(\calS) \xhookrightarrow{\ \varphi^*\ }
        \Seg_{\bmC}(\calS) 
        \xrightarrow[\ q_!\ ]{\simeq}
        \Seg_{\Dop \times \Fin_*}(\calS)_{/\St_{\Dop\times\Fin_*}(q)}
        \simeq \Seg_{\Dop}(\CMon)_{/\xN_\bullet\Csp}
    \]
    whose essential image consists of those $P_\bullet \to \xN_\bullet\Csp$ which are equifibered.
    These are automatically pre-properads because they are equifibered over a pre-properad and since $\xN_\bullet\Csp$ is the terminal pre-properad by \cref{thm:Csp-is-final} the essential image is equivalent to $\pPrpd$
    via the functor that forgets the map to $\xN_\bullet\Csp$.
    This proves the first part of the theorem. 
     
    For the second part, we need to show that a $\L$-Segal space $X\colon \L \to \calS$ is complete if and only if the pre-properad $\Env_\L(X) \in \Seg_\Dop(\CMon)$ is complete.
    The underlying Segal space $U_\bullet(X)$ of $X$ can be written as the composite
    \[
        \Dop \to \bmC \xrightarrow{\varphi} \L \xrightarrow{X} \calS
    \]
    where the first functor sends $[n]$ to $([n], \{1\}_+\colon  \Tw[n]^\rhd \to \Fin_*)$.
    Under the left Kan extension along $q\colon \bmC \to \Dop \times \Fin_*$ the space $U_n(X)$ corresponds to the fiber of 
    $(q_!\varphi^*X)([n],1_+) \to \xN_n\Csp$
    over the connected component of the terminal functor $\left\{\left(\ast \colon \Tw[n] \to \Fin \right)\right\} \subset \xN_n\Csp$.
    We thus have an equivalence
    \[
        U_\bullet(X) \simeq * \times_{\xN_\bullet\Csp} \Env_\L(X).
    \]
    The right side is the simplicial space $\Env_\L(X)^{(1,1)}$ from \cref{lem:completeness-of-pprd-on-11ary},
    so the claim follows from \cref{lem:completeness-of-pprd-on-11ary}
    where we showed that a pre-properad is complete if and only if its $(1,1)$-ary operations are complete.
\end{proof}

\subsubsection{Simplified formula for nerve of the envelope}
We now give a simpler formula for computing the $\simp^{\op,\act}$-part of the nerve of the envelope $\Env_\L(X) \in \Prpd$ for any $\L$-Segal space $X$, which will be useful in \cref{sect-3.4}.

\begin{cor}\label{cor:act-Env-formula}
    For $X \in \CSeg_\L(\calS)$ there is a natural equivalence
    \[
        \xN_\bullet \Env_\L(X)|_{\simp^{\op,\act}} \simeq p^\act_! (X|_{\L^\act})
    \]
    between the nerve of the $\infty$-category $\Env(X)$ restricted to $\simp^{\op,\act}$
    and the left Kan extension of $X|_{\L^\act}$ along $p^\act\colon \L^\act \to \simp^{\op,\act}$.
    In particular, for each $[n] \in \Dop$ we have an equivalence of spaces
    \[
        \xN_n \Env_\L(X) \simeq \colim_{A\colon \Tw[n] \to \Fin} X([n],A_+)
    \]
    where the colimit runs over the $1$-groupoid of pushout preserving functors $A\colon \Tw[n] \to \Fin$.
\end{cor}
\begin{proof}
    Since we assumed that $X$ is complete, so is $q_!\varphi^*X$ (by \cref{thm:segal-envelope}) and thus $\xN_\bullet\Env_\L(X) = q_!\varphi^*X$.
    (There is an abuse of notation here: we write $\Env_\L$ for both the envelope valued in $\Seg_{\Dop}(\CMon)$ and the envelope valued in $\Prpd \subset \SM$. These are related by the nerve, but since we assumed completeness this difference does not matter.)
    
    We define a functor $\tilde{j}\colon \L^\act \to \bmC$ that sends $([n], A\colon \Tw[n] \to \Fin_*^\act)$ to $([n], A')$ where $A'\colon \Tw[n]^\rhd \to \Fin_*^\act$ is defined by extending $A$ by $A'_\infty = 1_+$. (Note that this is well-defined on active maps $(\lambda, \alpha)$ in $\L$ as they are cocartesian and thus $\alpha\colon A\circ \Tw(\lambda) \to B$ is a natural isomorphism.)
    By unwinding definitions we see that the square in the diagram 
    \[\begin{tikzcd}
	{\L^{\act}} && \bmC & {\L} \\
	{\Delta^{\op,\act} } && {\Dop \times \Fin_\ast}
	\arrow["\tilde{j}", from=1-1, to=1-3]
	\arrow["{p^\act}"', from=1-1, to=2-1]
	\arrow["\varphi", from=1-3, to=1-4]
	\arrow["q", from=1-3, to=2-3]
	\arrow[""{name=0, anchor=center, inner sep=0}, "{j \coloneq (-,1_+)}"', from=2-1, to=2-3]
	\arrow["\lrcorner"{anchor=center, pos=0.125}, draw=none, from=1-1, to=0]
    \end{tikzcd}\]
    commutes and is cartesian and that the composite $\varphi \circ \tilde{j}$ is the inclusion $\L^\act \subset \L$.
    Since $q$ is a left fibration we, by \cite[Propositions 4.4.11 and 6.4.3]{cisinski}, get a Beck-Chevalley isomorphism $j^\ast q_! \simeq p^\act_! \tilde{j}^\ast$ and thus
    \[
    \xN_\bullet\Env_\L(-)|_{\Delta^{\op,\act}} \simeq j^\ast q_!\varphi^\ast \simeq p^\act_! \tilde{j}^\ast \varphi^\ast \simeq p^\act_! (\varphi \circ \tilde{j})^\ast \simeq p^\act_!(-|_{\L^{\act}}). 
    \]
    This shows the first claimed equivalence, and the second one follows as the left Kan extension along the left fibration $p^\act$ can be computed by taking colimits over the fibers, which are the groupoids
    $(p^\act)^{-1}([n]) = \Fun^{\rm po}(\Tw[n], \Fin_*^\act)^\simeq$.
    (Alternatively, base-change $p^\act$ along the inclusion $\{[n]\} \subset \simp^{\op,\act}$ and use the resulting Beck-Chevalley isomorphism.)
\end{proof}

\subsubsection{Graphical interpretations}
The original definition of Segal $\infty$-properads \cite{HRY15} was in terms of the ``properadic graphical category'' $\Gamma^\op$.
Chu--Hackney introduce a new category $\bfG^\op$ \cite[\S 2]{CH22} (based on work of Kock \cite{Koc16}),
and show that it is equivalent to $\Gamma^\op$ \cite[Appendix A]{CH22}.
Objects in $\bfG^\op$ are connected acyclic directed graphs and morphisms include both edge-collapses as well as restrictions to subgraphs.
We will not repeat the definition here as it is quite involved, 
but we will recall how the combinatorial data of $\L$ is related to graphs (see \cite[\S 2.3]{CH22}).

To an object $\left([n],A\right) \in \L$ we can associate a graph $\Gamma_A$ (see \cref{fig:level-graphs}) whose edges and vertices are given by
\[E(\Gamma_A) \coloneq \coprod_{0\le j \le n} A_{j,j}
\qquad \text{ and } \qquad  
V(\Gamma_A)\coloneq \coprod_{0\le j <n} A_{j,j+1}.\]
A vertex $v \in A_{j,j+1} \subseteq V(\Gamma_A)$ has incoming and outgoing edges given by
\[\mathrm{in}(v) \coloneq A_{j,j}\times_{A_{j,j+1}} \{v\}, 
\qquad \text{ and } \qquad  
 \mathrm{out}(v) \coloneq A_{j+1,j+1} \times_{A_{j,j+1}} \{v\}.\]
The evident map $E(\Gamma_A) \to \{0,\dots,n\}$ defines a ``levelling'' on $\Gamma_A$. 
Inert maps in $\L$ correspond to passing to level subgraphs,
whereas active maps correspond to level edge collapses.
(If one edge is collapsed, then all edges of the same level have to be collapsed as well.)
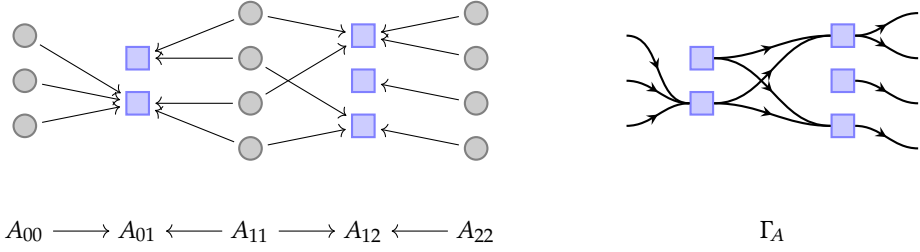
\begin{figure}[ht]
\centering
\usetikzlibrary{decorations.markings}
\begin{tikzpicture}[%
    edg/.style={circle,thick, draw=black!50, fill=black!20,
                 inner sep=1pt,minimum size=3mm},
    ver/.style={rectangle,draw=blue!50,fill=blue!20,thick,
                      inner sep=2pt,minimum size=3mm},
    arr/.style={->, shorten >= 2pt, shorten <= 2pt},
    directed/.style={thick, postaction = decorate},
    decoration={markings,
    mark=at position 0.5 with {\arrow{stealth}}}]
    \def\xspace{1.5}
    \def\yspace{.6}
    
    \foreach \i in {1,...,3}
    {        \node[edg] (A00-\i) at (0,{(\i+.5)*\yspace})  {};    }
    \foreach \i in {1,...,2}
    {        \node[ver] (A01-\i) at (\xspace,{(\i+1)*\yspace})  {};    }
    \foreach \i in {1,...,4}
    {        \node[edg] (A11-\i) at (2*\xspace,{\i*\yspace})  {};    }
    \foreach \i in {1,...,3}
    {        \node[ver] (A12-\i) at (3*\xspace,{(\i+.5)*\yspace})  {};    }
    \foreach \i in {1,...,4}
    {        \node[edg] (A22-\i) at (4*\xspace,{\i*\yspace})  {};    }

    \draw [arr] (A00-1) to (A01-1);
    \draw [arr] (A00-2) to (A01-1);
    \draw [arr] (A00-3) to (A01-1);
    
    \draw [arr] (A11-1) to (A01-1);
    \draw [arr] (A11-2) to (A01-1);
    \draw [arr] (A11-3) to (A01-2);
    \draw [arr] (A11-4) to (A01-2);
    
    \draw [arr] (A11-1) to (A12-1);
    \draw [arr] (A11-2) to (A12-3);
    \draw [arr] (A11-3) to (A12-1);
    \draw [arr] (A11-4) to (A12-3);
    
    \draw [arr] (A22-1) to (A12-1);
    \draw [arr] (A22-2) to (A12-2);
    \draw [arr] (A22-3) to (A12-3);
    \draw [arr] (A22-4) to (A12-3);
    
    \def\labeloffset{.5}
    \node (A00) at (0, -\labeloffset) {$A_{00}$};
    \node (A01) at (\xspace, -\labeloffset) {$A_{01}$};
    \node (A11) at (2*\xspace, -\labeloffset) {$A_{11}$};
    \node (A12) at (3*\xspace, -\labeloffset) {$A_{12}$};
    \node (A22) at (4*\xspace, -\labeloffset) {$A_{22}$};
    
    \draw [->] (A00) to [out = 0, in = 180] (A01);
    \draw [->] (A11) to [out = 180, in = 0] (A01);
    \draw [->] (A11) to [out = 0, in = 180] (A12);
    \draw [->] (A22) to [out = 180, in = 0] (A12);

    \def\hoffset{5*\xspace}    
    \foreach \i in {1,...,3}
    {        \node[] (gA00-\i) at (\hoffset + 0.25*\xspace,{(\i+.5)*\yspace})  {};    }
    \foreach \i in {1,...,2}
    {        \node[ver] (gA01-\i) at (\hoffset + \xspace,{(\i+1)*\yspace})  {};    }
    \foreach \i in {1,...,3}
    {        \node[ver] (gA12-\i) at (\hoffset + 2.25*\xspace,{(\i+.5)*\yspace})  {};    }
    \foreach \i in {1,...,4}
    {        \node[] (gA22-\i) at (\hoffset + 3*\xspace,{\i*\yspace})  {};    }

    \draw [directed] (gA00-1) to [out = 0, in = 180] (gA01-1);
    \draw [directed] (gA00-2) to [out = 0, in = 180] (gA01-1);
    \draw [directed] (gA00-3) to [out = 0, in = 180] (gA01-1);
    
    \draw [directed] (gA01-1) to [out = 0, in = 180] (gA12-1);
    \draw [directed] (gA01-1) to [out = 0, in = 180] (gA12-3);
    \draw [directed] (gA01-2) to [out = 0, in = 180] (gA12-1);
    \draw [directed] (gA01-2) to [out = 0, in = 180] (gA12-3);
    
    \draw [directed] (gA12-1) to [out = 0, in = 180] (gA22-1);
    \draw [directed] (gA12-2) to [out = 0, in = 180] (gA22-2);
    \draw [directed] (gA12-3) to [out = 0, in = 180] (gA22-3);
    \draw [directed] (gA12-3) to [out = 0, in = 180] (gA22-4);
    
    \node (Gamma) at (\hoffset + 1.625*\xspace, -\labeloffset) {$\Gamma_A$};
\end{tikzpicture}
\caption{An object $(A\colon \Tw[n] \to \Fin) \in \L_n$ (for $n=2$) interpreted as a level graph.}
\label{fig:level-graphs}
\end{figure}

Since objects in $\bfG^\op$ are connected graphs, to relate it to $\L$ we will first have to restrict our attention to the ``connected'' objects.
\begin{defn}\label{defn:Lc}
    We say that $([n],A) \in \L$ is \hldef{connected} if $A_{n,n} \cong 1_+$. 
    Note that this is equivalent to requiring that $\Gamma_A$ be a connected graph.
    We write $\hldef{\Lc} \subseteq \L$ for the full subcategory spanned by connected objects.
\end{defn}

The pattern structure on $\L$ canonically restricts to a pattern structure on $\Lc$.\footnote{In fact, $\Lc \subseteq \L$ is an example of an \textit{algebraic subpattern} in the sense of \cite[Definition 2.30]{Shaul-arity}.}
In the language of \cite{CH22} the objects of $\Lc$ are connected level graphs and forgetting the levelling defines a morphism of algebraic patterns 
$\tau \colon \Lc \to \bfG^\op$.

\begin{thm}[{\cite{CH22}}]\label{thm:L=G}
    The span of algebraic patterns 
    $\L \overset{j}{\hookleftarrow} \Lc \xrightarrow{\tau} \bfG^\op$ 
    gives rise to equivalences
    \[\begin{tikzcd}
	{\Seg_{\L}(\calS)} & {\Seg_{\Lc}(\calS)} & {\Seg_{\bfG^\op}(\calS).}
	\arrow["\simeq"', draw=none, from=1-1, to=1-2]
	\arrow["\simeq", draw=none, from=1-3, to=1-2]
	\arrow["{\,\tau^\ast}"', from=1-3, to=1-2]
	\arrow["{j^\ast}", from=1-1, to=1-2]
    \end{tikzcd}\]
    Moreover, these equivalences respect the notion of completeness.
\end{thm}

Restricting the envelope functor $\Env_{\L}$ (\cref{defn:env-Lop}) along this equivalence gives an envelope for $\bfG^\op$-Segal spaces.
For clarity, we spell out the resulting functor.

\begin{defn}
    The envelope functor for $\bfG^\op$-Segal spaces is defined as
    \[\hldef{\Env_{\bfG^\op}} \colon \Seg_{\bfG^\op}(\calS) \xrightarrow[\ \tau^\ast \ ]{\simeq}\Seg_{\Lc}(\calS) \xrightarrow[\ j_\ast \ ]{\simeq} \Seg_{\L}(\calS) \xhookrightarrow{\ \varphi^*\ }
    \Seg_{\bmC}(\calS) 
    \xrightarrow{\ q_!\ }
    \Seg_{\Dop \times \Fin_\ast}(\calS) \simeq \Seg_{\Dop}(\CMon).\]
\end{defn}

Combining \cref{thm:segal-envelope} and \cref{thm:L=G} we get a well-behaved envelope functor for Segal $\infty$-properads modelled on the pattern $\bfG^\op$.

\begin{cor}\label{cor:segal-G-envelope}
    The envelope for $\bfG^\op$-Segal spaces lands in pre-properads and gives an equivalence
    \[
        \Env_{\bfG^\op}\colon
        \Seg_{\bfG^\op}(\calS) \xrightarrow{\ \simeq\ }
        \pPrpd \subset \Seg_{\Dop}(\CMon).
    \]
    Moreover, this restricts to an equivalence
    \[
        \Env_{\bfG^\op}\colon
        \CSeg_{\bfG^\op}(\calS) \xrightarrow{\ \simeq\ } \Prpd,
    \]
    between the $\infty$-category of complete $\bfG^\op$ -Segal spaces and the $\infty$-category of $\infty$-properads.
\end{cor}

\section{\texorpdfstring{$n$}{n}-properads and projective \texorpdfstring{$\infty$-properads}{infinity-properads}}\label{sec:5}

The primary goal of this section is to study $n$-properads, which we define as those $\infty$-properads where all spaces of operations are $(n-1)$-truncated.
We prove that the resulting $(2,1)$-category of $1$-properads $\mrm{Prpd}_1$ is equivalent to the $(2,1)$-category of ``labelled cospan categories'' of \cite{Jan-tropical}, as conjectured there.
An analogous result was recently proven by Beardsley--Hackney \cite{BH22}, who use a more classical definition of $1$-properads.
Combining these results we will see that our $1$-properads are equivalent to the properads as defined in \cite{JohnsonYau}, based on \cite{Val07}.

In proving the above comparison the operations of arity $(0,0)$ will play a special role, as they prevent the underlying symmetric monoidal $\infty$-category of a $1$-properad from being a $1$-category.
Because of this, the first two subsections of this section will explore in detail how to delete and reintroduce $(0,0)$-ary operations in an $\infty$-properad.
This question is of independent interest and leads to notions of projective, reduced, and extended $\infty$-properads.
We will for example in \cref{cor:Prpd-as-pullback-factorization} show that $\Prpd$ can be written as a pullback that encodes an $\infty$-properad $\calP$ by remembering the projective $\infty$-properad $\smash{\Red{\calP}} \coloneq \calP/B\End_\calP(\unit)$, the space of $(0,0)$-ary operations $\calP(\emptyset; \emptyset)$, and a map that glues them together.
We will provide an effective criterion for deciding when an $\infty$-properad $\calP$ is extended, i.e.~when its $(0,0)$-ary operations are freely generated from positive arity operations, and we observe that this is the case for the bordism $\infty$-properad $\Bord_d$, which has implications for topological field theories.

\subsection{Projective \texorpdfstring{$\infty$-properads}{infinity-properads}}\label{sec:reduced-projective}

To motivate the study of $(0,0)$-ary operations, recall that the axioms of a bialgebra require $\varepsilon \circ \nu = \id_\unit$ for $\nu\colon \unit \to A$ the unit and $\varepsilon\colon A \to \unit$ the counit.
As discussed in \cref{rem:Span-not-prpd}, one cannot impose such a relation in an $\infty$-properad since $\varepsilon \circ \nu$ is a $(0,0)$-ary operation, but $\id_\unit$ is not.
We introduce a notion of ``projective $\infty$-properads'' where all $(0,0)$-ary operations are identified with $\id_\unit$.
Concretely, we will see that every projective $\infty$-properad can be obtained as $\Red{\calP} = \calP/\calP_0$ where $\calP_0 = B\End_\calP(\unit) \subset \calP$ is the full subcategory of an $\infty$-properad $\calP$ on the monoidal unit.

We show that the $\infty$-category of projective $\infty$-properads $\projPrpd$ is equivalent to the full subcategory $\Prpd^\rd \subset \Prpd$ of those $\infty$-properads that are ``reduced'' in the sense that they have a unique arity $(0,0)$ operation, and we show that there is a triple adjunction:
    \[
    \begin{tikzcd}
	\Prpd && {\Prpd^\rd} \ar[r, "\simeq"] & \projPrpd.
	\arrow["{(-)^\rd}"{description}, from=1-1, to=1-3]
	\arrow["{(-)^{\rm ext}}"', shift right = 3, hook', from=1-3, to=1-1]
	\arrow["{\text{include}}", shift left = 3, hook, from=1-3, to=1-1]
    \end{tikzcd}
    \]
We will further see that this exhibits $\Prpd$ as part of a semi-recollement and that $\Prpd$ can thus be written as a certain pullback $\Prpd^\rd \times_\calS \Ar(\calS)$.

\subsubsection{Sub-terminal $\infty$-properads}
Recall that an object $x$ in an $\infty$-category $\calC$ is called \hldef{subterminal} if for every other object $y \in \calC$ the mapping space $\Map_\calC(y,x)$ is either empty or contractible.
From this definition it follows that if there is a terminal object $t \in \calC$, then $x$ is subterminal if and only if $x \to t$ is a monomorphism.
By \cref{cor:Csp-is-final} $\Csp \in \Prpd$ is terminal, so sub-terminal objects of $\Prpd$ are precisely the subproperads 
$\calP \subseteq \Csp$.
Using \cref{cor:sub-properad}, it is straightforward to classify all such $\infty$-properads. 
This is similar to how we classified full subproperads in \cref{cor:subproperad-classification}.

\begin{lem}\label{cor:subterminal-prpd}
    Call a subset $S \subseteq \bbN \times \bbN$ admissible if it satisfies:
    \begin{enumerate}[(1)]
        \item $(1,1) \in S \text{ or } S = \emptyset \text{ or } S = \{(0,0)\}$.
        \item $(a,c),(b,d) \in S \text{ and } 1 \le k \le \min(b,c) \implies (a+b-k,c+d-k) \in S$.
    \end{enumerate}
    Then there is a canonical order preserving bijection
    \[\{\text{sub-terminal $\infty$-properads}\} \xtoo{\cong}  \{\text{admissible subsets } S \subseteq \bbN \times \bbN\}\]
    defined by sending $\calP \subset \Csp$ to $\pi_0(\Ar(\calP)^\simeq)^\el \subset \pi_0(\Ar(\Csp)^\simeq)^\el \simeq \bbN \times \bbN$.
\end{lem}
\begin{proof}
    By \cref{cor:sub-properad} giving a subproperad $\calP \subset \Csp$
    is equivalent to giving submonoids 
    $\calP^\simeq \subset \xN_0 \Csp$ and $\Ar(\calP)^\simeq \subset \xN_1 \Csp$
    such that they yield a well-defined subcategory (so they must be closed under passing to source/target, taking identity morphisms, and composition)
    and such that the submonoid inclusions are equifibered.
    Via \cref{lem:submonoid-eqf} we saw that the latter condition means that $f \otimes g \in \Ar(\calP)^\simeq$ if and only if $f,g \in \Ar(\calP)^\simeq$.
    
    The only equifibered submonoids of $\xN_0 \Csp = \xF(*)$ are $\calP^\simeq = 0$ or $\calP^\simeq = \xN_0 \Csp$,
    whereas choosing $\Ar(\calP)^\simeq \subset \xN_1\Csp \simeq \xF(\xF(*) \times \xF(*))$ is equivalent to choosing a full subspace of $\xF(*) \times \xF(*)$,
    or equivalently a subset $S \subset \pi_0(\xF(*) \times \xF(*)) = \bbN^2$.
    
    If $\calP^\simeq = 0$, then the only options for $S$ are $S = \emptyset$ or $S = \{(0,0)\}$.
    If $\calP^\simeq = \Csp^\simeq$ we know that $(1,1) \in S$ as $\calP$ has to contain the identity morphism on $* \in \Csp$.
    It remains to ensure that $\Ar(\calP)^\simeq$ is closed under composition.
    Since any two morphisms in an $\infty$-properad may be composed by iteratively composing operations along one or multiple colours, this exactly amounts to the condition that 
    for any two operations $(a,c),(b,d) \in S$ and any number of colours $k\ge 1$ with $k\le c$ and $k \le b$ the composite $(a+b-k,c+d-k)$ is still in $S$.
\end{proof}

\begin{defn}
    Given an admissible subset $S \subseteq \bbN \times \bbN$ we write $\hldef{\Prpd^S} \subseteq \Prpd$ for the full subcategory of those $\infty$-properads that only have operations of arities in $S$.
\end{defn}

\begin{example}
    The admissible subset 
    $\bbN \times \{1\} \subseteq \bbN \times \bbN$ 
    corresponds to the full subcategory of monic $\infty$-properads 
    $\Prpd^{\bbN \times \{1\}} = \Prpd^{\monic} \subseteq \Prpd$ from \cref{defn:monic},
    and similarly we have $\Prpd^{\{1\}\times\bbN} = \Prpd^{\comonic} \subseteq \Prpd$. 
\end{example}

\begin{example}
    The admissible subset $\{(1,1)\} \subseteq \bbN \times \bbN$ corresponds to the full subcategory $\Prpd^{(1,1)} \subseteq \Prpd$ spanned by $\infty$-properads $\calP$ for which $\calP(k,l) =\emptyset$ unless $k=l=1$. 
    By definition $\calP \in \Prpd^{(1,1)}$ if and only if the terminal map $\calP \to \Csp$ factors through $\Fin^\simeq \subseteq \Csp$.
    We thus have by \cref{cor:free-functor-on-Cat} an equivalence
    \[\xF \colon \Cat \iso \SMeq{\Fin^\simeq} \simeq \Prpd^{(1,1)}.  \]
\end{example}

\begin{example}\label{ex:nullary-properads}
    The admissible subset $\{(0,0)\} \subseteq \bbN \times \bbN$ corresponds to the full subcategory $\hldef{\Prpd^{(0,0)}} \subseteq \Prpd$ of \hldef{nullary $\infty$-properads}. 
    These are the $\infty$-properads $\calP$ such that $\calP(k,l) = \emptyset$ unless $k=l=0$,
    or equivalently, they are the symmetric monodial $\infty$-categories that admit an equifibered functor to the full subcategory of $\Csp$ on the unit.
    This full subcategory is exactly $\frB(\xF(*))$, and it then follows from \cref{lem:eqf-over-frB} that we have an equivalence
    \[
        \frB \circ \xF\colon \calS \xtoo{\simeq} \SMeq{\frB(\xF(*))} \simeq \Prpd^{(0,0)}
    \]
    whose inverse sends an $\infty$-properad $\calP$ to its space of $(0,0)$-ary operations $\calP(\emptyset; \emptyset)$.
\end{example}

\begin{lem}\label{lem:adjoint-PrpdS}
    For every admissible $S \subset \bbN \times \bbN$ the full inclusion $\Prpd^S \hookrightarrow \Prpd$ admits a right adjoint 
    \[\mrm{inc}\colon \Prpd^S \adj \Prpd \cocolon \hldef{(-)_S} \]
    Moreover, the counit $\calP_S \to \calP$ is the inclusion of the subproperad that contains exactly those operations whose arity lies in $S$.
\end{lem}
\begin{proof}
    We know that $\Csp_S \subset \Csp$ is a subterminal object in $\Prpd$ and an $\infty$-properad $\calP$ maps to $\Csp_S$ if and only if $\calP \in \Prpd^S$.
    Therefore, we have $\Prpd^S \simeq \SMeq{\Csp_S}$.
    Pullback along the inclusion $\Csp_S \to \Csp$ gives the desired right adjoint.
\end{proof}

\subsubsection{Reduced and projective properads}

Let $\calP$ be an $\infty$-properad.
Then the unique equifibered functor
$\calP \to \Csp$ induces an equifibered morphism of commutative monoids
$\Map_\calP(\unit,\unit) \too \Map_\Csp(\emptyset,\emptyset) \simeq \xF(\ast)$.
By \cref{cor: equifibered-over-free-equals-spaces}, this map is free. 
In fact, it can be identified with $\xF(-)$ applied to the terminal map
$\calP(\emptyset;\emptyset):=\Map_{\Prpd}(\mfr{c}_{\emptyset,\emptyset},\calP) \to \ast$.

\begin{defn}
    Let $\calP$ be an $\infty$-properad.
    We say that $\calP$ is \hldef{reduced} if the space $\calP(\emptyset;\emptyset)$ of nullary operations is contractible,
    or equivalently if $\Map_\calP(\unit, \unit) \simeq \Map_\Csp(\emptyset, \emptyset) \left(\simeq \xF(\ast)\right)$.
    We write $\hldef{\Prpd^\rd} \subseteq \Prpd$
    for the full subcategory of reduced $\infty$-properads. 
\end{defn}

\begin{prop}\label{prop:adjoints-to-rd-and-null}
    The inclusion 
    $\Prpd^\rd \hookrightarrow \Prpd$ admits a left adjoint preserving compact objects
    \[\hldef{(-)^\rd} \colon \Prpd \too \Prpd^\rd.\]
\end{prop}
\begin{proof}
    An $\infty$-properad $\calP$ is reduced if and only if
    $\Map_{\Prpd}(\mfr{c}_{\emptyset,\emptyset},\calP) \simeq \ast$. 
    Therefore, $\Prpd^\rd \hookrightarrow \Prpd$ is the pullback of the right adjoint $\Map_{\Prpd}(\mfr{c}_{\emptyset, \emptyset}, -)\colon \Prpd \to \calS$
    along the right adjoint $* \to \calS$.
    Since $\PrR \subset \Cat$ is closed under pullbacks \cite[Theorem 5.5.3.18]{HTT} we get that $\Prpd^\rd$ is presentable and its inclusion into $\Prpd$ admits a left adjoint.
    Moreover, because $\mfr{c}_{\emptyset,\emptyset}$ is compact by \cref{cor:Prpd-compact-generation}, the inclusion also preserves filtered colimits. 
    The left adjoint preserves compact objects since the right adjoint preserves filtered colimits \cite[Proposition 5.5.7.2.(1)]{HTT}.
\end{proof}

We will see later (\cref{lem:extra-adjoint-reduction}) that the left adjoint
$(-)^\rd \colon \Prpd \to \Prpd^\rd$ 
admits a further left adjoint $(-)^\ext \colon \Prpd^\rd \to \Prpd$ 
which, informally speaking, 
takes a reduced $\infty$-properad and equips it with the universal space of nullary operations generated from non-nullary operations.
We will show that any $\infty$-properad can be recovered by gluing its reduced and nullary pieces, see \cref{lem:cofracture-square}.
In order to establish these facts it will be useful to consider projective $\infty$-properads, obtained as a quotient of $\infty$-properads by their nullary operations.
\begin{defn}
    For $\calP \in \Prpd$ we write $\hldef{\calP_0} \subseteq \calP$ for the full subcategory spanned by $\unit \in \calP$.
    (This is the $S = \{(0,0)\}$ case of $\calP_S$ from \cref{lem:adjoint-PrpdS}.)
    We define the \hldef{projectivization} $\hldef{\Red{\calP}}$ to be the pushout
    \[\begin{tikzcd}
	{\calP_0} & \ast \\
	\calP & {\hldef{\Red{\calP}}}
	\arrow[from=1-1, to=1-2]
	\arrow[from=1-2, to=2-2]
	\arrow[hook, from=1-1, to=2-1]
	\arrow["\lrcorner"{anchor=center, pos=0.125, rotate=180}, draw=none, from=2-2, to=1-1]
	\arrow[from=2-1, to=2-2]
    \end{tikzcd}\]  
    in $\SM$.
\end{defn}

\begin{obs}\label{obs:proj-fiber-sequence}
    Let $\calP$ be an $\infty$-properad such that $\pi_0|\calP|$ is a group under $\otimes$.
    Then by applying $|-|$ to the definition of $\Red{\calP}$ we obtain a cofiber sequence of (group-like) symmetric monoidal $\infty$-groupoids, and thus a fiber sequence of infinite loop spaces
    \[  \Omega^\infty \Sigma^{\infty+1}\calP(\emptyset;\emptyset)_+
    \simeq |\calP_0| \too |\calP| \too |\overline{\calP}|.\]
    The identification of $|\calP_0|$ as a free infinite loop space follows from the adjunction $\frB \dashv \End_{-}(\unit)$ in \cref{lem:con-SM=CMon}.
    We write $\calP_0 = \frB(\xF(\calP(\emptyset;\emptyset)))$ using \cref{lem:eqf-over-frB} and then compute for every infinite loop space $X$, interpreted as a group-like symmetric monoidal $\infty$-category $X \in \SM$, that
    \begin{align*}
        \Map_{\SM}(|\calP_0|, X) 
        &= \Map_{\SM}(\frB(\xF(\calP(\emptyset;\emptyset))), X)
        = \Map_{\CMon}(\xF(\calP(\emptyset;\emptyset)), \Omega X)\\
        & = \Map_{\calS}(\calP(\emptyset;\emptyset), \Omega X)
        = \Map_{\calS_*}(\Sigma \calP(\emptyset;\emptyset)_+, X).
    \end{align*}
    Together with the identification $\Red{\calP} \simeq (h\calP)^\proj$ for $1$-properads $\calP$ (see the paragraph preceding \cref{prop:LCC-as-pullback}) this recovers and thus generalizes \cite[Proposition 3.4]{Jan-tropical}.
\end{obs}

\begin{example}\label{example:projective-cospans}
    We define the \hldef{projective cospan category} $\Csp^{\rm proj}$ as the symmetric monoidal $1$-category 
    with objects finite sets and 
    morphisms isomorphism classes of cospans $[A \to X \leftarrow B]$ such that $A \amalg B \to X$ is surjective.
    Equivalently, morphisms from $A$ to $B$ are equivalence relations on $A \amalg B$.
    The composition of two such cospans $[A \to X \leftarrow B]$ and $[B \to Y \leftarrow C]$ is defined as $[A \to Q \leftarrow C]$ where $Q \subset X \amalg_B Y$ is the image of $A \amalg C$.
    The symmetric monoidal structure is given by disjoint union.
    This category was called the reduced cospan category and denoted ``$\Csp^\rd$'' in \cite[Definition 2.9]{Jan-realisation-fibrations} and \cite{Jan-tropical},
    though this notation would be misleading here as $\Csp$ is already a reduced $\infty$-properad and hence $\Csp^\rd=\Csp$.
    This fits into the commutative square of symmetric monoidal $\infty$-categories
    \[\begin{tikzcd}
	{\Csp_0} & \ast \\
	\Csp & {\Csp^{\rm proj}}.
	\arrow[from=1-1, to=1-2]
	\arrow[from=1-2, to=2-2]
	\arrow[hook, from=1-1, to=2-1]
	\arrow[from=2-1, to=2-2]
    \end{tikzcd}\]  
    We will show in \cref{lem:nerve-splitting} that this square is a pushout in $\SM$ and hence that $\Csp^{\rm proj} \simeq \Red{\Csp}$.
\end{example}

The following lemma will be crucial as it (implicitly) describes the nerve of $\Red{\calP}$.
For this, let $\xN_\bullet^\xint(\calC)\colon \simp^{\op,\xint} \to \CMon$ denote the restriction of the nerve of $\calC \in \SM$ to the inert morphisms.
\begin{lem}\label{lem:nerve-splitting}\label{cor:Csp-proj}
    For $\calP \in \Prpd$ there is a natural equivalence 
    in $\Fun(\simp^{\op,\xint}, \CMon)_{\xN_\bullet^\xint(\calP_0)/}$
    \[ \xN^\xint_\bullet(\calP) \simeq \xN^\xint_\bullet(\calP_0) \oplus \xN^\xint_\bullet(\Red{\calP}).\]
    Moreover, $\Red{\Csp} \simeq \Csp^\proj$ via the map induced by the square in \cref{example:projective-cospans}.
\end{lem}
\begin{proof}
    Let $K_n(\Csp) \subseteq \xN^\xint_n(\Csp)$ 
    be the submonoid consisting of diagrams 
    $A \colon \Tw[n] \to \Fin$ 
    such that $A_{i,i} \amalg A_{i+1,i+1} \to A_{i,i+1}$ is surjective for all $0 \le i < n$.
    This condition is preserved by inert maps so we get a functor 
    $K_\bullet(\Csp) \colon \simp^{\op,\xint} \to \CMon$. 
    
    We claim that $K_\bullet(\Csp)$ satisfies the Segal condition.
    Indeed, $K_\bullet(\Csp)$ is a subfunctor of $\xN^\xint_\bullet(\Csp)$,
    for which the Segal map is an equivalence, and thus
    \[K_n(\Csp) \too K_1(\Csp) \times_{K_0(\Csp)} \cdots \times_{K_0(\Csp)} K_1(\Csp)\]
    is a monomorphism.
    To show surjectivity on $\pi_0$ it suffices to observe that if
    $A \colon \Tw[n]^\el \to \Fin$ 
    is such that
    $A_{i,i} \amalg A_{i+1,i+1} \to A_{i,i+1}$ is surjective for all $0 \le i < n$
    then its left Kan extension
    $\iota_!A \colon \Tw[n] \to \Fin$ lies in $K_n(\Csp)$ by definition.
    Note that the composite $K_n(\Csp) \to \xN_n(\Csp) \to \xN_n(\Csp^\proj)$ is an equivalence for $n=1$ and thus for all $n$.
    
    For an $\infty$-properad $\calP$ we write $K_\bullet(\calP) \colon \simp^{\op,\xint} \to \CMon$
    for the pullback
    \[\begin{tikzcd}
	{K_\bullet(\calP)} \ar[dr, phantom, very near start, "\lrcorner"] & {\xN^\xint_\bullet(\calP)} \\
	{K_\bullet(\Csp)} & {\xN_\bullet^\xint(\Csp)}.
	\arrow[hook, from=2-1, to=2-2]
	\arrow[from=1-1, to=2-1]
	\arrow[from=1-2, to=2-2]
	\arrow[hook, from=1-1, to=1-2]
    \end{tikzcd}\]
    The inclusion $K_\bullet(\calP) \hookrightarrow \xN_\bullet^\xint(\calP)$ uniquely extends to the map
    \[ \xN_\bullet^\xint(\calP_0) \oplus K_\bullet(\calP) \to \xN^\xint_\bullet(\calP)\]
    in $\Fun(\simp^{\op,\xint}, \CMon)_{\xN_\bullet^\xint(\calP_0)/}$.
    We claim that this map is an equivalence. 
    Because both sides satisfy the Segal condition, it will suffice to check $\bullet = 1$.
    The map in question is the base-change of 
    \( \xN_1^\xint(\Csp_0) \oplus K_1(\Csp) \to \xN^\xint_1(\Csp)\)
    along $\xN^\xint_1\calP \to \xN^\xint_1\Csp$.
    The former is an equivalence because it corresponds to the disjoint decomposition of the space of elementary cospans $(A_{00} \to \ast \leftarrow A_{11})$ 
    into those where $A_{00} \amalg A_{11}$ is empty or non-empty, respectively.
    
    To complete the proof we show that the composite map $K_\bullet(\calP) \to \xN^\xint_\bullet \calP \to \xN^\xint_\bullet(\Red{\calP})$ is an equivalence.
    We will show that the pushout square defining $\Red{\calP}$ is a level-wise colimit in the sense of \cref{obs:level-wise-colimits-in-SM}.
    To do so, we will show that $M_\bullet \in \Fun(\Delta^\op,\CMon)$ defined as the pushout
    \[\begin{tikzcd}
	{\xN_\bullet(\calP_0)} & \ast \\
	{\xN_\bullet(\calP)} & {M_\bullet}
	\arrow[from=1-1, to=1-2]
	\arrow[hook, from=1-2, to=2-2]
	\arrow[hook, from=1-1, to=2-1]
	\arrow["\lrcorner"{anchor=center, pos=0.125, rotate=180}, draw=none, from=2-2, to=1-1]
	\arrow[from=2-1, to=2-2]
    \end{tikzcd}\]
    is a complete Segal space.
    Restricting to $\simp^{\op,\xint}$ we have that $\xN_\bullet^\xint(\calP) \simeq \xN_\bullet^\xint(\calP_0) \oplus K_\bullet(\calP)$ as shown above.
    In a semi-additive category the cofiber of a summand inclusion is equivalent to the complementary summand, hence $(M_\bullet)_{|\simp^{\op,\xint}} \simeq K_\bullet(\calP)$.
    In particular $M_\bullet$ is Segal.
    Suppose, for a moment, that $M_\bullet$ is also complete.
    Then \cref{obs:level-wise-colimits-in-SM} implies that $M_\bullet \too N_\bullet\Red{\calP}$ is an equivalence, and therefore $K_n(\calP) \simeq M_n \to \xN_n\Red{\calP}$ is an equivalence for all $n$, as claimed.

    To prove completeness, first consider the special case of $\calP = \Csp$.
    The map $M_\bullet \to \xN_\bullet(\Csp^\proj)$ is an equivalence because its restriction to $\simp^{\op,\xint}$ is $K_\bullet(\Csp) \to \xN(\Csp^\proj)$, which we already observed to be an equivalence.
    Hence, $M_\bullet$ must be complete, as it is equivalent to the Rezk nerve of a category.
    It follows that $M_\bullet \simeq \xN_\bullet\Red{\Csp}$ and $\Red{\Csp} \simeq \Csp^\proj$.

    To see that $M_\bullet$ is complete in general,
    we apply \cref{lem:checking-completeness} to the map 
    \[ 
        M_\bullet \to \xN_\bullet(\Red{\calP}) \to \xN_\bullet(\Red{\Csp}) = \xN_\bullet(\Csp^\proj).
    \]
    To do so, we will need to show that $\xN_\bullet(\Csp^\simeq) \times_{\xN_\bullet(\Csp^\proj)} M_\bullet$ is complete. 
    (Here we used $(\Csp^\proj)^\simeq = \Csp^\simeq$.)
    Let $X_\bullet \coloneq \xN_\bullet(\Csp^\simeq \times_{\Csp} \calP) \simeq \xN_\bullet(\Csp^\simeq) \times_{\xN_\bullet(\Csp)} \xN_\bullet(\calP)$,
    which is a complete Segal space as it is the nerve of an $\infty$-category.
    Pullback pasting applied to the diagram
    \[\begin{tikzcd}
	{X^\xint_\bullet} & {K_\bullet(\calP)} & {\xN^\xint_\bullet(\calP)} & {M^\xint_\bullet} \\
	{\xN^\xint_\bullet(\Csp^\simeq)} & {K_\bullet(\Csp)} & {\xN^\xint_\bullet(\Csp)} & {\xN^\xint_\bullet(\Csp^{\rm proj})}
	\arrow[from=1-1, to=2-1]
	\arrow[from=1-1, to=1-2]
	\arrow[from=1-2, to=1-3]
	\arrow[from=1-3, to=1-4]
	\arrow[from=1-4, to=2-4]
	\arrow[from=2-3, to=2-4]
	\arrow[from=1-3, to=2-3]
	\arrow[from=1-2, to=2-2]
	\arrow[from=2-1, to=2-2]
	\arrow[from=2-2, to=2-3]
	\arrow["\lrcorner"{anchor=center, pos=0.125}, draw=none, from=1-2, to=2-3]
	\arrow["\lrcorner"{anchor=center, pos=0.125}, draw=none, from=1-1, to=2-2]
	\arrow["\simeq"{description}, curve={height=-12pt}, from=1-2, to=1-4]
	\arrow["\simeq"{description}, curve={height=12pt}, from=2-2, to=2-4]
\end{tikzcd}\]
    in $\Fun(\simp^{\op, \xint}, \calS)$ implies that the canonical map $X_\bullet \to \xN_\bullet(\Csp^\simeq) \times_{\xN_\bullet(\Csp^{\rm proj})} M_\bullet$ in $\Fun(\Dop,\calS)$ is an equivalence and shows the desired completeness.
\end{proof}

\begin{cor}\label{cor:mapping-space-in-Red}
    For any $\infty$-properad $\calP$ and all $x, y \in \calP$ we have a canonical equivalence
    \[
        \Map_\calP(x, y) \simeq \xF(\calP(\emptyset; \emptyset)) \times \Map_{\Red{\calP}}(x,y).
    \]
\end{cor}
\begin{proof}
    \cref{lem:nerve-splitting} gives $\xN_1\calP \simeq \xN_1(\calP_0) \oplus \xN_1(\Red{\calP})$.
    Taking fibers of the map to $\xN_0\calP \times \xN_0\calP$ yields the desired equivalence.
\end{proof}

\begin{cor}\label{cor:contra-eqf-reduced}
    For every $\infty$-properad $\calP$, in the following square in $\SM$ 
    \[\begin{tikzcd}
	{\calP} & \Red{\calP} \\
	{\Csp} & \Red{\Csp}
	\arrow[from=1-1, to=1-2]
	\arrow[from=1-2, to=2-2]
	\arrow[from=1-1, to=2-1]
	\arrow[from=2-1, to=2-2, "\rho"]
    \end{tikzcd}\]
    the horizontal functors are contrafibered and the vertical functors are equifibered.
\end{cor}
\begin{proof}
    The functor $\calP \to \Red{\calP}$ is contrafibered as it is the cobase change of the contrafibered functor $\calP_0 \to \ast$.
    Setting $\calP = \Csp$ we also get that the bottom arrow is contrafibered.
    The functor $\Red{\calP} \to \Red{\Csp}$ is equifibered since 
    $\xN^\xint_\bullet(\Red{\calP}) \to \xN^\xint_\bullet(\Red{\Csp})$ 
    can be identified with 
    $K_\bullet(\calP) \to K_\bullet(\Csp)$
    which is the base change of the equifibered map $\xN^\xint_\bullet(\calP) \to \xN^\xint_\bullet(\Csp)$.
\end{proof}

\begin{defn}\label{defn:projective-prpd}
    A \hldef{projective $\infty$-properad} is a symmetric monoidal $\infty$-category $\calQ$ equipped with an equifibered symmetric monoidal functor $\calQ \to \Red{\Csp}$.
    We let $\hldef{\projPrpd} \subset \SMeq{\Red{\Csp}}$ denote the full subcategory of projective $\infty$-properads.
\end{defn}

\begin{example}\label{ex:Span-proj}
    It was pointed out in \cref{rem:Span-not-prpd} that the $(2,1)$-category $\Span$ of spans of finite sets with its symmetric monoidal structure given by disjoint union is not an $\infty$-properad.
    However, it is a projective $\infty$-properad, as we shall argue now.
    Consider the functor $\pi\colon \Span \to \Csp^{\rm proj} = \Red{\Csp}$ defined on objects and morphisms by
    \[
        A \mapsto A
        \qquad \text{and} \qquad
        (A \leftarrow X \to B) \mapsto (A \to A \coproductlimits_X B \leftarrow B).
    \]
    Since $\Csp^{\rm proj}$ is a $1$-category we do not need to provide higher coherence, but only need to check the functoriality.
    In \cite[\S5.2]{Grandis-Pare} the authors argue this defines a lax $2$-functor from $\Span$ to a bicategory of cospans of sets.
    The lax functoriality is given by the natural comparison map
    \begin{align*}
        &\pi(A \leftarrow X \times_B Y \to C) = 
        (A \to A \coproductlimits_{X \times_B Y} C \leftarrow C ) \cong (A \to A\coproductlimits_X (X \coproductlimits_{X \times_B Y} Y) \coproductlimits_Y C \leftarrow C) \\
        &\too \pi(A \leftarrow X \to B) \coproductlimits_B \pi(B \leftarrow Y \to C)
        = (A \to (A\coproductlimits_X B) \coproductlimits_B (B \coproductlimits_Y C) \leftarrow C)
        \cong (A \to A\coproductlimits_X B \coproductlimits_Y C \leftarrow C).
    \end{align*}
    At the apex this comparison is given by the map
    \[
        A\coproductlimits_X (X \coproductlimits_{X \times_B Y} Y) \coproductlimits_Y C 
        \too
        A \coproductlimits_X B \coproductlimits_Y C,
    \]
    which is an injection whose image is precisely the image of $A \amalg C$.%
    \footnote{
        This map is not always a bijection, so $\pi$ does not lift to a functor $\Span \to \Csp$.
    }
    (Indeed, the map $X \amalg_{X \times_B Y} Y \to B$ is an injection whose image agrees with that of $X \amalg Y \to B$.)
    Recalling \cref{example:projective-cospans} we see that this image is exactly the definition of the composite $\pi(X) \circ \pi(Y)$ in $\Csp^{\rm proj}$ as a sub-cospan of the composite in $\Csp$.
    Therefore, we have indeed defined a functor $\pi\colon \Span \to \Csp^{\rm proj}$ and this functor is canonically symmetric monoidal with respect to disjoint union.
    To check that this functor is equifibered it suffices to consider $\xN_1(\pi)$.
    Writing $\Lambda_0^2$ for the diagram $(\bullet \leftarrow \bullet \to \bullet)$ we have
    \[
        \Fun(\Lambda_0^2, \Fin)^\simeq 
        \xrightarrow{\xN_1(\pi)} 
        \Fun((\Lambda_0^2)^\op, \Fin)^\simeq 
        \xrightarrow{\ev_1} 
        \Fin^\simeq.
    \]
    By \cref{lem:coproduct-disjunctive-colim} both the composite $\ev_1 \circ \xN_1(\pi)$ and $\ev_1$ are equifibered (as they are $\colim_{\Lambda_0^2}$ and $\colim_{(\Lambda_0^2)^\op}$) and thus $\xN_1(\pi)$ is equifibered by cancellation.
\end{example}

Note that projective $\infty$-properads are not $\infty$-properads.
However, the following proposition shows that the $\infty$-category of projective $\infty$-properads is equivalent to the $\infty$-category of reduced $\infty$-properads.

\begin{prop}\label{prop:projectivization}
    Let $\rho\colon \Csp \to \Red{\Csp}$ denote the quotient map.
    Then we have an adjunction
    \[
        \overline{(-)} \colon  \Prpd \simeq\SMeq{\Csp} \adj \SMeq{\Red{\Csp}} = \projPrpd \colon \rho^\ast.
    \]
    The right adjoint $\rho^\ast$ is fully faithful, and its essential image consists of reduced $\infty$-properads. In particular, we obtain an equivalence of $\infty$-categories  $\Prpd^{\rd} \simeq \projPrpd$.
\end{prop}
\begin{proof}
    The left adjoint of the composite functor
    \[   
    (\SM)_{/\Csp} \xleftarrow{\rho^\ast} 
    (\SM)_{/\Red{\Csp}} \hookleftarrow
    \SMeq{\Red{\Csp}} 
    \]
    is given by
    \[ (\SM)_{/\Csp} \xrightarrow{\rho_!} (\SM)_{/\Red{\Csp}} \xrightarrow{\bbL^\eqf} \SMeq{\Red{\Csp}}. \]
    Since $\rho^\ast$ preserves equifibered functors this restricts to an adjunction
    \[\bbL^\eqf \rho_! \colon \SMeq{\Csp} \adj \SMeq{\Red{\Csp}} \colon \rho^\ast.\]
    By \cref{cor:contra-eqf-reduced} we have 
    $\bbL^\eqf \rho_! (\calP) \simeq \Red{\calP}$
    for $\infty$-properads $\calP$, establishing the desired adjunction.
    
    To prove that $\rho^\ast$ is fully faithful it suffices to show that for all $\calQ \in \SMeq{\Red{\Csp}}$ the top horizontal functor in the pullback square
    \[\begin{tikzcd}
	{\rho^\ast \calQ} & \calQ \\
	\Csp & {\overline{\Csp}}
	\arrow[from=2-1, to=2-2]
	\arrow[from=1-1, to=2-1]
	\arrow[from=1-2, to=2-2]
	\arrow[from=1-1, to=1-2]
    \end{tikzcd}\]
    is contrafibered, as then $\calQ$ is the contrafibered-equifibered factorization of $\rho^*\calQ \to \Csp \to \Red{\Csp}$
    and hence $\bbL^\eqf \rho_! \rho^* \calQ \simeq \calQ$.
    By \cref{cor:detect-ctf}, it suffices to show that $\xN_n(\rho^\ast \calQ) \to \xN_n(\calQ)$ is contrafibered for all $n$, so fix some $n$.
    The bottom map in the square $\xN_n \Csp \to \xN_n \Red{\Csp}$ is the projection to a summand in
    $\xN_n \Csp \simeq \xN_n \Csp_0 \oplus \xN_n \Red{\Csp}$.
    Therefore, the top map, being the pullback, is
    $\xN_n \rho^*\calQ \simeq \xN_n \Csp_0 \oplus \xN_n \overline{\calQ} \to \xN_n \overline{\calQ}$.
    This is the projection to a factor, or in other words the coproduct of an equivalence and a map to $0$, and hence contrafibered.
    This shows that $\rho^*\calQ \to \calQ$ is contrafibered.
    
    Finally, to characterize the image of $\rho^*$ consider those $\infty$-properads $\calP$ for which the square 
    \[\begin{tikzcd}
	{\calP} & \overline{\calP} \\
	\Csp & {\overline{\Csp}}
	\arrow[from=2-1, to=2-2]
	\arrow[from=1-1, to=2-1]
	\arrow[from=1-2, to=2-2]
	\arrow[from=1-1, to=1-2]
    \end{tikzcd}\]
    is cartesian. 
    Using that $\xN_n \calP \simeq \xN_n \calP_0 \oplus \xN_n \Red{\calP}$ holds level-wise we see that this is exactly the case if and only if $\calP_0 \to \Csp_0$ is an equivalence,
    i.e.~if and only if $\calP$ is reduced.
\end{proof}

This description of the projectivization $\Red{\calP}$ also gives us information about the reduced $\calP^\rd$:
\begin{cor}\label{cor:reduction-doesnt-change-arities}\label{cor:reduction-0arity-jointly-conservative}
    The unit map $\calP \to \calP^\rd$ of the adjunction $(-)^\rd \dashv \mrm{include}$,
    induces an equivalence on colours and on operations of all arities except $(0,0)$.
    Moreover, the following functor is conservative:
    \[
    (-)^\rd \times (-)_0 \colon \Prpd \to \Prpd^\rd \times \Prpd^{(0,0)} \simeq \Prpd^\rd \times \calS.
    \]
\end{cor}
\begin{proof}
    By \cref{prop:projectivization} we can equivalently consider the unit $\calP \to  \rho^* \Red{\calP} \simeq \calP^\rd$.
    The claim now follows from \cref{lem:nerve-splitting} which gives, for any fixed $n$, compatible splittings
    \[
    \xN_n \calP \simeq \xN_n \calP_0 \oplus \xN_n \Red{\calP}, \qquad \xN_n \calP^\rd \simeq \xN_n \Csp_0 \oplus \xN_n \Red{\calP}.
    \]
    From this we can see that $\calP \to \calP^\rd$ is an equivalence in all arities but $(0,0)$.
    The corollas $\frc_{k,l}$ generate $\Prpd$ (\cref{cor:Prpd-compact-generation}), so it follows that $(-)^\rd$ and $(-)_0$ are indeed jointly conservative.
\end{proof}

\begin{lem}\label{lem:extra-adjoint-reduction}
    There is a triple-adjunction with two fully faithful functors:
    \[
    \begin{tikzcd}
	\Prpd && {\Prpd^\rd}
	\arrow["{(-)^\rd}"{description}, from=1-1, to=1-3]
	\arrow["{(-)^{\rm ext}}"', shift right = 3, hook', from=1-3, to=1-1]
	\arrow["{\inc}", shift left = 3, hook, from=1-3, to=1-1]
    \end{tikzcd}
    \]
    We refer to the left most adjoint \hldef{$(-)^{\rm ext}$} as the \hldef{extension} functor.
\end{lem}
\begin{proof}
    The bottom adjunction comes from \cref{prop:adjoints-to-rd-and-null} where we also saw that $\Prpd^\rd$ is presentable (as is $\Prpd$ by \cref{rem:Prpd-presentable}).
    To construct the additional left adjoint $(-)^\ext$,
    it suffices by the adjoint functor theorem \cite[Corollary 5.5.2.9]{HTT} to show that the left adjoint $(-)^\rd$ preserves limits.
    Recall from \cref{obs:N1el-conservative} and \cref{prop:free-properad} that $\xN_1^\el \colon \Prpd \to \calS_{/\xF(*\amalg *)}$ is a conservative right adjoint.
    \cref{cor:reduction-doesnt-change-arities}
    gives a commutative square
    \[\begin{tikzcd}
    	\Prpd & {\calS_{/\xF(* \amalg *)}} \\
    	{\Prpd^\rd} & {\calS_{/\xF(* \amalg *)_{>0}}}
    	\arrow["{\xN_1^\el}", from=2-1, to=2-2]
    	\arrow["{\xN_1^\el}", from=1-1, to=1-2]
    	\arrow["{(-)^\rd}"', from=1-1, to=2-1]
    	\arrow["{\mathrm{restr}}", from=1-2, to=2-2]
    \end{tikzcd}\]
    where all functors, except possibly $(-)^\rd$, preserve limits.
    Since the bottom horizontal functor is conservative it follows that $(-)^\rd$ preserves limits.

    Finally, to see that $(-)^\ext$ is fully faithful, recall that for any chain of adjunctions $F_1 \dashv F_2 \dashv F_3$ the functor $F_1$ is fully faithful if and only if $F_3$ is.
    To see this, recall that since $F_i$ participates in an adjunction, it is fully faithful if and only if the relevant (co)unit is an equivalence if and only if $h(F_i)$ is fully faithful.
    Thus, this claim can be checked on homotopy categories, where it can be found in \cite[Lemma 1.3]{Dyckhoff1987}.
\end{proof}

\subsubsection{\texorpdfstring{$\Prpd$}{Infinity-properads} as a semi-recollement}
We can use the above adjunctions to write the $\infty$-category $\Prpd$ in terms of $\Prpd^\rd \simeq \projPrpd$ and $\Prpd^{(0,0)}$.
This almost fits in the general setup of (unstable) recollements discussed in \cite[Appendix A.8]{HA}, but we need to make a mild generalization because the colocalization $(-)_0\colon \Prpd \to \Prpd^{(0,0)}$ is not right-exact, but only semi-right-exact.

\begin{defn}
    Let $\calC$ be an $\infty$-category that has finite limits.
    A \hldef{semi-recollement} consists of two full subcategories $\calC_0, \calC_1 \subset \calC$ such that
    \begin{enumerate}[(1)]
        \item the full inclusions $\calC_i \hookrightarrow \calC$ admit a left adjoints $L_i\colon \calC \to \calC_i$,
        \item the functor $L_1$ is left-exact, i.e.~it preserves finite limits,
        \item the localization $L_0$ is semi-left-exact,%
        \footnote{
            This terminology is taken from \cite[Theorem 4.3.(i)]{CHK85}.
        }
        i.e.~for every $x \in \calC$, $y \in \calC_0$ and $f\colon y \to L_0 x$ the functor $L_0$ preserves the pullback square
        \[\begin{tikzcd}
            y \times_{L_0 x} x \ar[r] \ar[d] & x \ar[d] \\
            y \ar[r, "f"] & L_0 x,
        \end{tikzcd}\]
        (or, equivalently, for all $x, y, f$ as above $L_0(y \times_{L_0 x} x) \to L_0(y)$ is an equivalence,)
        \item the functor $L_1$ sends every object of $\calC_0$ to the terminal object of $\calC$,
        \item if $f$ is a morphism in $\calC$ such that $L_0(f)$ and $L_1(f)$ are equivalences, then $f$ is an equivalence.
    \end{enumerate}
\end{defn}

\begin{thm}[Lurie]\label{thm:Lurie-recollement}
    For any semi-recollement $\calC_0, \calC_1 \subset \calC$ the natural transformation $\alpha\colon L_0 \to L_0 L_1$ gives rise to a cartesian square
    \[\begin{tikzcd}
        \calC \ar[r, "\alpha"] \ar[d, "L_1"'] \ar[dr, very near start, phantom, "\lrcorner"] & \Ar(\calC_0) \ar[d, "\ev_1"] \\
        \calC_1 \ar[r, "L_0"] & \calC_0.
    \end{tikzcd}\]
\end{thm}
\begin{proof}
    The proof given in \cite[Proposition A.8.11]{HA} applies verbatim.
    While Lurie assumes that $L_0$ is left-exact, i.e.~preserves all finite limits, this is only used twice in his proof and in both cases it is applied to cospans $C_0 \to C_{01} \leftarrow C_1$ that satisfy the conditions denoted (i) and (ii) in the proof.
    These conditions ensure that $C_0 \in \calC_0$, $C_1 \in \calC_1$ and that $C_1 \to C_{01}$ exhibits $C_{01}$ as the $\calC_0$-localization of $C_1$, i.e.~this morphism is $C_1 \to L_0C_1$.
    Our condition (3) ensures that $L_0$ still preserves such pullbacks squares, so the proof still works.
    The conclusion of \cite[Proposition A.8.11]{HA} is an equivalence
    \[
        \calC \simeq \Map_{\Delta^1}(\Delta^1, \calM)
    \]
    where $\pi\colon \calM \to \Delta^1$ is the unstraightening of the functor $\Delta^1 \to \Cat$ given by $(L_0\colon \calC_1 \to \calC_0)$.
    Using the factorization system $(\pi\text{-cocart}, \pi^{-1}((\Delta^1)^\simeq))$ (\cite[Proposition 2.1.2.5]{HA}) on $\calM$ we can rewrite this section space (using \cite[Proposition 5.2.8.17]{HTT}) as a pullback 
    \[
        \Map_{\Delta^1}(\Delta^1, \calM)
        \simeq \Map_{\Delta^1}^{\rm cocart}(\Delta^1, \calM) \times_{\calM_1} \Ar(\calM_1)
        \simeq \calM_0 \times_{\calM_1} \Ar(\calM_1).
    \]
    Using the identifications $\calM_i \simeq \calC_i$ we obtain the desired pullback square.
\end{proof}

In the case at hand we would now like to check that the two full inclusions
\[
    \Prpd^{(0,0)} \xhookrightarrow{\quad} \Prpd \xhookleftarrow{(-)^\ext} \Prpd^\rd
\]
form the opposite of a semi-recollement in the above sense.
We have two triple adjunctions
\[\begin{tikzcd}
	\Prpd^{(0,0)} && \Prpd && {\Prpd^\rd}
	\arrow[dashed, shift right = 3, from=1-3, to=1-1]
	\arrow["\inc"{description}, hook, from=1-1, to=1-3]
	\arrow["(-)_0", shift left = 3, from=1-3, to=1-1]
	\arrow["{(-)^\rd}"{description}, from=1-3, to=1-5]
	\arrow["{(-)^{\rm ext}}"', shift right = 3, hook', from=1-5, to=1-3]
	\arrow["{\text{include}}", dashed, shift left = 3, hook, from=1-5, to=1-3]
\end{tikzcd}\]
where the dashed functor on the left exists by the adjoint functor theorem, but we will not be needing it.
This shows that axiom (1) and (2) are satisfied, indeed $L_1=(-)^\rd$ preserves all colimits.
Ignoring the two dashed functors, we have two colocalizations.
Axiom (4) and (5) follow from \cref{cor:reduction-0arity-jointly-conservative}.
It remains to check axiom (3):

\begin{lem}\label{lem:0-semi-exact}
    For every $\calP, \calQ \in \Prpd$ and $f\colon \calP_0 \to \calQ$ the pushout square in $\Prpd$
    \[\begin{tikzcd}
        \calP_0 \ar[r] \ar[d] \ar[dr, very near end, phantom, "\ulcorner"] & \calP \ar[d] \\ 
        \calQ \ar[r] & \calQ' 
    \end{tikzcd}\]
    is a levelwise pushout in the sense of \cref{obs:level-wise-colimits-in-SM} and $\calQ_0 \to \calQ'_0$ is an equivalence.
\end{lem}
\begin{proof}
    We need to show that the levelwise pushout
    $M_\bullet \coloneq \xN_\bullet\calQ \cup_{\xN_\bullet\calP_0} \xN_\bullet\calP$
    is a complete Segal space.
    \cref{lem:nerve-splitting} gives us a decomposition
    $\xN_\bullet^\xint\calP = \xN_\bullet^\xint\calP_0 \oplus \xN_\bullet^\xint\Red{\calP}$
    as functors $\simp^{\op,\xint} \to \CMon$,
    so restricted to $\simp^{\op,\xint}$ the levelwise pushout is
    \[
        (M_\bullet)_{|\simp^{\op,\xint}}
        \simeq \xN_\bullet^\xint\calQ \cup_{\xN_\bullet^\xint\calP_0} \left( \xN_\bullet^\xint\calP_0 \oplus \xN_\bullet^\xint\Red{\calP}\right)
        \simeq \xN_\bullet^\xint\calQ \oplus \xN_\bullet^\xint\Red{\calP}
    \]
    which indeed satisfies the Segal condition.
    We will now check completeness of $M_\bullet$.
    Consider the map of simplicial commutative monoids
    \[
        g\colon M_\bullet \too Y_\bullet \coloneq \xN_\bullet(\Red{\calQ} \times \Red{\calP}) 
    \]
    that comes from the functor $\calQ' \too \Red{\calQ} \times \Red{\calP}$ that we can construct using the universal property of the pushout.
    On the $n$th simplicial level this map is the projection 
    \[
        M_n \simeq \xN_n(\calQ_0) \oplus \xN_n\Red{\calQ} \oplus \xN_n\Red{\calP} \simeq \xF(\xN_n^\el(\calQ_0)) \oplus Y_n
        \too Y_n
    \]
    A point in $M_1$ can only be an equivalence if its $\xN_1(\calQ_0)$-component is trivial, i.e.~if it lies in the disjoint summand $\{0\} \times Y_n$ of $M_n \simeq \xF(\xN_n^\el \calQ_0) \times Y_n$.
    (This is because the map $M_\bullet \to \xN_\bullet \calQ' \to \xN_\bullet \Csp$ has to send it to an equivalence.)
    This implies that $M_1^\eq \to Y_1$ is a monomorphism and thus so is $M_1^\eq \to Y_1^\eq$.
    On the other hand we have a commutative square
    \[
        \begin{tikzcd}
            {M_0} \dar["{s_0}"] \rar["{g_0}", "\simeq"'] & {Y_0} \dar["s_0", "\simeq"'] \\
            {M_1^\eq} \rar["{g_1}", hook] & {Y_1^\eq},
        \end{tikzcd}
    \]
    where the right vertical map is an equivalence since $Y_\bullet$ is complete.
    This implies that the left vertical map is an equivalence and thus $M_\bullet$ is complete.
%
\end{proof}

We may now apply the opposite of \cref{thm:Lurie-recollement} to obtain a description of $\Prpd$ as a pullback.
\begin{prop}\label{cor:Prpd-as-pullback}
    There is a cartesian square of $\infty$-categories
    \[\begin{tikzcd}
	\Prpd \ar[dr, phantom, very near start, "\lrcorner"] & {\Ar(\Prpd^\nul)} \\
	{\Prpd^\rd} & \Prpd^\nul.
	\arrow["{(-)^\ext_0}"',from=2-1, to=2-2]
	\arrow[from=1-1, to=1-2]
	\arrow["{(-)^\rd}"', from=1-1, to=2-1]
	\arrow["{\mathrm{ev}_0}", from=1-2, to=2-2]
    \end{tikzcd}\]
\end{prop}

We also record the following intermediate step in the proof of \cref{thm:Lurie-recollement}, which is tells us how to reconstruct an $\infty$-properad from $\calP^\rd$ and $\calP_0$ together with a gluing map $((\calP^\rd)^\ext)_0 \to \calP_0$.
\begin{cor}\label{lem:cofracture-square}
    Let $\calP$ be an $\infty$-properad. Then the following natural square is a pushout in $\Prpd$
    \[\begin{tikzcd}
	{((\calP^\rd)^{\ext})_0} \ar[dr, phantom, very near end, "\ulcorner"] & {\calP_0} \\
	{(\calP^\rd)^\ext} & \calP.
	\arrow[from=1-1, to=1-2]
	\arrow[hook, from=1-2, to=2-2]
	\arrow[from=2-1, to=2-2]
	\arrow[hook, from=1-1, to=2-1]
    \end{tikzcd}\] 
\end{cor}
\begin{proof}
    Applying $(-)_0$ or $(-)^\rd$ to this square yields a pushout square because either the vertical or the horizontal maps become equivalence.
    These two functors are jointly conservative and preserve pushouts of this shape by \cref{lem:0-semi-exact}, so the square must have already been a pushout.
\end{proof}

\subsection{Detecting extended \texorpdfstring{$\infty$-properads}{infinity-properads}}
\label{sect-3.4}

In the previous subsection we used the adjoint functor theorem to show that there is a fully faithful left adjoint $(-)^\ext\colon \Prpd^\rd \hookrightarrow \Prpd$ and this functor played an important role in the description of $\Prpd$ as a pullback in \cref{cor:Prpd-as-pullback}.
However, the description of $(-)^\ext$ that we have so far is not very explicit; for example it seems difficult to check whether a given $\infty$-properad $\calP$ is in the essential image of $(-)^\ext$.
We will now use \cref{thm:segal-envelope} and a more in-depth study of the level-graph category $\L$ and its variants to give a description of the triple adjunction from \cref{lem:extra-adjoint-reduction} in terms of algebraic patterns, see \cref{prop:ext-rd-pattern}.
This allows us to give a formula for the $(0,0)$-ary operations of $\calP^\ext$ and thus a criterion for when an $\infty$-properad $\calQ \in \SMeq{\Csp}$ is extended in \cref{prop:detect-extended}.
This will make use of the factorization categories $\calF(\calP)$ that already played a crucial role in \cite[Theorem F]{Jan-tropical}.
The new description allows us to check that the bordism $\infty$-properad $\Bord_d^\theta$ is extended, and we will also use it in the next section to compare our $1$-properads to more classical definitions.

\subsubsection{Extended \texorpdfstring{$\infty$}{infinity}-properads via factorization categories}
We begin by defining the factorization category that can be used to compute $\calP^\ext$.

\begin{defn}\label{fact-category-defn}
    For an $\infty$-properad $\calP$ we define the \hldef{factorization category} of $\calP$ to be the full subcategory
    \[
        \hldef{\calF(\calP)} \subset \calP_{\unit//\unit} = \calP_{\unit/} \times_{\calP} \calP_{/\unit}
    \]
    on those factorization $\unit \xto{f} x \xto{g} \unit$ where $x\not\cong \unit$ and $g \circ f$ is an indecomposable element of $\xN_1\calP$.
    Composing the two morphisms defines a map
    \[
    \hldef{\mathrm{comp}_\calP} \colon |\calF(\calP)|  
    \longrightarrow \calP(\emptyset;\emptyset)  = \Map_{\calP}(\unit, \unit)^\el .
    \]
\end{defn}

\begin{obs}\label{obs:nerve-calF}
    In order to describe the nerve of the factorization category, recall that if $\calC$ is an $\infty$-category and $x \in \calC$, then the nerve of the slice $\calC_{/x}$ is
    \[
        \xN_\bullet(\calC_{/x}) = (\xN_{\bullet+1}\calC) \times_{\xN_0 \calC} \{x\}
    \]
    were $\xN_{\bullet+1}\calC$ denotes the décalage of the simplicial space $\xN_\bullet \calC$ and the map $\xN_{\bullet+1}\calC \to \xN_0 \calC$ is the canonical augmentation of the décalage, which records the last vertex.
    Hence, for an $\infty$-properad $\calP$ we have
    \[
        \xN_\bullet(\calP_{\unit//\unit}) 
        \simeq \{\unit\} \times_{\xN_0\calP} \xN_{1+\bullet+1}\calP \times_{\xN_0\calP} \{\unit\}
        \hookrightarrow \xN_{1+\bullet+1}\calP
    \]
    and this map to the double-décalage is a levelwise monomorphism because
    $\{\unit\} \hookrightarrow \xN_0\calP$ is a monomorphism.
    (Indeed, we know that $\xN_0\calP$ is a free commutative monoid and $\{0\} \hookrightarrow \xF(X)$ is always a monomorphism.)
    Therefore, the nerve of the factorization category can be described as the subspace of the double-décalage of $\xN_\bullet\calP$: 
    \[
        \xN_n\calF(\calP) = \big\{ \big(x_{-\infty} \xto{f_{-\infty}} x_0 \xto{f_1} \dots \xto{f_n} x_n \xto{f_\infty} x_{\infty}\big) \in (\xN_{1+n+1}\calP)^\el \;\big|\; x_{-\infty} \cong \unit \cong x_{\infty} \text{ and } \forall 0\le i \le n: x_i \not\cong \unit\big\}.
    \]
    The assumption that the $(n+2)$-simplex lies in $(\xN_{n+2}\calP)^\el$ ensures that the total composite $f_\infty \circ \dots \circ f_{-\infty}$ is a $(0,0)$-ary operation in $\calP$.
    Under this identification the composition map $\mrm{comp}_\calP$ is induced by $d_1\colon \xN_2\calP \to \xN_1\calP$ and thus provides an augmentation $\xN_\bullet \calF(\calP) \to \xN_{-1} \calF(\calP) \coloneq \calP(\emptyset,\emptyset) \subset \xN_1 \calP$. 
\end{obs}

\begin{obs}\label{obs:NcalF-pullback}
    Because the conditions defining the factorization category and its nerve can be formulated in terms of connectedness we have pullback squares 
    \[
    \begin{tikzcd}
        \calF(\calP) \rar[hook] \dar \ar[dr, very near start, phantom, "\lrcorner"] & \calP_{\unit//\unit} \dar
        &&
        \xN_\bullet \calF(\calP) \rar[hook] \dar \ar[dr, very near start, phantom, "\lrcorner"] & \xN_{1+\bullet+1} \calP \dar
        \\
        \calF(\Csp) \rar[hook] & \Csp_{\emptyset//\emptyset} 
        &&
        \xN_\bullet \calF(\Csp) \rar[hook] & \xN_{1+\bullet+1} \Csp 
    \end{tikzcd}
    \]
    describing them in terms of the relevant notion for the terminal properad $\Csp$.
\end{obs}

\begin{cor}\label{cor:Fbar=F}
    The construction of the factorization category uniquely descends to a functor
    \[
        \Red{\calF}(-)\colon \projPrpd \too \Cat
    \]
    such that $\Red{\calF}(\Red{\calP}) \simeq \calF(\calP)$. 
\end{cor}
\begin{proof}
    As $\projPrpd$ is a Bousfield localization of $\Prpd$ (\cref{prop:projectivization}) it suffices (by \cite[Proposition 5.2.7.12]{HTT}) to check that if 
    $f\colon \calP_1 \to \calP_2$ in $\Prpd$ induces an equivalence on reduced $\infty$-properads $\calP_1^\rd \simeq \calP_2^\rd$, it also induces an equivalence on factorization categories $\calF(\calP_1) \simeq \calF(\calP_2)$.
    This is indeed the case because $\calF(\calP)$ is independent of the $(0,0)$-ary operations.
    More precisely, in \cref{obs:nerve-calF} we saw an embedding $\xN_\bullet \calF(\calP) \subset \xN_{\bullet+2} \calP$ and this lands in the subspace 
    $K_{n+2} \calP \subset \xN_{n+2} \calP$ discussed in \cref{lem:nerve-splitting}, which was shown there to be equivalent to $\xN_{n+2}\Red{\calP}$ and thus independent of $(0,0)$-ary operations.
\end{proof}

The main result of this section is the following characterization of extended $\infty$-properads in terms of the factorization category.
We will prove this at the end of this subsection after having established a description of $(-)^\ext$ in terms of algebraic patterns in \cref{prop:ext-rd-pattern}.
\begin{prop}\label{prop:detect-extended}
    An $\infty$-properad $\calP$ is extended, i.e.~the counit map $(\calP^\rd)^\ext \to \calP$ is an equivalence, if and only if the composition map induces an equivalence
    \[
        \mathrm{comp}_{\calP} \colon |\calF(\calP)| \xtoo{\simeq} \calP(\emptyset; \emptyset).
    \]
\end{prop}

The above characterization of extended properads is quite concrete:
through \cref{obs:nerve-calF} we can identify the augmented simplicial space $\xN_\bullet \calF(\calP) \to \calP(\emptyset; \emptyset)$ as a subspace of $\xN_{1+\bullet+1}\calP$
and $\calP$ is extended if and only if the map from the realization to the augmentation is an equivalence.
We can for example apply this to show that the $d$-dimensional $\theta$-structured bordism category $\Bord_d^\theta$, considered in \cref{ex:bord} and \cref{ex:tangential}, is an extended $\infty$-properad.

\begin{cor}\label{cor:Bord-extended}
    The $\infty$-properad $\Bord_d^\theta$ is extended for all $d\ge 1$ and all tangential structures $\theta$.
\end{cor}
\begin{proof}
    Using \cref{obs:nerve-calF} can identify augmented simplicial space $\xN_\bullet\calF(\Bord_d^\theta)$ with a subspace (in fact a levelwise a union of components) of the double-décalage $\xN_{1+\bullet+1}\Bord_d^\theta$.
    As discussed in \cref{ex:bord}, models for the bordism category are typically \textit{non-complete} Segal spaces $\mrm{PBord}_\bullet^{d,\theta}$, which need to be completed to obtain the nerve $\xN_\bullet\Bord_d^\theta$.
    If we define a subspace $X_\bullet \subset \mrm{PBord}_{1+\bullet+1}^{d,\theta}$ analogously to $\xN_\bullet\calF(\Bord_d^\theta)$, then its realization will still be equivalent to $|\calF(\Bord_d^\theta)|$ as the comparison map is a Dwyer--Kan equivalence of Segal spaces.
    In \cite{jan-Bord1} the Segal space $\mrm{PBord}_\bullet^{d,\theta}$ is given as the nerve of a topological poset $P\pcal{C}_{d,\theta}$.
    (See \cite[Remark 7.11]{jan-Bord1} and the reference therein for why the nerve of the topological poset $N_\bullet(P\pcal{C}_{d,\theta})$ considered there is equivalent to more standard constructions of the bordism category.)
    In this setting the augmented simplicial space $X_\bullet$ is almost exactly the nerve of the augmented topological poset $F_{d,\theta}$ defined in \cite[Definition 7.7]{jan-Bord1}, except that there the composite was also allowed to be empty.
    Let $F_{d,\theta}^{\neq \emptyset} \subset F_{d,\theta}$ denote the sub-poset of non-empty factorizations.
    (See \cite[Proof of Lemma 7.18]{jan-Bord1} for more details.)
    By \cite[Proposition 7.8]{jan-Bord1} the augmentation of this simplicial space induces an equivalence
    \[
        \|\xN_\bullet F_{d,\theta}^{\neq\emptyset}\| \simeq \Map_{\Bord_d^\theta}(\emptyset, \emptyset)^{\rm con}
    \]
    and therefore $\Bord_d^\theta$ is an extended $\infty$-properad by \cref{prop:detect-extended}.
    (Note that the statement of \cite[Proposition 7.8]{jan-Bord1} is missing the hypothesis that $d\ge 1$, but this is used in the proof because the vector space of smooth functions on a non-empty $d$-manifold is assumed to be infinite-dimensional.
    The assumption is also necessary because $\Red{\Bord_0} = *$.)
\end{proof}

\begin{example}
    For a symmetric monoidal $\infty$-category $\calC \in \SM$, one can define \textit{$\calC$-valued $d$-dimensional $\theta$-structured topological field theories} as symmetric monoidal functors $\Bord_d^\theta \to \calC$.
    Using \cref{cor:Bord-extended} together with the adjunctions from \cref{lem:underlying-properad} and \cref{lem:extra-adjoint-reduction} we get
    \[
        \Fun^\otimes(\Bord_d^\theta, \calC)^\simeq
        \simeq \Map_{\Prpd}(\Bord_d^\theta, \calU(\calC))
        \simeq \Map_{\Prpd^\rd}((\Bord_d^\theta)^\rd, \calU(\calC)^\rd),
    \]
    which can be understood as saying that the value of a TFT on closed manifolds is (coherently) uniquely determined by its value on manifolds with boundary.
    For example, let $\calZ \colon \Bord_d \to \Span(\calS)$ be a $d$-dimensional TFT valued in the $\infty$-category of spans of spaces with its monoidal structure $\times$.
    Then for every closed $d$-manifold $W \colon \emptyset \to \emptyset$, $\calZ$ defines a map
    \[
        \BDiff(W) \subset \End_{\Bord_d}(\emptyset) \xtoo{\calZ} \End_{\Span(\calS)}(*) \simeq \calS^\simeq,
    \]
    i.e.~$\calZ(W)$ is a space with a $\Diff(W)$-action. 
    \cref{cor:Bord-extended} tells us that this $\Diff(W)$-action is coherently uniquely determined by what $\calZ$ assigns to bordisms with non-empty boundary.
\end{example}

As another consequence of \cref{prop:detect-extended} we can also make more concrete the description of $\Prpd$ as a pullback, which we obtained using semi-recollements in \cref{cor:Prpd-as-pullback}.

\begin{cor}\label{cor:Prpd-as-pullback-factorization}
    The natural transformation $\mrm{comp}_\calP\colon |\calF(\calP)| \to \calP(\emptyset;\emptyset)$ induces a cartesian square
    \[\begin{tikzcd}
    	\Prpd && {\Ar(\calS)} \\
    	{ \SMeq{\Red{\Csp}} } && \calS
    	\arrow["{|\Red{\calF}(-)|}"', from=2-1, to=2-3]
    	\arrow["{\Red{(-)}}"', from=1-1, to=2-1]
    	\arrow["{\ev_0}", from=1-3, to=2-3]
    	\arrow["\mathrm{comp}_{(-)}", from=1-1, to=1-3]
    	\arrow[phantom, very near start, "\lrcorner", from=1-1, to=2-3]
    \end{tikzcd}\]
    in $\Cat$.
    In particular, as in \cref{lem:cofracture-square}, any $\infty$-properad $\calP \in \Prpd$ can be recovered from the triple
    \[\Red{\calP} \in \SM,\qquad \calP(\emptyset;\emptyset) \in \calS, \qquad |\Red{\calF}(\Red{\calP})| \to \calP(\emptyset;\emptyset) \in \Ar(\calS).\]
\end{cor}
\begin{proof}
    This follows by rewriting the cartesian square from \cref{cor:Prpd-as-pullback} using
    the equivalence $\Prpd^\nul \simeq \calS$ from \cref{ex:nullary-properads} and 
    the equivalence $\Prpd^\rd \simeq \projPrpd$ from \cref{prop:projectivization}.
    To see that the square is the same, it suffices to check that the natural transformation $\alpha\colon (\calP^\rd)^\ext(\emptyset;\emptyset) \too \calP(\emptyset;\emptyset)$ that comes from the counit of $(-)^\ext \dashv (-)^\rd$ agrees with $\mrm{comp}_\calP$.
    Since $\mrm{comp}_{(-)}$ and $\alpha$ both are natural we have a naturality square
    \[\begin{tikzcd}[column sep=large]
        {|\calF((\calP^\rd)^\ext)|} \ar[r, "\simeq", "\mrm{comp}_{(\calP^\rd)^\ext}"'] \ar[d, "\simeq"] &
        (\calP^\rd)^\ext(\emptyset; \emptyset) \ar[d, "\alpha"] \\
        {|\calF(\calP)|} \ar[r, "\mrm{comp}_\calP"] & 
        \calP(\emptyset; \emptyset)
    \end{tikzcd}\]
    where the top map is an equivalence by \cref{prop:detect-extended} and the left map is an equivalence by \cref{cor:Fbar=F}.
    This defines the desired equivalence between $\mrm{comp}_{(-)}$ and $\alpha$.
\end{proof}

\subsubsection{A pattern description reduced and extended \texorpdfstring{$\infty$}{infinity}-properads}
To prove \cref{prop:detect-extended} we will need to study the algebraic pattern $\L$ of level graphs from \cref{sec:segal} in a bit more detail.
This will allow us to describe the double adjunction involving $(-)^\ext$ in terms of left and right Kan extension along a morphism of patterns $\Lrd \hookrightarrow \Lc$, see \cref{prop:adjoints-to-rd-and-null}.
We begin by introducing the algebraic pattern whose complete Segal spaces will be the reduced $\infty$-properads.
\begin{defn}
    Let $\hldef{\Lrd} \subset \L$ denote the full subcategory on those diagrams $A\colon \Tw[n] \to \Fin_*$ where $A_{0,n} = 1_+$ and $\coprod_{i=0}^n A_{i,i} \neq 0_+$.
    (In particular, $\Lrd \subset \Lc$ is a full subcategory of the connected diagrams from \cref{defn:Lc}.)
\end{defn}

\begin{obs}\label{rem:Lzero}
    $\Lrd$ is the full subcategory of $\Lc$ whose objects are precisely those \textit{not} of the form:
    \[
        T(a,b)\colon \Tw[a+b+1] \to \Fin_*, \qquad T(a,b)_{i,j} = \begin{cases}
            1_+ & \text{ if } i \le a < j,\\
            0_+ & \text{ otherwise},
        \end{cases}
    \]
    for some $a, b \ge 0$.
    For example, $T(1,0)\colon \Tw[2] \to \Fin_*$ is given by the diagram
\[\begin{tikzcd}[row sep = 0]
	&& {A_{02}=1_+} \\
	& {A_{01}=0_+} && {A_{12}=1_+} \\
	{A_{00}=0_+} && {A_{11}=0_+} && {A_{22}=0_+}
	\arrow[from=2-2, to=1-3]
	\arrow[from=2-4, to=1-3]
	\arrow[from=3-1, to=2-2]
	\arrow[from=3-3, to=2-2]
	\arrow[from=3-3, to=2-4]
	\arrow[from=3-5, to=2-4]
\end{tikzcd}\]
    Using this description of $\Lrd$, one can check: if there is a map $([n], A) \to ([m], B)$ in $\Lc$, then $([n], A) \in \Lrd$ implies $([m], B) \in \Lrd$.
    In particular, the factorization system restricts to $\Lrd$ making it an algebraic pattern.
\end{obs}

We will also need the following category of closed level graphs, which is closely related to $\calF(\Csp)$, see \cref{rem:Lcl-full}.
\begin{defn}
    The category \hldef{closed level graphs} is defined as the left fibration $\hldef{\Lcl} \to \Dop$ that straightens to the functor
    \[
        [n] \longmapsto \Fun'(\Tw([n]^{\lhd\rhd}), \Fin)^\simeq
    \]
    where $[n]^{\lhd\rhd} = \{-\infty, 0,1, \dots, n, \infty\}$ and 
    $\Fun'$ denotes the full subcategory spanned by functors $A \colon \Tw([n]^{\lhd\rhd}) \to \Fin$ such that
    \begin{enumerate}[a)]
        \item $A$ preserves pushouts,
        \item $A_{-\infty,-\infty} = \emptyset = A_{\infty,\infty}$, and
        \item $A_{i,i} \neq \emptyset$ for $0 \le i \le n$ and $A_{-\infty,\infty} = *$.
    \end{enumerate}
\end{defn}

\begin{rem}\label{rem:Lcl-full}
    The left fibration $\Lcl \to \Dop$ is exactly the unstraightening of the simplicial space $\xN_\bullet\calF(\Csp)$.
    This follows from the description of $\xN_\bullet\calF(\calP)$ in \cref{obs:nerve-calF} and the description of the unstraightening of $\Csp$ in \cref{cor:boldC-description}.

    To relate this to $\L$, recall from the proof of \cref{cor:act-Env-formula} that the left fibration $p^\act\colon \L^\act \to \simp^{\op,\act}$ unstraightens to the active part of the simplicial space $\xN_\bullet\Csp$.
    Now consider its pullback along the functor $\Dop \to \simp^{\op,\act}$ that adds an initial and terminal element to each object.
    Then $\Lcl$ embeds fully faithfully into this pullback 
    \[
    \begin{tikzcd}[column sep = large]
        \Lcl \rar[hook]
        & {\Dop \times_{\simp^{\op,\act}} \L^\act} \rar \dar["q"'] \ar[dr,very near start, phantom, "\lrcorner"] & \L^\act \dar \\
        & \Dop \rar["{[0] \star [\bullet] \star [0]}"] & \simp^{\op, \act}
    \end{tikzcd}
    \]
    where $\star\colon \Dop \times \Dop \to \Dop$ denotes the join.
    Indeed, $q$ straightens to the functor that sends $[n]$ to the space of pushout-preserving functors $\Tw([n]^{\lhd\rhd}) \to \Fin$,
    and the straightening of $\Lcl \to \Dop$ is defined as a subfunctor of this.
\end{rem}

In order to compute the left Kan extension along $\Lrd \hookrightarrow \Lc$ below, 
we will need to understand the slice category
        $\Lrd \times_{\Lc} (\Lc)_{/T(a,b)}$.
An object of this category consists of $([l], B) \in \Lrd$, a map $d\colon [a+b+1] = [a] \star [b] \to [l]$ and an isomorphism between $B\circ \Tw(d)$ and the diagram $T(a,b)\colon \Tw[a+b+1] \to \Fin_*$.
This isomorphism is unique if it exists (since $T(a,b)$ has no non-trivial automorphisms) and it exists if and only if $B_{d(i),d(j)}$ is $1_+$ for $i \le a < j$ and $0_+$ for any other $i\le j$.
We will denote the objects of this slice as $(B, d\colon [a]\star [b] \to [l])$.

\begin{lem}\label{lem:Lrd-slice}
    For all $a, b \ge 0$ there is an equivalence of categories
    \begin{align*}
        (\Dop)_{/[a]} \times \Lcl \times (\Dop)_{/[b]}
        &\xtoo{\simeq}
        \Lrd \times_{\Lc} (\Lc)_{/T(a,b)} \\
        ([a] \xto{d^l} [n^l], 
        \Tw([n^c]^{\lhd\rhd}) \xto{A} \Fin, 
        [b] \xto{d^r} [n^r]) 
        &  
        \longmapsto
        (A_+ \circ \Tw(\lambda),
        [a] \star [b] \xto{d^l \star \emptyset \star d^r} [n^l] \star [n^c] \star [n^r])
    \end{align*}
    where $\lambda \coloneq (0 \star \id_{[n^c]} \star 0) \colon [n^l] \star [n^c] \star [n^r] \to [0] \star [n^c] \star [0] = [n^c]^{\lhd\rhd}$,
    and we write $A_+$ for the functor obtained by post-composing $A$ with $(-)_+\colon \Fin \to \Fin_*$.
\end{lem}
\begin{proof}
    For an object $(B, d\colon [a]\star[b] \to [l])$ in $\Lrd \times_{\Lc} (\Lc)_{/T(a,b)}$
    we have
    \[
        B_{d(a),d(a)} = 0_+ = B_{d(a+1),d(a+1)}
            \qquad\text{and}\qquad
        B_{d(a), d(a+1)} = 1_+.
    \]
    Because $([l], B)$ is in particular in $\Lc$ we know that this implies that $B_{i,j} = 0_+$ if either $j\le d(a)$ or $i \ge d(a+1)$
    and that $B_{x,y} \cong 1_+$ whenever $x \le d(a)<d(a+1) \le y$.
    
    A morphism $([l], B, d) \to ([k], C, e)$ consists of a morphism $f\colon [k] \to [l]$ such that $f \circ e = d$ and a natural isomorphism $\alpha\colon C \circ \Tw(f) \cong B$. (In general $\alpha$ would only be an inert natural transformation, but since $B_{f(0),f(k)} \cong 1_+ \cong C_{0,k}$ the inert morphism $\alpha_{0,k}\colon B_{f(0),f(k)} \to C_{0,k}$ must be an isomorphism and hence all values of $\alpha$ are active and therefore isomorphisms.)
    
    Using this description one can check the assignment given in the lemma indeed defines a functor, whose value on a morphism 
    \[
        \mu = (\mu^l, (\mu^c, \alpha\colon A \cong \Tw(\mu^{c,\lhd\rhd})^*B)), \mu^r) \colon ([a] \to [n^l], A, [b] \too [n^r]) \to ([a] \to [m^l], B, [b] \to [m^r])
    \]
     is described by the diagram
\[\begin{tikzcd}
	&& {[n^l]\star[n^c] \star [n^r]} && {[n^c]^{\triangleleft \triangleright}} & {\Tw([n^c]^{\triangleleft \triangleright})} \\
	{[a] \star [b]} &&&&&&& \Fin. \\
	&& {[m^l]\star[m^c] \star [m^r]} && {[m^c]^{\triangleleft \triangleright}} & {\Tw([m^c]^{\triangleleft \triangleright})}
	\arrow["{0 \star \id \star 0}", from=1-3, to=1-5]
	\arrow["{\mu^l \star \mu^c \star \mu^r}"', from=1-3, to=3-3]
	\arrow["{\mu^{c,\triangleleft \triangleright}}", from=1-5, to=3-5]
	\arrow["A", from=1-6, to=2-8]
	\arrow[""{name=0, anchor=center, inner sep=0}, "{\Tw(\mu^{c,\triangleleft \triangleright})}"', from=1-6, to=3-6]
	\arrow["{d^l \star \emptyset \star d^r}", from=2-1, to=1-3]
	\arrow["{d^l \star \emptyset \star d^r}"', from=2-1, to=3-3]
	\arrow["{0 \star \id \star 0}"', from=3-3, to=3-5]
	\arrow["B"', from=3-6, to=2-8]
	\arrow["\substack{\alpha\\\cong}"{marking, allow upside down}, draw=none, from=0, to=2-8]
\end{tikzcd}\]
    By this we mean that $\mu$ is sent to the morphism $(\mu^l \star \mu^c \star \mu^r, \beta)$ in $\Lrd \times_{\Lc} (\Lc)_{/T(a,b)}$ where 
    \[
        \beta\colon \Tw(0 \star \id_{[n^c]} \star 0)^*A_+ \cong \Tw(\mu^l \star \mu^c \star \mu^r)^* \Tw(0 \star \id_{[m^c]} \star 0)^*B_+
    \]
    is obtained from $\alpha \colon A \cong \Tw(\mu^{c,\lhd\rhd})^*B$ by adding basepoints and using the commutative square in the above diagram.
    
    We can define an inverse functor by sending a triple $([l], B, d)$ to the object defined by restricting $d$ and $B$ as follows:
    \[
        \left(d^{-1}(\{0,\dots,x\}) \xto{d} \{0,\dots,x\},\,
        \Tw([y-x-2]^{\lhd\rhd}) \cong \Tw(\{x,\dots,y\}) \xto{B^\circ} \Fin,\,
        d^{-1}(\{y,\dots,l\}) \xto{d} \{y,\dots,l\}\right)
    \]
    where $0\le x < y \le l$ are such that $B_{x,x} = \emptyset$, $B_{x+1,x+1} \neq \emptyset \neq B_{y-1,y-1}$ and $B_{y,y} = \emptyset$, 
    and $B^\circ$ is obtained from $B$ by removing basepoints.
    (Note that description of objects in $\Lrd \times_{\Lc} (\Lc)_{/T(a,b)}$ given above implies that $x$ and $y$ are unique and satisfy $x+2\le y$.)
    \end{proof}

\begin{prop}\label{prop:ext-rd-pattern}
    Restriction, left Kan extension, and right Kan extension along the full inclusion
    $j\colon \Lrd \hookrightarrow \Lc$
    all preserve Segal objects and hence yield a triple-adjunction
    \[\begin{tikzcd}
	{\Seg_{\Lc}(\calS)} && {\Seg_{\Lrd}(\calS).}
	\arrow["{j^\ast}"{description}, from=1-1, to=1-3]
	\arrow["{j_!}"', shift right = 2, hook', from=1-3, to=1-1]
	\arrow["{j_\ast}", shift left = 2, hook, from=1-3, to=1-1]
    \end{tikzcd}\]
    All three functors preserve the completeness condition, and restricting them to the full subcategory of complete Segal objects yields a triple adjunction that is equivalent to the one from \cref{lem:extra-adjoint-reduction}:
    \[\left(\begin{tikzcd}
	\CSeg_{\Lc}(\calS) && {\CSeg_{\Lrd}(\calS)}
	\arrow["{j^\ast}"{description}, from=1-1, to=1-3]
	\arrow["{j_!}"', shift right = 2, hook', from=1-3, to=1-1]
	\arrow["{j_\ast}", shift left = 2, hook, from=1-3, to=1-1]
    \end{tikzcd}
    \right) \simeq \left(
    \begin{tikzcd}
	\Prpd && {\Prpd^\rd}
	\arrow["{(-)^\rd}"{description}, from=1-1, to=1-3]
	\arrow["{(-)^{\rm ext}}"', shift right = 2, hook', from=1-3, to=1-1]
	\arrow["{\text{include}}", shift left = 2, hook, from=1-3, to=1-1]
    \end{tikzcd}\right).\]
\end{prop}
\begin{proof}
    Recall from \cref{rem:Lzero} that there are no maps from $T(a,b)$ to objects in $\Lrd$.
    This implies that for any $([n],A) \in \Lrd$ we have $(\Lrd)_{([n],A)/}^\xint = (\Lc)_{([n],A)/}^\xint$, and similarly for elementary slices.
    For any $T(a,b) \in \Lc$ there is a unique inert map to an elementary, namely $T(a,b) \intto T(0,0)$.
    In particular, $X \colon \Lc \to \calS$ satisfies the Segal condition if and only if $j^\ast X$ is Segal and this unique inert map induces an equivalence $X(T(a,b)) \simeq X(T(0,0))$.
    This shows that $j^\ast$ preserves Segal objects.
    Additionally, for any $X \colon \Lrd \to \calS$ we have $j^\ast j_\ast X \simeq X$ and $(j_\ast X)(T(a,b)) \simeq *$ (as there are no maps to $\Lrd$) and thus $j_\ast$ also preserves Segal objects.

    To show that $j_!$ preserves Segal objects, we first compute its value on $T(a,b)$ as
    \[
        (j_!X)(T(a,b)) \simeq \colim_{([n], A) \in \Lrd \times_{\Lc} (\Lc)_{/T(a,b)}} X([n],A)
        \simeq \colim_{([m], B) \in \Lcl} X([1+m+1], B)
    \]
    since by \cref{lem:Lrd-slice} $\Lcl$ is final in the relevant slice.
    In particular, the value is independent of $a$ and $b$ and one can check that the inert map $\alpha\colon T(a,b) \to T(0,0)$ indeed induces an equivalence when we evaluate $j_!X$ on it.
    Therefore, $j_!X$ is Segal if $X$ is.
    
    Recall that a $\L$-Segal space is called complete if its restriction along a certain functor $\Dop \to \L$ gives a complete $\Dop$-Segal space in the sense of Rezk.
    This functor from $\Dop$ factors through $\Lrd \subset \Lc \subset \L$, so we get compatible notions of completeness and $j^\ast$, $j_!$, and $j^\ast$ all preserve completeness since none of them changes the value of our Segal object on $\Dop$.
    We have therefore shown that $j^\ast$, $j_!$, and $j_\ast$ all preserve complete Segal objects and thus restrict to a triple adjunction 
    \[\begin{tikzcd}
	\CSeg_{\Lc}(\calS) && {\CSeg_{\Lrd}(\calS).}
	\arrow["{j^\ast}"{description}, from=1-1, to=1-3]
	\arrow["{j_!}"', shift right = 2, hook', from=1-3, to=1-1]
	\arrow["{j_\ast}", shift left = 2, hook, from=1-3, to=1-1]
    \end{tikzcd}\]

    For the final claim, we need to identify this with the other triple adjunction.
    From \cref{thm:segal-envelope} we get an equivalence between $\CSeg_{\Lc}(\calS)$ and $\Prpd$.
    It will suffice to show that under this equivalence the essential image of $j_\ast$ exactly corresponds to the reduced $\infty$-properads, as then the other adjoints must agree by the uniqueness of (left) adjoints.
    By the above discussion, a Segal object $X\colon \Lc \to \calS$ is in the essential image of $j_\ast$ if and only if $X(T(0,0)) = \ast$. 
    Tracing through the envelope equivalence from \cref{thm:segal-envelope} and noting that $\{T(0,0)\} \hookrightarrow \xN_1 \Csp$ is a monomorphism we have
    \[
        \Env(X)(\emptyset;\emptyset) \simeq \xN_1(\Env(X)) \times_{\xN_1(\Csp)} \{T(0,0)\} \simeq X(T(0,0)),
    \]
    so the condition $X(T(0,0)) \simeq \ast$ precisely says that $\Env(X)$ is reduced.
\end{proof}

In order to prove \cref{prop:detect-extended}, we now need to identify the formula for $(j_!X)(T(0,0))$ with the homotopy type of $\calF(\Env(X))$.

\begin{proof}[Proof of \cref{prop:detect-extended}]
    Given $X \in \CSeg_{\Lc}(\calS)$ we know that value of $j_!j^*X$ on $T(0,0)$ is given by the colimit
    \[
        j_!j^*X(T(0,0)) = \colim_{([n], A) \in \Lcl} X([1+n+1], A).
    \]
    We need to show the right-hand side is equivalent to the homotopy type $|\calF(\Env(X))|$ of the factorization category of the envelope.
    Expanding on \cref{cor:act-Env-formula} we have pullback squares
\[\begin{tikzcd}
	\Lcl & {\Dop\times_{\simp^{\op,\act}} \L^\act} && {\L^\act} && \bmC & \L \\
	& \Dop && {\simp^{\op,\act} } && {\Dop \times \Fin_\ast}
	\arrow["l", hook, from=1-1, to=1-2]
	\arrow["s"', from=1-1, to=2-2]
	\arrow["k", from=1-2, to=1-4]
	\arrow["r", from=1-2, to=2-2]
	\arrow["\lrcorner"{anchor=center, pos=0.125}, draw=none, from=1-2, to=2-4]
	\arrow["i", from=1-4, to=1-6]
	\arrow["{p^\act}"', from=1-4, to=2-4]
	\arrow["\varphi", from=1-6, to=1-7]
	\arrow["q", from=1-6, to=2-6]
	\arrow["{[0]\star-\star[0]}"', from=2-2, to=2-4]
	\arrow[""{name=0, anchor=center, inner sep=0}, "{(-, 1_+)}"', from=2-4, to=2-6]
	\arrow["\lrcorner"{anchor=center, pos=0.125}, draw=none, from=1-4, to=0]
\end{tikzcd}\]
    where the vertical functors are left fibrations.
    As in \cref{cor:act-Env-formula}, this implies that there are canonical Beck-Chevalley equivalences, eventually giving a natural equivalence of simplicial spaces
    \[
        r_!k^*i^*\varphi^* X \simeq (q_!\varphi^*X)([0] \star [\bullet] \star [0], 1_+) = \xN_{1+\bullet+1} \Env(X).
    \]
    Recall from \cref{rem:Lcl-full} that $\Lcl$ is a full subcategory of $\Dop \times_{\simp^{\op,\act}} \L^\act$ and therefore, if we apply $s_! l^*$ instead of $r_!$, then this corresponds to passing to a sub-simplicial space of $\xN_{1+\bullet+1}\Env(X)$.
    Indeed, the subspace we need to pass to is exactly the pullback
    \[
        s_!l^*k^*i^* \varphi^* X 
        \simeq s_!(\pt) \times_{r_!(\pt)} r_!k^*i^*\varphi^* X
        \simeq \St(s\colon\Lcl\to \Dop) \times_{\xN_{1+\bullet+1}\Env(\pt)} \xN_{1+\bullet+1}\Env(X).
    \]
    We know from \cref{rem:Lcl-full} that the straightening $\St(s\colon \Lcl \to \Dop)$ is the nerve of $\calF(\Csp)$, so the above expression, by \cref{obs:NcalF-pullback}, is exactly the nerve of the factorization category $\calF(\Env(X))$.
    Therefore, we get natural equivalences
    \[
        j_!j^*X(T(0,0)) \simeq \colim_{\Lcl} l^*k^*i^*\varphi^* X
        \simeq \colim_{\Dop} s_!k^*i^*\varphi^* X \simeq \colim_{[n] \in \Dop} \xN_n \calF(\Env(X)) = |\calF(\Env(X))|.
    \]
    Since these are compatible with the canonical map to $X(T(0,0))$ by \cref{obs:nerve-calF},
    it follows that $X$ is reduced if and only if the composition map
    \(
        |\calF(\Env(X))| \to \Env(X)(\emptyset;\emptyset)
    \)
    is an equivalence.
\end{proof}

\subsection{\texorpdfstring{$n$}{n}-properads and labelled cospan categories}\label{subsec:n-properads}
We now study $n$-properads and in particular $1$-properads, both of which we characterize as full subcategories of $\Prpd$.
Then we show that $1$-properads are equivalent to the labelled cospan categories of \cite[\S2]{Jan-tropical} and use an analogous result of Beardsley--Hackney \cite{BH22} to conclude that $1$-properads in our sense are equivalent to the more classical coloured properads of \cite{JohnsonYau}.

\subsubsection{$n$-properads}
In this section we study $n$-properads, which are to $\infty$-properads what $(n,1)$-categories are to $\infty$-categories.

\begin{defn}\label{defn:n-properad}
    For $n\ge 0$, we say that an $\infty$-properad $\calP$ is an \hldef{$n$-properad} if the spaces of operations $\calP(x_1,\dots, x_k; y_1, \dots, y_l)$ are $(n-1)$-truncated for all tuples of colours $x_i, y_j \in \xN_0^\el(\calP)$.
    We let $\hldef{\mrm{Prpd}_n} \subset \Prpd$ denote the full subcategory of $n$-properads.
\end{defn}

\begin{example}
    The morphism properad $\calU(\calC)$ of a symmetric monoidal $\infty$-category $\calC$
    is an $n$-properad if and only if $\calC$ an $(n,1)$-category.
\end{example}

\begin{lem}\label{lem:n-properads}
    For an $\infty$-properad $\calP$ and $n \ge 1$ the following are equivalent:
    \begin{enumerate}
        \item\label{it:n-prpd} $\calP$ is an $n$-properad,
        \item\label{it:Red} $\Red{\calP}$ is an $(n,1)$-category and $\calP(\emptyset; \emptyset)$ is $(n-1)$-truncated,
    \end{enumerate}
    If $n \ge 2$ we further have the following equivalent condition:
    \begin{enumerate}[resume]
        \item\label{it:ncat} $\calP$ is an $(n,1)$-category.
    \end{enumerate}
\end{lem}
\begin{proof}
    $(1 \Leftrightarrow 2)$
    Combining \cref{lem:mapping-space-in-properad} and \cref{cor:mapping-space-in-Red} we get 
    \[
        \Map_{\Red{\calP}}(x_1 \otimes \dots \otimes x_m, y_1 \otimes \dots \otimes y_l)
        \simeq  
        \coprod_{I \amalg J \twoheadrightarrow K} \prod_{k \in K} 
        \calP\left(\{x_i\}_{i \in I_k}; \{y_j\}_{j \in J_k} \right),
    \]
    so the mapping spaces in $\Red{\calP}$ are $(n-1)$-truncated (i.e.~$\Red{\calP}$ is an $(n,1)$-category) if and only if all non-$(0,0)$-ary operation spaces in $\calP$ are $(n-1)$-truncated (i.e.~$\calP^\rd$ is an $n$-properad).

    ($1 \Leftrightarrow 3$)
    Assume $n\ge 2$.
    Then a space $X$ is $(n-1)$-truncated if and only if $\xF(X)$ is $(n-1)$-truncated.
    (This can be seen using that $\xF(X) \to \xF(*)$ is a map with fibers $X^m$ and a $1$-truncated base.)
    Combining the description above with the equivalence
    \(
        \Map_{\calP}(x,y) \simeq \xF(\calP(\emptyset;\emptyset)) \times \Map_{\Red{\calP}}(x,y)
    \)
    from \cref{cor:mapping-space-in-Red} concludes the proof. 
\end{proof}

\begin{rem}\label{rem:UFCs}
    We believe that the hereditary unique factorization categories (hereditary UFCs) of Kaufmann--Monaco \cite{KM22} are precisely $\infty$-properads that have the property of being $1$-categories.
    Indeed, a symmetric monoidal $1$-category is a UFC if and only if $\xN_n\calC$ is free for $n = 0,1$ (and by \cref{cor:finite-limits} thus for all $n$),
    and it is hereditary if and only if $d_1:\xN_2\calC \to \xN_1\calC$ is a free map.

    If an $\infty$-properad $\calP$ is a (symmetric monoidal) $1$-category, then it in particular is a $1$-properad.
    However, for a $1$-properad to be a $1$-category the space $\Map_\calP(\unit, \unit) = \xF(\calP(\emptyset;\emptyset))$ has to be $0$-truncated.
    This is only possible if $\calP(\emptyset;\emptyset)$ is empty.
    Therefore, hereditary UFCs should be exactly those $1$-properads that do not have operations of arity $(0,0)$.
    (For example $\Csp$ is not a $1$-category and hence not a hereditary UFC.)
\end{rem}

Recall that for any $\infty$-category $\calC$, an object $x \in \calC$ is called $n$-truncated if and only if $\Map_\calC(y,x)$ is $n$-truncated for all $y \in \calC$, see \cite[\S 5.5.6]{HTT}.
While $n$-properads are not precisely the $n$-truncated objects in $\Prpd$, we do have the following implications.
\begin{lem}\label{lem:n-properad-n-truncated}
    Let $\calP$ be an $\infty$-properad and $n \ge 0$.
    \begin{enumerate}[(1)]
        \item If $\calP$ is an $(n-1)$-truncated in $\Prpd$, then it is an $n$-properad.
        \item If $\calP$ is an $n$-properad, then it is an $n$-truncated object in $\Prpd$.
    \end{enumerate}
    In particular, $\mrm{Prpd}_n$ is an $(n+1,1)$-category because all its objects are $n$-truncated.
\end{lem}
\begin{proof}
    For (1) suppose that $\calP$ is $(n-1)$-truncated.
    Then $\xN_0^\el\calP = \Map_{\Prpd}(\xF(*), \calP)$ and $\Map(\frc_{k,l},\calP)$ are $(n-1)$-truncated.
    By \cref{lem:free-corolla-pushout} we can compute the space of operations $\calP(x_1,\dots,x_k;y_1,\dots,y_l)$ as the fibers of the maps
    \[
        \Map_{\Prpd}(\frc_{k,l}, \calP) \too \Map_{\Prpd}(\xF(*), \calP)^{\times k+l}.
    \]
    A map between $(n-1)$-truncated spaces has $(n-1)$-truncated fibers, so $\calP$ is an $n$-properad.

    Conversely, suppose that $\calP$ is an $n$-properad.
    Then $\calP^\simeq$ is an $n$-groupoid and so $\xN_0^\el\calP \subset \calP^\simeq$ is $n$-truncated.
    We also know that the fibers of the above maps are $(n-1)$-truncated (and its base is $n$-truncated), so it follows that $\Map_{\Prpd}(\frc_{k,l}, \calP)$ is $n$-truncated for all $k$ and $l$.
    But since the corollas generate $\Prpd$ under colimits (by \cref{cor:Prpd-compact-generation}) this implies that $\Map_{\Prpd}(\calQ, \calP)$ is $n$-truncated for all $\calQ$.
\end{proof}

\begin{rem}
    This is analogous to the situation for $\infty$-categories.
    An $(n,1)$-category is an $\infty$-category with $(n-1)$-truncated mapping spaces.
    Every such $\infty$-category is $n$-truncated when considered as an object of $\Cat$.
    Conversely, every $n$-truncated object of $\Cat$ is an $(n+1,1)$-category.
    Both of these implications are strict, for further details see \cite[below example 2.3]{SY22}.
\end{rem}

As a special case of \cref{cor:Prpd-as-pullback-factorization} we can now write $\Prpdone$ as a pullback, which will be useful when comparing to ``labelled cospan categories'' in the next subsection.
\begin{cor}\label{cor:Prpd1-as-pullback}
    The $(2,1)$-category of $1$-properads fits into a pullback square
    \[\begin{tikzcd}[column sep = large]
	\Prpdone & {\Ar(\Sets)} \\
	{(\SMone)_{/\Red{\Csp}}^\eqf} & \Sets.
	\arrow["\pi_0 |\Red{\calF}(-)|"', from=2-1, to=2-2]
	\arrow["\pi_0 \mrm{comp}_{(-)}", from=1-1, to=1-2]
	\arrow["{\Red{(-)}}"', from=1-1, to=2-1]
	\arrow["{\mathrm{ev}_0}", from=1-2, to=2-2]
    \arrow["\lrcorner", phantom, very near start, from=1-1, to=2-2]
    \end{tikzcd}\]
\end{cor}
\begin{proof} 
    By \cref{lem:n-properads} the pullback square from \cref{cor:Prpd-as-pullback-factorization} restricts to the left cartesian square below 
    \[\begin{tikzcd}
    	{\Prpdone } && {\Ar(\calS)\times_{\calS}\Sets} & {\Ar(\Sets)} \\
    	{(\SMone)_{/\Red{\Csp}}^\eqf} && \calS & \Sets
    	\arrow["{\Red{(-)}}"', from=1-1, to=2-1]
    	\arrow["\mrm{comp}_{(-)}", from=1-1, to=1-3]
    	\arrow["{\ev_0}"', from=1-3, to=2-3]
    	\arrow[""{name=0, anchor=center, inner sep=0}, "{|\calF(-)|}"', from=2-1, to=2-3]
    	\arrow["{\pi_0}", from=1-3, to=1-4]
    	\arrow["{\pi_0}"', from=2-3, to=2-4]
    	\arrow["{\ev_0}", from=1-4, to=2-4]
        \arrow["\lrcorner", very near start, draw=none, from=1-3, to=2-4]
    	\arrow["\lrcorner", very near start, draw=none, from=1-1, to=0]
    \end{tikzcd}\]
    where the pullback $\Ar(\calS) \times_\calS \Sets$ is with respect to $\ev_1 \colon \Ar(\calS) \to \calS$. 
    The right square is a pullback because $\pi_0$ is left adjoint to the inclusion $\Sets \hookrightarrow \calS$.
    Pullback pasting implies that the outer rectangle is cartesian, as claimed.
\end{proof}

\begin{obs}[$1$-properads vs.~``classical'' properads]
\label{obs:1prpd-vs-classical}
    Let $\mathrm{PRPD}_1$ denote the $1$-category of (coloured) properads considered in \cite{HRY15}.
    It is shown there that this is equivalent to the category of ``strict $\infty$-properads'', which are the same as $\bfG^\op$-Segal sets.
    Combining this with work of \cite{CH22} we get that $\mathrm{PRPD}_1$ is equivalent the $1$-category $\Seg_{\L}(\Sets)$ where $\L$ is the category of levelled graphs that also appeared in \cref{sec:segal}.
    As pointed out in \cite[Proposition 1.15]{BH22} this implies that there is a fully faithful functor
    \[
        \mathbf{PRPD}_1 \simeq \Seg_{\bfG^\op}(\Sets) \simeq \Seg_{\L}(\Sets) \hookrightarrow \Seg_{\L}(\calS)
    \]
    whose essential image are precisely the $0$-truncated objects of $\Seg_{\L}(\calS)$.
    We prove in \cref{thm:segal-envelope} that there is an equivalence
    \[
        \Seg_{\L}(\calS) \simeq \pPrpd \subset \Seg_{\Dop}(\CMon)
    \]
    between these $\L$-Segal spaces and the $\infty$-category of ``pre-properads'' from \cref{defn:pre-properad}.
    Compiling all results one gets that the $1$-category of classical properads $\mrm{PRPD}_1$ is equivalent to the full subcategory $\tau_{\le 0}\pPrpd \subset \pPrpd$
    on the $0$-truncated pre-properads.
    This is analogous to how the $1$-category of $1$-categories is not a full subcategory of the $\infty$-category of $\infty$-categories, but rather a full subcategory of the $\infty$-category of (non-complete) Segal spaces, namely the Segal sets.
    To describe the full subcategory $\Prpdone \subset \Prpd$ in classical terms, we need $2$-morphisms, as $\Prpdone$ is a $(2,1)$-category. 
    See \cref{cor:prpd-classical}.
\end{obs}

\subsubsection{Labelled cospan categories}
In this section we compare $1$-properads to the ``labelled cospan categories'', which the second author defined in \cite{Jan-tropical} and which were in part the motivation for the current work.
An analogous comparison was proven in work of Beardsley--Hackney \cite{BH22} proving the first part of \cite[Conjecture 2.31]{Jan-tropical}.

A difficulty in comparing the two concepts is that they are set up in different models.
To remedy this, we will mostly be working with $2$-categories in this section and our goal will be to show that the $2$-category $\LCC$ of labelled cospan categories fits into a (homotopy) pullback square analogous to \cref{cor:Prpd1-as-pullback}.
(In order to talk about homotopy pullbacks, we use Lack's model structure on $2$-categories \cite{Lack2002}.)
Subject to \cref{hyp:SMC} this also identifies our $(2,1)$-category $\Prpdone$ with the ``classical'' $(2,1)$-category of properads that appears in \cite{HRY15}.

We let $\hldef{\mrm{SMC}}$ denote the $2$-category with objects symmetric monoidal ($1$-)categories, morphisms symmetric monoidal functors and $2$-morphisms symmetric monoidal natural transformations.
Based on this, we can assemble LCCs into a $2$-category following \cite[\S1.1]{BH22}.
For an $\infty$-category $\calC$ we let $h\calC$ denote its homotopy category,
and in particular we let $h\Csp$ be the homotopy category of $\Csp$ where morphisms are now isomorphism classes of cospans of finite sets.
\begin{defn}
    The $2$-category $\hldef{\mrm{LCC}}$ is defined as a certain full subcategory of the (non-strict) slice $\mrm{SMC}_{/h\Csp}$.
    Its objects are symmetric monoidal functors $\pi\colon \calC \to h\Csp$ that are LCCs in the sense of \cite[Definition 2.4]{Jan-tropical}.
    Morphisms $(\calC, \pi) \to (\calC',\pi')$ are symmetric monoidal functors $F\colon \calC \to \calC'$ such that there exists%
        \footnote{
            If it exists, this natural isomorphism is unique by \cite[Remarks 1.3 and 1.4]{BH22} so we do not need to remember it.
        }
    a symmetric monoidal natural isomorphism $\pi' \circ F \cong \pi$.
    $2$-morphisms $\gamma\colon F \Rightarrow G$ are symmetric monoidal natural transformations such that $\id_{\pi'} \circ \gamma\colon \pi' \circ F \Rightarrow \pi' \circ G$ is an isomorphism.
\end{defn}

Recall from \cite[\S2.1]{Jan-tropical} that for any LCC $\calC$ we can construct a projectivization $\calC^{\proj}$ (denoted $\calC^\mrm{red}$ there) by taking the quotient of each of the mapping sets $\Map_\calC(x,y)$ by the action of the commutative monoid $\Map_\calC(\unit,\unit)$. 
This quotient is in canonical bijection with the subset $\Map_\calC^{\rm red}(x, y) \subset \Map_\calC(x,y)$ of reduced morphisms of \cite[Definition 2.3]{Jan-tropical}, i.e.~those morphisms $f\colon x \to y$ for which the legs in the cospan $\pi(x) \to \pi(f) \leftarrow \pi(y)$ are jointly surjective.
This defines a $2$-functor
\[
    (-)^{\proj}\colon \LCC \too \mrm{SMC}_{/\Csp^\proj}.
\]
In the case of $h(\Csp)$ this recovers the symmetric monoidal $1$-category $\Csp^{\proj}$ from \cref{example:projective-cospans}, which is equivalent to $\Red{\Csp}$ by \cref{cor:Csp-proj}.
(More generally, one gets $h\Red{\calP} \simeq (h\calP)^\proj$ for all $\infty$-properads $\calP$.)
\begin{rem}\label{rem:proj-description}
    An alternative, more $\infty$-categorical, description of $(-)^\proj$ is as the composite 
    \[
        \LCC^{2\simeq} \hookrightarrow \SMover{h\Csp} 
        \too \SMover{\Red{\Csp}} 
        \too \SMeq{\Red{\Csp}} 
    \]
    where the first functor comes from \cref{hyp:SMC} below, the second functor composes with $h\Csp \to h\Red{\Csp} \simeq \Red{\Csp}$, and the third functor is the left adjoint to the inclusion of the equifibered slice.
    To see this, it suffices to argue that $\calC \to \calC^\proj \to \Csp^\proj = \Red{\Csp}$ is a contrafibered-equifibered factorization for every labelled cospan category $\calC$.
    The first functor is contrafibered via \cref{cor:detect-ctf} as on the $n$th level of the nerve it is a projection $\xF(X) \oplus \End_\calC(\unit)^{\times n}\to \xF(X)$ for some $1$-type $X$ and this is contrafibered because it is the sum of an equivalence and a map to $0$.
    The second functor is equifibered essentially by the definition of labelled cospan categories.
\end{rem}
We can now describe $\LCC$ as an analogous pullback to \cref{cor:Prpd1-as-pullback}.

\begin{prop}\label{prop:LCC-as-pullback}
    The $2$-category of labelled cospan categories fits into a homotopy pullback square of $2$-categories
    \[\begin{tikzcd}[column sep = large]
	\LCC & {\Ar(\Sets)} \\
	{\mrm{SMC}_{/\Csp^\proj}^\eqf} & \Sets
	\arrow["\pi_0 |\Red{\calF}(-)|"', from=2-1, to=2-2]
	\arrow["\Psi", from=1-1, to=1-2]
	\arrow["{(-)^\proj}"', from=1-1, to=2-1]
	\arrow["{\mathrm{ev}_0}", from=1-2, to=2-2]
    \arrow["\lrcorner", phantom, very near start, from=1-1, to=2-2]
    \end{tikzcd}\]
    where $\Psi$ sends $\calC \in \LCC$ to the natural map
    \[
        \mrm{comp}_\calC \colon \pi_0|\Red{\calF}(\calC^\proj)| \too \End_\calC^\el(\unit),
        \qquad [f,g] \mapsto \widetilde{g}\circ \widetilde{f}
    \]
    where $\End_\calC^\el(\unit)$ is the generating set of the endomorphisms of the unit and $\widetilde{f},\widetilde{g}$ are lifts of $f,g$ to $\calC$ that are reduced morphisms in the sense discussed above.
\end{prop}

In the proof of \cref{prop:LCC-as-pullback} we will use the following notion.
For $M$ a \emph{discrete} commutative monoid, an $M$-valued cocycle $\alpha$ on symmetric monoidal $1$-category $\calC$ is a map $\alpha\colon \pi_0(\xN_2\calC) \to M$ that takes as input isomorphism classes%
\footnote{
    The assumption that $\alpha$ is well-defined on isomorphism classes of tuples is a simplification that will suffice for our intended application.
    In principle, one might want to drop this assumption such that $\alpha$ can restrict to non-trivial group cocycles on $\Aut_\calC(x)$ for $x \in \calC$.
}
of tuples of composable morphisms and satisfies
\begin{enumerate}[(1)]
    \item $\alpha(f,g) + \alpha(g\circ f, h) = \alpha(f, h \circ g) + \alpha(g,h)$ for all $(f,g,h) \in \xN_3\calC$, and
    \item $\alpha(f_1 \otimes f_2, g_1 \otimes g_2) = \alpha(f_1, g_1) + \alpha(f_2, g_2)$ for all $(f_1,g_1), (f_2,g_2) \in \xN_2\calC$.
\end{enumerate}
There is a discrete $\bbN$-valued cocycle $\gamma$ on $\Csp^\proj$ defined by
\[
    \gamma( A \to X \leftarrow B, B \to Y \leftarrow C) \coloneq | (X \cup_B Y) \setminus \mrm{Image}(A \sqcup C \to X \cup_B Y)| \in \bbN
\]
which measures the difference between composing in $h\Csp$ and in $\Csp^\proj$.

\begin{proof}[Proof of \cref{prop:LCC-as-pullback}]
    Let $\mcal{P}$ denote the $2$-category obtained as the  strict pullback.
    This strict pullback is equivalent to the homotopy pullback in Lack's model structure on $2$-categories \cite{Lack2002} since $\ev_0\colon \Ar(\Sets) \to \Sets$, being a cartesian (and in fact cocartesian) fibration of $1$-categories, has lifts for equivalences and therefore is a fibration in the model structure.

    An object in $\mcal{P}$ may be described as a triple $(\pi\colon \calC \to \Csp^\proj, E, a\colon \pi_0|\Red{\calF}(\calC)| \to E)$
    where $\calC$ is an equifibered symmetric monoidal $1$-category over $\Csp^\proj$, $E$ is a set, and $a$ a map.
    The morphisms in $\mcal{P}$ are pairs of symmetric monoidal functors $F\colon \calC \to \calC'$ such that $\pi'\circ F \cong \pi$ 
    (unique if it exists, as before)
    and maps of sets $f\colon E \to E'$ satisfying $f\circ a = a' \circ \Red{\calF}(F)$.
    The $2$-morphisms $(F,f) \to (G,g)$ only exist if $f=g$, and then they are those symmetric monoidal natural transformations $\rho\colon F \Rightarrow G$ such that $\id_\pi\circ \rho\colon \pi\circ F \Rightarrow \pi\circ G$ is an isomorphism.

    The square yields a functor $\LCC \to \mcal{P}$, and we would like to construct an inverse $\Phi\colon \mcal{P} \to \LCC$.
    For each object $(\calC, E, a) \in \mcal{P}$ we now define $\calL = \Phi(\calC, E, a) \in \LCC$ as follows.
    One can show that every map $a \colon \pi_0|\Red{\calF}(\calC)| \to E$ uniquely extends to an $\bbN\langle E \rangle$-valued cocycle $\alpha$ that makes the squares
    \[\begin{tikzcd}
        \Red{\calF}(\calC)^\simeq \ar[d]  \ar[r, "a"] & 
        E  \ar[d, hook] \\
        \xN_2 \calC \ar[r, "\alpha"] & 
        \bbN\langle E \rangle
    \end{tikzcd}
    \qquad \text{ and } \qquad
    \begin{tikzcd}
        \xN_2 \calC \ar[d, "\pi"']  \ar[r, "\alpha"] & 
        \bbN\langle E \rangle \ar[d, "\nabla"] \\
        \xN_2 \Csp^\proj \ar[r, "\gamma"] & 
        \bbN
    \end{tikzcd}\]
    commute.
    Indeed, we can uniquely decompose any tuple $(f\colon x \to y, g\colon y \to z)$ under $\otimes$ into indecomposable pieces, using that $\xN_2 \Csp^\proj$ is a free commutative monoid.
    For indecomposable tuples either the cospan $(\pi x \to \pi f \cup_{\pi y} \pi g \leftarrow \pi z)$ is such that the maps are jointly surjective and the right square forces $\alpha(f,g) = 0$ (because $\gamma(\pi f, \pi g) = 0$), or we have $x = \unit = z$ so that $\alpha$ is determined by $a$ through the left square.

    The objects of $\calL = \Phi(\calC, E, a)$ are those of $\calC$ and the mapping sets are
    \[
        \Map_\calL(x, y) \coloneq \Map_\calC(x, y) \times \bbN\langle E \rangle.
    \]
    The composite of two morphisms $(f,e) \colon x \to y$ and $(g,e') \colon y \to z$ is defined as
    \[
        (g,e') \circ (f,e) \coloneq (g\circ f, e'+e+\alpha(f,g)).
    \]
    This is a well-defined symmetric monoidal $1$-category because $\alpha$ is a cocycle.
    It has a map to $\Phi(\Csp^\proj, *, \gamma) \cong h(\Csp)$ because we required that $\nabla \circ \alpha = \gamma \circ \pi$ and this map exhibits it as an object of $\LCC$.
    The construction of $\Phi$ defines a $1$-category enriched functor $\mcal{P} \to \LCC$:
    a symmetric monoidal functor $F\colon \calC \to \calC'$ and compatible map $f\colon E \to E'$ induce a symmetric monoidal functor $\Phi(\calC, E, a) \to \Phi(\calC', E', a')$, which on mapping sets is given by $F(-) \times \bbN\langle f \rangle$,
    and similarly symmetric monoidal natural transformations  $\psi \colon F \Rightarrow G$ (over $\Csp^\proj$) induce symmetric monoidal natural transformations $\Phi(\psi)\colon \Phi(F,f) \to \Phi(G,g)$ with components $\Phi(\psi)_c \coloneq (\psi_c, 0)$.
    Now the claim follows by observing that $\Phi$ is indeed inverse, up to natural isomorphism, to the functor $\LCC \to \mcal{P}$.
\end{proof}

In their paper \cite{BH22} Beardsley and Hackney define a $2$-category $\mrm{PRPD}$ (denoted $\mathbf{Ppd}$ in their paper) whose objects and morphisms are ``classical'' properads and properad maps, and where $2$-morphisms are defined via a suitable notion of natural transformation introduced in \cite[Definition 6.1]{BH22}.
Their main theorem relates this to the $2$-category $\LCC$.
\begin{thm}[\cite{BH22}]
    There is a biequivalence $\mrm{PRPD} \simeq \LCC$.
\end{thm}

If we ignore non-invertible $2$-morphisms, the combination of \cref{cor:Prpd1-as-pullback} and \cref{prop:LCC-as-pullback} yields a theorem similar this one, except with $\mrm{PRPD}$ replaced by $\mrm{Prpd}_1$, which was defined as the full subcategory of $\Prpd$ on the $1$-properads.
For any $2$-category $C$ let $C^{2\simeq} \subset C$ denote the sub-$2$-category that contains all objects and morphisms, but only the invertible $2$-morphisms.
Moreover, we can interpret each $(2,1)$-category as an $\infty$-category whose mapping spaces are $1$-types, for example via the Duskin nerve \cite[009P]{Kerodon}. 
We will work under the following hypothesis. This seems to be generally accepted, but we do not know of a proof in the literature, and we will not attempt one here. 
(See e.g.~\cite[p.484-485]{TV15} for the statement on the level of homotopy categories.)
\begin{hyp}\label{hyp:SMC}
    The construction in \cite[Example 2.1.1.5]{HA} extends to a fully faithful functor
    \[
        \mrm{SMC}^{2\simeq} \hookrightarrow \SM
    \]
    of $\infty$-categories whose essential image are those symmetric monoidal $\infty$-categories whose underlying $\infty$-categories are (equivalent to) $1$-categories.
\end{hyp}

Assuming this, we can slice the functor over $h\Csp$ to get a fully faithful functor
\[
    \LCC^{2\simeq} \hookrightarrow \SMover{h\Csp}
\]
whose essential image consists of those $\calC \to h\Csp$ for which $\calC$ is a $1$-category and satisfies \cite[Definition 2.4]{Jan-tropical}.
\cref{hyp:SMC} also implies that applying $(-)^{\simeq 2}$ to the square in \cref{prop:LCC-as-pullback} yields the same square as \cref{cor:Prpd1-as-pullback} and thus $\LCC^{\simeq 2} \simeq \Prpdone$.
Combining this with the main theorem of \cite{BH22} we get:
\begin{cor}\label{cor:prpd-classical}
    There are equivalences of $(2,1)$-categories
    \[
        \Prpdone \simeq \LCC^{2\simeq} \simeq  \mrm{PRPD}^{2\simeq}.
    \]
\end{cor}
This shows that our notion of $1$-properads as a full subcategory of $\Prpd$ agrees with the classical definition of coloured properads.


\printbibliography[heading=bibintoc]

\end{document}